\documentclass[11pt]{amsart}
\usepackage{tikz}
\usetikzlibrary{arrows.meta, decorations.pathmorphing}
\usepackage{amssymb}
\usepackage{amsmath, amsfonts, vmargin}
\usepackage{color}
\usepackage{amsthm}
\usepackage[colorlinks=true,citecolor=blue,linkcolor=blue]{hyperref}
\usepackage{bbm}
\usepackage{verbatim}
\usepackage{mathrsfs}
\usepackage{txfonts}
\usepackage{enumerate}
\usepackage{cleveref}
\usepackage{cite}
\usepackage{array}
\usepackage{multirow}
\usepackage{booktabs}
\usepackage{mathtools}

\usepackage[T1]{fontenc}
\usepackage[english]{babel}
\usepackage{dsfont}
\usepackage{ytableau}
\usepackage{float}
\usepackage{esint}

\newtheorem{theorem}{Theorem}
\newtheorem{proposition}[theorem]{Proposition}

\newtheorem{corollary}[theorem]{Corollary}
\newtheorem{lemma}[theorem]{Lemma}

\newtheorem{remark}[theorem]{Remark}
\newtheorem{definition}[theorem]{Definition}
\theoremstyle{definition}

\newtheorem{question}[theorem]{Question}
\newtheorem{example}[theorem]{Example}

\numberwithin{theorem}{section}
\numberwithin{theorem}{section}
\numberwithin{equation}{section}

\newenvironment{rk}{\begin{remark}\rm}{\end{remark}}

\def\exp{\operatorname{\rm exp}}

\def\Tr{{\rm Tr}}

\newcommand{\dbar}{{\mkern3mu\mathchar'26\mkern-10mu d}}

\newcommand{\dint}{\mathchoice{{\vbox{\hbox{$\textstyle-$}}\kern-7.5pt}}
    {{\vbox{\hbox{$\scriptstyle-$}}\kern-5.25pt}}
    {{\vbox{\hbox{$\scriptscriptstyle-$}}\kern-3.75pt}}
    {{\vbox{\hbox{$\scriptscriptstyle-$}}\kern-3.0pt}}\!\int}
\newcolumntype{C}[1]{>{\centering\arraybackslash}p{#1}}

\newcommand{\ncint}{\fint}

\def\Op{\mathrm{Op}}

\begin{document}

\title{Weyl's laws and Connes' Trace Theorem for operator-valued pseudo-differential operators}

\author{Edward McDonald}
\address{Mathematical Institute of the University of Bonn, 53115 Bonn, Germany}
\email{eamcd92@gmail.com}

\author{Xiao Xiong}
\address{Institute for Advanced Study in Mathematics, Harbin Institute of Technology, 150001 Harbin, China}
\email{xxiong@hit.edu.cn}

\author{Xinyu Zhang}
\address{Institute for Advanced Study in Mathematics, Harbin Institute of Technology, 150001 Harbin, China}
\email{xy.zhang@stu.hit.edu.cn}



\begin{abstract}
 We investigate the spectral asymptotic behavior of operator-valued classical pseudo-differential operators ($\Psi$DOs) for negative order with symbols taking values in a semifinite von Neumann algebran $\mathcal{M}$ equipped with a normal semifinite  faithful trace. Within the framework of Connes' noncommutative geometry, we extend Connes' trace theorem to this operator-valued (type II) setting. Our main results are as follows:
(i) a symbolic characterization of complex powers for operator-valued elliptic $\Psi$DOs, extending Seeley's classical construction;
(ii) a trace formula for localized Riemann $\zeta$-functions that links the spectral residues of operator-valued elliptic operators to their principal symbols, thereby providing an operator-valued extension of the Connes--Wodzicki residue;
(iii) Weyl's law for right-compactly supported operator-valued classical $\Psi$DOs of arbitrary negative order, which yields a direct spectral proof of the noncommutative integral that bypasses the use of Dixmier traces;
(iv) Weyl's law for operator-valued commutators of certain Fourier multipliers with multiplication operators.
\end{abstract}

\maketitle

\markboth{E. McDonald, X. Xiong and X. Zhang}%
{Weyl's laws and Connes' Trace Theorem for $\Psi$DO}

\tableofcontents

\setcounter{section}{-1}
\section{Introduction}

\subsection{Pseudo-differential operators}

Pseudo-differential operators (henceforth: $\Psi$DOs) were first explicitly introduced by Kohn-Nirenberg \cite{KN65} and H\"{o}rmander \cite{L.H65} to bridge singular integrals and differential operators.  Further foundational treatments appear in \cite{RT10, St93}. These operators can be regarded as natural generalizations of Fourier multipliers and are central to modern analysis and the theory of partial differential equations.

The definition of a $\Psi$DO proceeds as follows. Let $m, \rho, \delta$ be real numbers satisfying $0 \le \delta \le \rho \le 1$, excluding the case $(\rho, \delta) = (1,1)$. The symbol class $S_{\rho, \delta}^m\big(\mathbb{R}^d \times \mathbb{R}^d\big)$ is defined as the set of functions $\sigma \in C^{\infty}(\mathbb{R}^d \times \mathbb{R}^d)$ such that for all multi-indices $\alpha$ and $\beta$, there exists a constant $C_{\alpha, \beta}$ satisfying
$$
|\partial_{\xi}^\alpha \partial_x^\beta \sigma(x, \xi)| \le C_{\alpha, \beta}(1+|\xi|^2)^{\frac{m-\rho|\alpha|+\delta|\beta|}{2}}.
$$
Subsequently, the $\Psi$DO $\Op(\sigma)$ associated with $\sigma(x, \xi) \in S_{\rho, \delta}^m\big(\mathbb{R}^d \times \mathbb{R}^d\big)$ is defined as the linear operator on the Schwartz space $\mathscr{S}(\mathbb{R}^d)$ given by the integral formula
\begin{equation}\label{scalar pdo}
	\left(\Op(\sigma)f\right)(x)= \int_{\mathbb{R}^d}  \sigma(x, \xi) \hat{f}(\xi)e^{ ix \cdot \xi}  \bar{d} \xi, \quad x \in \mathbb{R}^d, f \in \mathscr{S}(\mathbb{R}^d),
\end{equation}
where $\hat{f}$ denotes the Fourier transform of $f$, and $\bar{d} \xi = (2 \pi)^{-d} d\xi$. It is well-known that $\Op(\sigma)$ defines a linear map from $\mathscr{S}(\mathbb{R}^d)$ to itself (\cite[Theorem 2.1.6]{RT10}). We denote by $\Psi_{\rho, \delta}^m(\mathbb{R}^d):=\{\Op(\sigma): \sigma \in S^m_{\rho,\delta}\big(\mathbb{R}^d \times \mathbb{R}^d\big)\}$ the class of $\Psi$DOs of order $m$.

\medskip

In many applications, it is necessary to allow the symbols $\sigma \in S_{\rho,\delta}^m$ to be matrix-valued, say $\sigma \in C^\infty(\mathbb{R}^d\times \mathbb{R}^d;M_N(\mathbb{C}))$ for some $N,$ and in this case $\Op(\sigma)$ is an operator on the matrix-valued Schwartz space $\mathscr{S}(\mathbb{R}^d,\mathbb{C}^N).$ The fact that the space is finite-dimensional allows most results to carry over from the scalar-valued to the matrix-valued case, with a cost of at most some $N$-dependent constant. Another generalization that sometimes occurs is that $\sigma$ should depend on some parameter $\omega\in \Omega$, where $\Omega$ is a set, we require that the functions $(x,\xi)\mapsto \sigma_{\omega}(x,\xi)$ belong to $S^{m}_{\rho,\delta}(\mathbb{R}^d\times \mathbb{R}^d),$ uniformly in $\omega$. Thus, it is natural to consider a mutual extension of these two extensions of pseudo-differential calculus: the case where $\sigma$ is valued in a semifinite von Neumann algebra $\mathcal{M}$. In this case, the corresponding operator $\Op(\sigma)$ will be affiliated to the von Neumann algebra tensor product $B(L_2(\mathbb{R}^d))\overline{\otimes} \mathcal{M}$, and its spectrum will be measured as the generalized singular value function---commonly referred
to as the singular value function in semifinite von Neumann algebras.

\medskip

Indeed, the study of $\Psi$DOs with symbols valued in operator algebras (more precisely, $C^*$-algebras) goes back at least to the work of Connes and Baaj \cite{AC80, Baaj88} on pseudo-differential calculus for $C^*$-dynamical systems. Their work aimed to generalize the Atiyah-Singer index theorem \cite{AS63} to Lie group actions on $C^*$-algebras. Approximately two decades later, Connes and Tretkoff \cite{CT11} employed this calculus to establish a version of the Gauss-Bonnet theorem for noncommutative 2-tori. This work initiated a significant wave of research applying pseudo-differential techniques to the differential geometry of noncommutative tori, including the development and analysis of noncommutative residues, $\zeta$-functions, and log-determinants of elliptic operators \cite{CT11, LM16,CF19,AC21}. Subsequent studies \cite{Tao18, LP20} provided detailed accounts of symbol calculus for $\Psi$DOs on noncommutative tori, and a weakly parametric pseudo-differential calculus for a twisted $C^*$-dynamical system was developed in \cite{LL25}.

Since the theory of von Neumann algebras generalizes classical measure theory to a more general setting, replacing the theory of $\Psi$DOs valued in $C^*$-algebras with that valued in von Neumann algebras facilitates the study of the mapping properties of $\Psi$DOs on noncommutative $L_p$ spaces and some of their generalizations such as Sobolev and Besov spaces. This is the central focus of works such as \cite{XX19, HLP19i, HLP19ii, GJP21}, where the authors have obtained the boundedness of $\Psi$DOs on function spaces in some noncommutative settings such as quantum tori and quantum euclidean spaces.

\medskip

This paper is devoted to the study of the spectral asymptotic properties of operator-valued $\Psi$DOs. We choose to work with $\Psi$DOs valued in von Neumann algebras. One reason is that,  they are mutual extensions of the above mentioned two extensions of $\Psi$DOs: the case where $\sigma$ is matrix-valued, and the case where $\sigma$ depends on a parameter $\omega\in \Omega$. Another reason, which is more crucial, is that, studying operators $\Op(\sigma)$ affiliated to the von Neumann algebra $B(L_2(\mathbb{R}^d))\otimes \mathcal{M}$ makes it easier to capture the generalized singular value function of the operators, thereby gaining a clearer understanding of their spectral asymptotic behavior.


\subsection{Noncommutative geometry}

The quantized calculus, introduced by Connes in \cite{AC85}, serves a replacement for the algebra of differential forms for applications in a noncommutative setting. Its relationship with the action functional of Yang-Mills theory was
subsequently investigated in \cite{AC88}. This calculus has been notably applied to derive formulas for the Hausdorff measure of Julia sets and for limit sets of quasi-Fuchsian groups in the plane \cite[Chapter 4, Section 3.$\gamma$]{AC94}; for a more recent exposition see \cite{CMSZ17,CSZ17}.

In the framework of \cite{AC85}, quantized calculus is built upon the notion of a Fredholm module, which in turn originates from Atiyah's work on $K$-homology \cite{A69} (see also \cite[Chapter 8]{HR00}). A Fredholm module consists of a separable Hilbert space $H$, a $C^*$-algebra $\mathcal{A}$ represented on $H$, and a unitary self-adjoint operator on $H$ such that the commutator $[F, a]$ is a compact operator on $H$ for every $a\in\mathcal{A}$. The quantized differential of an element $a\in\mathcal{A}$ is then defined by $\dbar a=i[F,a]$.
According to Connes, compact operators on $H$ serve as noncommutative analogues of infinitesimals, and the ``size'' of an infinitesimal $T$ is captured by the decay rate of its singular value sequence:
\[
\mu(n,T) := \inf\{\|T-R\|_{\infty}: \operatorname{rank}(R)\le n\},
\]
see \cite{AC95}. In this setting one can quantify the smoothness of an element $a \in \mathcal{A}$ in terms of the decay rate of $\{\mu(n, \dbar a)\}_{n=0}^{\infty}$. In operator theoretic language, if $\mu(n,\dbar a)  = O\big((n+1)^{-1/p}\big)$ when $ n \to \infty$, we say that $\dbar a$ belongs to the weak Schatten ideal $\mathcal{S}_{p,\infty}$; if $\sum_{n=0}^\infty \mu(n,\dbar a)^p  < \infty$, we say that $\dbar a$ belongs to the Schatten ideal $\mathcal{S}_{p}$; if $\sup_{n \geq 1} \frac{1}{\log(n+2)} \sum_{j=0}^n \mu(j,\dbar a)^p  < \infty$, we say that $\dbar a$ belongs to the Macaev-Dixmier ideal $\mathcal{M}_{1,\infty}$. A problem of particular interest in quantized calculus is to precisely quantify the Schatten (or Macaev-Dixmier) ideal properties of $\dbar a$ in terms of $a$.

In this context, integrals are implemented by positive (normalized) traces on the weak trace class ideal $\mathcal{S}_{1,\infty}$. A trace $\varphi : \mathcal{S}_{1,\infty} \to \mathbb{C}$ is a continuous linear functional invariant under unitary conjugation: $\varphi(UBU^*) = \varphi(B)$ for all unitaries $U$ and all $B \in \mathcal{S}_{1,\infty}$; equivalently, $\varphi(AB)=\varphi(BA)$ whenever $A \in \mathcal{S}_\infty$ and $B \in \mathcal{S}_{1,\infty}$. It is normalized if $\varphi\big(\operatorname{diag}\{(n+1)^{-1}\}_{n=0}^\infty\big) = 1$, and positive if $\varphi(A) \ge 0$ for all $A \in \mathcal{S}_{1,\infty}$ with $A \ge 0$. The most famous traces on $\mathcal{S}_{1, \infty}$ are the Dixmier traces $\operatorname{Tr}_\omega: \mathcal{S}_{1, \infty} \rightarrow \mathbb{C}$ (see \cite{D66}; see also \cite{AC94,LSZ12,LSZ19,LSZ21,P23,MP23}). Recall that, for $A\in \mathcal{S}_{1,\infty}$, the Dixmier trace $\operatorname{Tr}_\omega$ determined by the extended limit $\omega$ on $\ell_{\infty}$ is defined as
$$\operatorname{Tr}_\omega(A):=	\omega\Big(\big\{\frac{1}{\log (2+n)} \sum_{j=0}^n \mu(j, A)\big\}_{n=0}^{\infty}\Big).$$

The geometric interpretation of quantized calculus is elaborated in \cite{AC88}. A paradigmatic example arises from a compact $d$-dimensional Riemannian spin manifold $M$ ($d \geq 2$) equipped with its Dirac operator $D$.  Let $H$ denote the Hilbert space of square-integrable sections of the spinor bundle over $M$, modulo almost-everywhere equivalence. The algebra $\mathcal{A} = C(M)$ acts on $H$ by pointwise multiplication, and one defines $F := \chi_{[0,\infty)}(D) - \chi_{(-\infty,0)}(D)$, the difference of spectral projections of $D$. Then $\dbar f = i[F,M_f]$, where $M_f$ denotes the operator on $H$ of pointwise multiplication by $f \in C(M)$.
A fundamental problem in quantized calculus is to relate the degree of differentiability of $f \in C(M)$ to the decay rate of the singular values of $\dbar f$. Connes established the inclusion \cite[Theorem 3(1)]{AC88}:
\[
f\in C^\infty(M) \;\Longrightarrow\; |\dbar f|^d \in \mathcal{M}_{1,\infty}.
\]

On $\mathbb{R}^d$, the Schatten class membership of $\dbar f$ has been characterized. Janson-Wolff \cite{JW82} proved that, for $d<p<\infty$, $\dbar f\in\mathcal{S}_p$ if and only if $f\in L_{1,\mathrm{loc}}(\mathbb{R}^d)\cap B^{d/p}_{p,p}(\mathbb{R}^d)$, where $B^{d/p}_{p,p}$ denotes the Besov space. In the Appendix of Connes, Sullivan and Teleman's paper \cite[Page 679]{CST94}, a sketch is provided showing that $\dbar f \in \mathcal{S}_{d,\infty}$ if and only if $f \in L_{1,\mathrm{loc}}(\mathbb{R}^d)$ and $\nabla f \in L_d(\mathbb{R}^d)$; a proof using functional methods was later supplied in \cite{LMSZ17}. On fully noncommutative manifolds, characterizations of of the $\mathcal{S}_{d,\infty}$-ideal membership of quantized differentials via homogeneous Sobolev spaces have been established, for noncommutative $d$ dimensional tori \cite{MSX19}, and for noncommutative $d$ dimensional euclidean spaces \cite{MSX20}.

\medskip

Connes' trace theorem \cite{AC88}
states that every classical $\Psi$DO $P$ of order $-d$ on a compact manifold $M$ is measurable, and its noncommutative integral coincides with the Wodzicki residue:
\begin{equation}\label{connes_integral_Wod}
	\ncint P=\operatorname{Tr}_\omega(P)=\frac{1}{d(2 \pi)^d} \iint_{S^* M} \sigma(P)_{-d}(x, \xi) d x d \xi,
\end{equation}
where $\sigma(P)_{-d}$ is the principal symbol of $P$ (see Section \ref{section resol} below for concrete definition).  As a direct consequence, \cite[Theorem 3.3]{AC88} yields the geometric formula
\begin{equation}\label{connes_integral_formula}
\operatorname{Tr}_\omega(|\dbar f|^d) = c_d \int_M |df\wedge\star df|^{d/2},
\end{equation}
with $c_d$ a known constant, $d$ the exterior derivative, and $\star$ the Hodge star operator associated to the orientation of $M$. In the words of Connes, this formula ``shows how to pass from quantized 1-forms to ordinary forms, not by a classical limit, but by a direct application of the Dixmier trace'' \cite[Page 676]{AC88}.

It is observed by Connes that the formula \eqref{connes_integral_Wod} does not depend on the extended limit $\omega$. Motivated by this property, an operator $A \in \mathcal{S}_{1, \infty}$ is said to be measurable if the value of $\operatorname{Tr}_\omega(A)$ is independent of the choice of the Dixmier trace. Equivalently (see \cite{LSZ12,P23,MP23}), the operator $A$ is measurable if and only if it satisfies the following Tauberian condition,
$$
\ncint A:=\lim _{N \rightarrow \infty} \frac{1}{\log (N+1)} \sum_{j<N} \mu(j, A)\;\; \text { exists. }
$$
The limit $\ncint A$ is then called the noncommutative integral of $A$.


While Connes' trace theorem guarantees the independent of the choice of the extended limit $\omega$
for classical $\Psi$DOs of order $-d$, the definition remains non-constructive, relying on the existence of Dixmier traces. This limitation prompted Connes to formulate the following foundational question. During the conference ``Noncommutative geometry: state of the arts and future prospects'', held in honor of Alain Connes' 70th birthday at Fudan University in Shanghai, China, from March 29 to April 4, 2017, Alain Connes posed the following question:

\begin{question}[Connes]\label{Q-Connes}
	Is it possible to prove directly an asymptotic formula for the spectrums of a classical $\Psi$DO of order $-d$, which would allow one to deduce the integration theorem without employing ultrafilters (or extended limits, Banach limits, or similar tools)?
\end{question}

As it turns out there are numerous positive traces on $\mathcal{S}_{1, \infty}$ that are not Dixmier traces. In order to overcome this difficulty and answer Connes' Problem, \cite{SSUZ15} introduces a stronger notion of measurability, called $\mathcal{P T}$ measurability, by requiring the uniqueness of the trace value across all continuous normalized traces. It is proved in \cite[Proposition 7.2]{SSUZ15} that a positive operator $X \in \mathcal{S}_{1, \infty}$ is $\mathcal{P T}$ measurable with $\varphi(X)=c$ for all positive normalised traces $\varphi$ if and only if
$$
\lim _{n \rightarrow \infty} \frac{1}{n+1} \int_{2^m}^{2^{n+m+1}-2} \mu(t, X) \mathrm{d} t=c \cdot \log (2), \quad \text { uniformly in } m \geq 0.
$$

The $\mathcal{P T}$ measurability is by now well established for broad classes of operators arising in commutative and noncommutative settings, so that Connes' trace theorem holds for these operators. For instance, $M_f \Delta_g^{-\frac{d}{2}}$ for all $f \in L_2(M)$ \cite{KLPS13}, $\Delta_g^{-\frac{d}{4}} M_f \Delta_g^{-\frac{d}{4}}$ for any $f$ in the Orlicz space $L \log L(M)$ \cite{P22,R22,FZ23}, are $\mathcal{P T}$ measurable operators. It is also shown in \cite{LMSZ17} that for $f$ belongs to the homogeneous Sobolev space $\dot{W}_d^1(\mathbb{R}^d)$, the quantized derivative $|\dbar f|^{d}$ on $\mathbb{R}^d$ are n $\mathcal{P T}$ measurable. These results have also been partly extended to noncommutative and curved noncommutative tori \cite{MSZ19, MSX19, MSX20, P20ii, P20i, MP22, MP23}, and noncommutative euclidean spaces \cite{MSX20, MSX23}.

It is also worth mentioning that, in \cite{FZ18, MSZ19}, the authors develop a $C^*$-algebraic method to address Connes' trace theorem. Specifically, using representations $\pi_1: \mathcal{A}_1 \rightarrow B(H)$ and $\pi_2: \mathcal{A}_2 \rightarrow B(H)$, they show that a $C^*$-algebra homomorphism
$$
\operatorname{sym}: \Pi\left(\mathcal{A}_1, \mathcal{A}_2\right) \rightarrow \mathcal{A}_1 \otimes_{\min } \mathcal{A}_2
$$
exists, where $\Pi\left(\mathcal{A}_1, \mathcal{A}_2\right)$ is the $C^*$-algebra generated by $\pi_1\left(\mathcal{A}_1\right)$ and $\pi_2\left(\mathcal{A}_2\right)$. This approach enables them to study operators with non-smooth symbols and to derive Connes' trace theorem for $T(1-\Delta)^{-\frac{d}{2}}$ on noncommutative tori, commutative and noncommutative Euclidean spaces, and the Lie group $SU(2)$. Here $T \in \Pi(\mathcal{A}_1, \mathcal{A}_2)$ satisfies a compact support condition specifically in these cases.


The definitive resolution of Connes' Problem \ref{Q-Connes} lies in the spectral asymptotic approach pioneered by Birman and Solomyak \cite{BS70,BS77i,BS77,BS79,BS79a}. This approach culminates in a Weyl's law for classical $\Psi$DOs of negative order on closed manifolds and on $\mathbb{R}^d$, provided their Schwartz kernels have compact support.
Weyl's law provides a constructive definition of the noncommutative integral. Unlike the Dixmier trace, whose definition relies on the non-constructive theory of extended limits and the Hahn--Banach theorem, Weyl's law directly expresses the integral as the limit of the scaled singular value sequence. This limit is explicitly computable from the principal symbol via a classical Lebesgue integral, thereby resolving Connes' Problem \ref{Q-Connes} by eliminating the need for any singular trace machinery.

\medskip

In this paper, we will address Connes' Problem \ref{Q-Connes} directly by developing a purely spectral approach based on Weyl's law for $\mathcal{M}$-valued classical $\Psi$DOs, which provides the desired spectral justification of the noncommutative integral. Moreover, our approach applies to $\mathcal{M}$-valued classical $\Psi$DOs of arbitrary negative order, thereby extending the scope beyond order $-d$ operators with smooth symbols.
Our result constitutes a direct resolution of Connes' Problem \ref{Q-Connes} in the type II noncommutative geometry setting  introduced in \cite{BT06, CPRS04, CPRS06i, CPRS06ii, CGRS14}. First, it provides a definition of the noncommutative integral that is entirely independent of any trace construction; the integral is simply the limit of the scaled singular value function. Second, the uniqueness of the trace value across all continuous normalized traces follows as an immediate corollary, rather than serving as a definitional premise. Consequently, Weyl's law in the type II noncommutative geometry also eliminates the need for singular traces altogether, thereby fully answering Connes' Problem \ref{Q-Connes} in the affirmative.

\medskip


\subsection{Weyl's law}
Let us turn to the introduction of Weyl's law, which describes the asymptotic behavior of the number of eigenvalues of positive operators. As mentioned in the previous section, Weyl's law for classical $\Psi$DOs of order $-d$ will provide a direct answer to Connes' Problem \ref{Q-Connes}.

Weyl's law has a long history dating back to H. Weyl \cite{W11}. On a closed Riemannian manifold $M$, Weyl's law describes the asymptotic behavior of $N(\lambda)$, the number of eigenvalues of the positive Laplace-Beltrami operator $-\Delta_g$ on $M$ that are less than $\lambda$, as follows:
$$
N(\lambda)=(2 \pi)^{-d} \omega_d \operatorname{vol}(M) \lambda^{\frac{d}{2}}(1+o(1)) \quad \text{as } \lambda \rightarrow \infty,
$$
where $\omega_d=\frac{\pi^{\frac{d}{2}}}{\Gamma\left(\frac{d}{2}+1\right)}$ denotes the volume of the unit ball in $\mathbb{R}^d$. Several modern expositions are available; for the case of Euclidean domains, see e.g., \cite[Theorem 7.5.29]{Sim15}.
The inverse of $1-\Delta_g$ is compact and  belongs to a weak Schatten class. Weyl's law can equivalently be expressed in terms of the eigenvalues of $(1-\Delta_g)^{-1}$, as
\begin{equation}\label{weyllaw}
	\lim_{j\rightarrow\infty} j^{\frac{2}{d}} \mu(j,(1-\Delta_g)^{-1}) = (2 \pi)^{-2}(\omega_d \operatorname{vol}(M))^{\frac{2}{d}},
\end{equation}
where $\mu(j,(1-\Delta_g)^{-1})$ denotes the $j$-th largest singular value of the compact operator $(1-\Delta_g)^{-1}$. This restatement of Weyl's law in \eqref{weyllaw} is a very special case of a general result concerning the asymptotic behavior of the eigenvalues of $\Psi$DOs of negative order.

While Connes' Problem \ref{Q-Connes} is formulated for classical $\Psi$DOs of order $-d$ with smooth symbols, the original insight that Weyl's law persists for general negative orders, is due to the pioneering work of Birman and Solomyak \cite{BS70,BS77i,BS77,BS79,BS79a}. Later proofs and generalizations were obtained in \cite{DR87,Shu01,P08}.
Recent progress extends this spectral asymptotic approach further, demonstrating that Weyl's law holds for operators with non-smooth symbols and in various noncommutative settings:

On $\mathbb{R}^d$, Weyl's law for $T I^{-1}$ has been established in \cite{FSZ23}, where $T \in \Pi(C_0(\mathbb{R}^d) + \mathbb{C},\, C(\mathbb{S}^{d-1}))$ (as defined in the previous section) is compactly supported from the right and $I^{-1}$ denotes the Riesz potential of order $-1$, with $d \ge 2$.  As an application, Weyl's law has been established for the quantized derivative $\dbar f$ of a function $f$ from the homogeneous Sobolev space $\dot{W}_d^1(\mathbb{R}^d)$.
Weyl's law for the commutators $[(-\Delta)^{\frac{\epsilon}{2}}, M_f]$, where $d \ge 1$, $0 < \epsilon < 1$, or $d \ge 2$, $-\frac{d}{2} < \epsilon < 0$, has also been investigated in \cite{FSZ24}.

In the noncommutative case, Weyl's law for $T I_\theta^{-1}$ on noncommutative tori $\mathbb{T}_\theta^d$ has been established in \cite{SXZ23}, where $T \in \Pi(C(\mathbb{T}_\theta^d), C(\mathbb{S}^{d-1}))$, and $I_\theta^{-1}$ is the Riesz potential of order $-1$, with $d \ge 2$. As an application, Weyl's law has been established for the quantized derivative $\dbar x$ of an operator $x$ from the homogeneous Sobolev space $\dot{W}_d^1(\mathbb{T}_\theta^d)$. Weyl's law for $T I_\theta^{-1}$ on noncommutative euclidean spaces $\mathbb{R}_\theta^d$ in the $\det(\theta) \neq 0$ case has been established in \cite{Tian25}, where $T \in \Pi(C_0(\mathbb{R}_\theta^d)+ \mathbb{C}, C(\mathbb{S}^{d-1}))$ is compactly supported from the right, and $I_\theta^{-1}$ is the Riesz potential of order $-1$, with $d \ge 2$. As an application, Weyl's law has been established for the quantized derivative $\dbar x$ of an operator $x$ from the homogeneous Sobolev space $\dot{W}_d^1(\mathbb{R}_\theta^d)$.

For the case of $p$-summable spectral triples $(\mathcal{A}, H, D)$ under certain conditions, Weyl's law for the operators $|D|^{-\frac{1}{2}} a|D|^{-\frac{1}{2}}$ \cite{MSZ22} and $|D|^{-\frac{q}{2}} a|D|^{-\frac{q}{2}}$ with $q > 0$ \cite{P23} has bellipticen established, where $a = a^* \in \overline{\mathcal{A}}$, and these operators are compact operators. As pointed in \cite{MSZ22}, the  Weyl's law for $|D|^{-\frac{1}{2}} a|D|^{-\frac{1}{2}}$ implies directly the operator $|D|^{-\frac{1}{2}} a|D|^{-\frac{1}{2}}$
is strongly measurable, and its Dixmier trace is given by the associated noncommutative integral.

Note that the previous results are situated within the framework of type I noncommutative geometry, where the primary objects are compact operators $\mathcal{K}(H)$. In this setting, the spectra are discrete, and their asymptotic behavior is governed by the singular value sequence $\mu(j, \cdot)$  $(j \in \mathbb{N})$ within the usual weak Schatten ideals $\mathcal{S}_{p,\infty}(H)$.

\medskip

This paper shifts this perspective to the broader context of type II noncommutative geometry. Here, we move beyond ordinary compact operators and consider $\operatorname{Tr}$-compact operators $\mathcal{C}_0\big(B(L_2(\mathbb{R}^d)) \overline{\otimes} \mathcal{M}\big)$ (see Definition \ref{def-tau-compact} below for concrete definition). In this framework, the spectral data is encoded not by a discrete sequence, but by a continuous singular value function $\mu(t, \cdot)$ $(t > 0)$ within the noncommutative weak Schatten ideals $\mathcal{L}_{p,\infty}\big(B(L_2(\mathbb{R}^d)) \overline{\otimes} \mathcal{M}\big)$. A central objective of this work is to establish a version of Weyl's law tailored to this type II setting for $\mathcal{M}$-valued classical $\Psi$DOs of arbitrary negative order, thereby clarifying the continuous asymptotic behavior of their singular value functions.


\subsection{Main results}

We consider classical $\Psi$DOs with operator-valued symbols, specifically those taking values in a semifinite von Neumann algebran $\mathcal{M}$ equipped with a faithful, normal, semifinite trace $\tau$. Our primary objective is to establish Weyl's law for such operators. In pursuit of this goal, we extend Connes' trace theorem and integration formula to the operator-valued setting. The main results are summarized as follows:

Our first main result establishes the principle symbol mapping for complex powers of $\mathcal{M}$-valued elliptic  $\Psi$DOs:
\begin{theorem}\label{main_complex_powers_theorem}
	Let $A \in \mathrm{C}\Psi^m\big(\mathbb{R}^d; \mathcal{M}\big)$ with $ \mathfrak{m}=\Re m>0$ be elliptic, and $A^*A$ be strictly positive. Then $|A|^{z}\in \mathrm{C}\Psi^{ \mathfrak{m} z}\big(\mathbb{R}^d; \mathcal{M}\big)$ with principal symbol
	$$
	\sigma(|A|^{z})_{ \mathfrak{m}z}(x, \xi) = |\sigma(A)_{m}(x, \xi)|^{z}, \quad z\in \mathbb{C}.
	$$
\end{theorem}
Seeley's construction of complex powers \cite{Seeley67} did not assume that the symbols were scalar-valued, and in principle they could have been matrix-valued. The value
of Theorem \ref{main_complex_powers_theorem} is that it applies in the infinite-dimensional operator-valued case. While \cite[Theorem 3.3]{MSZ19} develops the theory of principal symbols for $\Psi$DOs of order 0, the above theorem provides a standard framework for symbolic calculus applicable to $\mathcal{M}$-valued elliptic  $\Psi$DOs. It thereby establishes the symbolic foundation necessary for defining localized Riemann $\zeta$-functions associated with such operator-valued $\Psi$DOs.

Our second main result provides the analytical property of localized Riemann $\zeta$-function for $\mathcal{M}$-valued elliptic $\Psi$DOs:

\begin{theorem}
	Let $A \in \mathrm{C}\Psi^m\big(\mathbb{R}^d; \mathcal{M}\big)$ with $m > 0$ be elliptic and strictly positive, and suppose its principal symbol $\sigma(A)_m(x, \xi)$ is positive for all $(x, \xi) \in \mathbb{R}^d \times (\mathbb{R}^d \setminus \{0\})$. For $\phi \in C_c(\mathbb{R}^d; L_1(\mathcal{M}) \cap \mathcal{M})$ and $\Re z > \frac{d}{m}$, define the localized Riemann $\zeta$-function
	$$
	\zeta_{A, \phi}(z) := \operatorname{Tr}(M_{\phi^{\ast}} A^{-z} M_{\phi})
	$$
	and the symbolic $\zeta$-function
	$$
	\zeta_{\sigma, \phi}(z) :=(2 \pi)^{-d} \int_{\mathbb{R}^d} \int_{\mathbb{R}^d} \tau\Big(\phi(x)^{\ast} \sigma(A)_m(x, \xi)^{-z} \phi(x)\Big) \, dx \, d\xi.
	$$
	
	Then both functions admit meromorphic continuations to $\mathbb{C}$, with the right-most simple pole at $z = \frac{d}{m}$, and their residues agree:
	\begin{equation*}
		\begin{split}
			\operatorname{Res}_{z=\frac{d}{m}} \zeta_{A, \phi}(z) &= \operatorname{Res}_{z=\frac{d}{m}} \zeta_{\sigma, \phi}(z) \\
			& = \frac{1}{m(2 \pi)^d} \int_{\mathbb{S}^{d-1} } \int_{\mathbb{R}^d} \tau\left(\phi(x)^{\ast} \sigma(A)_m(x, \xi)^{-\frac{d}{m}} \phi(x)\right) \, dx \, d\xi.
		\end{split}
	\end{equation*}
\end{theorem}

This result establishes an operator-valued analogue of the Connes-Wodzicki residue for $\mathcal{M}$-valued elliptic $\Psi$DOs thereby creating a direct connection between the spectral asymptotics of A and the symbolic integral over the cosphere bundle. It provides an analytic framework for computing noncommutative residues in terms of principal symbols and confirms that the noncommutative integral for such operators depends solely on local geometric data, in full consistency with the scalar-valued Connes-Wodzicki residue.

Our third main result establishes the precise asymptotic behavior of singular value functions for right-compactly supported classical $\Psi$DOs with operator-valued symbols:
\begin{theorem}\label{Thm-main}
	Let $T \in \mathrm{C}\Psi^{-m}\big(\mathbb{R}^d; \mathcal{M}\big)$ with $m > 0$, and $f \in C_c(\mathbb{R}^d; L_{\frac{d}{m}}(\mathcal{M}) \cap \mathcal{M})$. Let $M_f$ be the multiplication operator induced by an operator-valued function $f$. Then $T M_f \in \mathcal{L}_{\frac{d}{m}, \infty}\big(B(L_2(\mathbb{R}^d)) \overline{\otimes} \mathcal{M}\big)$ and
	\begin{align*}
		\lim_{t \to \infty} t^{\frac{m}{d}} \mu(t, T M_f) &= (2 \pi)^{-m} \lim_{t \to \infty} t^{\frac{m}{d}} \mu(t, \sigma(T)_{-m}f) \\
		&= d^{-\frac{m}{d}} (2 \pi)^{-m} \left[ \int_{\mathbb{S}^{d-1} } \int_{\mathbb{R}^d} \tau\Big(|\sigma(T)_{-m}(x, \xi)f(x)|^{\frac{d}{m}}\Big)\, dx \, d\xi \right]^{\frac{m}{d}}.
	\end{align*}
	In particular, if $T$ is compactly supported from the right (i.e. there exists $f \in C_c(\mathbb{R}^d; L_{\frac{d}{m}}(\mathcal{M}) \cap \mathcal{M})$ such that $T = T M_f$). Then $T \in \mathcal{L}_{\frac{d}{m}, \infty}\big(B(L_2(\mathbb{R}^d)) \overline{\otimes} \mathcal{M}\big)$ and
	\begin{align*}
		\lim_{t \to \infty} t^{\frac{m}{d}} \mu(t, T) &= (2 \pi)^{-m} \lim_{t \to \infty} t^{\frac{m}{d}} \mu(t, \sigma(T)_{-m}) \\
		&= d^{-\frac{m}{d}} (2 \pi)^{-m} \left[ \int_{\mathbb{S}^{d-1} } \int_{\mathbb{R}^d} \tau\Big(|\sigma(T)_{-m}(x, \xi)|^{\frac{d}{m}}\Big)\, dx \, d\xi \right]^{\frac{m}{d}}.
	\end{align*}
\end{theorem}

This is the central spectral asymptotic result of the paper, establishing the Weyl' law for $\mathcal{M}$-valued classical $\Psi$DOs of any negative order. It resolves Connes' Problem \ref{Q-Connes} by providing an explicit spectral formula for the noncommutative integral of such operators, while avoiding ultrafilter-based constructions such as Dixmier traces.

This theorem indeed extends Connes' trace theorem to the context of type II noncommutative geometry: Let $T \in \mathrm{C}\Psi^{-d}\big(\mathbb{R}^d; \mathcal{M}\big)$ be compactly supported from the right. Denote $\mathfrak{B}= B(L_2(\mathbb{R}^d)) \overline{\otimes} \mathcal{M}$, and define Dixmier trace for $A \in \mathcal{L}_{1,\infty}(\mathfrak{B})$ as
$$\operatorname{Tr}_\omega(A) = \lim\limits_\omega \dfrac{1}{\log (s+1)} \int_0^s \mu(t,A)  dt$$
Then we have
$$
\Tr_{\omega}(T) = d^{-1}(2\pi)^{-d} \int_{\mathbb{S}^{d-1} } \int_{\mathbb{R}^d} \tau \big( \sigma(T)_{-d}(x,\xi) \big)\,dx\,d\xi.
$$
In particular, the value is independent of the choice of $\omega.$

\subsection{Applications}

\subsubsection{Weyl's law for the  commutators}

The first application presented in this paper is a systematic study of the singular value function asymptotics of commutators.

We first consider commutators of the form $[T_\phi, M_f]$, where $T_\phi$ denotes a Calder\'{o}n--Zygmund singular integral operator on $\mathbb{R}^d$ ($d\ge 2$) and $M_f$ is the multiplication operator induced by an operator-valued function $f$.
Specifically, $T_\phi$ is the Fourier multiplier with symbol $\phi \in C^\infty(\mathbb{S}^{d-1})$, homogeneous of degree $0$ and not identically zero. Such operators include, in particular, the Riesz transforms $R_j$.
For a function $f$ taking values in a semifinite von Neumann algebra $\mathcal{M}$, we consider the commutator $[T_\phi, M_f]$ acting on the Hilbert space $L_2(\mathbb{R}^d;L_2(\mathcal{M}))$ in the operator-valued setting.

\begin{theorem}
	Let $f \in \dot{W}_d^1(\mathbb{R}^d; L_d(\mathcal{M}))$ and $d\ge 2$. Then
	\begin{equation}\label{Commutator-CZ}
		\begin{split}
			&  \lim_{t \to \infty} t^{\frac{1}{d}} \mu(t, [T_\phi, M_f]) \\
			=&(2\pi)^{-1} d^{-\frac{1}{d}}  \Big(\int_{\mathbb{S}^{d-1}} \int_{\mathbb{R}^d}  \tau \big( \big| \sum_{|\gamma|=1} \partial_s^\gamma \phi(s) D_x^\gamma f(x) \big|^d \big)\, dx \, ds \Big)^{\frac{1}{d}},
		\end{split}
	\end{equation}
	where the integral over $\mathbb{S}^{d-1}$ is taken with respect to the rotation-invariant measure $ds$ on $\mathbb{S}^{d-1}$, and $s = (s_1, \ldots, s_d)$.
\end{theorem}

In the scalar case ($\mathcal{M} = \mathbb{C}$), taking $\phi(s) =s_j/|s|$, we recover the Weyl's law for the commutator $[R_j ,  M_f]$, where $R_j$ is the $j$-th Riesz transform. In this case, formula \eqref{Commutator-CZ} reduces identically to that in \cite{FSZ23}. Furthermore, from a methodological point of view, our proof is based on a $\Psi$DO approach, different from that in \cite{FSZ23}.
We first approximate the homogeneous symbol $\phi$ by a smooth symbol $\widetilde{\phi}$ that is supported away from the origin, so that $T_{\widetilde{\phi}}$ becomes a classical $\Psi$DO of order $0$.
For smooth, compactly supported operator-valued functions $f$ we compute the full symbol expansion of the commutator $[T_{\widetilde{\phi}}, M_f]$ using the asymptotic expansion for $\mathcal{M}$-valued classical $\Psi$DOs, thereby identifying its  principal symbol.

A general spectral asymptotics theorem for $\mathcal{M}$-valued classical $\Psi$DOs (Theorem~\ref{asymp most}) then directly yields the singular value asymptotics for the approximated commutator.
Approximation errors are controlled via Cwikel-type estimates and are thus negligible in the weak Schatten class $\mathcal{L}_{d,\infty}$.
Finally, a density argument together with a priori bounds (Lemma~\ref{calderon upper}) extends the result to the full Sobolev space $\dot W^1_d(\mathbb{R}^d;L_d(\mathcal{M}))$.
The whole argument is self-contained and does not rely on any $C^*$-algebraic machinery; it is entirely within the framework of $\Psi$DO calculus, which turns out to be particularly well adapted to the operator-valued setting.

It is important to emphasise that even in the scalar-valued setting ($\mathcal{M}=\mathbb{C}$, $\tau = |\cdot|$) such a Weyl's law for commutators of Calder\'on--Zygmund operators has never been obtained before.

\medskip

We will also provide the singular value function asymptotics of commutators of the form $[I^{\alpha}, M_f]$,
where $I^{\alpha}=(-\Delta)^{\frac{\alpha}{2}}$ denotes the fractional Laplacian (i.e., the Riesz potential), and $M_f$ denotes the operator of multiplication by an operator-valued function $f$. Our goal is to establish sharp endpoint Schatten class properties and to obtain the asymptotics of the singular value function of $[I^{\alpha},M_f]$ under minimal regularity assumptions on $f$.

For scalar-valued functions $f$ on $\mathbb{R}^d$ and for parameters $\alpha\in (-\frac{d}{2},0)\cup(0,1)$,  the landmark work of Frank, Sukochev, and Zanin \cite{FSZ24} fully characterized the condition for $[(-\Delta)^{\frac{\alpha}{2}}, M_f]$ to belong to the weak Schatten ideal $\mathcal{L}_{\frac{d}{1-\alpha},\infty}$: namely, that $f \in \dot{W}_{\frac{d}{1-\alpha}}^1(\mathbb{R}^d)$.
Their proof relied on a powerful combination of double operator integrals, Cwikel-type estimates and a delicate $C^*$-algebraic approach to spectral asymptotics, where the algebra $\Pi$ generated by multiplication operators and by functions of $\nabla(-\Delta)^{-\frac{1}{2}}$ is equipped with a symbol homomorphism $\operatorname{sym}:\Pi\to C(\mathbb{S}^{d-1},\mathbb{C}+C_0(\mathbb{R}^d))$.

The present work extends these results to the non-commutative setting where $f$ takes values in a semifinite von Neumann algebra $\mathcal{M}$. More precisely, we prove the following:

\begin{theorem}
	Let $d \ge 2$ and let $\alpha \in (-\frac{d}{2}, 0) \cup (0,1)$ (alternatively, let $d=1$ and $\alpha \in(0,1)$ ). If $f \in \dot{W}_{\frac{d}{1-\alpha}}^1\Big(\mathbb{R}^d; L_{\frac{d}{1-\alpha}}(\mathcal{M})\Big) $, then
	\[
	\lim_{t \to \infty} t^{\frac{1-\alpha}{d}} \mu(t, [I^{\alpha}, M_f]) = C_{d,\alpha} \Bigg( \int_{\mathbb{S}^{d-1}}\int_{\mathbb{R}^d}  \tau\Big(\left|s \cdot \nabla f(x)\right|^{\frac{d}{1-\alpha}}\Big)\, dx \, ds \Bigg)^{\frac{1-\alpha}{d}}
	\]
	with
	$$
	C_{d,\alpha}=|\alpha| d^{\frac{\alpha-1}{d}} (2\pi)^{\alpha-1}.
	$$
\end{theorem}
In the scalar case ($\mathcal{M} = \mathbb{C}$), this formula reduces identically to that in \cite{FSZ24}.
From a methodological standpoint, our approach departs fundamentally from the $C^*$-algebraic framework of \cite{FSZ24}. Instead of relying on an abstract symbol homomorphism on a $C^*$-algebra, we adopt an operator-valued classical $\Psi$DO viewpoint.
Thus, although our final results parallel those of \cite{FSZ24}, the proof is entirely independent and grounded in the well-established calculus of operator-valued classical $\Psi$DOs, which is particularly well-suited for handling operator-valued symbols and for obtaining explicit asymptotics.  The $C^*$-algebraic method of \cite{FSZ24} is elegant and powerful, but our $\Psi$DO approach offers several advantages in the present noncommutative setting: it obviates the need to construct a separate symbol homomorphism for each von Neumann algebra $\mathcal{M}$, and it naturally accommodates the n.s.f trace $\tau$ in the asymptotic formula.

To the best of our knowledge, this is the first time that Weyl's law for these commutators at the endpoint Schatten class are obtained for operator-valued functions, and the first time that a purely $\Psi$DO-based proof is given even for the scalar case, providing an alternative to the $C^*$-algebraic approach treatment.

\subsubsection{Random and almost periodic operators}

We obtain a nontrivial application of the above result in the case that $\mathcal{M}$ is abelian, namely if $\mathcal{M} = L_{\infty}(\Omega)$, where $(\Omega,\Sigma,\mathbb{P})$ is a probability space. $\Psi$DOs with $L_{\infty}(\Omega)$-valued coefficients can be interpreted as random operators. A substantial literature exists on random operators, see e.g. \cite{CarmonaLacroix}. Our emphasis is on random operators that are equivariant with respect to an ergodic transformation.

To illustrate, let $T \in \mathrm{C}\Psi^{m}(\mathbb{R}^d;L_{\infty}(\Omega))$ be elliptic and positive, where $m>0.$ For $\omega\in \Omega,$ we write $T(\omega)$
for the corresponding scalar-valued $\Psi$DO on $\mathbb{R}^d.$ The \emph{density of states} of $T$ is defined as the random function
\[
N(\lambda,T)(\omega):= \lim_{R\to\infty} \frac{1}{(2R)^d} \mathrm{tr}(M_{\chi_{[-R,R]^d}}\chi_{(-\infty,\lambda)}(T(\omega)))
\]
provided that the limit exists. If $T$ is equivariant with respect to the group $\mathbb{Z}^d$ acting by translations and also ergodically on $\Omega,$ then the limit does exist almost surely in $T,$ and is equal to the tensor product trace on $L_{\infty}(\Omega)\otimes B(L_2(\mathbb{R}^d)).$ Under these assumptions, the operator-valued Weyl's law we prove in Section \ref{section asymp} below imply asymptotic formulas for the density of states. In this example, we have
\[
N(\lambda,A) = \lambda^{\frac{d}{m}}\frac{m}{d(2\pi)^d}\int_{|\xi|=1}\int_{[0,1]^d} \mathbb{E}(\sigma(A)_{m}(x,\xi,\cdot)^{-\frac{d}{m}})\,dxd\xi+o(\lambda^{\frac{d}{m}}),\quad \lambda\to\infty.
\]
For further details and more precise statement, see Section \ref{random_section}.

A related example comes from almost-periodic operators, in the sense of Shubin \cite{Shubin1976,Shubin1979}. Formally speaking, almost-periodic operators can be embedded into the class of random operators, with $\Omega = \mathbb{R}^d_B$ being the Bohr compactification of $\mathbb{R}^d$ and $\mathbb{P}$ being the Haar measure. We are able to recover some of Shubin's Weyl's law for the density of states of almost-periodic operators.

\section{Preliminaries on noncommutative analysis}\label{sec:preliminaries}

This section systematically introduces the analytic and operator-algebraic tools used throughout the paper.\ We recall the definition and basic properties of operator-valued $\Psi$DOs, as well as noncommutative $L_p$-spaces associated with a semifinite von Neumann algebra $(\mathcal{M}, \tau)$.\ Particular emphasis is placed on singular value functions, noncommutative Lorentz spaces, and weak Schatten ideals, which form the spectral framework of type II noncommutative geometry.

\subsection{Operator-valued $\Psi$DOs}

We begin with an introduction to the notation and fundamental concepts of vector-valued Fourier analysis, specifically Fourier analysis on functions taking values in a Banach space $X$.

Let $\mathscr{S}(\mathbb{R}^d ; X)$ denote the space of $X$-valued rapidly decreasing and infinitely differentiable functions on $\mathbb{R}^d$, equipped with the standard Fr\'{e}chet topology. In particular, $\mathscr{S}(\mathbb{R}^d ; \mathbb{C})$ is simply denoted as $\mathscr{S}(\mathbb{R}^d)$.

The space $\mathscr{S}^{\prime}(\mathbb{R}^d ; X)$ consists of continuous linear maps from $\mathscr{S}(\mathbb{R}^d)$ to $X$, known as $X$-valued tempered distributions. Operations on $\mathscr{S}(\mathbb{R}^d)$ such as differentiation, convolution, and Fourier transform extend naturally to $\mathscr{S}^{\prime}(\mathbb{R}^d ; X)$.

Moreover, for $1 \le p \le \infty$, the space $L_p(\mathbb{R}^d ; X)$ of strongly $p$-integrable functions from $\mathbb{R}^d$ to $X$ embeds naturally into $\mathscr{S}^{\prime}(\mathbb{R}^d ; X)$. Consequently, Fourier transforms and Fourier multipliers on $\mathbb{R}^d$ extend to vector-valued tempered distributions in a natural manner.

Next, we introduce some typical Fourier multipliers that will be frequently used in the sequel. For a real number $\alpha$, the Bessel potential operator $J^\alpha=\left(1-\Delta\right)^{\frac{\alpha}{2}}$ is defined on $\mathscr{S}^{\prime}(\mathbb{R}^d ; X)$, where $\Delta$ denotes the Laplacian on $\mathbb{R}^d$. When $\alpha=1$, we abbreviate $J^1$ as $J$. Additionally, we denote $J_\alpha(\xi)=(1+|\xi|^2)^{\frac{\alpha}{2}}$ on $\mathbb{R}^d$, which serves as the symbol of the Fourier multiplier $J^\alpha$.

\medskip

Let us now recall the definitions of operator-valued $\Psi$DOs in \cite{XX19}.


The symbols of $\Psi$DOs considered here are $B(X)$-valued, where $X$ is a Banach space and $B(X)$ denotes the space of all bounded linear operators on $X$.

\begin{definition}\label{def of sym}
	Let $m, \rho, \delta$ be real numbers such that  $0 \le \delta \le \rho \le 1$, excluding the case $(\rho, \delta) = (1,1)$. The symbol class $S_{\rho, \delta}^m \big(\mathbb{R}^d \times \mathbb{R}^d ; B(X) \big)$ consists of functions $\sigma(x, \xi) \in C^{\infty}\big(\mathbb{R}^d \times \mathbb{R}^d ; B(X)\big)$ satisfying the following condition: for arbitrary multi-indices $\alpha$ and $\beta$, there exists a constant $C_{\alpha, \beta}$ such that
	$$
	\left\|\partial_{\xi}^\alpha \partial_x^\beta \sigma(x, \xi)\right\|_{B(X)} \le C_{\alpha, \beta}(1+|\xi|^2)^{\frac{m-\rho|\alpha|+\delta|\beta|}{2}}.
	$$
	For simplicity, we write $S^m\big(\mathbb{R}^d \times \mathbb{R}^d ; B(X)\big)$ instead of $S_{1,0}^{m}\big(\mathbb{R}^d \times \mathbb{R}^d ; B(X)\big)$, or simply $S_{\rho, \delta}^m$. We also define
	$$
	S^{-\infty}\big(\mathbb{R}^d \times \mathbb{R}^d ; B(X)\big) = \bigcap_{m \in \mathbb{R}} S_{\rho, \delta}^m\big(\mathbb{R}^d \times \mathbb{R}^d ; B(X)\big),
	$$
	where the right hand side does not depend on $\rho$ and $\delta$.
\end{definition}

\begin{definition}\label{def of pdo}
	Let $\sigma(x,\xi) \in S_{\rho, \delta }^m\big(\mathbb{R}^d\times \mathbb{R}^d; B(X)\big)$. For a function $f \in \mathscr{S}(\mathbb{R}^d;X)$, the $\Psi$DO $\Op(\sigma)$ is defined as the mapping $f \mapsto \Op(\sigma) f$, given by
	\begin{equation} \label{form}
		\Op(\sigma)f(x) = \int_{\mathbb{R}^d} \sigma(x, \xi) \hat{f} (\xi)  e^{ ix \cdot \xi}  \bar{d} \xi.
	\end{equation}
	
	The class of $B(X)$-valued $\Psi$DOs, defined by $\sigma(x, \xi) \in S_{\rho, \delta }^m\big(\mathbb{R}^d \times \mathbb{R}^d; B(X)\big)$, is denoted by $\Psi_{\rho, \delta}^m\big(\mathbb{R}^d; B(X)\big)$ or simply by $\Psi_{\rho, \delta}^m$, i.e.,
	$$
	\Psi_{\rho, \delta}^m = \mathrm{Op} S_{\rho, \delta}^m.
	$$
	We also write $\Psi^m$ instead of $\Psi_{1,0}^m$ and define $\Psi^{-\infty} = \bigcap\limits_m \Psi^m$.
\end{definition}

As in scalar-valued case, we can also define the class of narrower $B(X)$-valued $\Psi$DOs, which can be asymptotically expanded as the sum of $\xi$-homogeneous symbols. Instead of studying this class of $\Psi$DOs in $B(X)$-valued setting,  we will mainly focus on such class for $\mathcal{M}$-valued $\Psi$DOs (with $\mathcal{M}$ being a von Neumann algebra), which will be crucial to symbol calculus for $\mathcal{M}$-valued $\Psi$DOs.

\subsection{Noncommutative $L_p$ space}

Let us recall the definition and some basic properties of noncommutative $L_p$-spaces (see \cite{PX03} and \cite{Xu07} for further
information about noncommutative $L_p$-spaces).

To begin with, let $\mathcal{M}$ be a von Neumann algebra equipped with a normal, semifinite, faithful  (n.s.f.) trace $\tau$, and let $\mathcal{S}_{+}(\mathcal{M})$ be the set of all positive elements $x$ in $\mathcal{M}$ with $\tau(s(x))<\infty$, where $s(x)$ denotes the support of $x \in \mathcal{M}$, i.e., the smallest projection $e \in \mathcal{M}$ satisfying $ex=x=xe$. Let $\mathcal{S}(\mathcal{M})=\operatorname{span}\big(\mathcal{S}_{+}(\mathcal{M})\big)$, then every $\mathcal{S}(\mathcal{M})$ has finite trace, and $\mathcal{S}(\mathcal{M})$ is a $w^*$-dense $*$-ideal of $\mathcal{M}$. Given $1 \le p<\infty$, for any $x \in \mathcal{S}(\mathcal{M})$, the operator $|x|^p$ belongs to $\mathcal{S}_{+}(\mathcal{M})$, recalling that $|x|=(x^* x)^{\frac{1}{2}}$ is the modulus of $x$. We define
$$
\|x\|_{p}=\big(\tau\big(|x|^p\big)\big)^{\frac{1}{p}}, x \in \mathcal{S}(\mathcal{M}).
$$
Then $\|\cdot\|_p$ is a norm on $\mathcal{S}(\mathcal{M})$. The completion of $\big(\mathcal{S}(\mathcal{M}),\|\cdot\|_p\big)$ is called as the noncommutative $L_p$ space associated with $(\mathcal{M}, \tau)$, denoted by $L_p(\mathcal{M})$. For convenience, we define $L_{\infty}(\mathcal{M})=\mathcal{M}$ equipped with the operator norm $\|\cdot\|_{\infty}$. Similar to classical $L_p$-spaces, noncommutative $L_p$-spaces retain fundamental properties including duality and interpolation.

If $\mathcal{M}_1$ and $\mathcal{M}_2$ are two semifinite von Neumann algebras equipped with $\tau_1$ and $\tau_2$, respectively, the von Neumann tensor algebra $\mathcal{M}_1 \overline{\otimes} \mathcal{M}_2$ is equipped with the tensor trace $\tau_1 \otimes \tau_2$. We will often consider the tensor $B(H) \overline{\otimes} \mathcal{M}$ ($H$ being a Hilbert space), equipped with $\operatorname{Tr}= \operatorname{tr} \otimes \tau$.

In most part of this paper, we are interested in operator-valued functions. The involved von Neumann algebra is the semi-commutative algebra $L_{\infty}(\mathbb{R}^d) \overline{\otimes} \mathcal{M}$ with its tensor trace $\int_{\mathbb{R}^d} d x \otimes \tau$, denoted by $\mathcal{N}$ in the sequel. Notice that $L_p(\mathcal{N})$ isometrically coincides with $L_p(\mathbb{R}^d ; L_p(\mathcal{M}))$, the space of strongly $p$-integrable functions from $\mathbb{R}^d$ to $L_p(\mathcal{M})$.

\subsection{Noncommutative Lorentz space, operator ideals and traces}

The elements in $L_p(\mathcal{M})$ can be described as closed densely defined operators on a Hilbert space $H$, where $\mathcal{M}$ acts. A closed densely defined operator on $H$ is said to be affiliated with $\mathcal{M}$ if it commutes with all unitary operators in the commutant $\mathcal{M}'$ of $\mathcal{M}$. Specifically, an operator $A \in \mathcal{M}$ is called $\tau$-measurable (or simply measurable) with respect to $(\mathcal{M}, \tau)$ if for any $\delta > 0$, there exists a projection $e \in B(H)$ such that
$$
e(H) \subset \operatorname{Dom}(A) \quad \text{and} \quad \tau(e^{\perp}) \le \delta,
$$
where $\operatorname{Dom}(A)$ denotes the domain of the operator $A$.

Subsequently, we denote the set of $\tau$-measurable operators by $L_0(\mathcal{M})$. For $A \in L_0(\mathcal{M})$, the distribution function of $A$ is defined by
$$
d_s(A):= \tau(e_s^{\perp}(|A|)), \quad s > 0.
$$
Here, $e_s^{\perp}(|A|)$ denotes the spectral projection of $|A|$ associated with the interval $(s, \infty)$.
For $A \in L_0(\mathcal{M})$, the singular value function $\mu(t, A)$ is defined by
\begin{align*}
	\mu(t, A) :&= \inf \{s > 0 : d_s(A) \le t\}, \quad t > 0\\
	&=\inf \{\|A(1-\mathrm p)\|_{\infty} : \mathrm p \in \operatorname{Proj}(\mathcal{M}),\, \tau(\mathrm p) \le t\}\\
	&=\inf \{\|A-B\|_{\infty} : B\in L_0(\mathcal{M}),\, d_0(B) \le t\}.
\end{align*}
The function $t \mapsto \mu(t, A)$ is decreasing and right-continuous.

The following fundamental properties of singular values will be frequently used in the sequel:
$$
\begin{gathered}
	\mu(t, U A V) \le \|U\|_{\infty}\, \mu(t, A)\, \|V\|_{\infty}, \\
	\mu(t+s, A+B) \le \mu(s, A) + \mu(t, B), \quad \mu(t+s, AB) \le \mu(s, A)\, \mu(t, B),
\end{gathered}
$$
and
$$
\mu(t, A) = \mu(t, A^*) = \mu\big(t, |A|^p\big)^{\frac{1}{p}}
$$
for any $0 < p < \infty$. See e.g. \cite[Section 2.3]{LSZ12} or \cite[Section 1.5]{Xu07}.

The norm on $L_p(\mathcal{M})$ can be expressed in terms of the singular value function, as shown in \cite[Example 2.4.2]{LSZ12}:
$$
\|A\|_{L_p(\mathcal{M})}=\|\,\mu(t, A)\|_{L_p(\mathbb{R}_{+})} =
\begin{cases}
	\displaystyle \left( \int_0^{\infty} \mu(t, A)^p \, dt \right)^{\frac{1}{p}}, & 1 \le p < \infty, \\[3mm]
	\mu(0, A), & p = \infty.
\end{cases}
$$
The case $p = \infty$ was established in \cite[Lemma 2.3.12]{LSZ12}.

For $0 < p < \infty$ and $0 < q < \infty$, the noncommutative Lorentz space $L_{p,q}(\mathcal{M})$ consists of all $A \in L_0(\mathcal{M})$ such that
$$
\|A\|_{L_{p,q}(\mathcal{M})} := \left( \int_0^{\infty} \big(t^{\frac{1}{p}} \mu(t, A)\big)^q \frac{dt}{t} \right)^{\frac{1}{q}} < \infty.
$$
When $q = \infty$, the space $L_{p,\infty}(\mathcal{M})$ is known as the noncommutative weak $L_p$-space for $0 < p < \infty$, with norm given by
$$
\|A\|_{L_{p,\infty}(\mathcal{M})} := \sup_{t>0} t^{\frac{1}{p}} \mu(t, A) = \sup_{s > 0} s \, \tau(e_s^{\perp}(|A|))^{\frac{1}{p}}.
$$

In the sequel, we focus primarily on the noncommutative weak $L_p$-space when $\mathcal{M}$ is taken as the tensor product $B(H) \overline{\otimes} \mathcal{M}$ ($H$ being a Hilbert space), equipped with $\operatorname{Tr} = \operatorname{tr} \otimes \tau$. We denote the corresponding noncommutative weak $L_p$-space by $\mathcal{L}_{p, \infty}(B(H) \overline{\otimes} \mathcal{M})$, which is also referred to as the noncommutative weak Schatten class, abbreviated by $\mathcal{L}_{p, \infty}$.

In type II noncommutative geometry, the concept of a compact operator is extended to that of a $\tau$-compact operator, as defined below (see \cite[Section 1.2]{BT06}).

\begin{definition}\label{def-tau-compact}
	An element $A \in \mathcal{M}$ is called $\tau$-compact if
	$$
	\lim_{t \to \infty} \mu(t,A) = 0.
	$$
	The set of all $\tau$-compact elements in $\mathcal{M}$ is a norm-closed ideal of $\mathcal{M}$ and is denoted by $\mathcal{C}_0(\mathcal{M})$.
\end{definition}

The ideal $\mathcal{C}_0(\mathcal{M})$ of all $\tau$-compact operators can also be characterized as the closure, in the norm $\|\cdot\|_\infty$, of the linear span of all $\tau$-finite projections (see \cite[Definition 2.6.8]{LSZ12}). Equivalently, $\mathcal{C}_0(\mathcal{M})$ is the set of all elements $A \in \mathcal{M}$ such that $\tau(e_s^{\perp}(|A|)) < \infty$ for every $s > 0$ (see e.g. \cite[Chapter II, Section 4]{DPS23}).
In particular, if $\mathcal{M}$ is finite, then $\mathcal{M} = \mathcal{C}_0(\mathcal{M})$ (see e.g. \cite[Page 64]{LSZ12}).

It follows from \cite[Lemma 1.17]{MSZ22} that the asymptotics of $\mu(t, A)$ as $t \rightarrow \infty$ are equivalent to those of $\tau\big(e_{s}^{\perp}(A)\big)$ as $s \rightarrow 0$, to be precise, is the following lemma.
\begin{lemma} \label{singular func}
	Let $0<p<\infty$ and let $0 \le A \in \mathcal{L}_{p, \infty}$, then
	$$
	\lim _{s\rightarrow 0} s^{p} \tau\big(e_{s}^{\perp}(A)\big)=\lim _{t \rightarrow \infty} t \ \mu(t, A)^p
	$$
	if any of the limits exist.
\end{lemma}

We will frequently use the following quasi-triangle inequality:
$$
\|A + B\|_{p, \infty} \le 2^{\frac{1}{p}} \big( \|A\|_{p, \infty} + \|B\|_{p, \infty} \big).
$$
Moreover, for $\frac{1}{r}=\frac{1}{p}+\frac{1}{q}$, the H\"{o}lder inequality holds:
$$
\|AB\|_{r, \infty} \le 2^{\frac{1}{r}} \|A\|_{p, \infty} \|B\|_{q, \infty}.
$$
The best constants in such H\"{o}lder inequalities have recently been determined in \cite{FZ21}.

In noncommutative geometry, classical analytical concepts are generalized through operator algebras and singular value techniques. Table \ref{tab:quantum_concepts} summarizes this correspondence across three frameworks: classical measure theory, Connes' Hilbert space operator formalism, and type II noncommutative geometry introduced in \cite{BT06, CPRS04, CPRS06i, CPRS06ii, CGRS14}. Key parallels include the characterization of infinitesimals, $L_p$ spaces, and integration via singular value sequence $\mu(j,A)$ for compact operators versus the singular value function $\mu(t,A)$ in the noncommutative setting.

The correspondence is unified by the relation $\mu(t,A) = \mu([t],A)$ when $(\mathcal{M},\tau) = (B(H),\operatorname{tr})$, which reduces the singular value function to the standard singular value sequence at integer points. This demonstrates how noncommutative measure spaces extend the operator-theoretic framework through a continuous parameterization of singular values, while maintaining the conceptual hierarchy of infinitesimals, function spaces, and integration.

Let us define the ideal
$$
\big(\mathcal{L}_{p, \infty}\big)_0=\left\{B \in \mathcal{L}_{p, \infty}: \lim _{t \rightarrow \infty} t^{\frac{1}{p}} \mu(t, B)=0\right\} .
$$

The following two lemmas quoted from \cite[Section 4]{FZ23}, which evaluate the singular value sequence ($\mu(j,\cdot), j\in \mathbb{N}_{0}$) case in \cite[Section 11.6]{BS87} to the singular value function ($\mu(t,\cdot), t\in \mathbb{R}_{+}$) case.

\begin{lemma} \label{weak weyl}
	Let $A \in \mathcal{L}_{p, \infty}$ and $B \in (\mathcal{L}_{p, \infty})_0$. If $\lim\limits_{t \to \infty} t^{\frac{1}{p}} \mu(t, A)$ exists, then
	$$
	\lim_{t \to \infty} t^{\frac{1}{p}} \mu(t, A+B) = \lim_{t \to \infty} t^{\frac{1}{p}} \mu(t, A).
	$$
\end{lemma}

\begin{lemma} \label{ideal approximate}
	Let $(A_n)_{n \ge 0} \subset \mathcal{L}_{p, \infty}$ satisfy the following conditions:
	
	\begin{enumerate}[$\rm (i)$]
		\item $A_n \to A$ in $\mathcal{L}_{p, \infty}$.
		
		\item For every $n \ge 0$, the limit
		$$
		\lim_{t \to \infty} t^{\frac{1}{p}} \mu(t, A_n) = c_n \quad \text{exists.}
		$$
	\end{enumerate}
	Then the following limits exist and are equal:
	$$
	\lim_{t \to \infty} t^{\frac{1}{p}} \mu(t, A) = \lim_{n \to \infty} c_n.
	$$
\end{lemma}

\begin{table}[H]
	\centering
	\caption{Classical and quantum concepts in N.C.G}
	\label{tab:quantum_concepts}
	\renewcommand{\arraystretch}{1.5}
	\begin{tabular}{|C{2.2cm}|C{5.4cm}|C{5.8cm}|}
		\hline
		\textbf{Classical Concept} &
		\textbf{Quantum Concept by Connes (operators on $\boldsymbol{H}$)} &
		\textbf{Quantum Concept in $\boldsymbol{(\mathcal{M},\tau)}$ (Type II N.C.G)} \\
		\hline
		Infinitesimal variable &
		compact operator \newline
		$\lim\limits_{j\to\infty} \mu(j,A) = 0$ &
		$\tau$-compact operator \newline
		$\lim\limits_{t\to\infty} \mu(t,A) = 0$ \\
		\hline
		Infinitesimal of order $\alpha>0$ &
		compact operator with \newline
		$\mu(j,A) = O(j^{-\alpha})$ &
		$\tau$-compact operator with \newline
		$\mu(t,A) = O(t^{-\alpha})$ \\
		\hline
		Function in $L_p$ space &
		Schatten class $\mathcal{S}_p{ \big(H\big)}$ \newline
		$\|A\|_{\mathcal{S}_p} = \Big( \sum\limits_{j=1}^{\infty} \mu(j,A)^p \Big)^{\frac{1}{p}} < \infty$ &
		N.C. Schatten class $\mathcal{L}_p{  \big(B(H)\overline{\otimes}\mathcal{M}\big)}$ \newline
		$\|A\|_{\mathcal{L}_p} = \Big( \int_0^\infty \mu(t,A)^p  dt \Big)^{\frac{1}{p}} < \infty$ \\
		\hline
		Function in $L_{p,\infty}$ space &
		weak Schatten class $\mathcal{S}_{p,\infty}$ \newline
		$\|A\|_{\mathcal{S}_{p,\infty}} = \sup\limits_{j\ge 1} j^{\frac{1}{p}} \mu(j,A) < \infty$ &
		N.C. weak Schatten class $\mathcal{L}_{p,\infty}$ \newline
		$\|A\|_{\mathcal{L}_{p,\infty}} = \sup\limits_{t>0} t^{\frac{1}{p}} \mu(t,A) < \infty$ \\
		\hline
		Integration &
		Dixmier trace for $T \in \mathcal{S}_{1,\infty}$ \newline
		$\operatorname{Tr}_\omega(A) = \lim\limits_\omega \dfrac{1}{\log (N+1)} \sum\limits_{j=1}^{N} \mu(j,A)$ &
		Dixmier trace for $T \in \mathcal{L}_{1,\infty}$ \newline
		$\operatorname{Tr}_\omega(A) = \lim\limits_\omega \dfrac{1}{\log (s+1)} \int_0^s \mu(t,A)  dt$ \\
		\hline
	\end{tabular}
	
\end{table}

\section{Operator-valued $\Psi$DOs with parameter}\label{section-parameter}

This section provides a systematic exposition of the calculus for operator-valued $\Psi$DOs with parameter $\lambda$. While the asymptotic expansions for composition and adjoints of unparameterized $B(X)$-valued $\Psi$DOs were previously obtained in \cite[Section 2]{XX19}, a unified, self-contained calculus of $B(X)$-valued $\Psi$DO with parameter calculus has not yet appeared in the literature. We fill this gap by detailing the algebra of $B(X)$-valued $\Psi$DOs with parameter in Section \ref{section sym}. In Section \ref{section l2}, we entend the existing $L_2$-boundedness results from \cite[Section 3]{XX19} to the parameterized setting, and carry out a refined analysis within the von Neumann tensor product $B(L_2(\mathbb{R}^d))\overline{\otimes}\mathcal{M}$, this refinement is essential for our subsequent definition of $\zeta$-functions in Section \ref{section asymp}.

\subsection{$B(X)$-valued $\Psi$DOs with parameter}\label{section sym}

\begin{definition}\label{sector}
	For a ray $\{\arg \lambda = \theta_0\}$ and constants $\theta, r > 0$, let $\Lambda = \Lambda(\theta_0, \theta, r)$ denote the following keyhole-shaped region:
	$$
	\Lambda = \Big\{\lambda \in \mathbb{C}\setminus \{0\} : |\lambda| < r \text{ or } |\arg \lambda - \theta_0| < \theta \Big\}.
	$$
\end{definition}

\begin{definition}
	Let $m, \rho, \delta, h$ be real numbers with $0 \le \delta \le \rho \le 1$ (excluding the case $(\rho, \delta) = (1,1)$). The class $S_{\rho, \delta ; h}^m\big(\mathbb{R}^d \times \mathbb{R}^d \times \Lambda ; B(X)\big)$ consists of the functions $\sigma(x, \xi, \lambda)$ such that

	\begin{enumerate}[$\rm (i)$]
		\item $\sigma (x, \xi, \lambda_0) \in C^{\infty}\big(\mathbb{R}^d \times \mathbb{R}^d ; B(X)\big)$ for every fixed $\lambda_0 \in \Lambda$.
		
		\item For arbitrary multi-indices $\alpha$ and $\beta$, there exist constants $C_{\alpha, \beta}$ such that
		
		$$
		\left\|\partial_\xi^\alpha \partial_x^\beta \sigma(x, \xi, \lambda)\right\|_{B(X)} \le C_{\alpha, \beta}(1+|\xi|^2)^{\frac{m-\rho|\alpha|+\delta|\beta|}{2}} (1+|\lambda|^2)^{\frac{h}{2}} .
		$$
	\end{enumerate}
	For $x \in \mathbb{R}^d, \xi \in \mathbb{R}^d, \lambda \in \Lambda$, we put
	$$
	S_{h}^{-\infty}\big(\mathbb{R}^d \times \mathbb{R}^d \times \Lambda ; B(X)\big)=\bigcap_{m \in \mathbb{R}} S_{\rho, \delta ; h}^m\big(\mathbb{R}^d \times \mathbb{R}^d \times \Lambda ; B(X)\big).
	$$
\end{definition}

The condition (ii) for parametric estimate in the above definition follows \cite{LP20}, although we omit an assumption of holomorphicity, as it is known that the approach in \cite{Shu01} works only for elliptic differential operator but not for general elliptic pseudo-differential operators.

The symbol $\sigma(x, \xi) \in S_{\rho, \delta}^m\big(\mathbb{R}^d \times \mathbb{R}^d ; B(X)\big)$ can be interpreted as an element of $S_{\rho,\delta;0 }^m\big(\mathbb{R}^d \times \mathbb{R}^d \times \Lambda ; B(X)\big)$ that is independent of $\lambda$.

\begin{definition}\label{symboldef}
	Let $\sigma(x,\xi,\lambda) \in S_{\rho, \delta;h }^m\big(\mathbb{R}^d\times \mathbb{R}^d\times \Lambda; B\big(X\big)\big)$. For function $f \in \mathscr{S}(\mathbb{R}^d;X)$, the $\Psi$DO $A_{\lambda}$ is a mapping $f \mapsto A_{\lambda} f$ given by
	\begin{equation}\label{eqsymdef}
		A_{\lambda}f(x)=\int_{\mathbb{R}^d} \sigma(x, \xi,\lambda) \hat{f} (\xi)  e^{ ix \cdot \xi}  \bar{d} \xi.
	\end{equation}
	The class of $\Psi D Os$ is denoted by $\Psi_{\rho, \delta;h}^m\big(\mathbb{R}^d\times \Lambda; B(X)\big)$ or simply by $\Psi_{\rho, \delta;h}^m$. We also put $\Psi_{h}^m$ instead of $\Psi_{1,0;h}^m$ and write $\Psi_{h}^{-\infty}=\bigcap\limits_m \Psi_{\rho, \delta;h}^m$.
\end{definition}

\begin{definition}Let $\sigma_j(x, \xi, \lambda) \in S_{\rho, \delta ; h}^{m_j} \big(\mathbb{R}^d \times \mathbb{R}^d \times \Lambda ; B(X)\big)$, where $j=0,1,2, \ldots$, $m_j \searrow-\infty$. For any $k$ ,if the function $\sigma(x, \xi,\lambda) \in C^{\infty} \big(\mathbb{R}^d \times \mathbb{R}^d \times \Lambda ; B(X)\big)$ holds for
	$$
	\sigma(x, \xi,\lambda)-\sum_{0\le j<k} \sigma_j(x, \xi,\lambda) \in S_{\rho, \delta ; h}^{m_k} \big(\mathbb{R}^d \times \mathbb{R}^d \times \Lambda ; B(X)\big),
	$$
	then $\sigma(x, \xi, \lambda)$ has asymptotic expansion $\sum_{j=0}^{\infty} \sigma_j(x, \xi, \lambda)$, denoted as
	$$
	\sigma(x, \xi,\lambda) \sim \sum_{j=0} ^{\infty} \sigma_j(x, \xi, \lambda).
	$$
\end{definition}

\begin{proposition}\label{asymp}
	Let $\sigma_j(x, \xi, \lambda) \in S_{\rho, \delta ; h}^{m_j}  \big(\mathbb{R}^d \times \mathbb{R}^d \times \Lambda ; B(X)\big)$, where $j=0,1,2, \ldots$, $m_j\searrow{-\infty}$, then there exits $\sigma(x, \xi,\lambda) \in  S_{\rho, \delta ; h}^{m_0} \big(\mathbb{R}^d \times \mathbb{R}^d \times \Lambda ; B(X)\big)$, and its asymptotic expansion is
	$$\sum_{j=0}^{\infty} \sigma_j(x, \xi, \lambda).$$
	Furthermore, if another function $\sigma^{\prime}(x, \xi,\lambda)$ has the same property
	$\sigma^{\prime}(x, \xi,\lambda) \sim \sum_{j=0} ^{\infty} \sigma_j(x, \xi, \lambda)$, then $\sigma(x, \xi, \lambda)-\sigma^{\prime}(x, \xi,\lambda)\in S^{-\infty}  \big(\mathbb{R}^d \times \mathbb{R}^d \times \Lambda ; B(X) \big)$.
\end{proposition}

\begin{proof}
	
	Let the function $\varphi(\xi) \in C^{\infty}(\mathbb{R}^d)$, which is equal to 0 when $|\xi| \le \frac{1}{2}$ , and is equal to 1 when $|\xi|\ge 1$ . The following proves that there exists sequence $ \{t_j \} $ such that for all  $\alpha, \beta$ satisfying $|\alpha|+|\beta| \le j$, we have
	
	\begin{equation} \label{t j estimate}
		\left\|\partial_\xi^\alpha \partial_x^\beta\big(\varphi(t_j^{-1}\xi ) \sigma_j(x, \xi, \lambda)\big)\right\|_{B(X)} \le 2^{-j} \cdot(1+|\xi|^{2})^\frac{m_{j-1}-\rho|\alpha|+\delta|\beta|}{2}(1+|\lambda|^2)^{\frac{h}{2}}.
	\end{equation}
	By using the sequence  $\{t_j\}$ ,we set
	\begin{equation} \label{t j sum}
		\sigma(x, \xi,\lambda)=\sum_{j=0}^{\infty} \varphi(t_j^{-1}\xi) \sigma_j(x, \xi, \lambda),
	\end{equation}
	then asymptotic expansion of $\sigma(x, \xi,\lambda)$ is $\sum_{j=0}^{\infty} \sigma_j(x, \xi, \lambda)$ .
	
	Firstly, we prove that the sequence $ \{t_j \}$ can be selected. For fixed $\alpha, \beta$, we have
	$$
	\left\|\partial_\xi^\alpha \partial_x^\beta\big(\varphi(t_j^{-1}\xi) \sigma_j(x, \xi, \lambda)\big)\right\|_{B(X)} \le C \sum_{\alpha_1 \le \alpha}\left\|\partial_\xi^{\alpha_1} \varphi(t_j^{-1}\xi) \cdot \partial_\xi^{\alpha-\alpha_1} \partial_x^\beta \sigma_j(x, \xi, \lambda)\right\|_{B(X)},
	$$
	$$
	|\partial_\xi^{\alpha_1} \varphi(t_j^{-1}\xi)| \le C t_j^{-\left|\alpha_1\right|},
	$$
	$$
	\left\|\partial_\xi^{\alpha-\alpha_1} \partial_x^\beta \sigma_j(x, \xi, \lambda)\right\|_{B(X)} \le C(1+|\xi|^{2})^\frac{ m_j-\rho|\alpha|+\delta|\beta|+\rho\left|\alpha_1\right|}{2}(1+|\lambda|^2)^{\frac{h}{2}}.
	$$
	Notice that the derivative of $\varphi(t_j^{-1}\xi)$ equals to 0 when $|\xi|<\frac{1}{2} t_j$ and $|\xi|>t_j$, so if $\alpha_1 \neq 0$, we have
	$$
	|\partial_\xi^{\alpha_1 }\varphi(t_j^{-1}\xi)| \le C(1+|\xi|^{2})^{-\frac{|\alpha_1|}{2}}.
	$$
	This also holds for $\alpha_1=0$, thus
	\begin{align*}	
		&\left\|\partial_\xi^\alpha \partial_x^\beta\big(\varphi(t_j^{-1}\xi) \sigma_j(x, \xi, \lambda)\big)\right\|_{B(X)} \\
		\le & C(1+|\xi|^{2})^{\frac{m_j-\rho|\alpha|+\delta|\beta|}{2}}(1+|\lambda|^2)^{\frac{h}{2}} \\
		\le & C(1+|\xi|^{2})^{\frac{m_{j-1}-\rho|\alpha|+\delta|\beta|}{2}}(1+\frac{1}{2} t_j)^{m_j-m_{j-1}}(1+|\lambda|^2)^{\frac{h}{2}}.
	\end{align*}
	As long as $t_j$ is sufficiently large, \eqref{t j estimate} holds for fixed $\alpha, \beta$. But for fixed $j$, $\alpha$ and $\beta$ are finite and satisfy the condition $|\alpha|+|\beta| \le j$. Therefore,  the above requirements sequence $ \{t_j \} $ can be chosen.
	
	Secondly we prove that the series \eqref{t j sum} is convergent, and $\sigma(x, \xi,\lambda) \sim \sum_{j=0} ^{\infty} \sigma_j(x, \xi, \lambda)$.
	For fixed $\xi$, $\varphi (t_j^{-1}\xi ) \equiv 0$ when $t_j$ is sufficiently large, so this series for each $(x, \xi, \lambda)$ is actually the sum of finite terms, thus $\sigma(x, \xi,\lambda)$ is determined.
	
	Next we estimate $\|\partial_\xi^\alpha \partial_x^\beta(\sigma-\sum_{0\le j<k} \sigma_j)\|_{B(X)}$. Set $l=\max(|\alpha|+|\beta|, k)$, then by \eqref{t j sum} we have
		\begin{align*}
		& \big\|\partial_\xi^\alpha \partial_x^\beta\big(\sigma(x, \xi,\lambda)-\sum_{0\le j<k} \sigma_j(x, \xi, \lambda)\big)\big\|_{B(X)} \\
		\le &  \big\|\partial_\xi^a \partial_x^\beta \sum_{0\le j<k}\big(\varphi(t_j^{-1}\xi)-1\big) \sigma_j(x, \xi, \lambda)\big\|_{B(X)} + \big\|\partial_\xi^a \partial_x^\beta \sum_{j=k}^l \varphi(t_j^{-1}\xi) \sigma_j(x, \xi, \lambda)\big\|_{B(X)} \\
		& + \big\|\partial_\xi^a \partial_x^\beta \sum_{j>l} \varphi(t_j^{-1}\xi) \sigma_j(x, \xi, \lambda)\big\|_{B(X)} \\
		:=& 	\text { (I) }+	\text { (II) }+	\text { (III) }.
	\end{align*}
	In term (I), only the finite terms of the series are included, and $\varphi(t_j^{-1}\xi) \equiv 1 $ if $|\xi|$ is sufficiently large,
	so this term is zero when $|\xi|\rightarrow \infty$. The term (II) also contains only limited items, each item is estimated by $C(1+|\xi|^{2})^\frac{m_j-\rho|\alpha|+\delta|\beta|}{2}(1+|\lambda|^2)^{\frac{h}{2}}$. We have $m_j \le m_k$ when $k \le j \le l$ , hence (II) is estimated by  $C(1+|\xi|^{2})^\frac{m_k-\rho|\alpha|+\delta|\beta|}{2}(1+|\lambda|^2)^{\frac{h}{2}}$. In term (III), since $|\alpha|+|\beta| \le j$, then by \eqref{t j estimate},
	\begin{align*}
		\text { (III) } & \le (1+|\xi|^{2})^\frac{m_{j-1}-\rho|\alpha|+\delta|\beta|}{2}\cdot(\sum_{j>l} 2^{-j})(1+|\lambda|^2)^{\frac{h}{2}} \\ &\le(1+|\xi|^{2})^\frac{m_l-\rho|\alpha|+\delta|\beta|}{2} \cdot(\sum_{j>l} 2^{-j})(1+|\lambda|^2)^{\frac{h}{2}} \\
		& \le(1+|\xi|^{2})^\frac{m_{l}-\rho|\alpha|+\delta|\beta|}{2}(1+|\lambda|^2)^{\frac{h}{2}} \\
		&\le(1+|\xi|^{2})^\frac{m_k-\rho|\alpha|+\delta|\beta|}{2}(1+|\lambda|^2)^{\frac{h}{2}}.
	\end{align*}

	
	Based on the above estimation of (I), (II) and (III), we obtain $\sigma-\sum_{0\le j<k} \sigma_j \in S_{\rho, \delta ; h}^{m_k}$, hence $\sigma(x, \xi,\lambda) \sim \sum_{j=0} ^{\infty} \sigma_j(x, \xi, \lambda)$.
\end{proof}

The above proposition is not convenient to apply, because to verify $\sigma \sim \sum_{j=0}^{\infty} \sigma_j$ we need to estimate all $\partial_\xi^\alpha \partial_x^\beta(\sigma-\sum_{0\le j<k} \sigma_j)$, which is troublesome. Using the following proposition we can simplify the verification process.

\begin{proposition}\label{asy after}
	Let $\sigma_j(x, \xi, \lambda) \in S_{\rho, \delta ; h}^{m_j} \big(\mathbb{R}^d \times \mathbb{R}^d \times \Lambda ; B(X)\big)$,  $j=0,1,2, \ldots$,  $  m_j \searrow-\infty $, and let $\sigma(x, \xi,\lambda) \in C^{\infty} \big(\mathbb{R}^d \times \mathbb{R}^d \times \Lambda ; B(X)\big)$ such that for arbitrary multi-indices $\alpha,\beta$ there exist constants $\mu=\mu(\alpha,\beta)$ and $C=C(\alpha,\beta)$ with
	\begin{equation} \label{mu}
		\big\|\partial_{\xi}^\alpha \partial_x^\beta \sigma(x, \xi, \lambda)\big\|_{B(X)} \le C(1+|\xi|^{2})^{\frac{\mu}{2}}(1+|\lambda|^2)^{\frac{h}{2}} .
	\end{equation}
	Furthermore assume there exit numbers $\mu_{k}\searrow-\infty $ and constants $C_{k}$ such that for arbitrary $k$, the following estimate holds
	$$
	\big\|\sigma(x, \xi, \lambda)-\sum_{0 \le j<k} \sigma_j(x, \xi, \lambda)\big\|_{B(X)} \le C_{k}(1+|\xi|^{2})^{\frac{\mu_{k}}{2}}(1+|\lambda|^2)^{\frac{h}{2}}.
	$$
	Then $\sigma(x, \xi, \lambda) \in S_{\rho, \delta;h}^{m_0}\big(\mathbb{R}^d \times \mathbb{R}^d \times \Lambda ; B(X)\big)$ , and $ \sigma(x, \xi, \lambda) \sim \sum_{j=0}^{\infty} \sigma_j(x, \xi, \lambda) .$
\end{proposition}

\begin{proof}
	Utilizing Proposition \ref{asymp} we set $b(x, \xi, \lambda)\sim \sum_{j=0} ^{\infty} \sigma_j(x, \xi, \lambda)$, then $\tilde{\sigma}(x, \xi, \lambda)=\sigma(x, \xi, \lambda)-b(x, \xi, \lambda)$ satisfies \eqref{mu} and
	\begin{align*}
		&\|\tilde{\sigma}(x, \xi, \lambda)\|_{B(X)} \\
		\le& \big\|\sigma(x, \xi, \lambda)-\sum_{j<k} \sigma_j(x, \xi, \lambda)\big\|_{B(X)}+\big\|b(x, \xi, \lambda)-\sum_{j<k} \sigma_j(x, \xi, \lambda)\big\|_{B(X)} \\
		\le& C_k\big(1+|\xi|^{2}\big)^{\frac{\mu_k^{\prime}}{2}}(1+|\lambda|^2)^{\frac{h}{2}}.
	\end{align*}
	where $\mu_k^{\prime}=\max \big(m_k, \mu_k\big)$ monotonically declines to $-\infty$. Additionally
	\begin{equation} \label{rapidde}
		\|\tilde{\sigma}(x, \xi, \lambda)\|_{B(X)} \le C_r(1+|\xi|^2)^{-\frac{r}{2}}, \quad \forall r \in \mathbb{N}.
	\end{equation}
	
	For a unit vector $\eta$, by  \eqref{mu} and  Taylor formula we have for $0<\varepsilon<1$,
	$$
	\big\|\tilde{\sigma}(x, \xi+\varepsilon \eta, \lambda)-\tilde{\sigma}(x, \xi, \lambda)- \sum_j \partial_{\xi_j}\tilde{\sigma} (x, \xi, \lambda) \cdot \varepsilon \eta_j\big\|_{B(X)} \le C\varepsilon^2\big((1+|\xi|^{2}\big)^{\frac{\mu}{2}}(1+|\lambda|^2)^{\frac{h}{2}},
	$$
	then for any $|\eta|=1$,
	\begin{align*}
		&\big\| \sum_j \partial_{\xi_j} \tilde{\sigma} (x, \xi, \lambda) \cdot   \eta_j \big\|_{B(X)} \\
		\le & C \varepsilon\big(1+ |\xi|^2\big)^{\frac{\mu}{2}}(1+|\lambda|^2)^{\frac{h}{2}}+\varepsilon^{-1}\|\tilde{\sigma}(x, \xi, \lambda)-\tilde{\sigma}(x, \xi+\varepsilon \eta, \lambda)\|_{B(X)} .
	\end{align*}
	Put $\varepsilon=(1+|\xi|^{2})^{-\frac{N}{2}}$ and by \eqref{rapidde} we have
	$$\|\partial_{\xi_j} \tilde{\sigma}(x, \xi, \lambda)\|_{B(X)} \le C^{'}\big(1+|\xi|^2\big)^{\frac{\mu-N}{2}}(1+|\lambda|^2)^{\frac{h}{2}},\quad \forall N \in \mathbb{N} ,$$
	thus $\partial_{\xi_j} \tilde{\sigma} (x, \xi, \lambda)$ is rapidly decreasing.
	Using the same method we see that $\partial_{x_j} \tilde{\sigma} (x, \xi, \lambda)$ is rapidly decreasing; then using iteration we know that  $\partial_{\xi}^\alpha \partial_x^\beta \tilde{\sigma}(x, \xi, \lambda)$ are rapidly decreasing for all $\alpha, \beta$. This concludes $\tilde{\sigma}(x, \xi, \lambda) \in S_{ h}^{-\infty}\big(\mathbb{R}^d \times \mathbb{R}^d \times \Lambda ; B(X)\big)$, completing the proof.
\end{proof}

Equation \eqref{eqsymdef} in Definition \ref{symboldef} can also be written as
\begin{equation}\label{pseudo-double-integral}
	A_{\lambda}f(x)=\int_{\mathbb{R}^d}\int_{\mathbb{R}^d}   \sigma(x, \xi,\lambda) f(y)  e^{ i(x-y) \cdot \xi} dy \ \bar{d} \xi.
\end{equation}
However, the above $\xi$-integral does not necessarily converge absolutely, even for $f\in \mathscr{S} (\mathbb{R}^d, X)$. A standard way to overcome this difficulty is to approximate $\sigma$ by symbols with compact support. Specifically, fix a compactly supported infinitely differentiable function  $\eta$ defined on $\mathbb{R}^d\times \mathbb{R}^d$ such that $\eta =1$ near the origin, and set $\sigma_j (x, \xi ,\lambda) =  \sigma (x, \xi ,\lambda) \ \eta(2^{-j}x, 2^{-j}\xi)$. Then the double integral
$$
A_{\lambda,j}f(x):= \int_{\mathbb{R}^d} \int_{\mathbb{R}^d} \sigma_j(x, \xi,\lambda) f(y) e^{ i(x-y) \cdot \xi} d y \ \bar{d} \xi
$$
converges absolutely for each $j$, and $A_{\lambda,j} f \rightarrow A_{\lambda } f$ in $\mathscr{S}(\mathbb{R}^d ; X)$ as $j \rightarrow \infty$, which provides a rigorous interpretation of the oscillatory integral as the limit of absolutely convergent integrals.

Using \eqref{pseudo-double-integral}, we are able to represent the adjoint operator of $A_\lambda$ on $\mathscr{S}(\mathbb{R}^d ; X^*)$ using double integral. Since $$\mathscr{S}(\mathbb{R}^d ; X^*)\subset L_p(\mathbb{R}^d ; X^*)\subset  \mathscr{S}'(\mathbb{R}^d ; X^*)=  \big(\mathscr{S}(\mathbb{R}^d ; X )\big)'$$
the dual bracket for $f \in \mathscr{S}(\mathbb{R}^d ; X)$ and $g \in \mathscr{S}(\mathbb{R}^d ; X^*)$ is simply written as
\begin{equation}\label{dual-product}
	\langle   f, g \rangle=  \int_{\mathbb{R}^d} f(x) g(x) dx .
\end{equation}
By this dual relation, we obtain
\begin{equation}\label{pseudo-double-integral-Ad}
	A_{\lambda}^* g(x)= \int_{\mathbb{R}^d} \int_{\mathbb{R}^d}  \sigma(y, \xi,\lambda)^* g(y) e^{i(x-y) \cdot \xi} d y \ \bar{d} \xi.
\end{equation}
It follows from \cite[Proposition 2.3]{XX19} that $A_{\lambda}^*$ is continuous on $\mathscr{S}(\mathbb{R}^d ; X^*)$.
Using this adjoint operator, the original operator $A_{\lambda}$ admits a unique extension by duality to the distribution space
$$
A_{\lambda}: \mathscr{S}(\mathbb{R}^d ; X^*)^{\prime} \cong \mathscr{S}^{\prime}(\mathbb{R}^d ; X^{**}).
$$
This is rigorously justified in Definition 2.4 and Proposition 2.5 in \cite{XX19}.

The factor $\sigma(x, \xi,\lambda)$ on the right hand side of \eqref{pseudo-double-integral} does not depend on the variable $y$. But the factor $\sigma(y, \xi,\lambda)^*$ on the right hand side of \eqref{pseudo-double-integral-Ad} depends on the variable $y$, and thus cannot be directly transformed into the form given in \eqref{eqsymdef}. To meet the requirements of operations within the class of operator-valued $\Psi$DOs, we will extend Definition \ref{symboldef}.

\begin{definition}
	Let $m, \rho, \delta, h$ be real numbers with $0 \le \delta \le \rho \le 1$ (excluding the case $(\rho, \delta) = (1,1)$). The class $S_{\rho, \delta ; h}^m\big(\mathbb{R}^d \times \mathbb{R}^d \times \mathbb{R}^d \times \Lambda ; B(X)\big)$ consists of functions $\sigma(x,y, \xi, \lambda)$ such that
	
	\begin{enumerate}[$\rm (i)$]
		\item $\sigma\big(x, y,\xi, \lambda_0\big) \in C^{\infty}\big(\mathbb{R}^d \times \mathbb{R}^d \times \mathbb{R}^d ; B(X)\big)$ for every fixed $\lambda_0 \in \Lambda$.
		
		\item For arbitrary multi-indices $\alpha$ and $\beta$, there exist constants $C_{\alpha, \beta}$ such that
		$$
		\big\|\partial_\xi^\alpha \partial_{x, y}^\beta \sigma(x,y, \xi, \lambda)\big\|_{B(X)} \le C_{\alpha, \beta}(1+|\xi|^2)^{\frac{m-\rho|\alpha|+\delta|\beta|}{2}} (1+|\lambda|^2)^{\frac{h}{2}}.
		$$
		
	\end{enumerate}
	When $\rho=1, \delta=0$, we still simply denote $S_{\rho, \delta;h}^m$ as $S_h^m$.
\end{definition}

\begin{definition}\label{ampdef}
	If the function $\sigma(x,y, \xi, \lambda) \in S_{\rho, \delta ; h}^m\big(\mathbb{R}^d \times \mathbb{R}^d \times \mathbb{R}^d \times \Lambda ; B(X)\big)$, then the operator defined by
	\begin{equation} \label{eqampdef}
		A_{\lambda}f(x)=\int_{\mathbb{R}^d}\int_{\mathbb{R}^d}  \sigma(x,y, \xi,\lambda) f(y)   e^{ i(x-y) \cdot \xi} dy \ \bar{d} \xi
	\end{equation}
	is continuous on $\mathscr{S}(\mathbb{R}^d ; X)$, and it is also called an operator-valued $\Psi$DO, where $\sigma(x,y, \xi, \lambda)$ is referred to as the amplitude of the operator $A_{\lambda}$.
\end{definition}

Next, we define the symbol of a general form of an operator-valued $\Psi$DO. According to Definition \ref{symboldef}, if an operator-valued $\Psi$DO is defined by \eqref{eqsymdef} (or in \eqref{eqampdef} the amplitude does not depend on $y$), then $\sigma(x,\xi, \lambda)$ is called the symbol of $A_{\lambda}$.

\begin{definition}\label{1symdef}
	For the operator-valued $\Psi$DO $A_{\lambda}$ of the form \eqref{eqampdef}, define
	\begin{equation}\label{symalldef}
		e^{-i x \cdot \xi} A_{\lambda} (e^{i x \cdot \xi})
	\end{equation}
	as the symbol of $A_{\lambda}$, denoted by $\sigma(A)(x, \xi,\lambda)$.
\end{definition}

Notice that Definition \ref{1symdef} is reasonable and consistent with Definition \ref{symboldef} when the amplitude $\sigma(x,y,\xi, \lambda)$ does not depend on $y$. By Definition \ref{1symdef}, the operator-valued $\Psi$DO of the form \eqref{eqampdef} can also be transformed into the form \eqref{eqsymdef}. In fact, for $f \in \mathscr{S}(\mathbb{R}^d ; X)$, the inverse Fourier transform formula holds:
$$
f(x)=\int_{\mathbb{R}^d}  \hat{f}(\xi) e^{i x \cdot \xi} \bar{d} \xi.
$$
Since $A_{\lambda}$ is a linear and continuous mapping from $\mathscr{S}(\mathbb{R}^d ; X)$ to $\mathscr{S}(\mathbb{R}^d ; X)$, it follows that
\begin{equation}\label{eqsymall}
	\begin{aligned}
		A_{\lambda} f(x) & =\int_{\mathbb{R}^d}(A_{\lambda}e^{i x \cdot \xi}) \hat{f}(\xi) \ \bar{d} \xi \\
		& =\int_{\mathbb{R}^d} \sigma(A)(x, \xi,\lambda) \hat{f}(\xi)  e^{i x \cdot \xi} \bar{d} \xi.
	\end{aligned}
\end{equation}

From \eqref{eqsymall}, we can say that the operator-valued $\Psi$DO of the form \eqref{eqampdef} is essentially the same as that of the form \eqref{eqsymdef}. However, due to the more flexible writing style of \eqref{eqampdef}, it is often more convenient in operations and discussions of some problems. The symbol given by \eqref{symalldef} is not convenient for calculation, so in the following, we derive a formula for calculating the symbol through the amplitude.

\begin{theorem} \label{asympthm}
	Let $A_ \lambda \in \Psi_{\rho, \delta;h}^m\big(\mathbb{R}^d\times \Lambda; B(X)\big)$, whose symbol is $\sigma(A)(x, \xi, \lambda)$. Then
	\begin{equation} \label{asympex}
		\begin{split}
			&\sigma(A)(x, \xi, \lambda) - \left.\sum_{|\alpha| \le N-1} \frac{1}{\alpha !} \partial_{\xi}^\alpha D_x^\alpha \sigma(x,y,\xi, \lambda)\right|_{y=x} \\
			&\qquad \in S_{\rho, \delta ; h}^{m-(\rho-\delta) N}\big(\mathbb{R}^d \times \mathbb{R}^d \times \Lambda ; B(X)\big).
		\end{split}
	\end{equation}
\end{theorem}

\begin{proof}
	The symbol $\sigma(A)(x, \xi,\lambda)$  can be reformulated as follows
	\begin{align*}
		\sigma(A)(x, \xi,\lambda) & =e^{-i x \cdot \xi} A_{\lambda} e^{i x \cdot \xi}\\
		&=e^{-i x \cdot \xi}\int_{\mathbb{R}^d}\int_{\mathbb{R}^d} \sigma(x, y, \theta,\lambda) e^{i(x-y) \cdot \theta} e^{i y\cdot \xi} d y \ \bar{d} \theta\\
		& =\int_{\mathbb{R}^d}\int_{\mathbb{R}^d} \sigma(x, y, \theta,\lambda) e^{i(x-y) \cdot (\theta-\xi)} d y \ \bar{d} \theta.
	\end{align*}
	Putting $z=y-x$ and $\eta=\theta-\xi$ to simplify the exponent, we obtain
	\begin{equation} \label{exp}
		\sigma(A)(x, \xi,\lambda)=\int_{\mathbb{R}^d}\int_{\mathbb{R}^d} \sigma(x, x+z, \xi+\eta,\lambda) e^{-i z \cdot \eta} d z \ \bar{d} \eta.
	\end{equation}
	Using Taylor's formula to expand $\sigma(x, x+z, \xi+\eta,\lambda)$ in $\eta$ near $\eta=0$, we have
	$$
	\sigma(x, x+z, \xi+\eta,\lambda)=\sum_{|\alpha| \le N-1} \frac{\eta^\alpha}{\alpha !}\partial_{\xi}^\alpha \sigma(x, x+z, \xi,\lambda)  +r_N(x, x+z, \xi, \eta,\lambda),
	$$
	where
	\begin{equation} \label{remain}
		r_N(x, x+z, \xi, \eta,\lambda)=\sum_{|\alpha|=N} \frac{N \eta^\alpha}{\alpha !} \int_0^1(1-t)^{N-1} \partial_{\xi}^\alpha \sigma(x, x+z, \xi+t \eta,\lambda) d t.
	\end{equation}
	Observe that
	$$
	\int_{\mathbb{R}^d}\int_{\mathbb{R}^d} \partial_{\xi}^\alpha \sigma(x, x+z, \xi,\lambda) \eta^\alpha e^{-i z \cdot \eta} d z \ \bar{d} \eta
	=\left.\partial_{\xi}^\alpha D_z^\alpha \sigma(x, x+z, \xi,\lambda)\right|_{z=0},
	$$
	by applying the Fourier inversion formula, this yields the finite terms in equation  \eqref{asympex}.
	
	Let us first derive a rough estimate for $\sigma(A)(x, \xi,\lambda)$ the type given in \eqref{mu}. To achieve this, we rewrite equation \eqref{exp} by integrating by parts
	\begin{equation} \label{int1}
		\sigma(A)(x, \xi,\lambda)=\int_{\mathbb{R}^d}\int_{\mathbb{R}^d} e^{-i z \cdot \eta}(1+|D_z|^{2})^{\frac{\nu}{2}} \sigma(x, x+z, \xi+\eta,\lambda) \cdot (1+|\eta|^{2})^{-\frac{\nu}{2}} d z \ \bar{d} \eta,
	\end{equation}
	where $\nu$ is even and non-negative.
	
	Utilizing the inequality $(1+|\xi+\eta|^{2})^{\frac{1}{2}} \le 2 (1+|\xi|^{2})^{\frac{1}{2}}(1+|\eta|^{2})^{\frac{1}{2}}$, we can deduce from \eqref{int1} that
	$$
	\big\|\partial_{\xi}^\alpha \partial_x^\beta \sigma(A)(x, \xi,\lambda)\big\|_{B(X)} \le C(1+|\xi|^{2})^{\frac{p+\delta \nu}{2}}(1+|\lambda|^2)^{\frac{h}{2}} \int_{\mathbb{R}^d}(1+|\eta|^{2})^{\frac{p-(1-\delta) \nu}{2}}d \eta,
	$$
	where $p=\max(m-\rho|\alpha|+\delta|\beta|, 0)$ and $\nu$ is sufficiently large. This provides us with the desired estimates for the derivatives of $\sigma(A)(x, \xi,\lambda)$ of the type \eqref{mu}. It remains to estimate the remainder term.

	Inserting in \eqref{exp} the expression $r_N$ (formula \eqref{remain}) for $\sigma(x, x+z, \xi+\eta,\lambda)$ and interchanging the orders of integration over $t$ and over $z, \eta$, we see that it is essential to establish a uniform estimate in $t \in(0,1]$ for the integral
	$$
	R_{\alpha, t}(x, \xi,\lambda)=\int_{\mathbb{R}^d}\int_{\mathbb{R}^d} e^{-i z \cdot \eta} \eta^\alpha \partial_{\xi}^\alpha \sigma(x, x+z, \xi+t \eta,\lambda) \ d z \ \bar{d} \eta,
	$$
	where $|\alpha|=N$. By integrating by parts again, we obtain
	\begin{equation} \label{remain2}
		R_{\alpha, t}(x, \xi,\lambda)=\int_{\mathbb{R}^d}\int_{\mathbb{R}^d} e^{-i z \cdot \eta} \partial_{\xi}^\alpha D_z^\alpha \sigma(x, x+z, \xi+t \eta,\lambda) \ d z \ \bar{d} \eta,
	\end{equation}
	
	We will now decompose the integral presented in equation \eqref{remain2} into two distinct parts
	\begin{equation} \label{remain3}
		R_{\alpha, t}=R_{\alpha, t}^{\prime}+R_{\alpha, t}^{\prime \prime},
	\end{equation}
	here in $R_{\alpha, t}^{\prime}$ denotes integration over the set $\{(z, \eta):|\eta|  \le\frac{1}{2}|\xi| \}$, while $R_{\alpha, t}^{\prime \prime}$ refers to integration over its complement set. Note that, if $|\eta| \le\frac{1}{2}|\xi| $, then $\frac{1}{2}|\xi|   \le|\xi+t \eta| \le \frac{3}{2}|\xi|$. Furthermore, within $R_{\alpha, t}^{\prime}$ ,the volume of integration domain concerning $\eta$ does not exceed $C|\xi|^d$, thus,
	\begin{equation} \label{remain4}
		\big\|R_{\alpha, t}^{\prime}(x, \xi,\lambda)\big\|_{B(X)} \le C(1+|\xi|^{2})^{\frac{m-(\rho-\delta) N+d}{2}} (1+|\lambda|^2)^{\frac{h}{2}} ,
	\end{equation}
	where $C$ is independent of $\xi$ and $t$.
	
	Next, let us estimate $R_{\alpha, t}^{\prime \prime}$. By integrating by parts once more and employing the formula
	$$
	(1+|\eta|^{2})^{-\frac{\nu}{2}}(1+|D_z|^{2})^{\frac{\nu}{2}} e^{-i z \cdot \eta}=e^{-i z \cdot \eta},
	$$
	where $\nu$ is an even and non-negative number, then $R_{\alpha, t}^{\prime \prime}$ can be written as a finite sum of terms of the form
	\begin{equation} \label{finitesum}
		\begin{split}
			R_{\alpha, \beta ,t}(x, \xi,\lambda) &=
			\int_{|\eta|>\frac{1}{2}|\xi|} e^{-i z \cdot \eta}(1+|\eta|^{2})^{-\frac{\nu}{2}} \\
			&\qquad \cdot \partial_{\xi}^\alpha D_z^{\alpha+\beta} \sigma(x, x+z, \xi+t \eta,\lambda) \, dz \, \bar{d}\eta .
		\end{split}
	\end{equation}
	where $|\beta| \le \nu$.
	
	For $|\eta|\ge\frac{1}{2}|\xi| $, $\big\|\partial_{\xi}^\alpha D_z^{\alpha+\beta} \sigma(A)(x, x+z, \xi+t \eta,\lambda)\big\|_{B(X)}$
	is estimated by $C(1+|\eta|^{2})^{\frac{m-(\rho-\delta) N+\delta \nu}{2}}(1+|\lambda|^2)^{\frac{h}{2}}$ if $m-(\rho-\delta) N+\delta \nu \ge 0$, and estimated by $C(1+|\lambda|^2)^{\frac{h}{2}}$ if $m-(\rho-\delta) N+\delta \nu<0$ (in both cases $C$ is independent of $\xi, \eta$ and $t$ ). Then we obtain from \eqref{finitesum} that for sufficiently large $\nu$
	$$
	\big\|R_{\alpha, \beta ,t}(x, \xi,\lambda)\big\|_{B(X)} \le C (1+|\lambda|^2)^{\frac{h}{2}} \int_{|\eta|>\frac{1}{2}|\xi|  }(1+|\eta|^{2})^{\frac{p-(1-\delta) \nu}{2}} d \eta,
	$$
	where $p=\max \{m-(\rho-\delta) N, 0\}$. If $p-(1-\delta) \nu+d+1<0$, it follows that
	\begin{align*}
		\big\|R_{\alpha, \beta, t}(x, \xi,\lambda)\big\|_{B(X)}  &\le C(1+|\xi|^{2})^{\frac{p-(1-\delta) \nu+d+1}{2}} (1+|\lambda|^2)^{\frac{h}{2}} \int_{\mathbb{R}^d}(1+|\eta|^{2})^{\frac{-d-1}{2}} d \eta \\
		&\le C(1+|\xi|^{2})^{\frac{p-(1-\delta) \nu+d+1}{2}} (1+|\lambda|^2)^{\frac{h}{2}},
	\end{align*}
	where $C$ is independent of $x, \xi$ and $t$ $(t \in(0,1])$.
	
	Choosing a sufficiently large  $\nu$ and taking \eqref{remain3} and \eqref{remain4} into account, we obtain the estimate for $R_{\alpha, t}$
	$$
	\big\|R_{\alpha, t}(x, \xi,\lambda)\big\|_{B(X)} \le C(1+|\xi|^{2})^{\frac{m-(\rho-\delta) N+d}{2}}(1+|\lambda|^2)^{\frac{h}{2}}, t \in(0,1],
	$$
	which guarantees the applicability under Proposition \ref{asy after}, thereby concluding our proof.
\end{proof}

\begin{theorem}\label{adformula}
	Let $A_{\lambda}\in \Psi_{\rho, \delta;h}^m\big(\mathbb{R}^d\times \Lambda; B(X)\big)$, whose symbol is $ \sigma(A)(x, \xi,\lambda)$. Then $A^*_{\lambda}\in \Psi_{\rho, \delta;h}^m\big(\mathbb{R}^d\times \Lambda; B(X^*)\big)$ whose symbol satisfies
	\begin{equation} \label{compo}
		\sigma(A^{*})(x, \xi,\lambda) \sim \sum_\alpha \frac{1}{\alpha !} \partial_{\xi}^\alpha D_x^\alpha \sigma(A)(x, \xi,\lambda)^{*}.
	\end{equation}
\end{theorem}

\begin{proof}
	$A_{\lambda}$ is given by
	$$
	A_{\lambda}f(x)=\int_{\mathbb{R}^d}\int_{\mathbb{R}^d}  \sigma(A)(x, \xi,\lambda) f(y) e^{i(x-y) \cdot \xi} d y  \ \bar{d} \xi.
	$$
	For any $f \in \mathscr{S}(\mathbb{R}^d; X)$ and $g \in \mathscr{S}(\mathbb{R}^d; X^{\ast})$, by the duality relation
	$$
	\langle A_{\lambda} f, g\rangle=\langle f,A^*_{\lambda} g\rangle,
	$$
	we have
	$$
	A^*_{\lambda}g(x)=\int_{\mathbb{R}^d}\int_{\mathbb{R}^d}  \sigma(A)(y, \xi,\lambda)^{*} g(y) e^{i(x-y) \cdot \xi} d y  \ \bar{d} \xi.
	$$
	Applying Theorem \ref{asympthm}, we obtain
	$$
	\sigma(A^{*})(x, \xi,\lambda) \sim \sum_\alpha \frac{1}{\alpha !} \partial_{\xi}^\alpha D_x^\alpha \sigma(A)(x, \xi,\lambda)^{*},
	$$
	completing the proof.
\end{proof}

\begin{theorem}\label{asympcomthm}
	Let $A_{\lambda}\in \Psi_{\rho, \delta;h_{1}}^{m_{1}}\big(\mathbb{R}^d\times \Lambda; B(X)\big)$ and $B_{\lambda}\in \Psi_{\rho, \delta;h_{2}}^{m_{2}}\big(\mathbb{R}^d\times \Lambda; B(X)\big)$, their symbols are $\sigma(A)(x, \xi,\lambda)$ and $\sigma(B)(x, \xi,\lambda)$ respectively. Then the composition $C_{\lambda} = B_{\lambda} \circ A_{\lambda}\in \Psi_{\rho, \delta;h_{1}+h_{2}}^{m_{1}+m_{2}}\big(\mathbb{R}^d\times \Lambda; B(X)\big)$ is an operator-valued $\Psi$DO with parameter, whose symbol satisfies the relation
	
	$$
	\sigma(C)(x, \xi,\lambda) \sim \sum_\alpha \frac{1}{\alpha !}\partial_{\xi}^\alpha \sigma(B)(x, \xi,\lambda) D_x^\alpha \sigma(A)(x, \xi,\lambda).
	$$
\end{theorem}

\begin{proof}
	For $f \in \mathscr{S}(\mathbb{R}^d;X)$,
	$$
	(B_{\lambda} \circ A_{\lambda}) f=\int_{\mathbb{R}^d} \sigma(B)(x, \xi,\lambda) \widehat{A_{\lambda}f} (\xi) e^{ix\cdot\xi} \bar{d} \xi.
	$$
	It follows from \cite[Proposition 2.3]{XX19} that $(A^{\ast}_{\lambda}) ^{\ast}$ is continuous on $\mathscr{S}^{\prime}(\mathbb{R}^d; X^{\ast\ast})$. Restricting $f \in \mathscr{S}(\mathbb{R}^d;X)$ we have
	$$
	A_{\lambda} f=\big(\big( A^{\ast}_{\lambda}\big) ^{\ast} f\big)=\int_{\mathbb{R}^d}\int_{\mathbb{R}^d} \sigma(A^{\ast})(y,\xi,\lambda)^{\ast} f(y) e^{i( x-y)\cdot \xi}  d y \ \bar{d} \xi,
	$$
	thus,
	$$
	\widehat{A_{\lambda} f}(\xi)=\int_{\mathbb{R}^d} \sigma(A^{\ast})(y,\xi,\lambda)^{\ast} f(y) e^{-i y\cdot\xi}  d y,
	$$
	then we obtain
	$$
	(B_{\lambda} \circ A_{\lambda}) f=\int_{\mathbb{R}^d}\int_{\mathbb{R}^d} \sigma(B)(x, \xi,\lambda) \sigma(A^{\ast})(y,\xi,\lambda)^{\ast} f(y) e^{i( x-y)\cdot\xi} d y \ \bar{d} \xi,
	$$
	therefore
	\begin{align*}
		\sigma(B \circ A)(x, \xi,\lambda) & \left.\sim \sum_\alpha \frac{1}{\alpha !} \partial_{\xi}^\alpha D_y^\alpha \sigma(B)(x, \xi,\lambda) \sigma(A^{\ast})(y,\xi,\lambda)^{\ast}\right|_{y=x} \\
		& =\sum_\alpha \frac{1}{\alpha !} \partial_{\xi}^\alpha\left[\sigma(B)(x, \xi,\lambda) D_x^\alpha \sigma(A^{\ast})(x,\xi,\lambda)^{\ast}\right].
	\end{align*}
	Using $\sigma(A^{\ast})(x,\xi,\lambda)^{\ast} \sim \sum_\beta \frac{(-1)^{|\beta|}}{\beta !}\partial_{\xi}^\beta D_x^\beta \sigma(A)(x, \xi,\lambda)$, we have
	$$
	\sigma(B \circ A)(x, \xi,\lambda) \sim \sum_{\alpha, \beta} \frac{(-1)^{|\beta|}}{\alpha ! \beta !} \partial_{\xi}^a\left[\sigma(B)(x, \xi,\lambda) \partial_{\xi}^\beta D_x^{\alpha+\beta} \sigma(A)(x, \xi,\lambda)\right].
	$$
	Moreover, using the following binomial formula
	$$
	\sum_{\alpha+\beta=\gamma} \frac{\eta^\alpha \theta^\beta}{\alpha ! \beta !}=\frac{(\eta+\theta)^\gamma}{\gamma !},
	$$
	we obtain that $\sum_{\alpha+\beta=\gamma} \frac{(-1)^{|\beta|}}{\alpha ! \beta !}=0$ when $|\gamma| \neq 0$ .
	Denote $D_x^{\alpha+\beta} \sigma(A)(x, \xi,\lambda)$ by $q(x, \xi,\lambda)$, then
	
	\begin{align*}
		\sigma(B \circ A)(x, \xi,\lambda) & \sim \sum_{\alpha, \beta} \frac{(-1)^{|\beta|}}{\alpha ! \beta !} \partial_{\xi}^\alpha\left[\sigma(B)(x, \xi,\lambda)\partial_{\xi}^\beta q(x, \xi,\lambda)\right] \\
		& =\sum_\gamma \sum_{\alpha+\beta=\gamma} \frac{(-1)^{|\beta|}}{\alpha ! \beta !} \partial_{\xi}^\alpha\left[\sigma(B)(x, \xi,\lambda) \partial_{\xi}^\beta q(x, \xi,\lambda)\right] \\
		& =\sum_\gamma \sum_{\alpha_1+\alpha_2+\beta=\gamma} \frac{(-1)^{|\beta|}}{\alpha_{1} ! \alpha_{2} ! \beta !} \partial_{\xi}^{\alpha_{1}} \sigma(B)(x, \xi,\lambda) \cdot \partial_{\xi}^{\alpha_{2}+\beta} q(x, \xi,\lambda)  \\
		& =\sum_\gamma \sum_{\alpha_1}\left[\frac{\partial_{\xi}^{\alpha_{1}} \sigma(B)(x, \xi,\lambda)}{\alpha_{1} !} \sum_{\alpha_2+\beta=\gamma-\alpha_1} \frac{(-1)^{|\beta|}}{\alpha_{2} ! \beta !} \partial_{\xi}^{\alpha_2+\beta} q(x, \xi,\lambda)\right].
	\end{align*}
	
	Since $\sum_{\alpha_2+\beta=\gamma-\alpha_1} \frac{(-1)^{|\beta|}}{\alpha_{2} ! \beta !}=0$ when $\gamma-\alpha_1 \neq 0$ , we get
	\begin{align*}
		\sigma(B \circ A)(x, \xi,\lambda) & \sim \sum_\gamma \frac{1}{\gamma !} \partial_{\xi}^\gamma \sigma(B)(x, \xi,\lambda) \cdot q(x, \xi,\lambda) \\
		& =\sum_\gamma \frac{1}{\gamma !} \partial_{\xi}^\gamma \sigma(B)(x, \xi,\lambda) D_x^\gamma \sigma(A)(x, \xi,\lambda).
	\end{align*}
\end{proof}

\subsection{$L_{2}$-boundedness}\label{section l2}

In the sequel, we will mainly consider $\Psi$DOs (with parameter) whose symbols take values in some von Neumann algebran $\mathcal{M}$. If we take $X=L_1(\mathcal{M})+\mathcal{M}$, then $\mathcal{M}$ admits an isometric embedding into $B(X)$ by left multiplication. In this way, these $\mathcal{M}$-valued symbols can be seen as a special case of the $B(X)$-valued symbols defined in the previous section.
In the sequel, the semi-commutative algebra $L_{\infty}(\mathbb{R}^d) \overline{\otimes} \mathcal{M}$ with tensor trace, denoted by $\mathcal{N}$ .

As in Section \ref{section sym}, we now provide the definitions for $\mathcal{M}$-valued $\Psi$DOs with parameter.

\begin{definition}
	Let $m, \rho, \delta, h$ be real numbers with $0 \le \delta \le \rho \le 1$ (excluding the case $(\rho, \delta) = (1,1)$). The class $S_{\rho, \delta ; h}^m\big(\mathbb{R}^d \times \mathbb{R}^d \times \Lambda ; \mathcal{M}\big)$ consists of functions $\sigma(x, \xi, \lambda)$ satisfying the following conditions:
	
	\begin{enumerate}[$\rm (i)$]
		\item For every fixed $\lambda_0 \in \Lambda$, $\sigma\big(x, \xi, \lambda_0\big) \in C^{\infty}(\mathbb{R}^d \times \mathbb{R}^d ; \mathcal{M})$.
		
		\item For arbitrary multi-indices $\alpha$ and $\beta$, there exist constants $C_{\alpha, \beta} > 0$ such that
		$$
		\big\|\partial_\xi^\alpha \partial_x^\beta \sigma(x, \xi, \lambda)\big\|_{\mathcal{M}} \le C_{\alpha, \beta}(1+|\xi|^2)^{\frac{m-\rho|\alpha|+\delta|\beta|}{2}} (1+|\lambda|^2)^{\frac{h}{2}},
		$$
		for all $x \in \mathbb{R}^d$, $\xi \in \mathbb{R}^d$, and $\lambda \in \Lambda$.
	\end{enumerate}
\end{definition}

\begin{definition}
	Let $\sigma(x,\xi,\lambda) \in S_{\rho, \delta;h }^m\big(\mathbb{R}^d\times \mathbb{R}^d\times \Lambda; \mathcal{M}\big)$. For a function $f \in \mathscr{S}(\mathbb{R}^d;L_{2}(\mathcal{M}))$, the $\Psi$DO $A_{\lambda}$ is defined as the mapping $f \mapsto A_{\lambda} f$, given by
	\begin{equation}\label{M valued pdo}
		A_{\lambda}f(x) = \int_{\mathbb{R}^d}  \sigma(x, \xi,\lambda) \hat{f} (\xi) e^{ix \cdot \xi} \bar{d} \xi.
	\end{equation}
	
\end{definition}

As for $B(X)$-valued setting, the symbol $\sigma(x, \xi) \in S_{\rho, \delta}^m\big(\mathbb{R}^d \times \mathbb{R}^d ;\mathcal{M}\big)$ can be interpreted as an element of $S_{\rho,\delta;0 }^m\big(\mathbb{R}^d \times \mathbb{R}^d \times \Lambda ; \mathcal{M}\big)$ that is independent of $\lambda$. Typical examples of $\mathcal{M}$-valued $\Psi$DOs are Fourier multipliers and pointwise multipliers by $f \in C_c^\infty (\mathbb{R}^d; \mathcal{M})$:
\begin{enumerate}[$\rm (i)$]
	\item The Bessel potential $J^\alpha=(1-\Delta )^{\frac{\alpha}{2}}$ for $\alpha \in \mathbb{R}$, which is initially defined on $L_2(\mathbb{R}^d)$, is a $\Psi$DO with symbol $J_\alpha(\xi)=(1+|\xi|^2)^{\frac{\alpha}{2}}$, with order $\alpha$.
	\item More generally, Fourier multipliers $T_\phi$ which are initially defined on $L_2(\mathbb{R}^d)$ as $T_\phi (f) =  \mathcal{F}^{-1} ( \phi \hat{f})$ are $\Psi$DOs with symbol $\phi(\xi)$, if $\phi\in C^\infty (\mathbb{R}^d )$ satisfies $|\partial_{\xi}^{\alpha}\phi(\xi)|\lesssim (1+|\xi|^2)^{\frac{m-|\alpha|}{2}}$ for some $m.$ The order of $T_\phi$ is $m.$
	\item For $f\in C_c^\infty (\mathbb{R}^d; \mathcal{M})$, define $M_f(g)= fg $ for $g \in \mathscr{S}(\mathbb{R}^d;L_{2}(\mathcal{M}))$. Then $M_f$ is a $\Psi$DO with symbol $f(x)$, with order $0$.
\end{enumerate}

In the following, let us state the $L_{2}$-boundedness of $\mathcal{M}$-valued $\Psi$DOs with parameter. In the commutative case, the $L_{2}$-boundedness of $\Psi$DOs is proved using the Cotlar-Stein Almost Orthogonality Lemma \cite{Co55,St93}, which can be found in numerous works, such as \cite{RT10,St93,Ta81}. In \cite{XX19}, the authors prove the $L_{2}$-boundedness of $\mathcal{M}$-valued $\Psi$DOs, also with the help of Cotlar-Stein Almost Orthogonality Lemma. Since the proof in \cite{XX19} easily extends to the current parametric situation, we record the $L_{2}$-boundedness below without proof.

\begin{proposition}\label{l2bdd}
	Let $\sigma(A)(x, \xi, \lambda) \in S_{\delta, \delta; h}^0\big(\mathbb{R}^d \times \mathbb{R}^d \times \Lambda ; \mathcal{M}\big)$ with $0 \le \delta < 1$ and $h \le 0$. Then $A_\lambda$ is bounded on $L_2(\mathcal{N})$.
\end{proposition}

\begin{rk}
	Let $h \le 0$ and $0 \le \delta \le \rho \le 1$ (excluding the case $(\rho, \delta) = (1,1)$). Since
	$$
	S_{\rho, \delta;h}^0\big(\mathbb{R}^d \times \mathbb{R}^d \times \Lambda ; \mathcal{M}\big) \subset S_{\delta, \delta;h}^0\big(\mathbb{R}^d \times \mathbb{R}^d \times \Lambda ; \mathcal{M}\big) \cap S_{\rho, \rho;h}^0\big(\mathbb{R}^d \times \mathbb{R}^d \times \Lambda ; \mathcal{M}\big),
	$$
	it follows that $A_\lambda \in B(L_2(\mathcal{N}))$ whenever $\sigma(A)(x, \xi,\lambda) \in S_{\rho, \delta;h}^0\big(\mathbb{R}^d \times \mathbb{R}^d \times \Lambda ; \mathcal{M}\big)$.
\end{rk}

Building upon the $L_2$-boundedness of operators $A_\lambda$ established in Proposition \ref{l2bdd}, we now investigate their intrinsic algebraic structure within the von Neumann algebra $B(L_2(\mathbb{R}^d)) \overline{\otimes} \mathcal{M}$. This progression elevates the analytic boundedness to operator algebraic containment.

\begin{theorem}\label{thm-tensor algebra}
	Let $\sigma(A)(x, \xi,\lambda) \in S_{\rho, \delta, h}^0\big(\mathbb{R}^d \times \mathbb{R}^d \times \Lambda ; \mathcal{M}\big)$ with $h \le 0$ and $0 \le \delta \le \rho \le 1$ (excluding the case $(\rho, \delta) = (1,1)$). Then
	$$
	A_\lambda \in B(L_2(\mathbb{R}^d)) \overline{\otimes} \mathcal{M}.
	$$
\end{theorem}

\begin{proof}
	The closed tensor product $B(L_2(\mathbb{R}^d)) \overline{\otimes} \mathcal{M}$ is defined as the weak operator topology (WOT) completion of the algebraic tensor product. By the double commutant theorem for von Neumann algebras, we obtain
	$$
	(B(L_2(\mathbb{R}^d)) \overline{\otimes} \mathcal{M})'' = B(L_2(\mathbb{R}^d)) \overline{\otimes} \mathcal{M}.
	$$
	The commutant of this tensor product is given by
	$$
	(B(L_2(\mathbb{R}^d)) \overline{\otimes} \mathcal{M})' = I \overline{\otimes} \mathcal{M}' \quad (I:= I_{L_2(\mathbb{R}^d)}).
	$$
	
	For any $a' \in \mathcal{M}'$ and $f \in L_2(\mathcal{N})$, we compute
	$$
	A_\lambda(I\otimes a')f(x) = \int_{\mathbb{R}^d} \sigma(x, \xi, \lambda) a' \hat{f}(\xi)  e^{i x \cdot \xi}  \bar{d} \xi = (I\otimes a') A_\lambda f(x).
	$$
	This computation shows that $A_\lambda$ commutes with all elements of the form $1 \overline{\otimes} \mathcal{M}'$, implying that
	$$
	A_\lambda \in (I\overline{\otimes} \mathcal{M}')'.
	$$
	Since $B(L_2(\mathbb{R}^d))$ is a factor (i.e., its commutant is trivial), it follows that
	$$
	(I \overline{\otimes} \mathcal{M}')' = B(L_2(\mathbb{R}^d)) \overline{\otimes} \mathcal{M}.
	$$
	Thus, we conclude that
	$$
	A_\lambda \in B(L_2(\mathbb{R}^d)) \overline{\otimes} \mathcal{M}.
	$$
\end{proof}

\begin{corollary}\label{inver coro}
	Let $A_\lambda \in \Psi_{1,0; h}^m\big(\mathbb{R}^d \times \Lambda ; \mathcal{M}\big)$ with $h \le 0$ and $m \le 0$. Then
	$$
	\big\|A_\lambda\big\|_{B(L_2(\mathbb{R}^d)) \overline{\otimes} \mathcal{M}} \le C\big(1+|\lambda|\big)^{h}.
	$$
\end{corollary}

\begin{proof}
	Consider the operator $\Phi_{m, h}(\lambda) \in \Psi_{1,0;h}^m\big(\mathbb{R}^d \times \Lambda; \mathcal{M}\big)$ whose symbol is given by $\varphi_m(x, \xi, \lambda) = (1+|\xi|^2)^{\frac{m}{2}}(1+|\lambda|^2)^{ \frac{h}{2}}$. The operator $\Phi_{m, h}(\lambda)$ acts as a multiplication operator by $(1+|\xi|^2)^{\frac{m}{2}}(1+|\lambda|^2)^{ \frac{h}{2}}$ on the Fourier transform $\hat{u}(\xi)$ in $L_2(\mathcal{N})$, and thus
	$$
	\big\|\Phi_{m, h }(\lambda)\big\|_{B(L_2(\mathbb{R}^d)) \overline{\otimes} \mathcal{M}} = \sup_{\xi \in \mathbb{R}^d}(1+|\xi|^2)^{\frac{m}{2}}(1+|\lambda|^2)^{ \frac{h}{2}} \le C\big(1+|\lambda|\big)^{h}.
	$$
	
	For $A_\lambda \in \Psi_{1,0;h}^m\big(\mathbb{R}^d \times \Lambda ; \mathcal{M}\big)$, by Theorem \ref{asympcomthm}, $\Phi_{-m, -h} \cdot A_\lambda\in \Psi_{1,0;0}^0\big(\mathbb{R}^d \times \Lambda; \mathcal{M}\big)$. By Theorem \ref{thm-tensor algebra}, we deduce that $\Phi_{-m,-h}(\lambda)\cdot A_\lambda \in B(L_2(\mathbb{R}^d)) \overline{\otimes} \mathcal{M}$, i.e.,
	$$
	\big\|\Phi_{-m,-h}(\lambda)\cdot A_\lambda\big\|_{B(L_2(\mathbb{R}^d)) \overline{\otimes} \mathcal{M}} \le C.
	$$
	Therefore,
	\begin{align*}
		\big\|A_\lambda\big\|_{B(L_2(\mathbb{R}^d)) \overline{\otimes} \mathcal{M}} &= \big\|\Phi_{ m, h}(\lambda) \cdot \big(\Phi_{-m,-h} (\lambda) \cdot A_\lambda\big)\big\|_{B(L_2(\mathbb{R}^d)) \overline{\otimes} \mathcal{M}} \\
		&\le \big\|\Phi_{ m, h}(\lambda)\big\|_{B(L_2(\mathbb{R}^d)) \overline{\otimes} \mathcal{M}} \cdot \big\|\Phi_{-m,-h}(\lambda) A_\lambda\big\|_{B(L_2(\mathbb{R}^d)) \overline{\otimes} \mathcal{M}} \\
		&\le C\big(1+|\lambda|\big)^{h}.
	\end{align*}
\end{proof}

\section{ Resolvent analysis for $\mathcal{M}$-valued elliptic classical  $\Psi$DOs}\label{section resol}

This section is devoted to the resolvent theory of $\mathcal{M}$-valued elliptic $\Psi$DOs. We construct parametrices and obtain a precise description of the resolvent as a $\Psi$DOs with parameter. These results extend the classical elliptic theory to the operator-valued setting and provide the analytic basis for defining complex powers.

From now on, we focus on $\Psi$DOs with symbols in the class $S_{\rho, \delta}^m $ for the special case $(\rho, \delta )  =  (1, 0)$, and drop the index $(1, 0 )$.

\subsection{$\mathcal{M}$-valued elliptic  $\Psi$DOs}

As discussed at the beginning of Section \ref{section l2}, $\mathcal{M}$-valued symbols can be regarded as a special case of $B(X)$-valued symbols. It is sometimes convenient to consider a narrower class of $\mathcal{M}$-valued $\Psi$DOs that are closed under most operational conditions.

\begin{definition}
	Let $m \in \mathbb{C}$.
	Fix a smooth function
	\begin{equation}
		\label{eq-cutoff}
		\varphi \in C^{\infty}(\mathbb{R}^d),
		\quad
		\varphi(\xi) =
		\begin{cases}
			0, & |\xi| \le \frac12,\\[1mm]
			1, & |\xi| \ge 1,
		\end{cases}
		\quad 0 \le \varphi(\xi) \le 1 \text{ for all } \xi \in \mathbb{R}^d.
	\end{equation}
	A smooth function $\sigma(x, \xi) \in C^{\infty}(\mathbb{R}^d \times \mathbb{R}^d; \mathcal{M})$ is called an $\mathcal{M}$-valued  classical symbol of order $m$, denoted by
	$$
	\sigma(x,\xi) \in CS^m\big(\mathbb{R}^d \times \mathbb{R}^d; \mathcal{M}\big),
	$$
	if there exists a sequence of symbols $\big\{\sigma_{m-j}(x,\xi)\big\}_{j\ge 0}$, with each $\sigma_{m-j}(x,\xi)$ positively homogeneous of degree $m-j$ in $\xi$ for $\xi \neq 0$, such that
	$$
	\sigma(x,\xi) \sim \sum_{j=0}^{\infty} \varphi(\xi)\, \sigma_{m-j}(x,\xi).
	$$
	
	Here, $\varphi$ is the fixed cutoff function \eqref{eq-cutoff}, whose sole purpose is to localize the homogeneous components away from the singularity at $\xi=0$, making $\varphi(\xi)\sigma_{m-j}(x,\xi)$ globally smooth. Although no radiality condition is required, it is common to choose $\varphi(\xi)=\chi(|\xi|)$ with $\chi \in C^\infty(\mathbb{R}_+)$ such that $\chi(r)=0$ for $r\le \tfrac12$ and $\chi(r)=1$ for $r\ge 1$.
\end{definition}

\begin{definition}
	Denote by $\mathrm{C}\Psi^m\big(\mathbb{R}^d;\mathcal{M}\big)$ the class of $\mathcal{M}$-valued $\Psi$DOs that can be written in the form \eqref{M valued pdo} (independent of $\lambda$) with $\sigma(x,\xi) \in CS^m\big(\mathbb{R}^d \times \mathbb{R}^d;\mathcal{M}\big)$. These operators are referred to as $\mathcal{M}$-valued classical $\Psi$DOs.
\end{definition}

\begin{remark}
	It is easily verified that
	$$
	CS^m(\mathbb{R}^d \times \mathbb{R}^d;\mathcal{M}) \subset S^{\Re m}(\mathbb{R}^d \times \mathbb{R}^d;\mathcal{M}).
	$$
\end{remark}

By adopting the same approach used for the $B(X)$-valued case, an $\mathcal{M}$-valued $\Psi$DO $A$ is defined as a linear and continuous mapping from $\mathscr{S}(\mathbb{R}^d ; L_{2}(\mathcal{M}))$ to itself. Since $\mathscr{S}(\mathbb{R}^d ; L_{2}(\mathcal{M}))$ is dense in $L_2(\mathcal{N})$, the $\mathcal{M}$-valued $\Psi$DO $A$ is densely defined. As discussed in Section 3, when $m \le 0$, the $\mathcal{M}$-valued $\Psi$DO $A$ (with parameter for $h \le 0$ or independent of the parameter) is bounded on $L_{2}(\mathcal{N})$.
When $\Re m > 0$, the $\mathcal{M}$-valued $\Psi$DO $A$ becomes an unbounded operator on $L_{2}(\mathcal{N})$. It follows from \eqref{dual-product} that the adjoint $A^*$ is uniquely determined by the duality relation
$$
\langle A u, v\rangle = \langle u, A^* v\rangle \quad \forall u \in \mathscr{S}(\mathbb{R}^d; L_2(\mathcal{M})), \, v \in D(A^*),
$$
where $D(A^*)$ consists of all $v \in L_2(\mathcal{N})$ for which the map $u \mapsto \langle A u, v\rangle$ is continuous in the $L_2(\mathcal{N})$-norm.

The class of $\mathcal{M}$-valued classical  $\Psi$DOs remains stable under both composition and adjoint operations. Specifically, the following proposition holds:

\begin{proposition}\label{cl M}
	\begin{enumerate}[$\rm (i)$]
		\item If $A \in \mathrm{C}\Psi^m\big(\mathbb{R}^d; \mathcal{M}\big)$, then $A^* \in \mathrm{C}\Psi^{\bar{m}}\big(\mathbb{R}^d; \mathcal{M}\big)$ with
		$$
		\sigma(A^{\ast})_{\bar{m}}(x,\xi) = \sigma(A)_{m}(x,\xi)^{\ast}.
		$$
		
		\item If $A \in \mathrm{C}\Psi^{m_1}\big(\mathbb{R}^d; \mathcal{M}\big)$ and $B \in \mathrm{C}\Psi^{m_2}\big(\mathbb{R}^d; \mathcal{M}\big)$, then $B A \in \mathrm{C}\Psi^{m_1+m_2}\big(\mathbb{R}^d; \mathcal{M}\big)$  with
		$$
		\sigma(BA)_{m_{1}+m_{2}}(x,\xi) = \sigma(B)_{m_{2}}(x,\xi)\sigma(A)_{m_{1}}(x,\xi).
		$$
		
	\end{enumerate}
\end{proposition}

\begin{proof}
	$\rm (i)$ Applying Theorem \ref{adformula}, we obtain
	$$
	\sigma(A^{\ast})(x,\xi)  \sim   \sum_{j=0}^{\infty} \sum_{|\alpha|+k=j}  \frac{1}{\alpha !} \partial_{\xi}^\alpha D_x^\alpha  \sigma(A)_{m-k}(x, \xi)^{*}.
	$$
	For fixed $j$, $\partial_{\xi}^\alpha D_x^\alpha  \sigma(A)_{m-k}(x, \xi)^{*}$ is  positively homogeneous of degree $\bar{m}-j$ in $\xi$, so $A^* \in \mathrm{C}\Psi^{\bar{m}}\big(\mathbb{R}^d; \mathcal{M}\big)$. The leading term of the above asymptotic expansion gives
	$$
	\sigma(A^{\ast})_{\bar{m}}(x,\xi) = \sigma(A)_{m}(x,\xi)^{\ast}.
	$$

	$\rm (ii)$ Using Theorem \ref{asympcomthm}, we derive
	\begin{equation}\label{compo classical}
		\sigma(BA)_{m_{1}+m_{2}}(x,\xi)   \sim  \sum_{j=0}^{\infty}\sum_{|\alpha|+k+l=j} \frac{1}{\alpha !}\partial_{\xi}^\alpha \sigma(B)_{m_{2}-k}(x, \xi) D_x^\alpha \sigma(A)_{m_{1}-l}(x, \xi).
	\end{equation}
	Similar to $\rm (i)$, it is straightforward to see that all terms in the above asymptotic expansion are homogeneous, so $B A \in \mathrm{C}\Psi^{m_1+m_2}\big(\mathbb{R}^d; \mathcal{M}\big)$
	with leading term
	$$
	\sigma(BA)_{m_{1}+m_{2}}(x,\xi) = \sigma(B)_{m_{2}}(x,\xi)\sigma(A)_{m_{1}}(x,\xi).
	$$
\end{proof}

Next, we introduce the concept of ellipticity for $\mathcal{M}$-valued classical  $\Psi$DOs.

We begin by discussing ellipticity in the context of scalar-valued classical $\Psi$DOs.

\begin{lemma}\label{scalar elliptic}
	Let $\sigma(A)_{m}(x, \xi) \in \mathbb{C}$ denote the scalar principal symbol of a classical $\Psi$DO of order $m$. The following conditions are equivalent:
	
	\begin{enumerate}[$\rm (i)$]
		\item $\sigma(A)_{m}(x, \xi) \neq 0$ for all $(x, \xi) \in \mathbb{R}^d \times (\mathbb{R}^d \setminus \{0\})$, and there exists a constant $C_1 > 0$ such that
		$$
		\left|\sigma(A)_{m}(x, \xi)^{-1}\right| \le C_1 |\xi|^{-\Re m}
		$$
		uniformly for all $(x, \xi) \in \mathbb{R}^d \times (\mathbb{R}^d \setminus \{0\})$.
		
		\item There exists a constant $C > 0$ such that
		$$
		\left|\sigma(A)_{m}(x, \xi)\right| \ge C |\xi|^{\Re m}
		$$
		uniformly for all $(x, \xi) \in \mathbb{R}^d \times (\mathbb{R}^d \setminus \{0\})$.
	\end{enumerate}
\end{lemma}

\begin{proof}
	$\rm (i)$ $\Rightarrow$ $\rm (ii)$:
	Since $\sigma(A)_{m}(x, \xi) \neq 0$, it is invertible. By the definition of the inverse, we have
	$$
	1 = \left|\sigma(A)_{m}(x, \xi) \cdot \sigma(A)_{m}(x, \xi)^{-1}\right| = \left|\sigma(A)_{m}(x, \xi)\right| \cdot \left|\sigma(A)_{m}(x, \xi)^{-1}\right|.
	$$
	Using the bound from $\rm (i)$, we deduce
	$$
	\left|\sigma(A)_{m}(x, \xi)\right| = \frac{1}{\left|\sigma(A)_{m}(x, \xi)^{-1}\right|} \ge C_1^{-1} |\xi|^{\Re m}.
	$$
	This inequality holds uniformly for all $(x, \xi) \in \mathbb{R}^d \times (\mathbb{R}^d \setminus \{0\})$, establishing (ii) with $C = C_1^{-1}$.
	
	$\rm (ii)$ $\Rightarrow$ $\rm (i)$:
	The inequality $\left|\sigma(A)_{m}(x, \xi)\right| \ge C |\xi|^{\Re m} > 0$ immediately implies that $\sigma(A)_{m}(x, \xi) \neq 0$ for all $(x, \xi) \in \mathbb{R}^d \times (\mathbb{R}^d \setminus \{0\})$. Therefore, $\sigma(A)_{m}(x, \xi)$ is invertible in $\mathbb{C}$. Using the bound from (ii), we obtain
	$$
	\left|\sigma(A)_{m}(x, \xi)^{-1}\right| = \frac{1}{\left|\sigma(A)_{m}(x, \xi)\right|} \le C^{-1} |\xi|^{-\Re m}.
	$$
	This inequality holds uniformly for all $(x, \xi) \in \mathbb{R}^d \times (\mathbb{R}^d \setminus \{0\})$, establishing (i) with $C_1 = C^{-1}$.
\end{proof}

For the $\mathcal{M}$-valued case, the conditions (i) and (ii) in Lemma \ref{scalar elliptic} now involve the operator norm $\|\cdot\|_{\mathcal{M}}$ and the operator modulus $|\sigma(A)_{m}| = (\sigma(A)_{m}^* \sigma(A)_{m})^{\frac{1}{2}}$. Specifically, they are stated as follows:
\begin{enumerate}[$\rm (E1)$]
	\item  \label{E1-M} $\sigma(A)_{m}(x, \xi)^{-1} \in \mathcal{M}$ for all $(x, \xi) \in \mathbb{R}^d \times (\mathbb{R}^d \setminus \{0\})$, and there exists a constant $C_1 > 0$ such that
	$$
	\big\|\sigma(A)_{m}(x, \xi)^{-1}\big\|_{\mathcal{M}} \le C_1 |\xi|^{-\Re m}
	$$
	uniformly for all $(x, \xi) \in \mathbb{R}^d \times (\mathbb{R}^d \setminus \{0\})$.
	\item\label{E2-M}  There exists a constant $C_2 > 0$ such that
	$$
	|\sigma(A)_{m}(x, \xi)| \ge C_2  |\xi|^{\Re m}
	$$
	uniformly for all $(x, \xi) \in \mathbb{R}^d \times (\mathbb{R}^d \setminus \{0\})$.
\end{enumerate}

The following analysis reveals a critical difference: while these conditions are equivalent in the scalar-valued case, this equivalence generally fails for operator-valued symbols.

\begin{proposition}\label{M value elliptic}
	(E1) implies (E2).
\end{proposition}

\begin{proof}
	Assume (E1) holds. Then $\sigma(A)_{m}(x, \xi)$ is invertible in $\mathcal{M}$, and
	$$
	\big\|\sigma(A)_{m}(x, \xi)^{-1}\big\|_{\mathcal{M}} \le C_1 |\xi|^{-\Re m}.
	$$
	For any vector $u \in L_2(\mathcal{M})$ with $\|u\|_{L_2(\mathcal{M})} = 1$, we have
	\begin{align*}
		&\big\|\sigma(A)_{m}(x, \xi) u\big\|_{L_2(\mathcal{M})} \\
		= &\big\|\sigma(A)_{m}(x, \xi) u\big\|_{L_2(\mathcal{M})} \cdot \big\|\sigma(A)_{m}(x, \xi)^{-1}\big\|_{\mathcal{M}} \cdot \big\|\sigma(A)_{m}(x, \xi)^{-1}\big\|_{\mathcal{M}}^{-1} \\
		\ge &C_1^{-1} |\xi|^{\Re m}.
	\end{align*}
	This implies that the smallest singular value of $\sigma(A)_{m}(x, \xi)$ is at least $C_1^{-1} |\xi|^{\Re m}$. The modulus $|\sigma(A)_{m}(x, \xi)| = (\sigma(A)_{m}^* \sigma(A)_{m})^{\frac{1}{2}}$ is a positive operator whose spectrum consists precisely of the squares of the singular values of $\sigma(A)_{m}(x, \xi)$. Consequently, the smallest point in the spectrum of $|\sigma(A)_{m}(x, \xi)|$ satisfies
	$$
	\inf \operatorname{Spec}\big(|\sigma(A)_{m}(x, \xi)|\big) \ge C_1^{-1} |\xi|^{\Re m}.
	$$
	The spectral infimum being bounded below by $C_1^{-1} |\xi|^{\Re m}$ is equivalent to the operator inequality
	$$
	|\sigma(A)_{m}(x, \xi)| \ge C_1^{-1} |\xi|^{\Re m} I.
	$$
	Thus, (E2) holds with $C_2 = C_1^{-1}$.
\end{proof}

\begin{remark}
	In general, (E2) does not imply (E1). A counterexample can be constructed on the separable Hilbert space $H=\ell^2(\mathbb{N})$ with orthonormal basis $\{e_n\}_{n\in\mathbb{N}}$. Define the unilateral shift operator $S:H\to H$ by
	$$
	S e_n = e_{n+1}, \quad n \in \mathbb{N}.
	$$
	Then $S^*S=I$, but $S$ is not surjective since $e_0$ does not belong to the range of $S$. Consider
	$$
	\sigma(A)_m(x,\xi) = S\,|\xi|^{m}.
	$$
	We have
	$$
	\|\sigma(A)_m(x,\xi)\|_{B(H)} = |\xi|^{\Re m},
	\qquad
	|\sigma(A)_m(x,\xi)| = (S^*S)^{\frac{1}{2}}\,|\xi|^{\Re m} = |\xi|^{\Re m} I.
	$$
	Thus (E2) holds with $C_2=1$. However, $\sigma(A)_m(x,\xi)$ is nowhere invertible since $S$ is not surjective, so (E1) fails.
\end{remark}

\begin{proposition}\label{sa M elliptic}
	If the principal symbol $\sigma(A)_{m}(x, \xi)$ is normal, i.e.,
	$$\sigma(A)_{m}(x, \xi) \sigma(A)_{m}(x, \xi)^* = \sigma(A)_{m}(x, \xi)^* \sigma(A)_{m}(x, \xi)$$
	for all $(x, \xi) \in \mathbb{R}^d \times (\mathbb{R}^d \setminus \{0\})$, then (E1) and (E2) are equivalent.
\end{proposition}

\begin{proof}
	For a normal operator, invertibility is equivalent to $0 \notin \operatorname{Spec}(\sigma(A)_{m})$. Since $\operatorname{Spec}(|\sigma(A)_{m}|) = |\operatorname{Spec}(\sigma(A)_{m})|$ for normal operators, $\sigma(A)_{m}$ is invertible if and only if $|\sigma(A)_{m}|$ is invertible. Norm estimates follow as in Proposition \ref{M value elliptic}.
\end{proof}

Therefore, we prefer to give the definition for $\mathcal{M}$-valued elliptic  $\Psi$DOs by (E1).

\begin{definition}\label{def M elliptic}
	An $\mathcal{M}$-valued classical  $\Psi$DO $A \in \mathrm{C}\Psi^m\big(\mathbb{R}^d; \mathcal{M}\big)$ is elliptic if its principal symbol satisfies
	$$
	\sigma(A)_m(x, \xi)^{-1} \in \mathcal{M}
	$$
	for all $(x, \xi) \in \mathbb{R}^d \times (\mathbb{R}^d \setminus \{0\})$, and there exists a constant $C > 0$ such that
	$$
	\big\|\sigma(A)_{m}(x, \xi)^{-1}\big\|_{\mathcal{M}} \le C |\xi|^{-\Re m}
	$$
	uniformly for all $(x, \xi) \in \mathbb{R}^d \times (\mathbb{R}^d \setminus \{0\})$.
\end{definition}

\subsection{The parametrices of $\mathcal{M}$-valued elliptic  $\Psi$DOs}

In the scalar-valued case, a key feature of elliptic $\Psi$DOs is the existence of parametrices. Analogously, in the $\mathcal{M}$-valued case, parametrices can be constructed for $\mathcal{M}$-valued elliptic  $\Psi$DOs.

\begin{definition}
	Let $A \in \mathrm{C}\Psi^m\big(\mathbb{R}^d; \mathcal{M}\big)$ be elliptic. If there exists an $\mathcal{M}$-valued $\Psi$DO $B$ such that
	\begin{align*}
		& B A = I + R, \quad R \in \Psi^{-\infty}\big(\mathbb{R}^d; \mathcal{M}\big), \\
		& A B = I + R', \quad R' \in \Psi^{-\infty}\big(\mathbb{R}^d; \mathcal{M}\big),
	\end{align*}
	then $B$ is called the parametrix of $A$.
\end{definition}

\begin{theorem}\label{parabelong}
	Let $A \in \mathrm{C}\Psi^m\big(\mathbb{R}^d; \mathcal{M}\big)$ be an $\mathcal{M}$-valued elliptic  $\Psi$DO. Then there exists an $\mathcal{M}$-valued $\Psi$DO $B \in \mathrm{C}\Psi^{-m}\big(\mathbb{R}^d; \mathcal{M}\big)$, such that $B$ is the parametrix of $A$.
\end{theorem}

\begin{proof}
	Let $\varphi$ be the fixed cutoff function \eqref{eq-cutoff}. Define an $\mathcal{M}$-valued $\Psi$DO $B_1$, whose symbol satisfies
	$$
	\sigma(B_1)(x,\xi) = \varphi(\xi) \sigma(A)_m(x,\xi)^{-1}.
	$$
	Then $B_1 \in \mathrm{C}\Psi^{-m}\big(\mathbb{R}^d; \mathcal{M}\big)$. For $|\xi| \ge 1$, by Theorem \ref{asympcomthm}, we have
	$$
	\sigma(B_{1} A)(x,\xi)
	\sim   \sum_{j=0}^{\infty}\sum_{|\alpha| =j} \frac{1}{\alpha !}\partial_{\xi}^\alpha \Big(\varphi(\xi)\sigma(A)_m(x, \xi)^{-1}\Big) D_x^\alpha \sigma(A) (x, \xi).
	$$
	When $|\alpha|+k=j$, the term $\partial_{\xi}^\alpha \big(\sigma(A)_m(x, \xi)^{-1}\big) D_x^\alpha \sigma(A)_{m-k}(x, \xi)$ is positively homogeneous of degree $-j$. Since $\varphi(\xi)$ does not affect the homogeneous degree, we have
	$$
	\sigma(B_{1} A)(x,\xi)
	\sim 1_\mathcal{M} + \sum_{j=1}^{\infty}\sum_{|\alpha|+k=j} \frac{1}{\alpha !}\partial_{\xi}^\alpha \big(\sigma(A)_m(x, \xi)^{-1}\big) D_x^\alpha \sigma(A)_{m-k}(x, \xi).
	$$
	By Proposition \ref{asymp}, there exists an $\mathcal{M}$-valued $\Psi$DO $R_1$ whose symbol $\sigma(R_1)(x,\xi)$ satisfies
	$$
	\sigma(R_1)(x,\xi)
	\sim \sum_{j=1}^{\infty}\sum_{|\alpha|+k=j} \frac{1}{\alpha !}\partial_{\xi}^\alpha \big(\sigma(A)_m(x, \xi)^{-1}\big) D_x^\alpha \sigma(A)_{m-k}(x, \xi).
	$$
	Thus,
	$$
	B_1 A = I + R_1, \quad R_1 \in \mathrm{C}\Psi^{-1}\big(\mathbb{R}^d; \mathcal{M}\big).
	$$
	Consequently, for any given $N \in \mathbb{Z}_{+}$, we have
	\begin{equation}\label{left-paramet}
		\big(\sum_{j=0}^{N-1} (-1)^j R_1^j\big) B_1 A = I - (-1)^{N} R_1^{N}.
	\end{equation}

	We are going to construct the required $\mathcal{M}$-valued $\Psi$DO $B \in \mathrm{C}\Psi^{-m}\big(\mathbb{R}^d; \mathcal{M}\big)$ with the help of $R_1^j$. Let $\tilde{\sigma}_j (x, \xi) $ be the symbol of $(-1)^j R_1^j  B_1$. Using the sequence $\{t_j\}$ selected in the proof of Proposition \ref{asymp}, we find a symbol of the form $ \sum_{j=0}^{\infty} \varphi(t_j^{-1}\xi) \tilde{\sigma}_j(x, \xi )$, and denote $B$ the corresponding $\Psi$DO, i.e.,
	$$\sigma(B)=  \sum_{j=0}^{\infty} \varphi(t_j^{-1}\xi) \tilde{\sigma}_j(x, \xi ).$$
	Since multiplying $\varphi(t_j^{-1}\xi)$ does not affect the homogeneous degree, by the above construction,
	$$\sigma(B)- \sigma\Big(\sum_{j= 0 }^{N-1}(-1)^j R_1^j  B_1\Big) \in CS^{-N-m}$$
	for any $N>0$. Therefore,
	$$\sigma(BA)- \sigma\Big(\sum_{j= 0 }^{N-1}(-1)^j R_1^j  B_1 A\Big) \in CS^{-N}.$$
	But by \eqref{left-paramet},
	$$\sigma\Big(\sum_{j= 0 }^{N-1}(-1)^j R_1^j  B_1 A\Big)=1_{\mathcal{M}}      -\sigma\Big( (-1)^{N+1} R_1^{N+1}\Big),$$
	forcing $\sigma(BA)-1_{\mathcal{M}}\in CS^{-N-1} $ for any $N>0$. Consequently, $B A = I + R$ with $ R \in \Psi^{-\infty}\big(\mathbb{R}^d; \mathcal{M}\big)$.

	Using the same construction method as for $B$, we can similarly construct an $\mathcal{M}$-valued $\Psi$DO $B^{\prime} \in \mathrm{C}\Psi^{-m}\big(\mathbb{R}^d; \mathcal{M}\big)$ such that
	$$
	A B^{\prime} = I + R^{\prime}, \quad R^{\prime} \in \Psi^{-\infty}\big(\mathbb{R}^d; \mathcal{M}\big).
	$$
	It follows that
	$$
	B - B^{\prime} = R B^{\prime} - B R^{\prime} \in \Psi^{-\infty}\big(\mathbb{R}^d; \mathcal{M}\big).
	$$
	Hence, either $B$ or $B^{\prime}$ can be chosen as the parametrix of $A$, which is unique modulo $\Psi^{-\infty}\big(\mathbb{R}^d; \mathcal{M}\big)$.
\end{proof}

Having constructed the parametrix for an $\mathcal{M}$-valued elliptic  $\Psi$DO, our next goal is to compare the parametrix for an $\mathcal{M}$-valued elliptic  $\Psi$DO with its inverse, if the $\Psi$DO is assumed to be invertable in some sense. To this end, we proceed to examine the equivalence of different definitions of smoothing operators, which will play a crucial role in the analysis of the remainder terms of parametrices.

For any $A \in \mathrm{C}\Psi^m\big(\mathbb{R}^d; \mathcal{M}\big)$ and any $s \in \mathbb{R}$, it follows from \cite[Corollary 3.3]{XX19} that the operator
$$
A : H^{s}(\mathbb{R}^d; L_2(\mathcal{M})) \to H^{s - \Re m}(\mathbb{R}^d; L_2(\mathcal{M}))
$$
is bounded, as a bounded linear operator. Here, the Sobolev space $H^s(\mathbb{R}^d ; L_2(\mathcal{M}))$ is defined as the space of all tempered distributions in $\mathscr{S}^{\prime}(\mathbb{R}^d ; L_2(\mathcal{M}))$ with finite Sobolev norm
\begin{equation}\label{potential-sobolev}
	\|u\|_{H^s(\mathbb{R}^d; L_2(\mathcal{M}))} = \|J^s u\|_{L_2(\mathcal{N})}\ ,
\end{equation}
where $J^s=\left(1-\Delta\right)^{\frac{s}{2}}$ denotes the Bessel potential operator.

For all $s \in \mathbb{R}$, it follows from \cite[Proposition 5.6.4]{Hyt16} that we have the following dense and continuous embeddings:
$$
\mathscr{S}(\mathbb{R}^d ; L_2(\mathcal{M})) \hookrightarrow H^{s}(\mathbb{R}^d ; L_2(\mathcal{M})) \hookrightarrow \mathscr{S}^{\prime}(\mathbb{R}^d ; L_2(\mathcal{M})).
$$

We will need a form of the Schwartz kernel theorem for von Neumann-algebra valued kernels. First, recall that for a Banach space $X$ we have a canonical isomorphism
\[
\mathscr{S}(\mathbb{R}^d;X) = \mathscr{S}(\mathbb{R}^d)\otimes_{\pi} X
\]
where $\otimes_{\pi}$ is the completed projective tensor product. Similarly, for a Banach space $Y$ we have defined $\mathscr{S}'(\mathbb{R}^d;Y^*)$ as the space of continuous linear maps
\[
\mathcal{L}(\mathscr{S}(\mathbb{R}^d),Y^*).
\]
This is isomorphically the dual of $\mathscr{S}(\mathbb{R}^d)\otimes_{\pi} Y,$ see \cite[Chapter 43]{Tre67} for details.
It follows that we have an isomorphism
\[
\mathcal{L}(\mathscr{S}(\mathbb{R}^d;X),\mathscr{S}'(\mathbb{R}^d;Y^*)) = (\mathscr{S}(\mathbb{R}^d)\otimes_{\pi} X\otimes_{\pi}\mathscr{S}(\mathbb{R}^d)\otimes_{\pi} Y)'
\]
where $(\cdot)'$ denotes the topological dual. Making the obvious identifications and using $\mathscr{S}(\mathbb{R}^d)\otimes_{\pi}\mathscr{S}(\mathbb{R}^d) = \mathscr{S}(\mathbb{R}^{2d}),$ we conclude that
\begin{equation}\label{Bochner_Schartz_kernel_theorem}
	\mathcal{L}(\mathscr{S}(\mathbb{R}^d;X),\mathscr{S}'(\mathbb{R}^d;Y^*)) = \mathscr{S}'(\mathbb{R}^{2d}; (X \otimes_{\pi} Y)^*).
\end{equation}
The above identity asserts that if $T$ is a continuous linear map from $\mathscr{S}(\mathbb{R}^d;X)$ to $\mathscr{S}'(\mathbb{R}^d;Y^*),$ then there exists a unique distributional kernel
$K_T \in \mathscr{S}^{\prime}(\mathbb{R}^{2 d} ; (X\otimes_{\pi} Y)^*)$ for $T$,
$$
(T f)(x) = \int_{\mathbb{R}^d} K_T(x, y) f(y) \, d y
$$
In the distribution sense, i.e. for any $f\in \mathscr{S}(\mathbb{R}^d;X),\, g\in\mathscr{S}(\mathbb{R}^d;Y)$,
$$ \langle T(f),g\rangle_{\mathscr{S}'(\mathbb{R}^d;Y^*);\mathscr{S}(\mathbb{R}^d;Y)} = \langle K_T,g\otimes f\rangle_{\mathscr{S}'(\mathbb{R}^{2d};(X\otimes_{\pi} Y)^*),\mathscr{S}(\mathbb{R}^{2d};X\otimes_{\pi} Y)} ,
$$
where the left hand side is the scalar-valued pairing of $\mathscr{S}'(\mathbb{R}^d;Y^*)$ with $\mathscr{S}(\mathbb{R}^d;Y)$ and the right hand side is the scalar-valued pairing of $\mathscr{S}'(\mathbb{R}^{2d},(X\otimes_{\pi} Y)^*)$ with $\mathscr{S}(\mathbb{R}^{2d};X\otimes_{\pi} Y).$ If we instead think of $K_T$ as a continuous linear functional
\[
K_T:\mathscr{S}(\mathbb{R}^{2d})\to (X\otimes_{\pi} Y)^*
\]
then \eqref{Bochner_Schartz_kernel_theorem} is to be interpreted as follows: if $x\in X, y\in Y$ and $f,g \in \mathscr{S}(\mathbb{R}^d),$ we have
\begin{equation}\label{dual_mapping_interpretation}
	\langle K_T(f\otimes g),x\otimes y\rangle_{(X\otimes_{\pi} Y)^*,X\otimes_{\pi} Y} = \langle T(xf),gy\rangle_{\mathscr{S}'(\mathbb{R}^d;Y^*),\mathscr{S}(\mathbb{R}^d;Y)}.
\end{equation}

Note that in the above we are using the prime $(\cdot)'$ to denote the topological dual of a locally convex space, but in what follows $\mathcal{M}'$ denotes the commutant of a von Neumann algebra $\mathcal{M}.$ This coincidence of notation should not cause problems since we will never refer to the topological dual of a von Neumann algebra.

\begin{lemma}\label{BS-Ed}
	Let
	$T: \mathscr{S}(\mathbb{R}^d ; L_2(\mathcal{M})) \rightarrow \mathscr{S}^{\prime}(\mathbb{R}^d ; L_2(\mathcal{M}))$
	be a continuous linear operator such that $T$ commutes with every $u\in \mathcal{M}'.$ That is, for any $f,g \in \mathscr{S}(\mathbb{R}^d;L_2(\mathcal{M}))$ and $u\in \mathcal{M}',$ we have
	\[
	\langle T(uf),g\rangle_{\mathscr{S}'(\mathbb{R}^d;L_2(\mathcal{M})),\mathscr{S}(\mathbb{R}^d;L_2(\mathcal{M}))} = \langle uT(f),g\rangle_{\mathscr{S}'(\mathbb{R}^d;L_2(\mathcal{M})),\mathscr{S}(\mathbb{R}^d;L_2(\mathcal{M}))}.
	\]
	Here, $\mathcal{M}'$ is the commutant of $\mathcal{M}$ in its left multiplication representation on $L_2(\mathcal{M}).$
	Then the distributional kernel $K_T$ of $T$ belongs to
	\[
	\mathscr{S}'(\mathbb{R}^{2d};\mathcal{M}),
	\]
	where we have identified $\mathcal{M}$ with its image in the embedding
	\[
	\iota:\mathcal{M} \to (L_2(\mathcal{M})\otimes_{\pi} L_2(\mathcal{M}))^*
	\]
	given by
	\[
	\langle \iota(x),y\otimes z\rangle_{(L_2(\mathcal{M})\otimes_{\pi}L_2(\mathcal{M}))^*,L_2(\mathcal{M})\otimes_{\pi} L_2(\mathcal{M})} = \tau(yxz),\quad y,z \in L_2(\mathcal{M}),\; x \in \mathcal{M}.
	\]
\end{lemma}
\begin{proof}
	First we note that $\iota$ is an isometric embedding. Indeed, since $\mathcal{M}$ is the dual of $L_1(\mathcal{M})$ and every $h \in L_1(\mathcal{M})$ can be factorized as $h = gf$ where $g,f \in L_2(\mathcal{M})$ satisfy $\|h\|_1 = \|f\|_2\|g\|_2,$ we have
	\[
	\|\iota(x)\|_{(L_2(\mathcal{M})\otimes_{\pi} L_2(\mathcal{M}))^*} =\|x\|_{\mathcal{M}}.
	\]
	In particular, $\iota(\mathcal{M})$ is a closed subspace of $(L_2(\mathcal{M})\otimes_{\pi}L_2(\mathcal{M}))^*.$
	Secondly, we note that by the Schmidt decomposition, an arbitrary element of $L_2(\mathcal{M})\otimes_{\pi} L_2(\mathcal{M})$ can be identified with a trace-class operator on $L_2(\mathcal{M}).$ That is, for any element $\omega \in (L_2(\mathcal{M})\otimes_{\pi} L_2(\mathcal{M}))^*$ there exists a unique $X_{\omega} \in B(L_2(\mathcal{M}))$ such that
	\[
	\langle \omega,y\otimes z\rangle_{(L_2(\mathcal{M})\otimes_{\pi} L_2(\mathcal{M}))^*,L_2(\mathcal{M})\otimes_{\pi} L_2(\mathcal{M})} = \tau(yX_{\omega}(z)).
	\]
	We will have that $X_{\omega}\in \iota(\mathcal{M})$ if and only if $X_{\omega}$ commutes with every element of the commutant of $\mathcal{M},$ i.e. if and only if for all $u\in \mathcal{M}'$ we have
	\begin{align*}
		&\langle \omega,u(y)\otimes z\rangle_{(L_2(\mathcal{M})\otimes_{\pi} L_2(\mathcal{M}))^*,L_2(\mathcal{M})\otimes_{\pi} L_2(\mathcal{M})} \\
		=& \langle \omega,y\otimes u(z)\rangle_{(L_2(\mathcal{M})\otimes_{\pi} L_2(\mathcal{M}))^*,L_2(\mathcal{M})\otimes_{\pi} L_2(\mathcal{M})},\quad y,z \in L_2(\mathcal{M}).
	\end{align*}
	So we are now reduced to checking this commutativity. Let
	\[
	K_T \in \mathscr{S}'(\mathbb{R}^{2d};(L_2(\mathcal{M})\otimes_{\pi} L_2(\mathcal{M}))^*)
	\]
	be the Schwartz kernel of $T$ in the sense of \eqref{dual_mapping_interpretation}. For every $f,g \in \mathscr{S}(\mathbb{R}^d),$ we have
	\begin{align*}
		&\langle K_T(f\otimes g),y\otimes z\rangle_{(L_2(\mathcal{M})\otimes_{\pi} L_2(\mathcal{M}))^*,L_2(\mathcal{M})\otimes_{\pi} L_2(\mathcal{M})} \\
		=& \langle T(yf),zg\rangle_{\mathscr{S}'(\mathbb{R}^d;L_2(\mathcal{M})),\mathscr{S}(\mathbb{R}^d;L_2(\mathcal{M}))}.
	\end{align*}
	Hence, for $u\in \mathcal{M}'$ we have
	\begin{align*}
		&\langle K_T(f\otimes g),u(y)\otimes z\rangle_{(L_2(\mathcal{M})\otimes_{\pi} L_2(\mathcal{M}))^*,L_2(\mathcal{M})\otimes_{\pi} L_2(\mathcal{M})}\\
		=& \langle T(u(y)f),zg\rangle_{\mathscr{S}'(\mathbb{R}^d;L_2(\mathcal{M})),\mathscr{S}(\mathbb{R}^d;L_2(\mathcal{M}))}\\
		=& \langle T(yf),u(z)g\rangle_{\mathscr{S}'(\mathbb{R}^d;L_2(\mathcal{M})),\mathscr{S}(\mathbb{R}^d;L_2(\mathcal{M}))}\\
		=& \langle K_T(f\otimes g),y\otimes u(z)\rangle_{(L_2(\mathcal{M})\otimes_{\pi} L_2(\mathcal{M}))^*,L_2(\mathcal{M})\otimes_{\pi} L_2(\mathcal{M})}
	\end{align*}
	Hence,
	\[
	K_T(f\otimes g) \in \iota(\mathcal{M})
	\]
	for all $f,g \in \mathscr{S}(\mathbb{R}^d).$ By continuity, it follows that $K_T$ is a continuous map
	\[
	K_T:\mathscr{S}(\mathbb{R}^{2d})\to \iota(\mathcal{M}).
	\]
	That is, $K_T \in \mathscr{S}'(\mathbb{R}^{2d};\mathcal{M})$ in the required sense.
\end{proof}

\begin{proposition}\label{smooth-op-kernel}
	Let $T:\mathscr{S}(\mathbb{R}^d;L_2(\mathcal{M}))\to\mathscr{S}'(\mathbb{R}^d;L_2(\mathcal{M}))$
	be a continuous linear operator. The following statements are equivalent:
	
	\begin{enumerate}[$\rm (i)$]
		\item $T\in \Psi^{-\infty}\big(\mathbb{R}^d;\mathcal{M}\big)$, i.e., its symbol $\sigma(T)(x,\xi)\in S^{-\infty}\big(\mathbb{R}^d\times\mathbb{R}^d;\mathcal{M}\big)$ decays faster than any polynomial in $\xi$, along with all its derivatives.
		
		\item $T$ is smoothing in the sense that for every $s,t\in\mathbb{R}$, there exists a constant $C_{s,t} > 0$ such that
		$$
		\|T u\|_{H^t(\mathbb{R}^d;L_2(\mathcal{M}))} \le C_{s,t}\|u\|_{H^s(\mathbb{R}^d;L_2(\mathcal{M}))},\quad \forall u\in H^s(\mathbb{R}^d;L_2(\mathcal{M})),
		$$
		and $T$ commutes with every $u\in \mathcal{M}'$ in the sense that for any $f \in L_2(\mathcal{N})$, it holds that $
		T(I\otimes u)f  = (I\otimes u) T f
		$.
		\item $T$ has an $\mathcal{M}$-valued Schwartz kernel $K_T\in C^\infty(\mathbb{R}^{2d};\mathcal{M})$ which is rapidly decaying in $x - y$, along with all its derivatives.
	\end{enumerate}
	Moreover, if any of the three statements is fulfilled, the symbol and kernel are related by partial Fourier transforms:
	$$
	K_T(x,y) = \int_{\mathbb{R}^d} e^{ i(x-y)\cdot \xi} \sigma(T)(x,\xi)\,\bar{d}\xi,\qquad
	\sigma(T)(x,\xi) = \int_{\mathbb{R}^d} e^{- i(x-y)\cdot \xi} K_T(x,y)\,dy,
	$$
	interpreted as Bochner integrals in $\mathcal{M}$.
\end{proposition}

\begin{proof}
	(i) $\Rightarrow$ (ii):
	Assume $ T \in \Psi^{-\infty}\big(\mathbb{R}^d; \mathcal{M}\big) $. By definition, for every multi-index $ \alpha, \beta $ and any $ N > 0 $, there exists a constant $ C_{\alpha, \beta, N} > 0 $ such that
	$$
	\big\|\partial_x^\beta \partial_{\xi}^\alpha \sigma(T)(x, \xi)\big\|_{\mathcal{M}} \le C_{\alpha, \beta, N}(1 + |\xi|)^{-N}.
	$$
	Using the definition of the potential Sobolev space \eqref{potential-sobolev}, for any $ u \in H^s(\mathbb{R}^d; L_2(\mathcal{M})) $, we have
	$$
	\|T u\|_{H^t}^2 = \int_{\mathbb{R}^d} (1 + |\xi|^2)^t \|\sigma(T)(x, \xi) \hat{u}(\xi)\|_{\mathcal{M}}^2 \, d\xi.
	$$
	Due to the rapid decay of $ \sigma(T) $ in $ \xi $, it follows that for any $ s, t \in \mathbb{R} $, there exists a constant $ C_{s,t} > 0 $ such that $ \|T u\|_{H^t} \le C_{s,t} \|u\|_{H^s} $. This confirms that $ T $ is a smoothing operator.  The commutativity is clear, see the proof of Theorem \ref{thm-tensor algebra}.

	(ii) $\Rightarrow$ (iii):
	Assume that
	$$
	T: \mathscr{S}(\mathbb{R}^d ; L_2(\mathcal{M})) \rightarrow \mathscr{S}^{\prime}(\mathbb{R}^d ; L_2(\mathcal{M}))
	$$
	is a smoothing operator, and assume that $T$ commutes with every $u\in \mathcal{M}'$.
	By the Bochner-Schwartz kernel theorem for vector-valued distributions (Lemma \ref{BS-Ed}), there exists a unique distribution kernel $K_T \in \mathscr{S}^{\prime}(\mathbb{R}^{2 d} ; \mathcal{M})$ such that
	$$
	(T u)(x) = \int_{\mathbb{R}^d} K_T(x, y) u(y) \, d y
	$$
	in the sense of distributions.
	
	The smoothing property of $T$ implies that for every integer $N > 0$, the operator $T$ extends continuously from $L_2(\mathbb{R}^d ; L_2(\mathcal{M}))$ to the Sobolev space $H^N(\mathbb{R}^d ; L_2(\mathcal{M}))$. It follows from \cite[Corollary 14.4.27]{Hyt23} that if $N > \frac{d}{2} + s$ for some $s > 0$, we have a continuous embedding
	$$
	H^N(\mathbb{R}^d ; L_2(\mathcal{M})) \hookrightarrow C_{\mathrm{ub}}^s(\mathbb{R}^d ; L_2(\mathcal{M})),
	$$
	where $C_{\mathrm{ub}}^s(\mathbb{R}^d ; L_2(\mathcal{M}))$ denotes the space of functions of non-integer smoothness $s = k + \gamma$, with $k = \lfloor s \rfloor \in \mathbb{N}_{0}$ and $\gamma = s - k \in (0, 1)$. More precisely, $C_{\mathrm{ub}}^s(\mathbb{R}^d ; L_2(\mathcal{M}))$ consists of functions whose derivatives up to order $k$ are uniformly continuous and bounded, and whose $k$-th order derivatives are bounded and $\gamma$-H\"{o}lder continuous, with seminorm
	$$
	\|f\|_{C_{\mathrm{ub}}^\gamma} := \sup_{x \neq y} \frac{\|f(x) - f(y)\|_{L_2(\mathcal{M})}}{|x - y|^\gamma}.
	$$
	Thus, for arbitrarily large $s$, we have $T: L_2(\mathcal{N}) \rightarrow C_{\mathrm{ub}}^s(\mathbb{R}^d ; L_2(\mathcal{M}))$, which implies that $T u$ is smooth and all derivatives are uniformly continuous and bounded. Consequently, one can differentiate under the integral defining $T u$ to obtain
	$$
	\partial_x^\alpha(T u)(x) = \int_{\mathbb{R}^d} \partial_x^\alpha K_T(x, y) u(y) \, d y
	$$
	for all multi-indices $\alpha$. Considering the adjoint operator $T^*$ and the action of $T$ on negative Sobolev spaces ensures that the derivatives with respect to $y$ also exist as continuous functions. Therefore, all mixed partial derivatives $\partial_x^\alpha \partial_y^\beta K_T(x, y)$ exist and are continuous.
	
	To establish rapid decay in $x - y$, fix an arbitrary $N > 0$ and consider the continuity of
	$$
	T: H^{-N}(\mathbb{R}^d ; L_2(\mathcal{M})) \rightarrow L_2(\mathbb{R}^d ; L_2(\mathcal{M})).
	$$
	Taking the partial Fourier transform of $K_T$ in the variable $x - y$,
	$$
	\hat{K}_T(x, \xi) := \int_{\mathbb{R}^d} e^{-i(x - y) \cdot \xi} K_T(x, y) \, d y,
	$$
	the smoothing property implies that $(1 + |\xi|)^N \hat{K}_T(x, \xi)$ is uniformly bounded in $x$ and $\xi$. By Fourier inversion, this uniform control in frequency space guarantees that $K_T(x, y)$ decays faster than any polynomial in $x - y$.
	
	
	(iii) $\Rightarrow$ (i):
	Finally, assume (iii), and define the symbol $\sigma(T)(x, \xi)$ via the partial Fourier transform
	$$
	\sigma(T)(x, \xi) := \int_{\mathbb{R}^d} e^{-i(x - y) \cdot \xi} K_T(x, y) \, dy.
	$$
	Differentiating under the integral sign and using the rapid decay of all derivatives of $K_T$, we obtain
	$$
	\big\|\partial_x^\beta \partial_{\xi}^\alpha \sigma(T)(x, \xi)\big\|_{\mathcal{M}} \le C_{\alpha, \beta, N}(1+|\xi|)^{-N}
	$$
	for all multi-indices $\alpha, \beta$ and for every $N > 0$. Therefore, $\sigma(T)(x, \xi) \in S^{-\infty}\big(\mathbb{R}^d \times \mathbb{R}^d ; \mathcal{M}\big)$, that is, $T \in \Psi^{-\infty}\big(\mathbb{R}^d;\mathcal{M}\big)$.
\end{proof}

Combining the parametrix construction in Theorem \ref{parabelong} with Proposition \ref{smooth-op-kernel}, we obtain the following result concerning the inverse operator of an $\mathcal{M}$-valued elliptic  $\Psi$DO.

\begin{theorem}\label{inverse A PDO}
	Let $A \in \mathrm{C}\Psi^{m}\big(\mathbb{R}^d; \mathcal{M}\big)$ with $\Re m>0$ be an $\mathcal{M}$-valued elliptic  $\Psi$DO. Suppose that for some $s \in \mathbb{R}$, the operator
	$$
	A : H^s(\mathbb{R}^d; L_2(\mathcal{M})) \to H^{s-\Re m}(\mathbb{R}^d; L_2(\mathcal{M}))
	$$
	is bounded and invertible as a bounded linear operator. Then the inverse operator
	$$
	A^{-1} : H^{s-\Re m}(\mathbb{R}^d; L_2(\mathcal{M})) \to H^s(\mathbb{R}^d; L_2(\mathcal{M}))
	$$
	is also an $\mathcal{M}$-valued classical  $\Psi$DO of order $-m$, i.e.,
	$$
	A^{-1} \in \mathrm{C}\Psi^{-m}\big(\mathbb{R}^d; \mathcal{M}\big).
	$$
\end{theorem}

\begin{proof}
	Let $B \in \mathrm{C}\Psi^{-m}\big(\mathbb{R}^d;\mathcal{M}\big)$ be a parametrix of $A$, so that both the left and right parametrix relations hold:
	$$
	BA = I - R, \quad AB = I - R',
	$$
	where $R, R' \in \Psi^{-\infty}\big(\mathbb{R}^d;\mathcal{M}\big)$ are smoothing operators.
	
	Fix $t \in \mathbb{R}$, and suppose $u \in H^t(\mathbb{R}^d; L_2(\mathcal{M}))$ satisfies $Au = 0$. Using the left parametrix relation $BA = I - R$, we obtain
	$$
	u = BAu + Ru = Ru.
	$$
	Since $R$ is a smoothing operator, for any $N \ge 0$,
	$$
	\|u\|_{H^{t+N}} = \|Ru\|_{H^{t+N}} \le C_{t,N} \|u\|_{H^t}.
	$$
	Choose $N$ sufficiently large so that $t + N \ge s$. Then $\|u\|_{H^s} \le C \|u\|_{H^t}$. But $A$ is injective on $H^s(\mathbb{R}^d; L_2(\mathcal{M}))$, hence $u = 0$. This shows that $A$ is injective on $H^t(\mathbb{R}^d; L_2(\mathcal{M}))$.
	
	Next, let $f \in H^{t-\Re m}(\mathbb{R}^d; L_2(\mathcal{M}))$, and define
	$$
	u_0 := B f \in H^t(\mathbb{R}^d; L_2(\mathcal{M})).
	$$
	Then by the right parametrix relation $AB = I - R'$,
	$$
	A u_0 = A B f = f - R' f.
	$$
	Since $R'$ is smoothing, $R' f \in H^\infty(\mathbb{R}^d; L_2(\mathcal{M}))$. Here
	$H^{\infty}(\mathbb{R}^d ; L_2(\mathcal{M})):=\bigcap_{k \ge 0} H^k(\mathbb{R}^d ; L_2(\mathcal{M}))$,
	endowed with the Fr\'{e}chet topology defined by the seminorms $\|\cdot\|_{H^k}$.
	By the surjectivity of $A$ on $H^s(\mathbb{R}^d; L_2(\mathcal{M}))$, there exists $v \in H^s(\mathbb{R}^d; L_2(\mathcal{M}))$ such that
	$$
	A v = R' f.
	$$
	Elliptic regularity implies $v \in H^\infty(\mathbb{R}^d; L_2(\mathcal{M})) \subset H^t(\mathbb{R}^d; L_2(\mathcal{M}))$. Set $u := u_0 + v \in H^t(\mathbb{R}^d; L_2(\mathcal{M}))$. Then
	$$
	A u = A(u_0 + v) = (f - R'f) + R'f = f,
	$$
	so $A$ is surjective on $H^t(\mathbb{R}^d; L_2(\mathcal{M}))$.
	
	Combining injectivity and surjectivity,
	$$A : H^t(\mathbb{R}^d; L_2(\mathcal{M})) \to H^{t-\Re m}(\mathbb{R}^d; L_2(\mathcal{M}))$$
	is bijective with a bounded inverse. By the bounded inverse theorem, there exists a constant $C_t > 0$ such that
	$$
	\|A^{-1} g\|_{H^t} \le C_t \|g\|_{H^{t-\Re m}}, \quad \forall g \in H^{t-\Re m}(\mathbb{R}^d; L_2(\mathcal{M})).
	$$
	
	Now consider the identity
	$$
	A^{-1} = B + A^{-1} R'.
	$$
	We already know $B \in \mathrm{C}\Psi^{-m}\big(\mathbb{R}^d;\mathcal{M}\big)$. For the second term, note that $R' : H^{t-\Re m}(\mathbb{R}^d; L_2(\mathcal{M})) \to H^\infty(\mathbb{R}^d; L_2(\mathcal{M}))$ is continuous for all $t$, and $A^{-1} : H^\infty(\mathbb{R}^d; L_2(\mathcal{M})) \to H^\infty(\mathbb{R}^d; L_2(\mathcal{M}))$ is also continuous by elliptic regularity. Therefore,
	$$
	A^{-1} R' : H^{t-\Re m}(\mathbb{R}^d; L_2(\mathcal{M})) \to H^\infty(\mathbb{R}^d; L_2(\mathcal{M}))
	$$
	is continuous for all $t \in \mathbb{R}$.
	By Proposition \ref{smooth-op-kernel}, any operator mapping $H^t(\mathbb{R}^d; L_2(\mathcal{M}))$ continuously into $\bigcap_k H^k(\mathbb{R}^d; L_2(\mathcal{M}))$ belongs to $\Psi^{-\infty}\big(\mathbb{R}^d;\mathcal{M}\big)$. Hence
	$$
	A^{-1} R' \in \Psi^{-\infty}\big(\mathbb{R}^d;\mathcal{M}\big).
	$$
	Therefore, we conclude
	$$
	A^{-1} = B + A^{-1} R' \in \mathrm{C}\Psi^{-m}\big(\mathbb{R}^d;\mathcal{M}\big).
	$$
	This completes the proof.
\end{proof}

\subsection{The resolvent of an $\mathcal{M}$-valued elliptic  $\Psi$DOs }\label{section-resolvent}

As before, let
$$
A \in \mathrm{C}\Psi^m(\mathbb{R}^d;\mathcal{M})
$$
be an $\mathcal{M}$-valued elliptic $\Psi$DO of order $\Re m > 0$, acting on the dense domain $\mathscr{S}(\mathbb{R}^d; L_2(\mathcal{M})) \subset L_2(\mathcal{N})$. For $f \in \mathscr{S}(\mathbb{R}^d; L_2(\mathcal{M}))$, the quadratic form of $A^*A$ satisfies
\begin{equation}\label{eq:quadratic}
	\langle A^*A f, f \rangle
	= \int_{\mathbb{R}^d} \tau\!\left((A f(x))^* A f(x)\right)\,dx
	= \|A f\|_{L_2(\mathcal{N})}^2
	\ge 0.
\end{equation}
Thus, $A^*A$ is symmetric, densely defined, and bounded below by zero.

Although the Friedrichs extension theorem ensures that a symmetric operator which is bounded below admits a unique self-adjoint extension with the same lower bound, this extension does not necessarily coincide with the closure of $A^*A$ on the Schwartz space, nor is $A^*A$ automatically essentially self-adjoint without additional assumptions.

In the present setting, the ellipticity of $A$ implies that $A^*A$ is also elliptic, of order $2\Re m > 0$. By elliptic regularity, any solution $u \in L_2(\mathcal{N})$ to the equation
$$
(A^*A \pm i)u = 0
$$
must belong to $\mathscr{S}(\mathbb{R}^d; L_2(\mathcal{M}))$. Substituting such $u$ into the above equation and applying~\eqref{eq:quadratic}, we deduce that $u = 0$. Therefore,
$$
\ker(A^*A - i) = \ker(A^*A + i) = \{0\}.
$$
It follows that $A^*A$ is essentially self-adjoint on $\mathscr{S}(\mathbb{R}^d; L_2(\mathcal{M}))$, and hence its closure
$$
\overline{A^*A} = (A^*A)^*
$$
is the unique self-adjoint extension. Since this extension is bounded below by zero, it coincides with the Friedrichs extension of $A^*A$.

Finally, if $f \in D(\overline{A^*A})$, then by definition there exists a sequence $f_n \in \mathscr{S}(\mathbb{R}^d; L_2(\mathcal{M}))$ such that
$$
f_n \to f \quad \text{and} \quad A^*A f_n \to \overline{A^*A}\, f \quad \text{in } L_2(\mathcal{N}).
$$
By continuity of the inner product and~\eqref{eq:quadratic},
$$
\langle \overline{A^*A}\, f , f \rangle
= \lim_{n\to\infty} \langle A^*A f_n, f_n \rangle
= \lim_{n\to\infty} \|A f_n\|_{L_2(\mathcal{N})}^2
\ge 0.
$$
Thus, $\overline{A^*A}$ is a positive operator on its domain, and the Friedrichs extension agrees with the closure of $A^*A$ on $\mathscr{S}(\mathbb{R}^d; L_2(\mathcal{M}))$.

By Proposition \ref{cl M}, the principal symbol of $A^* A$ is $\sigma(A)_m(x, \xi)^{\ast}\sigma(A)_m(x, \xi)$. Consequently, $\sigma(A)_m(x, \xi)^{\ast}\sigma(A)_m(x, \xi)$ is a positive element in $\mathcal{M}$ for all $(x, \xi) \in \mathbb{R}^d \times (\mathbb{R}^d \setminus \{0\})$.
By ellipticity and Proposition \ref{sa M elliptic}, there exists a constant $C_e>0$ such that
$$
\sigma(A)_m(x,\xi)^* \sigma(A)_m(x,\xi) \ge C_e^2 |\xi|^{2\Re m} ,
$$
which ensures that $\sigma(A)_m(x,\xi)^* \sigma(A)_m(x,\xi)$ is invertible in $\mathcal{M}$ for $\xi\neq 0$.

Since our main interest is the singular spectrum of $A$, i.e., the spectrum of $|A| = (A^* A)^{\frac{1}{2}}$, it suffices to study the positive operator $A^* A$.
From now on, for notational simplicity, we will assume without loss of generality that $A$ is positive, and that its principal symbol $\sigma(A)_m(x, \xi)$ is positive for all $(x, \xi) \in \mathbb{R}^d \times (\mathbb{R}^d \setminus \{0\})$. In this case, we also know from Proposition \ref{cl M} that the order $m$ of $A$ is real.

\begin{rk}
	If $A$ is positive, then we may write $A=B^2$ with (densely defined unbounded) self-adjoint operator on $L_2(\mathcal{N}) $. If $B$ were known to be a $\Psi$DO, then we could see from Proposition \ref{cl M} that, the principal symbol $ \sigma(A)_m(x,\xi)$ is positive. We will indeed show that $A^{\frac 1 2 }$ is a $\Psi$DO under some certain assumption; but currently, we keep the assumptions that $A$ is positive, and that its principal symbol $\sigma(A)_m(x, \xi)$ is positive for all $(x, \xi) \in \mathbb{R}^d \times (\mathbb{R}^d \setminus \{0\})$.
\end{rk}

We then define the keyhole-shaped region $\Lambda_{\xi}$ as
$$
\Lambda_{\xi} = \Big\{ \lambda \in \mathbb{C} \setminus \{0\} : |\lambda| < \tfrac{1}{2}C_e(\xi) \text{ or } |\arg \lambda - \pi| < \tfrac{\pi}{4} \Big\},
$$
where $C_e(\xi) = C_e |\xi|^m$ for $0<|\xi| \le 1$, and $C_e(\xi) = C_e$ for $|\xi| \ge 1$.
As a result, for all $(x, \xi) \in \mathbb{R}^d \times (\mathbb{R}^d \setminus \{0\})$, the operator
$$
\sigma(A)_m(x, \xi) - \lambda 1_{\mathcal{M}} \quad (\text{abbreviated as } \sigma(A)_m(x, \xi) - \lambda)
$$
is invertible in $\mathcal{M}$ for any $\lambda \in \Lambda_{\xi}$.

We define the parameter-dependent principal symbol of $A-\lambda$ by
$$
\sigma(A)_m(x,\xi,\lambda) := \sigma(A)_m(x,\xi) - \lambda.
$$
For the lower order terms we set
$$
\sigma(A)_{m-j}(x,\xi,\lambda) := \sigma(A)_{m-j}(x,\xi), \quad j \ge 1.
$$
Moreover, for $\xi \neq 0$ and $t>0$,
\begin{align*}
	\sigma(A)_m(x, t\xi, t^m\lambda) &= \sigma(A)_m(x, t\xi) - t^m\lambda \\
	&= t^m (\sigma(A)_m(x, \xi) - \lambda) \\
	&= t^m \sigma(A)_m(x, \xi, \lambda).
\end{align*}
Thus, $\sigma(A)_m(x, \xi, \lambda)$ is positive homogeneous in $(\xi, \lambda^{\frac{1}{m}})$ of degree $m$.

Next, we are going to construct an $\mathcal{M}$-valued $\Psi$DO with parameter, that will be denoted by $B(\lambda)$, such that $B(\lambda) (A - \lambda ) \sim 1_{B(L_2(\mathbb{R}^d)) \overline{\otimes} \mathcal{M}}$, i.e., a left parametrix for $A - \lambda $. To achieve this construction, we solve the following equations:
\begin{equation} \label{mexpan}
	\sigma(B)_{-m}^0(x, \xi, \lambda) \sigma(A)_m(x, \xi, \lambda) = 1_{\mathcal{M}},
\end{equation}
\begin{equation} \label{rexpan}
	\begin{split}
		&\sigma(B)_{-m-j}^0(x, \xi, \lambda) \sigma(A)_{m}(x, \xi, \lambda)\\
		+ \sum_{\substack{k+l+|\alpha| = j \\ k < j}}& \frac{1}{\alpha!} \partial_\xi^\alpha \sigma(B)_{-m-k}^0(x, \xi, \lambda) D_x^\alpha \sigma(A)_{m-l}(x, \xi, \lambda) = 0,
	\end{split}
\end{equation}
where $j \ge 1$ and $k,l \ge 0$. The existence of the solution to this system is ensured in the following

\begin{proposition}\label{existence-sB}
	Let $A \in \mathrm{C}\Psi^m(\mathbb{R}^d;\mathcal{M})$ be elliptic with $  m>0$, and suppose that both $ A $ and its principal symbol $\sigma(A)_m(x, \xi)$ are positive. Let $\Lambda_\xi$ be the parameter region defined above. Then there exists a unique family of $\mathcal{M}$-valued smooth functions
	$$
	\big\{\sigma(B)_{-m-j}^0(x,\xi,\lambda)\big\}_{j\ge 0},
	$$
	such that, for each  $(x,\xi)\in \mathbb{R}^d\times(\mathbb{R}^d\setminus\{0\})$ and $\lambda \in \Lambda_\xi$,  \eqref{mexpan} and \eqref{rexpan} hold.
\end{proposition}

\begin{proof}
	Since $A$ is elliptic, $\sigma(A)_m(x,\xi,\lambda) = \sigma(A)_m(x,\xi)-\lambda$ is invertible in $\mathcal{M}$ for $\lambda\in \Lambda_\xi$.
	Equation \eqref{mexpan} then has the unique solution
	$$
	\sigma(B)_{-m}^0(x,\xi,\lambda) = \sigma(A)_m(x,\xi,\lambda)^{-1}.
	$$

	Suppose inductively that $\sigma(B)_{-m-k}^0(x,\xi,\lambda)$ are already defined for $k<j$.
	In \eqref{rexpan}, all terms except the first involve only these lower-order coefficients.
	Rearranging gives
	$$
	\sigma(B)_{-m-j}^0(x,\xi,\lambda)\,\sigma(A)_m(x,\xi,\lambda) = -r_j(x,\xi,\lambda),
	$$
	where $r_j(x,\xi,\lambda)$ is explicitly computable from $\big\{\sigma(B)_{-m-k}^0(x,\xi,\lambda) : k<j\big\}$ and $\big\{\sigma(A)_{m-l}(x,\xi,\lambda):l\ge 0 \big\}$.
	Right-multiplying by the invertible $\sigma(A)_m(x,\xi,\lambda)^{-1}$ yields
	$$
	\sigma(B)_{-m-j}^0(x,\xi,\lambda) = -r_j(x,\xi,\lambda)\,\sigma(A)_m(x,\xi,\lambda)^{-1}.
	$$
	Thus $\sigma(B)_{-m-j}^0(x,\xi,\lambda)$ exists uniquely.
\end{proof}

To continue, we need the following formula for derivatives of inverses.

\begin{lemma}\label{inverse derivative lemma}
	Let $F(x,\xi) \in \mathcal{M}$ be a smooth $\mathcal{M}$-valued function such that $F(x,\xi)$ is invertible for all $(x,\xi) \in \mathbb{R}^d \times (\mathbb{R}^d \setminus \{0\})$.
	Then for any multi-indices $\alpha,\beta$, the mixed derivatives of $F^{-1}(x,\xi)$ satisfy
	\begin{equation}\label{inverse derivative formula}
		\begin{split}
			\partial_\xi^\alpha \partial_x^\beta \big(F^{-1}\big)
			=& \sum_{k=1}^{|\alpha|+|\beta|} \
			\sum_{\substack{\alpha_1+\cdots+\alpha_k=\alpha \\ \beta_1+\cdots+\beta_k=\beta}}
			C_{\alpha_1\cdots \alpha_k;\, \beta_1\cdots \beta_k} \\
			&\cdot  F^{-1} \big(\partial_\xi^{\alpha_1}\partial_x^{\beta_1} F\big) F^{-1} \cdots
			F^{-1} \big(\partial_\xi^{\alpha_k}\partial_x^{\beta_k} F\big) F^{-1},
		\end{split}
	\end{equation}
	where the constants $C_{\alpha_1\cdots \alpha_k;\, \beta_1\cdots \beta_k}$ are combinatorial coefficients depending only on the partitions of $(\alpha,\beta)$.
\end{lemma}

\begin{proof}
	The case $|\alpha|+|\beta|=1$ follows from the basic identity
	$$
	\partial(F^{-1}) = -F^{-1} (\partial F) F^{-1}.
	$$
	Assume the formula \eqref{inverse derivative formula} holds for all multi-indices with total order $\le N$.
	For total order $N+1$, differentiating \eqref{inverse derivative formula} once more and applying Leibniz' rule yields a sum of products of the form
	$$
	F^{-1} \big(\partial F\big) F^{-1} \cdots F^{-1} \big(\partial F\big) F^{-1},
	$$
	with the correct combinatorial factors. This establishes the induction step.
	Hence the formula \eqref{inverse derivative formula} holds for all multi-indices $(\alpha,\beta)$.
\end{proof}

The properties of the family $\big\{\sigma(B)_{-m-j}^0(x,\xi,\lambda)\big\}_{j\ge 0}$ determined by \eqref{mexpan} and \eqref{rexpan} are provided in the following proposition.

\begin{proposition}\label{para sym prop}
	For $(x, \xi) \in \mathbb{R}^d \times (\mathbb{R}^d \backslash \{0\})$, and $\lambda$ in the keyhole-shaped region $\Lambda_{\xi}$, the following properties hold:
	
	\begin{enumerate}[$\rm (i)$]
		\item $\big\| \sigma(B)_{-m}^0(x,\xi, \lambda) \big\|_{\mathcal{M}} \le C (1 + |\xi|^2)^{-\frac{ m }{2}} (1 + |\lambda|)^{-1}$.
		
		\item For each $j \ge 0$, multi-indices $\alpha, \beta$, and  $t > 0$,
		$$
		\partial_{\xi}^\alpha \partial_x^\beta \sigma(B)_{-m-j}^0\big(x, t \xi, t^m \lambda\big) = t^{-m-j-|\alpha|} \partial_{\xi}^\alpha \partial_x^\beta \sigma(B)_{-m-j}^0(x,\xi, \lambda).
		$$
		
		\item For $j + |\alpha| + |\beta| > 0$,
		$$
		\big\| \partial_\xi^\alpha \partial_x^\beta \sigma(B)_{-m-j}^0(x,\xi, \lambda) \big\|_{\mathcal{M}} \le C (1 + |\xi|^2)^{\frac{- m -j-|\alpha|}{2}} (1 + |\lambda|)^{-2}.
		$$
	\end{enumerate}
\end{proposition}

\begin{proof}
	$\rm (i)$ By the positivity and ellipticity of $ \sigma(A)_m(x, \xi) $, the spectrum satisfies
	$$
	\operatorname{Spec}\big(\sigma(A)_{m}(x, \xi)\big) \subset [C_e|\xi|^m, \infty).
	$$
	For $ 0 < |\lambda| < \tfrac{1}{2}C_e(\xi) $, the distance from $ \lambda $ to $ \operatorname{Spec}\big(\sigma(A)_{m}(x, \xi)\big) $ is given by
	\begin{align*}
		\operatorname{dist}\Bigl(\lambda, \operatorname{Spec}\bigl(\sigma(A)_{m}(x, \xi)\bigr)\Bigr)
		&= \inf_{\mu \in [C_e|\xi|^m, \infty)} |\lambda - \mu| \notag \\
		&\ge C_e|\xi|^m - |\lambda| \notag \\
		&\ge \tfrac{1}{2}C_e|\xi|^m \notag \\
		&\ge \tfrac{1}{4}C_e|\xi|^m + \tfrac{1}{2}|\lambda|.
	\end{align*}
	For $ |\arg \lambda - \pi| < \tfrac{\pi}{4} $ with $ \lambda \ne 0 $, the distance from $ \lambda $ to $ \operatorname{Spec}\big(\sigma(A)_{m}(x, \xi)\big) $ is given by
	$$
	\operatorname{dist}\Big(\lambda, \operatorname{Spec}\big(\sigma(A)_{m}(x, \xi)\big)\Big) = \inf_{\mu \in [C_e|\xi|^m, \infty)} |\lambda - \mu| \ge C_e|\xi|^m + \tfrac{\sqrt{2}}{2}|\lambda|.
	$$
	Therefore, by choosing a sufficiently large constant $ C > 0 $, we obtain
	$$
	\big\| \sigma(B)_{-m}^0(x, \xi, \lambda) \big\|_{\mathcal{M}} \le C (1 + |\xi|^2)^{-\frac{m}{2}} (1 + |\lambda|)^{-1},
	$$
	which establishes the estimate in $\rm (i)$.
	
	$\rm (ii)$ For  $t > 0$, it holds that
	$$
	\sigma(B)_{-m}^0\big(x, t \xi, t^m \lambda\big) = t^{-m} \sigma(B)_{-m}^0(x, \xi, \lambda).
	$$
	Applying Lemma \ref{inverse derivative lemma},
	\begin{equation*}
		\begin{split}
			\partial_{\xi}^\alpha \partial_x^\beta \sigma(A)_m^{-1} &= \sum_{k=1}^{|\alpha|+|\beta|} \sum_{\substack{\sum \alpha_j=\alpha \\ \sum \beta_j=\beta}} C_{\alpha_1 \cdots \alpha_k ; \beta_1 \cdots \beta_k} \\
			&\cdot \sigma(A)_m^{-1} \big(\partial_{\xi}^{\alpha_1} \partial_x^{\beta_1} \sigma(A)_m\big) \sigma(A)_m^{-1} \cdots \sigma(A)_m^{-1} \big(\partial_\xi^{\alpha_k} \partial_x^{\beta_k} \sigma(A)_m\big) \sigma(A)_m^{-1},
		\end{split}
	\end{equation*}
	combined with the positive homogeneity of $\sigma(A)_m(x, \xi, \lambda)$, we deduce that for $|\xi| \ge 1$ and $t > 0$,
	$$
	\partial_{\xi}^\alpha \partial_x^\beta \sigma(B)_{-m}^0\big(x, t \xi, t^m \lambda\big) = t^{-m-|\alpha|} \partial_{\xi}^\alpha \partial_x^\beta \sigma(B)_{-m}^0(x, \xi, \lambda).
	$$
	This confirms the claimed homogeneity for $j = 0$.
	
	When $j \ge 1$, assume by induction that the homogeneity property holds for all symbols of order $-m-k$ with $k < j$. The recursive definition of $\sigma(B)_{-m-j}^0(x, \xi, \lambda)$ is given by
	\begin{align*}
		\sigma(B)_{-m-j}^0(x, \xi, \lambda)
		&= -\sigma(B)_{-m}^0(x, \xi, \lambda) \notag \\
		&\qquad \cdot \sum_{\substack{k+l+|\gamma|=j \\ l<j}}
		\frac{1}{\gamma!} \partial_{\xi}^\gamma \sigma(B)_{-m-k}^0(x, \xi, \lambda)
		\cdot D_x^\gamma \sigma(A)_{m-l}(x, \xi).
	\end{align*}
	Under the scaling $(t \xi, t^m \lambda)$, each term transforms as
	$$
	\partial_{\xi}^\gamma \sigma(B)_{-m-k}^0\big(x, t \xi, t^m \lambda\big) \cdot D_x^\gamma \sigma(A)_{m-l}(x, t \xi).
	$$
	By the induction hypothesis,
	$$
	\partial_{\xi}^\gamma \sigma(B)_{-m-k}^0\big(x, t \xi, t^m \lambda\big) = t^{-m-k-|\gamma|} \partial_{\xi}^\gamma \sigma(B)_{-m-k}^0(x, \xi, \lambda),
	$$
	and since $\sigma(A)_{m-l}(x, t \xi) = t^{m-l} \sigma(A)_{m-l}(x, \xi)$, it follows that
	$$
	D_x^\gamma \sigma(A)_{m-l}(x, t \xi) = t^{m-l} D_x^\gamma \sigma(A)_{m-l}(x, \xi).
	$$
	Thus, each term scales as
	$$
	t^{-m-k-|\gamma|} \cdot t^{m-l} = t^{-(k+l+|\gamma|)} = t^{-j}.
	$$
	This implies
	$$
	\sigma(B)_{-m-j}^0\big(x, t \xi, t^m \lambda\big) = t^{-m-j} \sigma(B)_{-m-j}^0(x, \xi, \lambda),
	$$
	and consequently,
	$$
	\partial_{\xi}^\alpha \partial_x^\beta \sigma(B)_{-m-j}^0\big(x, t \xi, t^m \lambda\big) = t^{-m-j-|\alpha|} \partial_{\xi}^\alpha \partial_x^\beta \sigma(B)_{-m-j}^0(x, \xi, \lambda).
	$$
	By induction, the homogeneity of all symbols $\sigma(B)_{-m-j}^0(x, \xi, \lambda)$ is rigorously established.
	
	$\rm (iii)$
	For $j=0$ and $|\alpha|+|\beta|>0$, we have
	\begin{align*}
		&\bigl\| \partial_{\xi}^\alpha \partial_x^\beta \sigma(B)_{-m}^0 \bigr\|_{\mathcal{M}} \\
		=& \Bigl\| \sum_{k=1}^{|\alpha|+|\beta|} \sum_{\substack{\sum \alpha_j=\alpha \\ \sum \beta_j=\beta}}
		C_{\alpha_1 \cdots \alpha_k ; \beta_1 \cdots \beta_k}
		\sigma(A)_m^{-1} \bigl(\partial_{\xi}^{\alpha_1} \partial_x^{\beta_1} \sigma(A)_m\bigr) \sigma(A)_m^{-1} \notag \\
		&\hspace{2em} \cdots \sigma(A)_m^{-1} \bigl(\partial_\xi^{\alpha_k} \partial_x^{\beta_k} \sigma(A)_m\bigr) \sigma(A)_m^{-1} \Bigr\|_{\mathcal{M}} \\
		\le &C \bigl\| \sigma(A)_m^{-1} \bigr\|_{\mathcal{M}}
		\bigl\|\partial_{\xi}^{\alpha_1}\partial_x^{\beta_1} \sigma(A)_m \bigr\|_{\mathcal{M}}
		\bigl\|\sigma(A)_m^{-1} \bigr\|_{\mathcal{M}} \notag \\
		&\hspace{2em} \cdots \bigl\|\sigma(A)_m^{-1} \bigr\|_{\mathcal{M}}
		\bigl\|\partial_\xi^{\alpha_k} \partial_x^{\beta_k} \sigma(A)_m\bigr\|_{\mathcal{M}}
		\bigl\|\sigma(A)_m^{-1} \bigr\|_{\mathcal{M}} \\
		\le& C (1+|\xi|^2)^{\frac{- m -|\alpha|}{2}}(1+|\lambda|)^{-2}.
	\end{align*}
	
	When $j\ge1$, assume the estimate holds for all $k<j$. From the recursive equation
	\begin{align*}
		\partial_{\xi}^\alpha \partial_x^\beta \sigma(B)_{-m-j}^0
		= -\sum \binom{\alpha}{\alpha_1}& \binom{\beta}{\beta_1}
		\partial_{\xi}^{\alpha_1} \partial_x^{\beta_1} \sigma(B)_{-m}^0 \\
		&\cdot \partial_{\xi}^{\alpha_2} \partial_x^{\beta_2} \Bigl(
		\sum_{\substack{k+l+|\gamma|=j \\ l<j}}
		\frac{1}{\gamma!} \partial_{\xi}^\gamma \sigma(B)_{-m-k}^0 \cdot D_x^\gamma \sigma(A)_{m-l}
		\Bigr).
	\end{align*}
	each term in the second sum splits via Leibniz rule into
	\begin{equation} \label{leib product}
		\partial_{\xi}^{\alpha_2+\gamma_1} \partial_x^{\beta_2} \sigma(B)_{-m-k}^0 \cdot D_x^{\gamma_2} \partial_x^{\beta_2} \sigma(A)_{m-l}.
	\end{equation}
	By the induction hypothesis,
	$$
	\big\|\partial_{\xi}^{\alpha_2+\gamma_1} \partial_x^{\beta_2} \sigma(B)_{-m-k}^0\big\|_{\mathcal{M}} \le C\big(1+|\xi|^2\big)^{\frac{- m -k-\left|\alpha_2+\gamma_1\right|}{2}}(1+|\lambda|)^{-2}.
	$$
	For $\sigma(A)_{m-l}$, classical symbol estimates yield
	$$
	\big\|D_x^{\gamma_2} \partial_x^{\beta_2} \sigma(A)_{m-l}\big\|_{\mathcal{M}} \le C\big(1+|\xi|^2\big)^{\frac{ m -l-\left|\gamma_2\right|}{2}}.
	$$
	The norm of the product \eqref{leib product} is bounded by
	$$
	C\big(1+|\xi|^2\big)^{\frac{- m -k-|\alpha_2|-|\gamma_1|+ m -l-|\gamma_2|}{2}}(1+|\lambda|)^{-2} = C\big(1+|\xi|^2\big)^{\frac{-j-|\alpha_2|}{2}}(1+|\lambda|)^{-2}.
	$$
	Absorbing lower-order terms into $C$, we obtain
	$$
	\Big\|\partial_{\xi}^{\alpha_2} \partial_x^{\beta_2}\Big(\sum_{\substack{k+l+|\gamma|=j \\ l<j}} \frac{1}{\gamma!} \partial_{\xi}^\gamma \sigma(B)_{-m-k}^0 \cdot D_x^\gamma \sigma(A)_{m-l}\Big)\Big\|_{\mathcal{M}} \le C\big(1+|\xi|^2\big)^{\frac{-j-\left|\alpha_2\right|}{2}}(1+|\lambda|)^{-2}.
	$$
	Combining with the estimate for $\partial_{\xi}^{\alpha_1} \partial_x^{\beta_1} \sigma(B)_{-m}^0$:
	$$
	\big\|\partial_{\xi}^{\alpha_1} \partial_x^{\beta_1} \sigma(B)_{-m}^0\big\|_{\mathcal{M}} \le C\big(1+|\xi|^2\big)^{\frac{- m -\left|\alpha_1\right|}{2}}(1+|\lambda|)^{-1},
	$$
	the total product satisfies
	$$
	\big\|\partial_{\xi}^\alpha \partial_x^\beta \sigma(B)_{-m-j}^0\big\|_{\mathcal{M}} \le C\big(1+|\xi|^2\big)^{\frac{- m -j-|\alpha|}{2}}(1+|\lambda|)^{-2}.
	$$
	By induction, all higher-order symbols satisfy the claimed derivative estimates. The factor $(1+|\lambda|)^{-2}$ arises from the recursive accumulation of parametric inverses.
\end{proof}

In order to construct $B(\lambda)$ from the family $\big\{\sigma(B)_{-m-j}^0(x,\xi,\lambda)\big\}_{j\ge 0}$, we should select a keyhole-shaped region that is suitable for all $|\xi| \ge \frac{1}{2}$. To achieve this, we define the keyhole-shaped region
$$
\Lambda_{0} := \Big\{ \lambda \in \mathbb{C} \setminus \{0\} : |\lambda| < \tfrac{1}{2^{m+1}} C_e \text{ or } |\arg \lambda - \pi| < \tfrac{\pi}{4} \Big\}.
$$

In the following theorem, we are able to construct $B(\lambda)$, and then establish that the resolvent of an $\mathcal{M}$-valued elliptic  $\Psi$DO is invertible provided $|\lambda|$ is sufficiently large.

\begin{theorem}\label{inverse resolvent}
	Let $A \in \mathrm{C}\Psi^m\big(\mathbb{R}^d; \mathcal{M}\big)$ be elliptic with $  m>0$, and suppose that both $ A $ and its principal symbol $\sigma(A)_m(x, \xi)$ are positive.  Suppose $\lambda$ lies in the keyhole-shaped region $\Lambda_{0}$. Then there exists $l > 0$ such that for $\lambda \in \Lambda_l = \Lambda_{0} \cap \{\lambda : |\lambda| \ge l\}$, the operator $A - \lambda $ is invertible, and the following norm estimate holds:
	\begin{equation} \label{inver estimate}
		\big\| (A - \lambda )^{-1} \big\|_{B(L_2(\mathbb{R}^d)) \overline{\otimes} \mathcal{M}} \le C (1 + |\lambda|)^{-1}.
	\end{equation}
\end{theorem}

\begin{proof}
	In order to construct a real parametrix from the $\mathcal{M}$-valued functions $\sigma(B)_{-m-j}^0(x, \xi, \lambda)$, it is essential to eliminate their singularities at $\xi = 0$. This can be achieved by multiplying these functions by a smooth cut-off function $\varphi$, which is the fixed cutoff function defined in \eqref{eq-cutoff}.
	
	We define the regularized resolvent symbols via spatial truncation as
	$$
	\sigma(B)_{-m-j}(x, \xi, \lambda) = \varphi(\xi) \sigma(B)_{-m-j}^0(x, \xi, \lambda).
	$$
	Let $B_{-m-j}(\lambda)$ denote the corresponding $\mathcal{M}$-valued $\Psi$DO, this operator is well-defined for all $\lambda \in \Lambda_{0}$.

	Using the sequence $\{t_j\}$ selected in the proof of Proposition \ref{asymp}, we find a symbol of the form $ \sum_{j=0}^{\infty} \varphi(t_j^{-1}\xi) \sigma(B)_{-m-j}(x, \xi, \lambda)$, and denote $B(\lambda)$ the corresponding $\Psi$DO, i.e.,
	\begin{equation}
		\label{B-lambda}
		\sigma(B)(x, \xi, \lambda) =  \sum_{j=0}^{\infty} \varphi(t_j^{-1}\xi) \sigma(B)_{-m-j}(x, \xi, \lambda) .\end{equation}
	Since multiplying $\varphi(t_j^{-1}\xi)$ does not affect the homogeneous degree, by the above construction,
	$$\sigma(B)(x, \xi, \lambda)-\sum_{j= 0 }^N\sigma(B)_{-m-j}(x, \xi, \lambda) \in S_{-1}^{-N-1-m}\big(\mathbb{R}^d \times \mathbb{R}^d \times \Lambda_{0}; \mathcal{M}\big)$$
	for any $N>0$. Therefore,
	$$\sigma\Big(B(\lambda) (A - \lambda )\Big)- \sigma\Big(\sum_{j= 0 }^N B_{-m-j}(\lambda)(A - \lambda ) \Big) \in S_{-1}^{-N-1}.$$
	By the construction of $B(\lambda)$ and Proposition \ref{para sym prop}, it follows that
	$$
	\sigma\Big(\sum_{j= 0 }^N B_{-m-j}(\lambda)(A - \lambda ) \Big) - 1_{\mathcal{M}} \in S_{-1}^{-N-1}.
	$$
	Therefore, we conclude that $\sigma\Big(B(\lambda) (A - \lambda )\Big)-1_{\mathcal{M}}\in S_{-1}^{-N-1} $ for any $N>0$. Consequently, $B(\lambda) (A - \lambda ) = I + R(\lambda)$ with $ R(\lambda) \in \Psi_{-1}^{-\infty}\big(\mathbb{R}^d \times \Lambda_{0}; \mathcal{M}\big)$.

	Similarly, using the same construction method, it can be demonstrated that
	$$
	(A - \lambda )B(\lambda) = I + R'(\lambda),
	$$
	where $R(\lambda)' \in \Psi_{-1}^{-\infty}\big(\mathbb{R}^d \times \Lambda_{0}; \mathcal{M}\big)$.
	
	By Corollary \ref{inver coro}, there exist constants $C_1, C_2, C_3 > 0$ such that the following norm estimates hold:
	$$
	\big\| B(\lambda) \big\|_{B(L_2(\mathbb{R}^d)) \overline{\otimes} \mathcal{M}} \le C_1 (1 + |\lambda|)^{-1},
	$$
	\begin{equation}\label{R lamda estimate}
		\big\| R(\lambda) \big\|_{B(L_2(\mathbb{R}^d)) \overline{\otimes} \mathcal{M}} \le C_2 (1 + |\lambda|)^{-1},
	\end{equation}
	$$
	\big\| R'(\lambda) \big\|_{B(L_2(\mathbb{R}^d)) \overline{\otimes} \mathcal{M}} \le C_3 (1 + |\lambda|)^{-1}.
	$$
	Hence, there exists $l > 0$ such that for $|\lambda| \ge l$ we have
	\[
	\| R(\lambda) \|_{B(L_2(\mathbb{R}^d)) \overline{\otimes} \mathcal{M}} < \frac{1}{2}
	\quad \text{and} \quad
	\| R'(\lambda) \|_{B(L_2(\mathbb{R}^d)) \overline{\otimes} \mathcal{M}} < \frac{1}{2}.
	\]
	Consequently, for $|\lambda| \ge l$,
	$$
	\big\| (I + R(\lambda))^{-1} \big\|_{B(L_2(\mathbb{R}^d)) \overline{\otimes} \mathcal{M}} \le (1 - \| R(\lambda) \|_{B(L_2(\mathbb{R}^d)) \overline{\otimes} \mathcal{M}})^{-1} < 2,
	$$
	$$
	\big\| (I + R'(\lambda))^{-1} \big\|_{B(L_2(\mathbb{R}^d)) \overline{\otimes} \mathcal{M}} \le (1 - \| R(\lambda)' \|_{B(L_2(\mathbb{R}^d)) \overline{\otimes} \mathcal{M}})^{-1} < 2.
	$$
	This implies that $I + R(\lambda)$ and $I + R(\lambda)'$ are invertible. Therefore, a left inverse of $A - \lambda $ on $L_2(\mathcal{N})$ is given by $(I + R(\lambda))^{-1}B(\lambda)$. Since
	\begin{align*}
		&(A - \lambda )(I + R(\lambda))^{-1}B(\lambda) \\
		= & (A - \lambda ) \Big[\sum_{j=0}^{\infty}\big(I-B(\lambda)(A - \lambda )\big)^{j} \Big] B(\lambda) \\
		= &  \Big[\sum_{j=0}^{\infty}\big(I-(A - \lambda )B(\lambda)\big)^{j} \Big] (A - \lambda )B(\lambda)= I,
	\end{align*}
	the left inverse coincides with the right inverse, implying that $A - \lambda $ is invertible. Moreover, we obtain the desired norm estimate:
	\begin{equation} \label{inver A estimate}
		\begin{split}
			\big\| (A - \lambda )^{-1} \big\|_{B(L_2(\mathbb{R}^d)) \overline{\otimes} \mathcal{M}}
			&\le \big\| (I + R(\lambda))^{-1} \big\|_{B(L_2(\mathbb{R}^d)) \overline{\otimes} \mathcal{M}} \cdot \big\| B(\lambda) \big\|_{B(L_2(\mathbb{R}^d)) \overline{\otimes} \mathcal{M}} \\
			&\le C (1 + |\lambda|)^{-1}.
		\end{split}
	\end{equation}
	for $\lambda \in \Lambda_l$.
\end{proof}

In the following corollary, we provide the connection between the inverse and the parametrix of $A-\lambda $.

\begin{corollary}\label{inverse-para}
	Let $A \in \mathrm{C}\Psi^m\big(\mathbb{R}^d; \mathcal{M}\big)$ be elliptic with $  m>0$, and suppose that both $ A $ and its principal symbol $\sigma(A)_m(x, \xi)$ are positive.  Suppose $\lambda$ lies in the keyhole-shaped region $\Lambda_{0}$. Then there exists $l>0$ such that for $\lambda \in \Lambda_l$,
	$$
	\big\|B(\lambda)-(A-\lambda )^{-1}\big\|_{B(L_2(\mathbb{R}^d)) \overline{\otimes} \mathcal{M}} \le C(1+|\lambda|)^{-1}.
	$$
\end{corollary}
\begin{proof}
	It follows from \eqref{R lamda estimate} and \eqref{inver A estimate} that
	$$
	\begin{aligned}
		\big\|B(\lambda)-(A-\lambda )^{-1}\big\|_{B(L_2(\mathbb{R}^d)) \overline{\otimes} \mathcal{M}} & =\big\|R(\lambda)(A-\lambda )^{-1}\big\|_{B(L_2(\mathbb{R}^d)) \overline{\otimes} \mathcal{M}} \\
		& \le \big\|R(\lambda)\big\|_{B(L_2(\mathbb{R}^d)) \overline{\otimes} \mathcal{M}}\big\|(A-\lambda )^{-1}\big\|_{B(L_2(\mathbb{R}^d)) \overline{\otimes} \mathcal{M}} \\
		& \le C(1+|\lambda|)^{-1}.
	\end{aligned}
	$$
\end{proof}

\begin{rk}\label{inver A-lamda}
	If $\lambda \in \Lambda_0$, and $\lambda$ also belongs to the resolvent set of $A$, then the proof of Theorem \ref{inverse A PDO} also indicates that $(A - \lambda)^{-1}$ defines an $\mathcal{M}$-valued $\Psi$DO, i.e.,
	$$
	(A - \lambda)^{-1} \in \Psi^{-  m}(\mathbb{R}^d ; \mathcal{M}).
	$$
	Indeed, the resolvent
	$$
	(A - \lambda)^{-1} : L_2(\mathbb{R}^d; L_2(\mathcal{M})) \to L_2 (\mathbb{R}^d; L_2(\mathcal{M}))
	$$
	is known a priori to be bounded, ensuring $ (A - \lambda)^{-1} = B(\lambda) - R(\lambda) (A - \lambda )^{-1}   $. But by Proposition \ref{smooth-op-kernel}, $R(\lambda) (A - \lambda )^{-1} $ still belongs to $\Psi^{-\infty}\big(\mathbb{R}^d;\mathcal{M}\big)$, so $
	(A - \lambda)^{-1} \in \Psi^{-  m}(\mathbb{R}^d ; \mathcal{M})
	$.
\end{rk}

\section{Structure theory for complex powers of $\mathcal{M}$-valued elliptic classical  $\Psi$DOs }\label{Section complex1}

In this section we develop the theory of complex powers of $\mathcal{M}$-valued elliptic $\Psi$DOs. Extending Seeley's construction, we show that complex powers remain within the classical pseudo-differential calculus and admit explicit symbolic descriptions. The meromorphic continuation of the Schwartz kernel of $A^z$ is analyzed, identifying the poles and residues that are essential for the $\zeta$-function approach in Section \ref{section asymp}.

\subsection{Complex powers of $\mathcal{M}$-valued elliptic  $\Psi$DOs }\label{Section complex}

\subsubsection{Holomorphic semigroup construction} \label{semititle}

In order to develop the parameter-dependent pseudo-differential calculus and to perform analytic continuation arguments for zeta functions, we shall work with $X$-valued holomorphic functions on domains of the complex plane. For convenience, we recall the following standard notion (see \cite[Section B.3]{Hyt16}), which will be used repeatedly in the sequel.

\begin{definition}
	A domain is a connected open subset of $\mathbb{C}$. If $D \subseteq \mathbb{C}$ is a domain and $X$ is a Banach space, a function $f: D \to X$ is called holomorphic if for all $z_0 \in D$, the limit
	$$
	f'(z_0) := \lim_{z \to z_0} \frac{f(z) - f(z_0)}{z - z_0}
	$$
	exists in $X$.
\end{definition}
This definition coincides with the usual one in the scalar-valued case $X = \mathbb{C}$, and ensures that familiar tools from complex analysis, such as Cauchy's integral formula and analytic continuation, remain valid in the $\mathcal{M}$-valued setting.

As in the previous section, let $ A \in \mathrm{C}\Psi^m\big(\mathbb{R}^d; \mathcal{M}\big) $ be elliptic with $  m>0$, and suppose that both $ A $ and its principal symbol $\sigma(A)_m(x, \xi)$ are positive for all $(x, \xi) \in \mathbb{R}^d \times (\mathbb{R}^d \setminus \{0\})$.
For the remainder of the analysis, we may assume without loss of generality that $A$ is strictly positive, i.e., $A\ge 0$, and there exists $\rho>0$ such that
$$
\operatorname{dist}\Big( \operatorname{Spec}A,0\Big)\ge \rho.
$$
This assumption is not essential, as it can always be fulfilled by replacing $ A $ with $ A + \rho $ for some $ \rho > 0 $.

Now consider the keyhole-shaped region defined by
$$
\Lambda_{\rho} = \left\{ \lambda \in \mathbb{C} \setminus \{0\} : |\lambda| < \rho \text{ or } |\arg \lambda - \pi| < \tfrac{\pi}{4} \right\}.
$$

We now construct a spectral contour $ \Gamma = \Gamma_{\frac{\rho}{2}} $ as follows:

\begin{definition}\label{az definition}
	Let $\Gamma$ be the contour consisting of a descending radial segment from $-\infty + i\epsilon$ to $-\sqrt{(\frac{\rho}{2})^2 - \epsilon^2} + i\epsilon$ along $\arg \lambda = \pi$, a clockwise circular arc around the origin with radius $\frac{\rho}{2}$ connecting $-\sqrt{(\frac{\rho}{2})^2 - \epsilon^2} + i\epsilon$ to $-\sqrt{(\frac{\rho}{2})^2 - \epsilon^2} - i\epsilon$, and an ascending radial segment from $-\sqrt{(\frac{\rho}{2})^2 - \epsilon^2} - i\epsilon$ to $-\infty - i\epsilon$ along $\arg \lambda = -\pi$, where $\epsilon > 0$ is a fixed small parameter. Define the complex powers of $A$ through the Dunford integral
	\begin{equation} \label{az}
		A_z = \frac{i}{2 \pi} \int_{\Gamma} \lambda^z (A - \lambda)^{-1} d\lambda,
	\end{equation}
	where $\Re z < 0$, and the complex logarithm is defined using the standard branch cut along $(-\infty, 0]$, with $\lambda^z := e^{z(\ln |\lambda| + i \arg \lambda)}$ for $\lambda \in \mathbb{C} \setminus (-\infty, 0]$. Along $\Gamma$, we explicitly define:
	$$
	\lambda^z =
	\begin{cases}
		|\lambda|^z e^{i z \pi}, & \text{on the upper radial segment}; \\
		|\lambda|^z e^{-i z \pi}, & \text{on the lower radial segment}; \\
		|\lambda|^z e^{i z \theta}, & \text{on the circular arc, } \theta \in (-\pi, \pi].
	\end{cases}
	$$
\end{definition}

\begin{figure}[htbp]
	\centering
	\begin{tikzpicture}[>={Stealth[scale=1.2]}, every node/.style={font=\small}]
		
		\draw[->] (-4,0) -- (2,0) node[below] {$\mathrm{Re}\,\lambda$};
		\draw[->] (0,-1.5) -- (0,1.5) node[left] {$\mathrm{Im}\,\lambda$};

		\draw[thick] (-4,0) -- (0,0) node[midway, below=8pt] {};
		
		\def\radius{1}

		\draw[thick, ->] (-3.8,0.08) -- (-\radius,0.08)
		node[midway, above=2pt] {$\arg\lambda = \pi$};

		\draw[thick, ->] (180:\radius) arc (180:-180:\radius);
		\node at (0.98*\radius,-0.98*\radius) {\scriptsize $r = \frac{\rho}{2}$};
		
		\draw[thick, ->] (-\radius,-0.08) -- (-3.8,-0.08)
		node[midway, below=2pt] {$\arg\lambda = -\pi$};
		
		\node at (-\radius,0) [above left=-1pt] {};
	\end{tikzpicture}
	\caption{Integration contour $\Gamma$ for the Dunford integral (\ref{az})}.
	\label{fig:contour}
\end{figure}
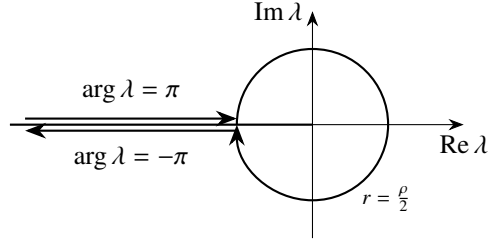

\begin{rk}\label{remark of az}
	$\rm (i)$ Let $l > 0$ be the constant from Theorem \ref{inverse resolvent}. For $\Re z < 0$ and $|\lambda| \ge l$, the radial integrals converge absolutely due to the decay of order $O(|\lambda|^{\Re z - 1})$. Since $\Re z - 1 < -1$, the integral converges at infinity.
	
	$\rm (ii)$
	Consider the circular arc $|\lambda| = \frac{\rho}{2}$. For each $\lambda_0$ on the arc, $\lambda_0 \notin \operatorname{Spec}(A)$, so $(A - \lambda_0)^{-1} \in B(L_2(\mathbb{R}^d)) \overline{\otimes} \mathcal{M}$ is well-defined. For $\lambda$ in a sufficiently small neighborhood of $\lambda_0$, we can write
	$$
	A - \lambda = (A - \lambda_0)\Big(I - (\lambda - \lambda_0)(A - \lambda_0)^{-1}\Big).
	$$
	If $|\lambda - \lambda_0| < \big\|(A - \lambda_0)^{-1}\big\|_{B(L_2(\mathbb{R}^d))\overline{\otimes} \mathcal{M}}^{-1}$, the Neumann series
	$$
	\Big(I - (\lambda - \lambda_0)(A - \lambda_0)^{-1}\Big)^{-1} = \sum_{k=0}^\infty (\lambda - \lambda_0)^k (A - \lambda_0)^{-k}
	$$
	converges in the operator norm, giving the local expansion
	$$
	(A - \lambda)^{-1} = \sum_{k=0}^{\infty} (\lambda - \lambda_0)^k (A - \lambda_0)^{-k-1}.
	$$
	This shows that $(A - \lambda)^{-1}$ is analytic in a neighborhood of $\lambda_0$. By compactness of the circular arc, $(A - \lambda)^{-1}$ is analytic along the entire arc. Moreover, by continuity, the integral
	$\int_{|\lambda| = \frac{\rho}{2}} \lambda^z (A - \lambda)^{-1} \, d\lambda$ converges absolutely in the operator norm, with
	$$
	\sup_{|\lambda| = \frac{\rho}{2}} \big\|\lambda^z (A - \lambda)^{-1}\big\|_{B(L_2(\mathbb{R}^d))\overline{\otimes} \mathcal{M}} \le \Big(\tfrac{\rho}{2}\Big)^{\Re z} M(\rho),
	$$
	where $M(\rho) = \sup_{|\lambda| = \frac{\rho}{2}} \big\|(A - \lambda)^{-1}\big\|_{B(L_2(\mathbb{R}^d))\overline{\otimes} \mathcal{M}}$.

	$\rm (iii)$ Consider an alternative contour $ \Gamma_{\frac{\rho'}{2}}$ with $\rho < \rho' < 2\rho$, where the radial segments lie along $\arg \lambda = \pi - \delta$ and $\arg \lambda = -\pi + \delta$ for $0 < \delta < \tfrac{\pi}{4}$. Since the resolvent $(A - \lambda)^{-1}$ is holomorphic in $\mathbb{C} \setminus \operatorname{Spec}(A)$ and $\operatorname{Spec}(A) \cap \Lambda = \emptyset$, both $\Gamma_{\frac{\rho}{2}}$ and $\Gamma_{\frac{\rho'}{2}}$ are homotopic in $\Lambda_\rho$. By Cauchy's integral theorem,
	$$
	\frac{i}{2\pi} \int_{\Gamma_{\frac{\rho'}{2}}} \lambda^z (A - \lambda)^{-1} d\lambda = \frac{i}{2\pi} \int_{\Gamma_{\frac{\rho}{2}}} \lambda^z (A - \lambda)^{-1} d\lambda.
	$$
	Thus, Definition \ref{az definition} is independent of the choice of $\rho$ and $\delta$.
\end{rk}

\begin{proposition}\label{azprop}
	The complex powers $A_z$ defined in \eqref{az} satisfy the following properties:
	
	\begin{enumerate}[$\rm (i)$]
		\item For $\Re z < 0$ and $\Re w < 0$, the semigroup property holds:
		\begin{equation} \label{azwsemi}
			A_z A_w = A_{z+w}.
		\end{equation}
		
		\item For $k \in \mathbb{N}$,
		\begin{equation} \label{ak}
			A_{-k} = \big(A^{-1}\big)^k.
		\end{equation}
		
		\item $A_z$ defines a $B(L_2(\mathbb{R}^d)) \overline{\otimes} \mathcal{M}$-valued holomorphic function  for $\Re z < 0$.
	\end{enumerate}
\end{proposition}

\begin{proof}
	$\rm (i)$
	Let $\Gamma^{\prime}$ denote the modified contour consisting of a radial segment from $-\infty$ to $-\frac{3\rho}{4}$ along $\arg \lambda=\frac{7\pi}{8}$, a clockwise circular arc around the origin with radius $\frac{3\rho}{4}$, and a radial segment from $-\frac{3\rho}{4}$ to $-\infty$ along $\arg \lambda=-\frac{7\pi}{8}$, where $\Gamma^{\prime}$  maintains spectral avoidance.

	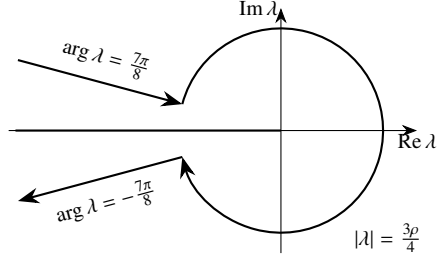
\begin{figure}[htbp]
		\centering
		\begin{tikzpicture}[>={Stealth[scale=1.2]},
			every node/.style={font=\scriptsize, inner sep=1pt},
			scale=0.9]
			\draw[->] (-4,0) -- (2,0) node[below] {$\mathrm{Re}\,\lambda$};
			\draw[->] (0,-1.8) -- (0,1.8) node[left] {$\mathrm{Im}\,\lambda$};
			\draw[thick] (-3.9,0) -- (0,0);
			\def\radius{1.5}
			\def\deltaAngle{15}
			\draw[thick, ->] (180-\deltaAngle:4) -- (180-\deltaAngle:\radius)
			node[midway, above=3pt, sloped] {$\arg\lambda = \frac{7\pi}{8}$};
			\draw[thick, ->] (180-\deltaAngle:\radius) arc (180-\deltaAngle:-180+\deltaAngle:\radius);
			\node at (1.05*\radius,-1.05*\radius) {$|\lambda| = \frac{3\rho}{4}$};
			\draw[thick, ->] (-180+\deltaAngle:\radius) -- (-180+\deltaAngle:4)
			node[midway, below=3pt, sloped] {$\arg\lambda = -\frac{7\pi}{8}$};
		\end{tikzpicture}
		\caption{Modified contour $\Gamma'$ with angular offset $\frac{\pi}{8}$ and
			radius $\frac{3\rho}{4}$. }
		\label{fig:modified_contour}
	\end{figure}

	The integral \eqref{az} remains invariant under contour deformation by Remark \ref{remark of az} $\rm (iii)$. Applying the resolvent identity and Fubini's theorem, we obtain
	\begin{align*}
		A_z A_w &= -\frac{1}{4 \pi^2} \int_{\Gamma'} \int_{\Gamma} (A - \lambda )^{-1} (A - \mu )^{-1} \lambda^z \mu^w d\mu d\lambda \\
		&= -\frac{1}{4 \pi^2} \int_{\Gamma'} \int_{\Gamma} \frac{\lambda^z \mu^w}{\lambda - \mu} \left[(A - \lambda )^{-1} - (A - \mu )^{-1}\right] d\mu d\lambda \\
		&= \frac{i}{2 \pi} \int_{\Gamma'} \lambda^{z+w} (A - \lambda )^{-1} d\lambda + \frac{1}{4 \pi^2} \int_{\Gamma}(A - \mu )^{-1} \int_{\Gamma'}  \frac{\lambda^z \mu^w}{\lambda - \mu} d\lambda d\mu \\
		&= A_{z+w} + 0 = A_{z+w}.
	\end{align*}
	
	$\rm (ii)$ If $z = -1, -2, \ldots$, then $(r e^{i \pi})^z = (r e^{-i \pi})^z$, and the integrals along the straight line parts of $\Gamma$ in \eqref{az} cancel. Therefore,
	$$
	A_{-k}=\frac{i}{2 \pi} \oint_{|\lambda|=\frac{\rho}{2}} \lambda^{-k}(A-\lambda )^{-1} d \lambda.
	$$
	Under the substitution $\mu=\lambda^{-1}$ with $d \mu=-\lambda^{-2} d \lambda$, we have
	\begin{align*}
		A_{-k} &= -\frac{i}{2\pi} \oint_{|\mu|=\frac{2}{\rho}} \mu^{k-2} (A - \mu^{-1}I)^{-1} d\mu \\
		&= -\frac{i}{2\pi} A^{-1} \oint_{|\mu|=\frac{2}{\rho}} \mu^{k-1} (\mu  - A^{-1})^{-1} d\mu \\
		&= A^{-1}(A^{-1})^{k-1} = (A^{-1})^k .
	\end{align*}
	
	$\rm (iii)$
	For $\Re z \le-\varepsilon<0$, by differentiating the integral \eqref{az} with respect to $z$, we obtain
	$$
	\frac{d^n}{d z^n} A_z=\frac{i}{2 \pi} \int_{\Gamma} \lambda^z(\ln \lambda)^{n}(A-\lambda )^{-1} d \lambda, \quad n\in \mathbb{Z}_{+}.
	$$
	Due to the exponential decay $O(|\lambda|^{\Re z-1})$ from \eqref{inver estimate} on radial segments, and the uniform boundedness of $(\ln \lambda)^{n}(A-\lambda )^{-1}$ on the circular arc, the integral above converges absolutely in operator norm. This establishes the holomorphy of $z \mapsto A_z$ through norm-differentiability.
\end{proof}

\subsubsection{Extension to general complex exponents} \label{entiretitle}

For complex exponents $z \in \mathbb{C}$ and integer regulators $k \in \mathbb{Z}$ satisfying $\Re z<k$, we extend the complex power operators through the recursive relation.

\begin{definition}
	Let $z \in \mathbb{C}$ and $k= \lfloor \Re z \rfloor +1$, we define
	\begin{equation} \label{azw}
		A^{z}:=A^{k} A_{z-k},
	\end{equation}
	where $A_{z-k}$ is defined via the integral representation \eqref{az}.
\end{definition}

The following theorem establishes the consistency and fundamental properties of this extension.

\begin{theorem}\label{groupaz}
	
	\begin{enumerate}[$\rm (i)$]
		\item The operator $A^z$ as defined by \eqref{azw} is independent of the choice of integer $k$, provided $\Re z<k$.
		
		\item If $\Re z<0$, then $A^z=A_z$.
		
		\item For all $z, w \in \mathbb{C}$, the group property holds:
		\begin{equation}\label{agroup}
			A^z A^w=A^{z+w}.
		\end{equation}
		
		\item For $k \in \mathbb{Z}$, \eqref{azw} reduces to standard operator powers, i.e. $A^0 = I$, $A^1 = A$, and $A^{-1}$ coincides with the inverse operator.
	\end{enumerate}
\end{theorem}

\begin{proof}
	$\rm (i)$ Let $k, l \in \mathbb{Z}$ with $k, l>\Re z$. Without loss of generality, assume $k>l$ and set $p:=k-l \in \mathbb{N}, w:=z-k$. The equality $A^k A_{z-k}=A^l A_{z-l}$ reduces through left-multiplication by $A^{-l}$ to demonstrating
	\begin{equation} \label{awp}
		A_{w} = A^{-p} A_{w+p}.
	\end{equation}
	Since $\Re(w+p)=\Re(z-l)<0$, Proposition \ref{azprop} $\rm (i)$ and the integer power relation \eqref{ak} yield
	$$
	A^{-p} A_{w+p}=(A^{-1})^{p} A_{w+p}=A_{-p} A_{w+p}=A_{w}.
	$$

	$\rm (ii)$ Immediate from $\rm (i)$ with $k=0$ and $\Re z<0$.
	
	$\rm (iii)$ Select $k, j \in \mathbb{Z}$ with $k>\Re z, j>\Re w$, then
	\begin{align*}
		A^z A^w &= (A^k A_{z-k})(A^j A_{w-j}) \\
		&= A^{k+j} A_{z-k} A_{w-j}  \\
		&= A^{k+j} A_{(z+w)-(k+j)}  \\
		&= A^{z+w}.
	\end{align*}
	
	$\rm (iv)$ For $k=0$,
	$$
	A^0=A^1 A_{-1}=A A^{-1}= I.
	$$
	Positive integers follow inductively
	$$
	A^{k+1}=A^k A^1=A^k A.
	$$
	Negative integers are covered by Proposition \ref{azprop} $\rm (ii)$. The inverse property $A^{-1} A= I$ follows from \eqref{ak} with $k=1$.
\end{proof}

\subsection{Symbolic calculus for complex powers with $\Re z < 0$}

As before, let $A \in \mathrm{C}\Psi^m\big(\mathbb{R}^d; \mathcal{M}\big)$ be elliptic with $  m>0$, and suppose that both $ A $ and its principal symbol $\sigma(A)_m(x, \xi)$ are positive for all $(x, \xi) \in \mathbb{R}^d \times (\mathbb{R}^d \setminus \{0\})$.
By Proposition \ref{sa M elliptic}, there exists a constant $C_e > 0$ such that $\sigma(A)_m(x, \xi) \ge C_e$. To construct the homogeneous components of the symbol of $A^z$, we extend the parametrix construction via resolvent integration.

Recall the keyhole-shaped region $\Lambda_{\xi}$ defined as
$$
\Lambda_{\xi} = \Big\{ \lambda \in \mathbb{C} \setminus \{0\} : |\lambda| < \tfrac{1}{2}C_e(\xi) \text{ or } |\arg \lambda - \pi| < \tfrac{\pi}{4} \Big\},
$$
where $C_e(\xi) = C_e|\xi|^m$ for $0 < |\xi| \le 1$, and $C_e(\xi) = C_e$ for $|\xi| \ge 1$.

We define the integration contours $\Gamma_{\xi}$ and $\Gamma_{\xi}'$ to be chosen inside $\Lambda_{\xi}$, having the same geometric shape as those in Section \ref{Section complex}, with the parameter $\rho$ replaced by $\frac{1}{2} C_e(\xi)$.

\begin{definition}\label{complex_symbols}
	For $\Re z < 0$, define the homogeneous symbol components of order $m z - j$ by
	\begin{equation} \label{bmzj}
		\sigma(B)_{m z - j}^{(z),0}(x, \xi) = \frac{i}{2\pi} \int_{\Gamma_{\xi}} \lambda^z \sigma(B)_{-m-j}^0(x, \xi, \lambda)  d\lambda, \quad j \in \mathbb{N}_0,
	\end{equation}
	which is well-defined for all $\xi \neq 0$. In particular, for $j = 0$, Cauchy's integral formula yields
	\begin{equation} \label{Bmzj,0z}
		\sigma(B)_{m z}^{(z),0}(x, \xi) = \frac{i}{2\pi} \int_{\Gamma_{\xi}} \lambda^z (\sigma(A)_m(x, \xi) - \lambda)^{-1} d\lambda = \sigma(A)_m(x, \xi)^z.
	\end{equation}
\end{definition}

\begin{proposition}  \label{bzhomo}
	For $\Re z < 0$ and $\xi\neq 0$, the symbols $\sigma(B)_{m z - j}^{(z),0}$ exhibit positive homogeneity of degree $m z - j$:
	$$
	\sigma(B)_{m z - j}^{(z),0}(x, t \xi) = t^{m z - j} \sigma(B)_{m z - j}^{(z),0}(x, \xi), \quad \forall t > 0.
	$$
\end{proposition}

\begin{proof}
	Let $\Gamma_{\xi}^t = t^{-m} \Gamma_{\xi}$. Using the homogeneity $\sigma(B)_{-m-j}^0(x, t\xi, t^m \lambda) = t^{-m-j} \sigma(B)_{-m-j}^0(x, \xi, \lambda)$, we have
	\begin{align*}
		\sigma(B)_{m z - j}^{(z),0}(x, t \xi)
		& = \frac{i}{2\pi} \int_{\Gamma_{\xi}} \lambda^z \sigma(B)_{-m-j}^0(x, t \xi, \lambda)  d\lambda \\
		& = \frac{i}{2\pi} \int_{\Gamma_{\xi}^t} (t^m \mu)^z \sigma(B)_{-m-j}^0(x, t \xi, t^m \mu) \cdot t^m  d\mu \quad (\text{setting } \lambda = t^m \mu) \\
		&= t^{m z - j} \sigma(B)_{m z - j}^{(z),0}(x, \xi).
	\end{align*}
\end{proof}

The symbolic calculus extends naturally to complex exponents:
\begin{proposition}\label{prop-semigroup}
	\begin{enumerate}[$\rm (i)$]
		\item For $\Re z, \Re w < 0$ and $j \in \mathbb{N}_0$, the semigroup property
		\begin{equation} \label{sym semigroup}
			\sum_{|\alpha|+p+q=j} \frac{1}{\alpha!} \partial_\xi^\alpha \sigma(B)_{m z-p}^{(z),0}(x, \xi) D_x^\alpha \sigma(B)_{m w-q}^{(w),0} (x, \xi) = \sigma(B)_{m(z+w)-j}^{(z+w),0}(x, \xi),
		\end{equation}
		valid for all $|\xi| \ge 1$.
		
		\item For $k \in \mathbb{N}$, the sequence $\Big\{\sigma(B)_{-mk-j}^{(-k),0}(x, \xi)\Big\}_{j\ge0}$ constitutes homogeneous components of the parametrix for $A^{k}$.
		
		\item For any multi-indices $\alpha,\beta$, the derivative $\partial_\xi^\alpha\partial_x^\beta\sigma(B)_{mz-j}^{(z),0}(x, \xi)$ is a holomorphic $\mathcal{M}$-valued function in $z$ for $\Re z < 0$ and $(x, \xi) \in \mathbb{R}^d \times (\mathbb{R}^d \backslash \{0\})$.
	\end{enumerate}
\end{proposition}

\begin{proof}
	$\rm (i)$
	From the resolvent identity, we have
	$$
	\sigma\big((A - \lambda )^{-1} (A - \mu )^{-1}\big) = \frac{1}{\lambda - \mu} \left[\sigma((A - \lambda )^{-1}) - \sigma((A - \mu )^{-1})\right].
	$$
	Given that
	$$
	\sigma\big((A - \lambda )^{-1} (A-\mu )^{-1}\big) \sim \sum_{j=0}^{\infty} \sum_{|\alpha|+p+q=j}\frac{1}{\alpha!}\partial_\xi^\alpha \sigma(B)_{-m-p}^0(x, \xi, \lambda) D_x^\alpha \sigma(B)_{-w-q}^0(x, \xi, \mu),
	$$
	and
	$$
	\sigma\big((A - \lambda )^{-1}\big) - \sigma((A-\mu )^{-1}) \sim \sum_{j=0}^{\infty} \left[\sigma(B)_{-m-j}^0(x, \xi, \lambda) - \sigma(B)_{-w-j}^0(x, \xi, \mu)\right],
	$$
	by matching the homogeneous components for $j \in \mathbb{N}_{0}$, we obtain
	\begin{align*}
		&\sum_{|\alpha|+p+q=j} \frac{1}{\alpha!} \partial_\xi^\alpha \sigma(B)_{-m-p}^0(x, \xi, \lambda) D_x^\alpha \sigma(B)_{-w-q}^0(x, \xi, \mu) \\
		&= \frac{\sigma(B)_{-m-j}^0(x, \xi, \lambda) - \sigma(B)_{-w-j}^0(x, \xi, \mu)}{\lambda - \mu}, \quad |\xi| \ge 1.
	\end{align*}
	Applying Definition \ref{bmzj}, we derive
	\begin{align*}
		&\sum_{|\alpha|+p+q=j} \frac{1}{\alpha!} \partial_\xi^\alpha \sigma(B)_{m z-p}^{(z),0}(x, \xi) D_x^\alpha \sigma(B)_{m w-q}^{(w),0} (x, \xi)\\
		=& -\frac{1}{4 \pi^2} \int_{\Gamma_{\xi}'} \int_{\Gamma_{\xi}} \lambda^z \mu^w \Biggl(
		\sum_{\substack{|\alpha|+p+q=j}} \frac{1}{\alpha!} \partial_\xi^\alpha \sigma(B)_{-m-p}^0(x, \xi, \lambda) \\
		&\hspace{15em} \cdot D_x^\alpha \sigma(B)_{-w-q}^0(x, \xi, \mu)
		\Biggr) d\mu \, d\lambda\\
		= & -\frac{1}{4 \pi^2} \int_{\Gamma_{\xi}'} \int_{\Gamma_{\xi}} \frac{\lambda^z \mu^w}{\lambda - \mu} \Big(\sigma(B)_{-m-j}^0(x, \xi, \lambda) - \sigma(B)_{-w-j}^0(x, \xi, \mu)\Big) d\mu d\lambda \\
		= & \frac{i}{2 \pi} \int_{\Gamma_{\xi}'} \lambda^{z+w} \sigma(B)_{-m-j}^0(x, \xi, \lambda) d\lambda + \frac{1}{4 \pi^2} \int_{\Gamma_{\xi}}\sigma(B)_{-w-j}^0(x, \xi, \mu) \int_{\Gamma_{\xi}'}  \frac{\lambda^z \mu^w}{\lambda - \mu} d\lambda d\mu \\
		= & \sigma(B)_{m (z+w)-j}^{(z+w),0}(x, \xi).
	\end{align*}

	$\rm (ii)$ For $k \in \mathbb{N}$, Proposition \ref{azprop} $\rm (ii)$ implies $A_{-k} = (A^{-1})^k$. The construction of the parametrix for $A^k$ mirrors the symbol recursion described in Definition \ref{complex_symbols}.
	
	$\rm (iii)$ For any multi-indices $\alpha, \beta$, we have
	$$
	\partial_{\xi}^\alpha \partial_x^\beta \sigma(B)_{mz-j}^{(z), 0}(x, \xi) = \frac{i}{2 \pi} \int_{\Gamma_{\xi}} \lambda^z \partial_{\xi}^\alpha \partial_x^\beta \sigma(B)_{-m-j}^0(x, \xi, \lambda) d \lambda.
	$$
	Differentiating this integral with respect to $z$, we obtain
	\begin{equation} \label{dz}
		\frac{d^n}{d z^n}  \partial_{\xi}^\alpha \partial_x^\beta \sigma(B)_{mz-j}^{(z), 0}(x, \xi) = \frac{i}{2 \pi} \int_{\Gamma_{\xi}} \lambda^z (\ln\lambda)^{n} \partial_{\xi}^\alpha \partial_x^\beta \sigma(B)_{-m-j}^0(x, \xi, \lambda) d \lambda, \quad n\in \mathbb{Z}_{+}.
	\end{equation}
	By Proposition \ref{para sym prop} $\rm (i)$ and $\rm (iii)$, the integral in \eqref{dz} converges in $\mathcal{M}$ for $\Re z \le -\varepsilon < 0$. Hence, the derivative $\partial_{\xi}^\alpha \partial_x^\beta \sigma(B)_{mz-j}^{(z), 0}(x, \xi)$ is a holomorphic $\mathcal{M}$-valued function of $z$ for $\Re z < 0$ and $\xi \neq 0$.
\end{proof}

\subsection{Regularized symbols for $\Re z < 0$}

The homogeneous components $\sigma(B)_{mz-j}^{(z),0}(x,\xi)$ constructed in the previous section via \eqref{bmzj} are smooth for $\xi\neq0$ but typically singular at $\xi=0$. To eliminate this singularity while preserving the asymptotic expansion as $|\xi|\to\infty$, we regularize them by multiplying with the smooth cutoff function $\varphi(\xi)$ from \eqref{eq-cutoff}.
The resulting symbols are globally smooth and belong to the standard $\mathcal{M}$-valued classical symbol classes.

For $k\in \mathbb{N}$, let $\sigma(A^k)_j(x, \xi)$ denote the degree-$j$ homogeneous components of the symbol $\sigma(A^k)(x, \xi)$, satisfying the asymptotic expansion
$$
\sigma(A^k)(x, \xi) \sim \sum_{j=-\infty}^{m k} \sigma(A^k)_j(x, \xi), \quad |\xi| \ge 1.
$$

We now introduce regularized symbols by applying the fixed smooth cutoff function $\varphi$, defined in \eqref{eq-cutoff}, to the homogeneous symbol components constructed in Section \ref{Section complex}.

\begin{definition}\label{def:regularized-sym}
	Let $\Re z <0$. For each $j \in \mathbb{N}_0$, define the \emph{regularized symbol} by
	$$
	\sigma(B)_{mz-j}^{(z)}(x, \xi) \coloneqq \varphi(\xi) \, \sigma(B)_{mz-j}^{(z),0}(x, \xi),
	$$
	where $\sigma(B)_{mz-j}^{(z),0}$ is the homogeneous component from \eqref{bmzj}. The corresponding $\mathcal{M}$-valued $\Psi$DO is denoted by $B_{mz-j}^{(z)}$.
\end{definition}

\begin{definition}\label{def:truncated-sum}
	Let $\Re z <0$. For $N \in \mathbb{N}$, define the truncated sum
	$$
	B_{(N)}^{(z)} \coloneqq \sum_{j=0}^{N-1} B_{mz-j}^{(z)} \in \mathrm{C}\Psi^{mz}(\mathbb{R}^d; \mathcal{M}).
	$$
\end{definition}

For $ \Re z < 0 $, it follows from Proposition \ref{asymp} that there exists a classical $ \Psi DO $
$$
B^{(z)} \in \mathrm{C}\Psi^{mz}\big(\mathbb{R}^d; \mathcal{M}\big)
$$
whose symbol admits the asymptotic expansion
$$
\sigma(B^{(z)})(x, \xi) \sim \sum_{j=0}^\infty \sigma(B)_{mz-j}^{(z)}(x, \xi), \quad |\xi| \ge 1.
$$
Equivalently, for every $ N $,
$$
B^{(z)} - B_{(N)}^{(z)} \in \mathrm{C}\Psi^{mz - N}\big(\mathbb{R}^d; \mathcal{M}\big).
$$

In order to relate $A^z$ and $B_{(N)}^{(z)}$, let us first relate $(A - \lambda)^{-1}$ and $B(\lambda)$ given by \eqref{B-lambda}. Recall that $B_{-m-j}(\lambda)$ denote the $\mathcal{M}$-valued $\Psi$DO with symbols $
\sigma(B)_{-m-j}(x, \xi, \lambda) = \varphi(\xi) \sigma(B)_{-m-j}^0(x, \xi, \lambda)$, and denote the truncated sum
$$
B_{(N)}(\lambda)=\sum_{j=0}^{N-1}B_{-m-j}(\lambda) .
$$

\begin{proposition}\label{inverse-bn}
	For $N \in \mathbb{N}$ and $\lambda \in \Lambda_0$, where
	$$
	\Lambda_0= \Big\{ \lambda \in \mathbb{C} \setminus \{0\} : |\lambda| < \tfrac{1}{2^{m+1}} C_e \text{ or } |\arg \lambda - \pi| < \tfrac{\pi}{4} \Big\},
	$$
	we have
	$$
	(A - \lambda)^{-1} - B_{(N)}(\lambda) \in \Psi_{-2}^{- m -N}(\mathbb{R}^d \times \Lambda_0; \mathcal{M}).
	$$
\end{proposition}

\begin{proof}
	By \eqref{B-lambda}, the symbol of $B(\lambda)$  admits the asymptotic expansion
	$$
	\sigma(B)(x, \xi, \lambda) \sim \sum_{j=0}^\infty \sigma(B)_{-m-j}(x, \xi, \lambda).
	$$
	Denote the symbol of the remainder $(A-\lambda)^{-1} - B_{(N)}(\lambda)$ by $r_N(x, \xi, \lambda)$. We decompose
	$$
	(A-\lambda)^{-1} - B_{(N)}(\lambda) = \big(B(\lambda) - B_{(N)}(\lambda)\big) + \big((A-\lambda)^{-1} - B(\lambda)\big).
	$$
	
	From the asymptotic expansion
	$$
	\sigma(B)(x, \xi, \lambda) - \sum_{j=0}^{N-1} \sigma(B)_{-m-j}(x, \xi, \lambda) \in S_{-1}^{- m -N}(\mathbb{R}^d \times \mathbb{R}^d \times \Lambda_0; \mathcal{M}),
	$$
	the first term satisfies
	$$
	B(\lambda) - B_{(N)}(\lambda) \in \Psi_{-1}^{- m -N}(\mathbb{R}^d \times \Lambda_0; \mathcal{M}).
	$$
	
	From Theorem \ref{inverse resolvent}, define
	$$
	R(\lambda) := I - B(\lambda)(A-\lambda) \in \Psi_{-1}^{-\infty}(\mathbb{R}^d \times \Lambda_0; \mathcal{M}).
	$$
	Then
	$$
	(A-\lambda)^{-1} - B(\lambda) = R(\lambda) (A-\lambda)^{-1} \in \Psi_{-1}^{-\infty}(\mathbb{R}^d \times \Lambda_0; \mathcal{M}).
	$$
	Consequently, the total remainder satisfies
	\begin{equation}\label{rn estimate}
		\begin{split}
			r_N(x, \xi, \lambda) &\in S_{-1}^{- m -N}(\mathbb{R}^d \times \mathbb{R}^d \times \Lambda_0; \mathcal{M}), \\
			\| r_N(x, \xi, \lambda) \|_{\mathcal{M}} &\le C (1+|\xi|^2)^{\frac{- m -N}{2}} (1+|\lambda|)^{-1}.
		\end{split}
	\end{equation}
	This completes the proof.
\end{proof}

\begin{definition}
	A symbol family $\sigma(A)(x, \xi, z) \in C^{\infty}\big(\mathbb{R}^d \times \mathbb{R}^d \times G ; \mathcal{M}\big)$ belongs to
	$$
	\operatorname{Hol}\left(G, S_{\rho, \delta}^m\big(\mathbb{R}^d \times \mathbb{R}^d ; \mathcal{M}\big)\right)
	$$
	if the following conditions hold:
	\begin{enumerate}[$\rm (i)$]
		\item $z \mapsto \sigma(A)(x, \xi, z)$ is analytic in $G \subset \mathbb{C}$;
		\item For any multi-indices $\alpha, \beta$, any $k \in \mathbb{N}$, and compact $K \subset G$, there exists a constant $C=C(\alpha, \beta, k, K)>0$ such that
		$$
		\big\|\partial_z^k\partial_{\xi}^\alpha \partial_x^\beta  \sigma(A)(x, \xi, z) \big\|_{\mathcal{M}} \le C (1+|\xi|^2 )^\frac{m-\rho|\alpha|+\delta|\beta|}{2},
		$$
		uniformly for $(x, \xi, z) \in \mathbb{R}^d \times \mathbb{R}^d \times K$.
	\end{enumerate}
	An operator family $A(z)$ belongs to
	$$
	\operatorname{Hol}\left(G, \Psi_{\rho, \delta}^m\big(\mathbb{R}^d; \mathcal{M}\big)\right)
	$$
	if its symbol satisfies the above conditions.
\end{definition}

\begin{rk}
	Repeating the arguments in Sections \ref{section sym} and \ref{section resol}, it follows that the class of $\mathcal{M}$-valued holomorphic pseudo-differential families is closed under composition, adjoint operations, and taking the parametrix of an $\mathcal{M}$-valued elliptic  $\Psi$DO, provided the ellipticity condition holds uniformly in $z$.
\end{rk}

We now relate $A^z$ and $B^{(z)}$ for $\Re z < 0$.

\begin{proposition}\label{appro-Re-leq0}
	Let $A \in \mathrm{C}\Psi^m\big(\mathbb{R}^d; \mathcal{M}\big)$ with $  m > 0$ be elliptic and strictly positive, and suppose its principal symbol $\sigma(A)_m(x, \xi)$ is positive for all $(x, \xi) \in \mathbb{R}^d \times (\mathbb{R}^d \setminus \{0\})$. Then for any $N \in \mathbb{N}_0$,
	$$
	A^z - B_{(N)}^{(z)} \in \operatorname{Hol}\left(\Re z < 0, \Psi ^{- N}\big(\mathbb{R}^d; \mathcal{M}\big)\right).
	$$
\end{proposition}

\begin{proof}
	Let $\rho$ in Section \ref{Section complex} be $\rho = \tfrac{1}{2^{m+1}} C_e$. Then the integration contours $\Gamma$ and $\Gamma'$ are chosen inside $\Lambda_0$ and have the same geometric shape as in Section \ref{Section complex}, with the parameter $\rho$ replaced by $\frac{1}{2^{m+1}} C_e$.
	
	The regularization enables operator integrals over $\Gamma$:
	\begin{align*}
		B_{m z-j}^{(z)} &= \frac{i}{2\pi} \int_{\Gamma} \lambda^z B_{-m-j}(\lambda) \, d\lambda, \\
		B_{(N)}^{(z)} &= \frac{i}{2\pi} \int_{\Gamma} \lambda^z B_{(N)}(\lambda) \, d\lambda.
	\end{align*}
	Define
	$$
	R_N^{(z)} := A^z - B_{(N)}^{(z)} = \frac{i}{2 \pi} \int_{\Gamma} \lambda^z \Big((A-\lambda )^{-1} - B_{(N)}(\lambda)\Big) d \lambda,
	$$
	with symbol
	$$
	r_N^{(z)}(x, \xi) = \frac{i}{2 \pi} \int_{\Gamma} \lambda^z r_N(x, \xi, \lambda) \, d \lambda.
	$$
	By \eqref{rn estimate}, we have
	$$
	\big\|\partial_{\xi}^\alpha \partial_x^\beta r_N(x, \xi, \lambda)\big\|_{\mathcal{M}} \le C(1+|\xi|^2)^{\frac{-m-N-|\alpha|}{2}}(1+|\lambda|)^{-1}.
	$$
	Differentiating under the integral for $\Re z < 0$ and $k \in \mathbb{N}$, we obtain
	\begin{align*}
		\big\|\partial_z^{k}\partial_{\xi}^\alpha \partial_x^\beta r_N^{(z)}(x, \xi)\big\|_{\mathcal{M}} & \le C^{\prime}(1+|\xi|^2)^{\frac{-m-N-|\alpha|}{2}} \int_{\Gamma} |\lambda|^{\Re z} |\ln \lambda|^k (1+|\lambda|)^{-1}|d \lambda| \\
		& \le C^{\prime \prime}(1+|\xi|^2)^{\frac{-m-N-|\alpha|}{2}}.
	\end{align*}
	Hence,
	$$
	R_N^{(z)} \in \operatorname{Hol}\left(\Re z < 0, \Psi ^{-N}\big(\mathbb{R}^d; \mathcal{M}\big)\right).
	$$
\end{proof}


\subsection{Extension to entire complex plane}
We now extend the definition of $ \sigma(B)_{mz-j}^{(z), 0} $ to all $ z \in \mathbb{C} $ by following the operator extension procedure in Section \ref{entiretitle}. Let $ k \in \mathbb{Z} $ with $ \Re z < k $. Proposition \ref{appro-Re-leq0} yields
$$
A^{z-k} - B_{(N)}^{(z-k)} \in \Psi^{-N}\big(\mathbb{R}^d; \mathcal{M}\big),
$$
while the asymptotic definition of $ B^{(z-k)} $ gives
$$
B^{(z-k)} - B_{(N)}^{(z-k)} \in \Psi^{m(z-k) - N}\big(\mathbb{R}^d; \mathcal{M}\big).
$$
By choosing $ N $ sufficiently large, both terms become arbitrarily smoothing, hence
\begin{equation}\label{Az-Bz}
	A^{z-k} - B^{(z-k)} \in \Psi^{-\infty}(\mathbb{R}^d; \mathcal{M}).
\end{equation}

For $ k = \lfloor \Re z \rfloor + 1 $, we define $ B^{(z)} := A^k B^{(z-k)} $. By\eqref{azw} and \eqref{Az-Bz}, we have
$$
A^k B^{(z-k)} \equiv B^{(z-k)} A^k \quad (\bmod \Psi^{-\infty}).
$$
Correspondingly, for all $ j \in \mathbb{N}_0 $ and $ |\xi| \ge 1 $, the symbol components $ \sigma(B)_{mz-j}^{(z),0}(x,\xi) $ can be recursively defined by
\begin{equation}\label{bmzjj}
	\sigma(B)_{mz-j}^{(z),0}(x,\xi)
	= \sum_{p+q+|\alpha|=j} \frac{1}{\alpha!} \, \partial_\xi^\alpha \sigma(A^k)_{mk-p}(x,\xi) \, D_x^\alpha \sigma(B)_{m(z-k)-q}^{(z-k),0}(x,\xi),
\end{equation}
where the terms $ \sigma(A^k)_{mk-p}(x,\xi) $ and $ \sigma(B)_{m(z-k)-q}^{(z-k),0}(x,\xi) $ may be interchanged in position, i.e.,
\begin{equation*}
	\sigma(B)_{mz-j}^{(z),0}(x,\xi)
	= \sum_{p+q+|\alpha|=j} \frac{1}{\alpha!} \, D_x^\alpha \sigma(B)_{m(z-k)-q}^{(z-k),0}(x,\xi)\, \partial_\xi^\alpha \sigma(A^k)_{mk-p}(x,\xi) .
\end{equation*}

The following two definitions are therefore the natural counterparts of Definition \ref{def:regularized-sym} and Definition \ref{def:truncated-sum}, now formulated for the $z \in \mathbb{C}$.

\begin{definition}
	Let $z \in \mathbb{C}$. For each $j \in \mathbb{N}_0$, define the \emph{regularized symbol} by
	$$
	\sigma(B)_{mz-j}^{(z)}(x, \xi) \coloneqq \varphi(\xi) \, \sigma(B)_{mz-j}^{(z),0}(x, \xi),
	$$
	where $\sigma(B)_{mz-j}^{(z),0}$ is the homogeneous component from \eqref{bmzj}. The corresponding $\mathcal{M}$-valued $\Psi$DO is denoted by $B_{mz-j}^{(z)}$.
\end{definition}

\begin{definition}
	Let $z \in \mathbb{C}$. For $N \in \mathbb{N}$, define the truncated sum
	$$
	B_{(N)}^{(z)} \coloneqq \sum_{j=0}^{N-1} B_{mz-j}^{(z)} \in \mathrm{C}\Psi^{mz}(\mathbb{R}^d; \mathcal{M}).
	$$
\end{definition}

\begin{theorem}\label{bmz sym group}
	The extended symbols satisfy the following properties:
	\begin{enumerate}[$\rm (i)$]
		\item The $\mathcal{M}$-valued function $\sigma(B)_{m z-j}^{(z), 0}(x, \xi)$ defined by \eqref{bmzjj} is independent of the choice of the integer $k$, as long as $\Re z < k$.
		
		\item If $\Re z < 0$, then the $\mathcal{M}$-valued functions $\sigma(B)_{m z-j}^{(z), 0}(x, \xi)$ obtained by formula \eqref{bmzjj} coincide with the $\mathcal{M}$-valued functions (denoted by the same symbol) obtained by the contour integral \eqref{bmzj}.
		
		\item The group property \eqref{sym semigroup} holds for all $z, w \in \mathbb{C}$.
		
		\item If $k \in \mathbb{Z}$, then
		$$
		\sigma(B)_{mk-j}^{(k), 0}(x, \xi) = \sigma(A^{k})_{m k-j}(x, \xi), \quad |\xi| \ge 1.
		$$
		
		\item For any multi-indices $\alpha, \beta$, any $(x, \xi) \in \mathbb{R}^d \times (\mathbb{R}^d \setminus \{0\})$, and any $j \in \mathbb{N}_0$, the function $\partial_{\xi}^\alpha \partial_x^\beta \sigma(B)_{mz-j}^{(z),0}(x, \xi)$ is an entire $\mathcal{M}$-valued function of $z$.
	\end{enumerate}
\end{theorem}

\begin{proof}
	(i) Let $k_1, k_2 \in \mathbb{Z}$ with $\Re z < k_1, k_2$. Then it follows from Theorem \ref{groupaz} (i) and \eqref{Az-Bz} that
	$$
	A^{k_1} B^{(z-k_1)} \equiv A^{k_2} B^{(z-k_2)} \quad (\bmod \Psi^{-\infty}),
	$$
	which implies
	$$
	\sigma(B)_{mz-j}^{(z),0}(x,\xi) \text{ is independent of } k, \quad |\xi|\ge 1.
	$$
	
	(ii) This follows immediately from (i) by choosing $k = 0$ and using the assumption $\Re z < 0$.
	
	(iii) Let $z,w \in \mathbb{C}$ and choose $k,l \in \mathbb{Z}$ such that $\Re z < k$ and $\Re w < l$. Then it follows from \eqref{agroup} that
	$$
	B^{(z)} B^{(w)} \equiv (A^k B^{(z-k)}) (A^{l} B^{(w-l)}) \equiv A^{k+l} (B^{(z-k)} B^{(w-l)}) \quad (\bmod \Psi^{-\infty}).
	$$
	By Proposition \ref{prop-semigroup} (i) and the construction of $B^{(z-k)}$ and $B^{(w-l)}$ at the operator level, we have
	$$
	B^{(z-k)} B^{(w-l)} \equiv B^{(z+w-k-l)} \quad (\bmod \Psi^{-\infty}),
	$$
	so that
	$$
	B^{(z)} B^{(w)} \equiv A^{k+l} B^{(z+w-k-l)} \equiv B^{(z+w)} \quad (\bmod \Psi^{-\infty}).
	$$
	Therefore, for each $j \in \mathbb{N}_0$, the corresponding symbols satisfy
	$$
	\sum_{|\alpha|+p+q=j} \frac{1}{\alpha!} \partial_\xi^\alpha \sigma(B)_{mz-p}^{(z),0}(x,\xi) D_x^\alpha \sigma(B)_{mw-q}^{(w),0}(x,\xi)
	= \sigma(B)_{m(z+w)-j}^{(z+w),0}(x,\xi), \quad |\xi|\ge 1.
	$$
	
	(iv) For $z = k \in \mathbb{Z}$, the recursive formula directly yields
	$$
	\sigma(B)_{m k-j}^{(k), 0}(x, \xi) = \sigma(A^k)_{m k-j}(x, \xi), \quad |\xi| \ge 1.
	$$
	
	(v) Choose $k \in \mathbb{Z}$ such that $\Re z < k$. For any multi-indices $\alpha, \beta$, any $(x, \xi) \in \mathbb{R}^d \times (\mathbb{R}^d \setminus \{0\})$, and any $j \in \mathbb{N}_0$, we have
	\begin{align*}
		\partial_{\xi}^\alpha \partial_x^\beta \sigma(B)_{m z-j}^{(z), 0}(x, \xi)
		&= \sum_{p+q+|\gamma|=j} \frac{i}{2 \pi \gamma!} \partial_{\xi}^\alpha \partial_x^\beta \Bigl(
		\partial_{\xi}^\gamma \sigma(A^k)_{m k-p}(x, \xi) \\
		&\qquad \cdot \int_{\Gamma_{\xi}} \lambda^{z-k} D_x^\gamma \sigma(B)_{-m-q}^0(x, \xi, \lambda) \, d\lambda \Bigr).
	\end{align*}
	Differentiating with respect to $z$, for any $n \in \mathbb{N}$, we obtain
	\begin{align*}
		\frac{d^n}{d z^n} \partial_{\xi}^\alpha \partial_x^\beta \sigma(B)_{m z-j}^{(z), 0}(x, \xi)
		&= \sum_{p+q+|\gamma|=j} \frac{i}{2 \pi \gamma!} \partial_{\xi}^\alpha \partial_x^\beta \Bigl(
		\partial_{\xi}^\gamma \sigma(A^k)_{m k-p}(x, \xi) \\
		&\qquad \cdot \int_{\Gamma_{\xi}} \lambda^{z-k} (\ln \lambda)^n D_x^\gamma \sigma(B)_{-m-q}^0(x, \xi, \lambda) \, d\lambda \Bigr).
	\end{align*}
	By Proposition \ref{para sym prop} (i) and (iii), this integral converges absolutely in $\mathcal{M}$ for $\Re(z-k) \le -\varepsilon < 0$. Since $k$ can be chosen arbitrarily large for fixed $z$, the derivative exists for all $z \in \mathbb{C}$ and is holomorphic. This proves that $\partial_{\xi}^\alpha \partial_x^\beta \sigma(B)_{m z-j}^{(z),0}(x, \xi)$ is an entire $\mathcal{M}$-valued function of $z$.
\end{proof}

\subsection{Structure Theorem}

Utilizing the construction developed in the previous sections, we now obtain the following structure theorem:

\begin{theorem}\label{structurethm}
	Let $A \in \mathrm{C}\Psi^m\big(\mathbb{R}^d; \mathcal{M}\big)$ with $  m > 0$ be elliptic and strictly positive, and suppose its principal symbol $\sigma(A)_m(x, \xi)$ is positive for all $(x, \xi) \in \mathbb{R}^d \times (\mathbb{R}^d \setminus \{0\})$. Then the complex powers of $A$ satisfy:
	\begin{enumerate}[$\rm (i)$]
		\item For any $N \in \mathbb{N}_0$ and $t \in \mathbb{R}$,
		$$
		A^z - B_{(N)}^{(z)} \in \operatorname{Hol}\left(\Re z < t, \Psi^{  m  t - N}\big(\mathbb{R}^d; \mathcal{M}\big)\right).
		$$
		
		\item $A^z \in \mathrm{C}\Psi^{m z}\big(\mathbb{R}^d; \mathcal{M}\big)$ for all $z \in \mathbb{C}$.
	\end{enumerate}
\end{theorem}

\begin{proof}
	(i) In Proposition \ref{appro-Re-leq0}, we have established the case for $t = 0$. For the general case, set $w = z + t$. Then
	$$
	R_N^{(w-t)} \in \operatorname{Hol}\left(\Re w < t, \Psi^{-N}\big(\mathbb{R}^d; \mathcal{M}\big)\right).
	$$
	Using the group property in Theorems \ref{groupaz} and \ref{bmz sym group}, along with the fact that the composition of holomorphic families yields again a holomorphic family, we conclude
	$$
	R_N^{(w)} = A^{w} - B_{(N)}^{(w)} = A^{t} R_N^{(w-t)} \in \operatorname{Hol}\left(\Re w < t, \Psi^{mt-N}\big(\mathbb{R}^d; \mathcal{M}\big)\right).
	$$
	
	(ii) Fix $z \in \mathbb{C}$ and choose $t > \Re z$. By Proposition \ref{asymp}, there exists an $\mathcal{M}$-valued $\Psi$DO $B^{(z)}$ with symbol $\sigma(B^{(z)})(x, \xi)$ such that
	$$
	\sigma(B ^{(z)})(x, \xi) \sim \sum_{j=0}^{\infty} \sigma(B)_{m z-j}^{(z)}(x, \xi),\quad |\xi| \ge 1.
	$$
	Then
	$$
	B^{(z)} \in \mathrm{C}\Psi^{m z}\big(\mathbb{R}^d; \mathcal{M}\big).
	$$
	From part (i), as $N \to \infty$,
	$$
	A^z - B^{(z)} \in \bigcap_{k \in \mathbb{Z}} \Psi^k\big(\mathbb{R}^d; \mathcal{M}\big) = \Psi^{-\infty}\big(\mathbb{R}^d; \mathcal{M}\big).
	$$
	Therefore,
	$$
	A^z \in \mathrm{C}\Psi^{m z}\big(\mathbb{R}^d; \mathcal{M}\big).
	$$
\end{proof}

Theorem \ref{structurethm} establishes that the complex powers $A^z$ are $\mathcal{M}$-valued classical  $\Psi$DOs of order $m z$.
A natural next step is to identify the principal symbol of such operators.
In particular, since $|A| = (A^\ast A)^{\frac{1}{2}}$ plays a fundamental role in spectral theory,
we now turn to the structure of its complex powers $|A|^z$.
The following corollary shows that the principal symbol of $|A|^z$ is simply the $z$-th power of the modulus of the principal symbol of $A$.

\begin{corollary}\label{prin map ellipaz}
	Let $A \in \mathrm{C}\Psi^m\big(\mathbb{R}^d; \mathcal{M}\big)$ with $ \mathfrak{m}=\Re m>0$ be elliptic, and $A^*A$ be strictly positive. Then $|A|^{z}\in \mathrm{C}\Psi^{ \mathfrak{m} z}\big(\mathbb{R}^d; \mathcal{M}\big)$ with principal symbol
	$$
	\sigma(|A|^{z})_{ \mathfrak{m}z}(x, \xi) = |\sigma(A)_{m}(x, \xi)|^{z}, \quad z\in \mathbb{C}.
	$$
\end{corollary}

\begin{proof}
	Applying Proposition \ref{cl M}, we conclude that $A^{\ast}A \in \mathrm{C}\Psi^{2 \mathfrak{m} }\big(\mathbb{R}^d; \mathcal{M}\big)$ with principal symbol
	$$
	\sigma(A^{\ast}A)_{2 \mathfrak{m}}(x,\xi) = \sigma(A)_{m}(x,\xi)^{\ast}\sigma(A)_{m}(x,\xi).
	$$
	Combining this with Theorem \ref{structurethm}, we obtain that $|A|^{z}= (A^{\ast}A)^{\frac{z}{2}}\in \mathrm{C}\Psi^{ \mathfrak{m} z}\big(\mathbb{R}^d; \mathcal{M}\big)$, and it follows from \eqref{Bmzj,0z} that
	$$
	\sigma(|A|^{z})_{ \mathfrak{m} z}(x, \xi)= (\sigma(A)_{m}(x,\xi)^{\ast}\sigma(A)_{m}(x,\xi))^{\frac{z}{2}} = |\sigma(A)_{m}(x, \xi)|^{z}.
	$$
\end{proof}

\begin{rk}
	From the above corollary, we see that, if A is elliptic and strictly positive, then automatically its principal symbol $\sigma(A)_m(x, \xi)$ is positive for all $(x, \xi) \in \mathbb{R}^d \times (\mathbb{R}^d \setminus \{0\})$. So in the rest part of the paper, we usually assume that the involved $\Psi$DOs to be elliptic and strictly positive, without assuming the positivity of the principal symbols.
\end{rk}

\subsection{Analytic continuation of the kernels of complex powers }\label{section ker}

We begin by recalling the definition of $\mathcal{M}$-valued $\Psi$DOs. For a symbol $\sigma(x,\xi) \in S_{\rho,\delta}^m\big(\mathbb{R}^d\times \mathbb{R}^d; \mathcal{M}\big)$ and a test function $f \in \mathscr{S}(\mathbb{R}^d ; L_{2}(\mathcal{M}))$, the corresponding $\Psi$DO $A$ is defined as the mapping $f \mapsto A f$, given by
$$
A f(x) = \int_{\mathbb{R}^d}  \sigma(A)(x, \xi) \hat{f} (\xi) e^{i x \cdot \xi} \bar{d} \xi,
$$
where $\ \bar{d} \xi = (2 \pi)^{-d} d\xi$.

If we express the $\Psi$DO $A \in \Psi_{\rho,\delta}^m\big(\mathbb{R}^d; \mathcal{M}\big)$ in the form
$$
A f(x) = \int_{\mathbb{R}^d} \int_{\mathbb{R}^d} \sigma(A)(x, \xi) f(y)   e^{i(x-y) \cdot \xi} dy \ \bar{d} \xi,
$$
where $\sigma(A)(x, \xi) \in S_{\rho,\delta}^m\big(\mathbb{R}^d\times \mathbb{R}^d; \mathcal{M}\big)$, then when $m < -d$, the distributional kernel $A(x, y)$, which takes values in $ \mathcal{M}$, is continuous and has the form
$$
A(x, y) = \int_{\mathbb{R}^d} \sigma(A)(x, \xi) e^{i(x-y) \cdot \xi}  \bar{d} \xi.
$$

\begin{theorem}\label{holoextend}
	Let $A \in \mathrm{C}\Psi^m\big(\mathbb{R}^d; \mathcal{M}\big)$ with $  m > 0$ be elliptic and strictly positive. Let $A_z(x, y)$ denote the kernels of the complex powers $A^z$, initially defined for $\Re z<-\frac{d}{m}$ and $x, y \in \mathbb{R}^d$. Then $A_z(x, x)$ admits a meromorphic continuation to $\mathbb{C}$ with  at most simple poles located at $z_j = \frac{j-d}{m}$ for $j \in\mathbb{N}_0$. The $\mathcal{N}$-valued residue of $A_z(x, x)$ at $z_j$ is given by
	\begin{equation}\label{residuemost}
		-\frac{1}{m} \int_{\mathbb{S}^{d-1} } \sigma(B)_{-d}^{(z_j), 0}(x, \xi)  \ \bar{d} \xi,
	\end{equation}
	where $\ \bar{d} \xi = (2 \pi)^{-d} d \xi$, and $d \xi$ denotes the surface measure on the unit sphere $\mathbb{S}^{d-1} \subset \mathbb{R}^d$.
\end{theorem}

\begin{proof}
	Following the notation from Theorem \ref{structurethm}, consider the remainder operator
	$R_{(N)}^{(z)} = A^z - B_{(N)}^{(z)}$ with symbol $r_{(N)}^{(z)}(x, \xi)$. Its kernel is given by
	$$
	R_{(N)}^{(z)}(x, y) = \int_{\mathbb{R}^d} r_{(N)}^{(z)}(x, \xi) e^{i(x-y) \cdot \xi} \ \bar{d} \xi,
	$$
	which converges in operator norm for $\Re z < \frac{N-d}{m}$ and defines a holomorphic $L_{\infty}(\mathbb{R}^d\times \mathbb{R}^d)\overline{\otimes}\mathcal{M}$-valued function. Thus, the analytic properties of $A^z$ are reduced to studying those of $B_{(N)}^{(z)}$.
	
	The homogeneous components of $B_{(N)}^{(z)}$ have kernels
	$$
	B_{m z-j}^{(z)}(x, y) = \int_{\mathbb{R}^d} \sigma(B)_{m z-j}^{(z)}(x, \xi) e^{i(x-y) \cdot \xi}  \bar{d} \xi.
	$$
	For $x = y$,
	\begin{equation} \label{bmzj2}
		B_{m z-j}^{(z)}(x, x) = \int_{\mathbb{R}^d} \varphi(\xi) \sigma(B)_{m z-j}^{(z), 0}(x, \xi) \ \bar{d} \xi,
	\end{equation}
	where $\varphi \in C^{\infty}(\mathbb{R}^d)$, $\varphi(\xi) = 0$ for $|\xi| \le \frac{1}{2}$, and $\varphi(\xi) = 1$ for $|\xi| \ge 1$.
	Passing to polar coordinates in \eqref{bmzj2}, we obtain
	$$
	B_{m z-j}^{(z)}(x, x) = \Big( \int_0^{\infty} \varphi(r) r^{m z - j + d - 1} dr \Big) \Big( \int_{\mathbb{S}^{d-1} } \sigma(B)_{m z-j}^{(z), 0}(x, \xi) \ \bar{d} \xi \Big).
	$$
	Here we abbreviate the notion $\varphi(|\xi|) : = \varphi(\xi)$.
	The second factor is an entire function of $z$. The first factor can be decomposed into the sum
	\begin{equation} \label{omega}
		\int_0^{\infty} \varphi(r) r^{m z - j + d - 1} dr = \int_0^1 \varphi(r) r^{m z - j + d - 1} dr + \int_1^{\infty} r^{m z - j + d - 1} dr.
	\end{equation}
	The first integral in \eqref{omega} is an entire function of $z$, and the second one can be computed as
	\begin{equation} \label{2int}
		\int_1^{\infty} r^{m z - j + d - 1} dr = -\frac{1}{m z - j + d} = -\frac{1}{m} \cdot \frac{1}{z - z_j},
	\end{equation}
	where $z_j = \frac{j-d}{m}$. Therefore, $ B_{m z-j}^{(z)}(x, x)$ has a simple pole at $z_j$ with $\mathcal{N}$-valued residue given by the formula \eqref{residuemost}.
\end{proof}

\begin{corollary}\label{sym holoextend}
	Let $A \in \mathrm{C}\Psi^m\big(\mathbb{R}^d; \mathcal{M}\big)$ with $  m > 0$ be elliptic and strictly positive. The integral $\int_{\mathbb{R}^d} \sigma(A)_m(x, \xi)^z d\xi$ is initially defined for $\Re z < -\frac{d}{m}$ and $x \in \mathbb{R}^d$. Then $\int_{\mathbb{R}^d} \sigma(A)_m(x, \xi)^z d\xi$ admits a meromorphic continuation to $\mathbb{C}$ with simple poles located at $z_j = \frac{j-d}{m}$ ($j \in \mathbb{N}_0$). The $\mathcal{N}$-valued residue of $\int_{\mathbb{R}^d} \sigma(A)_m(x, \xi)^z d\xi$ at $z_j$ is given by
	$$
	-\frac{1}{m} \int_{\mathbb{S}^{d-1} } \sigma(B)_{-d}^{(z_j), 0}(x, \xi) d\xi.
	$$
\end{corollary}

\begin{proof}
	Following the notation in Theorem \ref{structurethm}, consider the remainder symbol
	$$r_{(N)}^{(z)}(x, \xi) = \sigma(A)_m(x, \xi)^z - \sum_{j=0}^{N-1} \varphi(\xi) \sigma(B)_{m z-j}^{(z), 0}(x, \xi).$$
	The integral $\int_{\mathbb{R}^d} r_{(N)}^{(z)}(x, \xi) d\xi$ converges in operator norm for $\Re z < \frac{N-d}{m}$ and defines a holomorphic $\mathcal{N}$-valued function. Thus, the analytic properties of $\int_{\mathbb{R}^d} \sigma(A)_m(x, \xi)^z d\xi$ reduce to studying those of $\sum_{j=0}^{N-1} \int_{\mathbb{R}^d}\varphi(\xi) \sigma(B)_{m z-j}^{(z), 0}(x, \xi)d\xi$. So we can complete the proof as that of Theorem \ref{holoextend}.
	%
	%
\end{proof}

\section{Weyl's law and Connes' Trace Theorem  for $\mathcal{M}$-valued  classical $\Psi$DOs}\label{section asymp}

This section presents the main results of the paper. We establish Weyl's law for $\mathcal{M}$-valued classical $\Psi$DOs of negative order in the type II noncommutative geometry setting. We will first introduce the localized Riemann $\zeta$-functions associated to $\mathcal{M}$-valued elliptic $\Psi$DOs, then study the simple poles of these $\zeta$-functions, and finally deduce the Weyl's law for $\mathcal{M}$-valued classical $\Psi$DOs of negative order, thereby obtaining a direct spectral characterization of the noncommutative integral. As a consequence, Connes' trace theorem is recovered without the use of Dixmier traces.

\subsection{Traces and Schatten estimates of operators in $B(L_2(\mathbb{R}^d)) \overline{\otimes} \mathcal{M}$}

Recall that $\mathcal{M}$ is a semifinite von Neumann algebra equipped with a normal, semifinite, faithful (n.s.f.) trace $\tau$. The operator trace on $B(L_2(\mathbb{R}^d)) \overline{\otimes} \mathcal{M}$ is denoted by $\operatorname{Tr}$, where $\operatorname{Tr} = \operatorname{tr} \otimes \tau$, and $\operatorname{tr}$ denotes the canonical trace on $B(L_2(\mathbb{R}^d))$.

The algebra $L_{1}(\mathcal{M}) \cap \mathcal{M}$ is a two-sided $\ast$-ideal of $\mathcal{M}$ and is dense in $L_{p}(\mathcal{M})$ for $1 \le p < \infty$.

\begin{theorem}\label{op-tr=intker}
	Let $ \mathcal{K} \in \mathcal{L}_1\big(B(L_2(\mathbb{R}^d)) \overline{\otimes} \mathcal{M})\big) $, and assume that $ \mathcal{K} $ admits a continuous kernel
	$$
	K \in C^0(\Omega \times \Omega;\, L_1(\mathcal{M}) \cap \mathcal{M}),
	$$
	where $\Omega$ in a compact set in $\mathbb{R}^d$, i.e.,
	$$\mathcal{K} f(x)  =  \int _{\Omega} K(x,y)f(y) dy ,\quad \forall \, f\in L_2( \mathcal{N}).$$
	Then
	$$
	\operatorname{Tr}(\mathcal{K}) = \int_\Omega \tau(K(x,x)) \, dx.
	$$
\end{theorem}

\begin{proof}
	The proof follows the structure of the scalar case presented in \cite[Theorem 3.9]{Si05}, and also see \cite[Theorem 3.1]{JYE25} for the situation $\mathcal{M}=B(H)$.
	
	For simplicity, assume $ \Omega = [0,1] $. However, the proof can easily be extended first to $\Omega=[0,1]^d$ and then to any compact set $\Omega \subset \mathbb{R}^d$.
	
	For each $ n \in \mathbb{N} $, consider the dyadic partition of $ [0,1] $ into intervals
	$$
	\chi_{n,\ell} = \Big[\frac{\ell - 1}{2^n}, \frac{\ell}{2^n}\Big), \quad \ell = 1, \dots, 2^n,
	$$
	and define the orthonormal functions
	$$
	\varphi_{n,\ell}(x) = \begin{cases}
		2^{\frac{n}{2}}, & x \in \chi_{n,\ell}, \\
		0, & \text{otherwise.}
	\end{cases}
	$$
	Let $ \mathcal{P}_n $ denote the orthogonal projection from $ L_2(\Omega) $ onto $\mathcal{H}_n:= \operatorname{Span}\{\varphi_{n,\ell}\}_{\ell=1}^{2^n} $, and denote
	$$\mathcal{K}_n  =  (\mathcal{P}_n \otimes 1_{\mathcal{M}}   )\cdot  \mathcal{K} \cdot (\mathcal{P}_n \otimes 1_{\mathcal{M}}  ) .$$
	Then for any $e\in L_2( \mathcal{M})$,
	$$\mathcal{K}_n (\varphi_{n,\ell} \otimes e ) =\sum_{k, \ell=1}^{2^n} \int_{\Omega\times \Omega }  \varphi_{n,k }  (x) K(x, y) \varphi_{n,\ell}(y)  \ e \ dx \ dy.$$
	Therefore, denoting
	$$
	x_{k,\ell}^{(n)} := \int_{\Omega\times \Omega }  \varphi_{n,k }  (x) K(x, y) \varphi_{n,\ell}(y)   \ dx \ dy
	$$
	we are able to write $\mathcal{K}_n $ in block-diagonal form
	$$
	\mathcal{K}_n := \sum_{k, \ell=1}^{2^n} (\varphi_{n,k } \otimes \varphi_{n,\ell}) \otimes x_{k,\ell}^{(n)} .
	$$
	Using this block-diagonal form, we see that
	$$
	\operatorname{Tr}(\mathcal{K}_n) = \sum_{\ell=1}^{2^n} \tau(x_{\ell,\ell}^{(n)} )
	= \sum_{\ell=1}^{2^n} 2^n \int_{\chi_{n,\ell}} \int_{\chi_{n,\ell}} \tau(K(x,y)) \, dx \, dy.
	$$
	
	By the trace-class property of $ \mathcal{K} $ and the normality of $ \operatorname{Tr} $, we have
	\begin{equation}\label{Tr-int-K}
		\operatorname{Tr}(\mathcal{K}) = \lim_{n\to\infty} \operatorname{Tr}(\mathcal{K}_n).
	\end{equation}
	Let us now analyse this limit.
	Since $ K $ is continuous on the compact set $ \Omega \times \Omega $ and takes values in $ L_1(\mathcal{M})\cap \mathcal{M}$, it is uniformly continuous. Therefore, for any $ \varepsilon > 0 $, there exists $ N \in \mathbb{N} $ such that for all $ n > N $ and for all $ \ell = 1, \dots, 2^n $, we have
	$$
	\sup_{x, y \in \chi_{n,\ell}} \|K(x, y) - K(x, x)\|_{L_1(\mathcal{M})} \le \varepsilon.
	$$
	This implies that for all $ x, y \in \chi_{n,\ell} $,
	$$
	|\tau(K(x, y)) - \tau(K(x, x))| \le \|K(x, y) - K(x, x)\|_{L_1(\mathcal{M})} \le \varepsilon.
	$$
	Consider the difference between $ \operatorname{Tr}(\mathcal{K}_n) $ and the Riemann sum:
	\begin{align*}
		&\Bigg| \operatorname{Tr}(\mathcal{K}_n) - \sum_{\ell=1}^{2^n} \int_{\chi_{n,\ell}} \tau(K(x,x))\,dx \Bigg| \\
		&= \Bigg| \sum_{\ell=1}^{2^n} 2^n \int_{\chi_{n,\ell}} \int_{\chi_{n,\ell}} \big( \tau(K(x,y)) - \tau(K(x,x)) \big) \, dx \, dy \Bigg| \\
		&\le \sum_{\ell=1}^{2^n} 2^n \int_{\chi_{n,\ell}} \int_{\chi_{n,\ell}} \big| \tau(K(x,y)) - \tau(K(x,x)) \big| \, dx \, dy \\
		&\le \sum_{\ell=1}^{2^n} 2^n \int_{\chi_{n,\ell}} \int_{\chi_{n,\ell}} \varepsilon \, dx \, dy \\
		&= \sum_{\ell=1}^{2^n} 2^n \cdot (\varepsilon \cdot 2^{-2n}) = \varepsilon.
	\end{align*}
	Thus, for all $ n > N $,
	$$
	\Bigg|\operatorname{Tr}(\mathcal{K}_n) - \int_{\Omega} \tau(K(x, x))\, dx \Bigg| \le \varepsilon.
	$$
	This proves that
	\begin{equation}\label{lim-Tr-int}
		\lim_{n\to\infty} \operatorname{Tr}(\mathcal{K}_n) = \int_{\Omega} \tau(K(x,x)) \, dx.
	\end{equation}
	Combining \eqref{Tr-int-K} with \eqref{lim-Tr-int}, we obtain
	$$
	\operatorname{Tr}(\mathcal{K}) = \int_\Omega \tau(K(x,x)) \, dx.
	$$
	This completes the proof.
\end{proof}

However, there is no simple necessary and sufficient condition for $ \mathcal{K} \in \mathcal{L}_1\big(B(L_2(\mathbb{R}^d)) \overline{\otimes} \mathcal{M}\big) $. Thus, in order to make use of Theorem \ref{op-tr=intker}, we include the following characterization of $ \mathcal{K} \in \mathcal{L}_2\big(B(L_2(\mathbb{R}^d)) \overline{\otimes} \mathcal{M}\big) $, which is an easy corollary of its scalar-valued counterpart \cite[Theorem~2.11]{Si05}.

\begin{proposition}\label{op-S=intker}
	If $ \mathcal{K} \in \mathcal{L}_2\big(B(L_2(\mathbb{R}^d)) \overline{\otimes} \mathcal{M}\big) $, then there exists a unique function $K \in L_2(\mathbb{R}^d \times \mathbb{R}^d; L_2(  \mathcal{M}))$ such that
	\begin{equation}\label{op-S-Ker}
		\mathcal{K} f(x)  =  \int K(x, y )  f(y) dy.
	\end{equation}
	Conversely, any $K \in L_2(\mathbb{R}^d \times \mathbb{R}^d; L_2(  \mathcal{M}))$ defines an operator
	$$\mathcal{K} \in \mathcal{L}_2\big(B(L_2(\mathbb{R}^d)) \overline{\otimes} \mathcal{M}\big) \quad \text{with} \quad \|\mathcal{K}\|_{\mathcal{L}_2}  =  \|K\|_{L_2(\mathbb{R}^d \times \mathbb{R}^d; L_2(  \mathcal{M}))}.$$
\end{proposition}
\begin{proof}
	Since $\mathcal{L}_2\big(B(L_2(\mathbb{R}^d)) \overline{\otimes} \mathcal{M}\big) =  \mathcal{S}_2( L_2(\mathbb{R}^d  ))\otimes_2 L_2( \mathcal{M}) $, the Hilbert space tensor product of $\mathcal{S}_2( L_2(\mathbb{R}^d  ))$ and $ L_2( \mathcal{M})$, the operator $ \mathcal{K} \in \mathcal{L}_2\big(B(L_2(\mathbb{R}^d)) \overline{\otimes} \mathcal{M}\big)  $ may be approximated by sequences in the algebraic tensor product of $\mathcal{S}_2( L_2(\mathbb{R}^d  ))$ and $ L_2( \mathcal{M})$. But by \cite[Theorem~2.11]{Si05}, elements $\mathcal{K}\in \mathcal{S}_2( L_2(\mathbb{R}^d  ))$ are represented as \eqref{op-S-Ker}, with $K \in L_2(\mathbb{R}^d \times \mathbb{R}^d)$. The desired assertion then follows from the easy fact $L_2(\mathbb{R}^d \times \mathbb{R}^d)\otimes_2 L_2( \mathcal{M})= L_2(\mathbb{R}^d \times \mathbb{R}^d; L_2(  \mathcal{M}))$.
\end{proof}

Unlike in the scalar-valued case, $ \mathcal{K} $ with kernel $K \in L_2(\mathbb{R}^d \times \mathbb{R}^d; L_2(  \mathcal{M}))$ may not be a bounded operator in $B(L_2(\mathbb{R}^d)) \overline{\otimes} \mathcal{M}$.

\medskip

Proposition \ref{op-S=intker} gives necessary and sufficient condition for an operator $ \mathcal{K} \in \mathcal{L}_2\big(B(L_2(\mathbb{R}^d)) \overline{\otimes} \mathcal{M}\big) $. In order to estimate a priori the (weak) $\mathcal{L}_p$ norms of $\Psi$DOs of different orders, we will also need the so called Cwikel-type estimates, that we include below.

Let $d \ge 1$. Denote by
$$
\mathcal{L}_p\big(B(L_2(\mathbb{R}^d)) \overline{\otimes} \mathcal{M}\big)
\quad \text{and} \quad
\mathcal{L}_{p,\infty}\big(B(L_2(\mathbb{R}^d)) \overline{\otimes} \mathcal{M}\big)
$$
the noncommutative (weak) Schatten classes associated with the semifinite trace $\operatorname{Tr}=\operatorname{tr} \otimes\tau$ on $B(L_2(\mathbb{R}^d)) \overline{\otimes} \mathcal{M}$.
Let $f:\mathbb{R}^d \to \mathcal{M}$ be a strongly measurable operator-valued function and $g:\mathbb{R}^d \to \mathbb{C}$ a scalar-valued measurable function.

Recall that, for $\phi: \mathbb{R}^d \rightarrow\mathbb{C}$, we use $T_\phi$ to denote the Fourier multiplier with symbol $\phi(\xi)$; For $f: \mathbb{R}^d \rightarrow \mathcal{M}$, define $M_f(g)= fg $ for $g \in \mathscr{S}(\mathbb{R}^d;L_{2}(\mathcal{M}))$ the (left) multiplication operator. Initially we do not assume any smoothness or integrability conditions on $\phi$ or $f$.

We now present the following Cwikel-type estimates:

\begin{lemma}\label{Cwikel}
	If  $\phi \in L_{p  }(\mathbb{R}^{d})$, $f \in L_p(\mathbb{R}^d; L_p(\mathcal{M}))$ with $2\le p <\infty$, then $T_\phi M_f \in \mathcal{L}_{p }\big(B(L_2(\mathbb{R}^d))\overline{\otimes} \mathcal{M}\big) $ with
	\begin{equation}\label{Cwikel-p}
		\|T_\phi M_f\|_{\mathcal{L}_{p } }  \le C_{p,d  } \,  \|f\|_{L_p(\mathbb{R}^d; L_p(\mathcal{M}))} \|\phi \|_{L_{p  }(\mathbb{R}^{d})} .
	\end{equation}
	If $\phi \in L_{p,\infty }(\mathbb{R}^{d})$, $f \in L_p(\mathbb{R}^d; L_p(\mathcal{M}))$ with $2<p <\infty$, then $T_\phi M_f \in \mathcal{L}_{p, \infty}\big(B(L_2(\mathbb{R}^d))\overline{\otimes} \mathcal{M}\big) $ with
	\begin{equation}\label{Cwikel-wp}
		\|T_\phi M_f\|_{\mathcal{L}_{p, \infty} }  \le C_{p,d  } \,  \|f\|_{L_p(\mathbb{R}^d; L_p(\mathcal{M}))} \|\phi \|_{L_{p,\infty }(\mathbb{R}^{d})} .\end{equation}
\end{lemma}
\begin{proof}
	This assertion is the operator-valued version of the Cwikel inequality due to \cite{Cwikel77}. Its proof follows closely from the proof of the abstract Cwikel estimate \cite[Corollary~1.2]{LSZ20}, by an interpolation argument. When $p= \infty$, we have easily
	$$\|T_\phi M_f\|_{B(L_2(\mathbb{R}^d))\overline{\otimes} \mathcal{M} }  \le     \|f\|_{L_\infty(\mathbb{R}^d)\overline{\otimes} \mathcal{M} } \|\phi \|_{L_{ \infty }(\mathbb{R}^{d})} .$$
	When $p= 2$, it follows from Proposition \ref{op-S=intker} that
	$$\|T_\phi M_f\|_{\mathcal{L}_2\big(B(L_2(\mathbb{R}^d))\overline{\otimes} \mathcal{M} \big)}  =    \|f\|_{L_{ 2 }(L_\infty(\mathbb{R}^d)\overline{\otimes} \mathcal{M} )} \|\phi \|_{L_{ 2 }(\mathbb{R}^{d})} .$$
	The desired estimates then follows from Corollary 1.2 and Lemma 2.3 in \cite{LSZ20}.
\end{proof}

To formulate the Schatten estimate for $0 < p < 2$ and the weak Schatten estimate for $0 < p \le 2$, we employ the Birman--Solomyak block spaces $\ell_p(L_q)(\mathbb{R}^d)$ and $\ell_{p,\infty}(L_q)(\mathbb{R}^d)$ (see \cite{BS77}, \cite[Chapter~4]{Si05}, \cite[Subsection~5.6]{BGS90}), as well as their logarithmic refinement $\ell_{2,\log}(L_\infty)(\mathbb{R}^d)$ defined in \cite[Definition~5.5]{LSZ20}. All these function spaces contain $C_c(\mathbb{R}^d)$. By analogy, we now define their operator-valued counterparts.

Fix the unit cube $Q = [0,1)^d$, and for $\mathbf{m} \in \mathbb{Z}^d$, define
$$
f_{\mathbf{m}}(x) := f(x)\,\mathbf{1}_{Q + \mathbf{m}}(x),
\qquad
\phi_{\mathbf{m}}(\xi) := \phi(\xi)\,\mathbf{1}_{Q + \mathbf{m}}(\xi).
$$
For $0 < p < 2$, we introduce the sequence-space norms with block $L_2$-based norms
(with operator-valued $f$ and scalar $\phi$):
\begin{align*}
	\|f\|_{\ell_p(L_2)(\mathbb{R}^d;\mathcal{M})}
	&:= \biggl( \sum_{\mathbf{m} \in \mathbb{Z}^d} \|f_{\mathbf{m}}\|_{L_2(\mathbb{R}^d;\mathcal{M})}^p \biggr)^{\frac{1}{p}}, \\[2ex]
	\|\phi\|_{\ell_p(L_2)(\mathbb{R}^d)}
	&:= \biggl( \sum_{\mathbf{m} \in \mathbb{Z}^d} \|\phi_{\mathbf{m}}\|_{L_2(\mathbb{R}^d)}^p \biggr)^{\frac{1}{p}} .
\end{align*}
where
$$
\|f_{\mathbf{m}}\|_{L_2(\mathbb{R}^d;\mathcal{M})}
= \left( \int_{Q + \mathbf{m}} \|f(x)\|_{\mathcal{M}}^2 \, dx \right)^{\frac{1}{2}}.
$$
For the weak space, let $a_{\mathbf{m}} := \|\phi_{\mathbf{m}}\|_{L_2(\mathbb{R}^d)}$ and denote by $(a_k^*)_{k \ge 1}$ the decreasing rearrangement of the sequence $(a_{\mathbf{m}})_{\mathbf{m} \in \mathbb{Z}^d}$. We then set
\begin{equation}\label{def:weaklpL2}
	\|\phi\|_{\ell_{p,\infty}(L_2)(\mathbb{R}^d)}
	:= \sup_{k \ge 1} k^{\frac{1}{p}} a_k^*,
	\qquad 0 < p < \infty.
\end{equation}
For the logarithmic space, we define
$$
\|f\|_{\ell_{2,\log}(L_\infty)(\mathbb{R}^d;\mathcal{M})}
:= \left( \sum_{\mathbf{m}} (1 + \log(1 + |\mathbf{m}|)) \|f_{\mathbf{m}}\|_{L_\infty(\mathbb{R}^d)\overline{\otimes} \mathcal{M}}^2 \right)^{\frac{1}{2}}.
$$

\begin{lemma}\label{Cwikel-0-2}
	Let $T_\phi M_f$ be the operator defined by multiplication by $f$ followed by Fourier multiplier with symbol $\phi$. Then:
	\begin{enumerate}[$\rm(i)$]
		\item If $f \in \ell_p(L_2)(\mathbb{R}^d;\mathcal{M})$, $\phi \in \ell_p(L_2)(\mathbb{R}^d)$ with $0 < p < 2$, then $T_\phi M_f \in \mathcal{L}_p\big(B(L_2(\mathbb{R}^d))\overline{\otimes} \mathcal{M}\big)$ and
		$$
		\|T_\phi M_f\|_{\mathcal{L}_p}
		\le C_{p,d}\,
		\|f\|_{\ell_p(L_2)(\mathbb{R}^d;\mathcal{M})}\,
		\|\phi\|_{\ell_p(L_2)(\mathbb{R}^d)}.
		$$
		
		\item If $f \in \ell_p(L_2)(\mathbb{R}^d;\mathcal{M})$, $\phi \in \ell_{p,\infty}(L_2)(\mathbb{R}^d)$ with $0 < p < 2$, then $T_\phi M_f \in \mathcal{L}_{p,\infty}\big(B(L_2(\mathbb{R}^d))\overline{\otimes} \mathcal{M}\big)$ and
		$$
		\|T_\phi M_f\|_{\mathcal{L}_{p,\infty}}
		\le C_{p,d}\,
		\|f\|_{\ell_p(L_2)(\mathbb{R}^d;\mathcal{M})}\,
		\|\phi\|_{\ell_{p,\infty}(L_2)(\mathbb{R}^d)}.
		$$
		
		\item If $f \in \ell_{2,\log}(L_\infty)(\mathbb{R}^d;\mathcal{M})$ and $\phi \in \ell_{2,\infty}(L_4)(\mathbb{R}^d)$ (i.e., $p=2$), then $T_\phi M_f \in \mathcal{L}_{2,\infty}\big(B(L_2(\mathbb{R}^d))\overline{\otimes} \mathcal{M}\big)$ and
		$$
		\|T_\phi M_f\|_{\mathcal{L}_{2,\infty}}
		\le C_{2,d}\,
		\|f\|_{\ell_{2,\log}(L_\infty)(\mathbb{R}^d;\mathcal{M})}\,
		\|\phi\|_{\ell_{2,\infty}(L_4)(\mathbb{R}^d)}.
		$$
	\end{enumerate}
\end{lemma}

\begin{proof}
	The proof of the lemma following easily from  the scalar-valued Birman--Solomyak and Cwikel counterparts. Indeed, for $f: \mathbb{R}^d \rightarrow \mathcal{M}$, 	Define the scalar majorant
	$$
	h(x) := \|f(x)\|_{\mathcal{M}}  \quad (\text{pointwise operator norm}).
	$$
	For each $x$, we have $0 \le |f(x)| \le h(x) \otimes 1_{\mathcal{M}}$, hence as multiplication operators on $L_2(\mathbb{R}^d) \otimes L_2(\mathcal{M})$,
	$$
	0 \le M_{|f|} \le M_{h \otimes 1_{\mathcal{M}}}.
	$$
	Therefore,
	$$
	\big| (T_{\phi } M_{f })^*\big|^2
	= T_{\phi} M_{|f|^2} T_{\phi}^*
	\le T_{\phi} M_{h ^2 \otimes 1_{\mathcal{M}}} T_{\phi}^*.
	$$
	Since $\mu(t,X) \le \mu(t,Y)$ for all $t \ge 0$ whenever $0 \le X \le Y$, it follows that for all $t \ge 0$,
	$$
	\mu\big(t,|T_{\phi } M_{f }|\big)=	\mu\big(t,|(T_{\phi } M_{f })^*|\big)
	\le \mu\big(t,T_{\phi}M_{h^2 \otimes 1_{\mathcal{M}}} T_{\phi}^* \big)^{\frac{1}{2}}= \mu (t, M_{h \otimes 1_{\mathcal{M}}} T_{\phi}^* ).
	$$
	Furthermore, since
	$$
	\mu(t,A \otimes B) \le \|B\|_{\mathcal{M}}\,\mu(t,A), \qquad t \ge 0
	$$
	for general operator $A\otimes B \in \mathcal{N} \otimes \mathcal{M}$, we have
	$$
	\mu\big(t,|T_{\phi } M_{f }|\big)\leq  \mu (t, M_{h  } T_{\phi}^* ).
	$$
	The desired assertions then follow from the scalar-valued Birman--Solomyak and Cwikel counterparts.
\end{proof}

\begin{rk}\label{BS-example}
	Even though the operator-valued norms of $\ell_p(L_2)(\mathbb{R}^d;\mathcal{M})$ and $\ell_{2,\log}(L_\infty)(\mathbb{R}^d;\mathcal{M})$ do not seem easy to be calculated, the following facts are evident, and will frequently be used in the sequel:
	\begin{enumerate}[$\rm (i)$]
		\item $\displaystyle C_c (\mathbb{R}^d; \mathcal{M})\subset \ell_p(L_2)(\mathbb{R}^d;\mathcal{M})\cap \ell_{2,\log}(L_\infty)(\mathbb{R}^d;\mathcal{M});$
		\item $f(x)= \mathbf{1}_{\Omega}(x) \mathrm{p} \in \ell_p(L_2)(\mathbb{R}^d;\mathcal{M})   \cap    \ell_{2,\log}(L_\infty)(\mathbb{R}^d;\mathcal{M})$, where $\Omega\subset \mathbb{R}^d$ is compact, and $\mathrm{p}\in \mathcal{M} $ is a finite projection;
		\item the symbol function $\phi(\xi)=  (1+|\xi|^2)^{-\frac{m}{2}}$ of Bessel potential $J^{-m}$ satisfies
		$$ (1+|\xi|^2)^{-\frac{m}{2}} \in \ell_{\frac d m ,\infty}(L_2)(\mathbb{R}^d), \quad (1+|\xi|^2)^{-\frac{m}{2}} \in \ell_{\frac d m ,\infty}(L_4)(\mathbb{R}^d) , \quad m>0. $$
	\end{enumerate}
\end{rk}

\subsection{Localized Riemann $\zeta$-functions}
\label{section-def-zeta}

Appealing to Theorem \ref{op-tr=intker}, we are able to calculate the trace of elliptic $\Psi$DOs using its kernel.

\begin{proposition}
	Let $A \in \mathrm{C}\Psi^m\big(\mathbb{R}^d; \mathcal{M}\big)$ with $  m <-d  $. For $\phi \in C_c(\mathbb{R}^d ; L_1(\mathcal{M}) \cap \mathcal{M})$ with compact support $\Omega \subset \mathbb{R}^d$, we have
	$$
	\operatorname{Tr}(M_{\phi^{\ast}} A M_\phi) = \int_{\mathbb{R}^d} \tau(\phi(x)^{\ast} A(x, x) \phi(x)) \, dx.
	$$
	Similarly, for $\mathrm{P} (x) =   \mathbf{1}_{\Omega}(x) \mathrm{p} $, where $\mathrm{p}\in \mathcal{M} $ is a finite projection, we have
	$$
	\operatorname{Tr}(M_{\mathrm{P}} A M_{\mathrm{P}} ) = \int_{\Omega } \tau(  A(x, x) \mathrm{p}) \, dx.
	$$
\end{proposition}
\begin{proof}
	Define
	$$
	K_{A,\phi}(x,y) := \phi(x)^* A(x,y) \phi(y), \quad (x,y) \in \Omega\times\Omega,
	$$
	which is the kernel of the operator $M_{\phi^*} A M_\phi$.
	
	Since $ m  < -d$, the kernel $A(x,y)$ is continuous on compact subsets of $\mathbb{R}^d \times \mathbb{R}^d$ and takes values in $\mathcal{M}$. Given that  $\phi \in C_c(\mathbb{R}^d ; L_1(\mathcal{M}) \cap \mathcal{M})$, it follows that
	$$
	K_{A,\phi} \in C^0(\Omega\times\Omega; L_1(\mathcal{M})\cap \mathcal{M}).
	$$
	On the other hand, write
	$$M_{\phi^{\ast}} A M_\phi =  M_{\phi^{\ast}} J^{\frac{ m }{2}}J^{-\frac{ m }{2}} AJ^{-\frac{ m }{2}}  J^{\frac{ m }{2}}M_\phi . $$
	with a bounded $\Psi$DO $J^{-\frac{ m }{2}} AJ^{-\frac{ m }{2}}  $. For $M_{\phi^{\ast}} J^{\frac{ m }{2}}$, since its kernel is the operator valued function $$\phi(x)^{\ast}\  \mathcal{F}^{-1} \big((1+|\xi|^2) ^{\frac{ m }{4}}\big)(x-y)\in L_2(\mathbb{R}^d \times \mathbb{R}^d; L_2(  \mathcal{M})) ,$$
	Proposition \ref{op-S=intker} ensures that
	$$M_{\phi^{\ast}} J^{\frac{ m }{2}} \in \mathcal{L}_2\big(B(L_2(\mathbb{R}^d)) \overline{\otimes} \mathcal{M}\big) .$$
	It then follows from the H\"older inequality that
	$$
	M_{\phi^*} A M_\phi \in \mathcal{L}_1\big(B(L_2(\mathbb{R}^d)) \,\overline{\otimes}\, \mathcal{M}\big).
	$$
	Now we are able to apply Theorem \ref{op-tr=intker} to get
	$$
	\operatorname{Tr}\big(M_{\phi^*} A M_\phi\big) = \int_\Omega \tau\big(\phi(x)^* A(x,x) \phi(x)\big)\,dx.
	$$
	The assertion for $
	\operatorname{Tr}(M_{\mathrm{P}} A M_{\mathrm{P}} ) $ is treated similarly, so we omit the details.
\end{proof}

Therefore, we can define the localized $\zeta$-functions for $\mathcal{M}$-valued elliptic  $\Psi$DOs.

\begin{definition}\label{az zeta def}
	Let $A \in \mathrm{C}\Psi^m\big(\mathbb{R}^d; \mathcal{M}\big)$ with $ m  > 0$ be elliptic and strictly positive. For $\phi \in C_c(\mathbb{R}^d ; L_1(\mathcal{M}) \cap \mathcal{M})$, the localized Riemann $\zeta$-functions of $A$ and $\sigma(A)_m(x, \xi)$ are defined as
	$$
	\zeta_{A,\phi}(z) := \operatorname{Tr}(M_{\phi^{\ast}} A^{-z} M_\phi) = \int_{\mathbb{R}^d} \tau(\phi(x)^{\ast} A_{-z}(x, x) \phi(x)) dx, \quad \Re z > \tfrac{d}{m},
	$$
	and
	$$
	\zeta_{\sigma,\phi}(z) :=(2 \pi)^{-d} \int_{\mathbb{R}^d} \int_{\mathbb{R}^d} \tau(\phi(x)^{\ast} \sigma(A)_m(x, \xi)^{-z} \phi(x)) dx d\xi, \quad \Re z > \tfrac{d}{m}.
	$$
	Similarly, for $\mathrm{P} (x) =   \mathbf{1}_{\Omega}(x) \mathrm{p} $, we define
	$$
	\zeta_{A,\mathrm{P}}(z) := \operatorname{Tr}(M_{\mathrm{P}} A^{-z} M_{\mathrm{P}}) = \int_{\Omega} \tau(  A_{-z}(x, x) \mathrm{p}) dx, \quad \Re z > \tfrac{d}{m},
	$$
	and
	$$
	\zeta_{\sigma,\phi}(z) :=(2 \pi)^{-d} \int_{\mathbb{R}^d} \int_{\Omega} \tau(  \sigma(A)_m(x, \xi)^{-z} \mathrm{p}) dx d\xi, \quad \Re z > \tfrac{d}{m}.
	$$
\end{definition}

Combining this with Theorem \ref{holoextend} and Corollary \ref{sym holoextend}, we derive the analogous properties of the localized Riemann $\zeta$-functions in the following theorem.

\begin{theorem}\label{zeta pole}
	Let $A \in \mathrm{C}\Psi^m\big(\mathbb{R}^d; \mathcal{M}\big)$ with $ m  > 0$ be elliptic and strictly positive.
	\begin{enumerate}[$\rm (i)$]
		\item  For $\phi \in C_c(\mathbb{R}^d ; L_1(\mathcal{M}) \cap \mathcal{M})$, the localized Riemann $\zeta$-functions $\zeta_{A,\phi}(z)$ and $\zeta_{\sigma,\phi}(z)$ admit meromorphic continuations to $\mathbb{C}$, with the right-most simple pole located at $z = \frac{d}{m}$. The residues at $z = \frac{d}{m}$ are given by
		$$
		\frac{1}{m(2 \pi)^d} \int_{\mathbb{S}^{d-1} } \int_{\mathbb{R}^d} \tau(\phi(x)^{\ast} \sigma(A)_m(x, \xi)^{-\frac{d}{m}} \phi(x)) dx d \xi.
		$$
		\item  For $\mathrm{P} (x) =   \mathbf{1}_{\Omega}(x) \mathrm{p} $ with $\Omega\subset \mathbb{R}^d$ compact and $\mathrm{p}\in \mathcal{M} $ finite projection, the localized Riemann $\zeta$-functions $\zeta_{A,\mathrm{P}}(z)$ and $\zeta_{\sigma,\mathrm{P}}(z)$ admit meromorphic continuations to $\mathbb{C}$, with the right-most simple pole located at $z = \frac{d}{m}$. The residues at $z = \frac{d}{m}$ are given by
		$$
		\frac{1}{m(2 \pi)^d} \int_{\mathbb{S}^{d-1} } \int_{\Omega} \tau( \sigma(A)_m(x, \xi)^{-\frac{d}{m}} \mathrm{p}) dx d \xi.
		$$
	\end{enumerate}
	
\end{theorem}

\begin{proof}
	By Theorem \ref{holoextend}, the kernel $A_{-z}(x,x)$ of the complex power
	$A^{-z}$ (which is of order $-m z$) admits a meromorphic continuation to
	$\mathbb{C}$ with simple poles at $z_j = \frac{d-j}{m}$, $j\in\mathbb{N}_0$.
	Equivalently, as an operator-valued function of $z$, $A_{-z}(x,x)$ has the right-most simple pole at
	$z = \frac{d}{m}$, and its $\mathcal{N}$-valued residue there is obtained from
	\eqref{residuemost} by substituting $z_j = -\frac{d}{m}$ (i.e., $j=0$):
	\begin{equation}\label{residue-A}
		\mathrm{Res}_{z=\frac{d}{m}} A_{-z}(x,x)
		= -\frac{1}{m} \int_{\mathbb{S}^{d-1}} \sigma(B)_{-d}^{(-\frac{d}{m}),0}(x,\xi) \, \bar{d}\xi.
	\end{equation}
	According to \eqref{Bmzj,0z} (applied with $z = -\frac{d}{m}$), we have
	\[
	\sigma(B)_{-d}^{(-\frac{d}{m}),0}(x,\xi) = \sigma(A)_m(x,\xi)^{-\frac{d}{m}}.
	\]
	Inserting this into \eqref{residue-A} and recalling that $\bar{d}\xi = (2\pi)^{-d}d\xi$,
	we obtain
	\[
	\mathrm{Res}_{z=\frac{d}{m}} A_{-z}(x,x)
	= \frac{1}{m(2\pi)^d} \int_{\mathbb{S}^{d-1}} \sigma(A)_m(x,\xi)^{-\frac{d}{m}} \, d\xi.
	\]
	
	Similarly, Corollary \ref{sym holoextend} shows that the symbolic integral
	$\int_{\mathbb{R}^d} \sigma(A)_m(x,\xi)^{-z} d\xi$ has a meromorphic continuation
	with a simple pole at $z = \frac{d}{m}$, and its $\mathcal{N}$-valued residue is
	\[
	\mathrm{Res}_{z=\frac{d}{m}} \int_{\mathbb{R}^d} \sigma(A)_m(x,\xi)^{-z} d\xi
	= \frac{1}{m} \int_{\mathbb{S}^{d-1}} \sigma(A)_m(x,\xi)^{-\frac{d}{m}} \, d\xi.
	\]
	
	Now, for $\phi \in C_c(\mathbb{R}^d; L_1(\mathcal{M})\cap\mathcal{M})$, the localized
	$\zeta$-function is given by
	$\zeta_{A,\phi}(z) = \int_{\mathbb{R}^d} \tau\big(\phi(x)^* A_{-z}(x,x) \phi(x)\big) dx$.
	Multiplying the residue of $A_{-z}(x,x)$ by $\phi(x)^*$ and $\phi(x)$, integrating
	over $\mathbb{R}^d$, and taking the trace $\tau$ yields the residue formula stated in
	the theorem. The computation for $\zeta_{\sigma,\phi}(z)$ is entirely parallel,
	using the definition in Definition \ref{az zeta def} and the residue of the
	symbolic integral computed above.
\end{proof}

\subsection{Microlocal Weyl's law}
In this section we prove the following \emph{microlocal Weyl's law} for operator-valued pseudo-differential operators.
\begin{theorem}\label{microlocal_weyl_law}
	Let $A \in \mathrm{C}\Psi^m(\mathbb{R}^d;\mathcal{M})$ with $m>0$ be elliptic and positive semidefinite. Let $Q \in \mathrm{C}\Psi^0(\mathbb{R}^d;\mathcal{M}).$ For $\phi \in C_c(\mathbb{R}^d;L_1(\mathcal{M})\cap \mathcal{M}),$ we have
	\begin{align*}
		&\lim_{\lambda\to\infty} \lambda^{-\frac{d}{m}} \Tr\bigl(M_\phi Q \chi_{[0,\lambda]}(A)\bigr)\\
		=& \frac{m}{d(2\pi)^{d}} \int_{\mathbb{R}^d} \int_{\mathbb{S}^{d-1}}
		\tau\bigl(\phi(x) \sigma(Q)_0(x,\xi) \sigma(A)_m(x,\xi)^{-\frac{d}{m}}\bigr) dx \, d\xi.
	\end{align*}
\end{theorem}

Theorem \ref{microlocal_weyl_law} can be deduced, via the Wiener-Ikehara theorem, from the following result on the analytic continuation of zeta functions.
\begin{theorem}\label{microlocal_zeta_function}
	Let $A \in \mathrm{C}\Psi^m(\mathbb{R}^d;\mathcal{M})$ with $m>0$ be elliptic and strictly positive. Let $Q \in \mathrm{C}\Psi^0(\mathbb{R}^d;\mathcal{M}).$ For $\phi \in C_c(\mathbb{R}^d;L_1(\mathcal{M})\cap \mathcal{M}),$ the microlocalized zeta function
	\[
	\zeta_{\phi,Q,A}(s) := \Tr(M_\phi QA^{-s}),\quad \Re(s)>\frac{d}{m}
	\]
	admits a meromorphic continuation to the complex plane, with at most simple poles located at the points
	\[
	\{\frac{d-j}{m}\}_{j\geq 0}.
	\]
	The residue of $\zeta_{\phi,Q,A}$ at $s=\frac{d}{m}$ is given by
	\[
	\mathrm{Res}_{s=\frac{d}{m}} \zeta_{\phi,Q,A}(s) = (2\pi)^{-d}\int_{\mathbb{R}^d}\int_{\mathbb{S}^{d-1} } \tau(\phi(x)\sigma(Q)_0(x,\xi)\sigma(A)_m(x,\xi)^{-\frac{d}{m}})\,dxd\xi.
	\]
\end{theorem}
The proof is identical to that of Theorem \ref{zeta pole}, and is hence omitted.

Let us now introduce the Wiener-Ikehara Tauberian theorem, which was originally published in \cite{Ike31} by Ikehara. It is a special case of Wiener's Tauberian theorems, which were published by Wiener in \cite{Wie32} one year later. This theorem plays an important role in analytic number theory.
Let us state the Wiener-Ikehara Tauberian theorem in the following:

Let $F: \mathbb{R}_{+} \rightarrow \mathbb{R}_{+}$ be a monotone non-decreasing function such that $F(t)=0$ for $t \le \epsilon$, where $\epsilon>0$. Suppose
$$
\zeta(z)=\int_\epsilon^{\infty} t^{-z} d F(t)
$$
converges absolutely for $\Re z>k_0$ and there is a constant $\gamma$ such
that
$$\zeta(z)-\frac{\gamma}{z-k_0}$$
extends to a continuous function in $\Re z \ge k_0$. Then
$$
F(t) \sim \frac{\gamma}{k_0} t^{k_0} \quad \text { as } t \rightarrow \infty.
$$

\begin{proof}[Proof of Theorem \ref{microlocal_weyl_law}]
	Assume initially that $Q$ and $\phi$ are non-negative. Let
	\[
	N_{\phi,Q,A}(\lambda) := \Tr(M_{\phi}Q\chi_{[0,\lambda]}(A)).
	\]
	Since $Q$ and $\phi$ are non-negative, $N$ is non-decreasing.
	Since
	\[
	\zeta_{\phi,Q,A}(s) = \int_0^\infty \lambda^{-s}\,dN_{\phi,Q,A}(\lambda)
	\]
	the result of Theorem \ref{microlocal_zeta_function} and the Wiener-Ikehara theorem ensures that
	\[
	N_{\phi,Q,A}(\lambda) \sim \frac{m}{d}\lambda^{\frac{d}{m}}\mathrm{Res}_{s=\frac{d}{m}}\zeta_{\phi,Q,A}(s).
	\]
	This proves the result in the case that $\phi$ and $Q$ are positive. If $Q$ is merely self-adjoint, then since it has order zero it is bounded on $L_2(\mathbb{R}^d),$ so there exists a constant $C\geq 0$ such that $Q+C\geq 0.$ We have
	\begin{align*}
		N_{\phi,Q,A}(\lambda) &= N_{\phi,Q+C,A}(\lambda)-N_{\phi,C,A}(\lambda)\\
		&\sim \frac{m}{d}\lambda^{\frac{d}{m}}\mathrm{Res}_{s=\frac{d}{m}}(\zeta_{\phi,Q+C,A}(s)-\zeta_{\phi,C,A}(s))\\
		&= \frac{m}{d}\lambda^{\frac{d}{m}}\mathrm{Res}_{s=\frac{d}{m}}\zeta_{\phi,Q,A}(s).
	\end{align*}
	If $Q$ is non-self-adjoint, write $Q = \frac12(Q+Q^*)+\frac12(Q-Q^*)$ and again exploit the linearity of $\zeta_{\phi,Q,A}$ and $N_{\phi,Q,A}$ in $Q.$
	
	Similarly, for non-positive $\phi$ we decompose $\phi$ into a linear combination of four positive functions and deduce the result by linearity.
\end{proof}

\medskip

\subsection{Weyl's law for $\mathcal{M}$-valued  classical $\Psi$DOs of negative order}\label{weyl-pdo-negative}

In the semifinite von Neumann algebraic setting, the notion of singular value sequences is generalized to the singular value function $\mu(\cdot, A)$, which is defined in terms of the distribution function of $|A|$ with respect to the trace. For $\Re m  \ge 0$, by Proposition \ref{l2bdd} and Theorem \ref{thm-tensor algebra}, the following inclusions hold:
$$
\mathrm{C}\Psi^{-m}\big(\mathbb{R}^d; \mathcal{M}\big) \subset \Psi^{-\Re m}\big(\mathbb{R}^d \times \Lambda ; \mathcal{M}\big) \subset B(L_2(\mathbb{R}^d)) \overline{\otimes} \mathcal{M}.
$$
Then for $A \in \mathrm{C}\Psi^{-m}\big(\mathbb{R}^d; \mathcal{M}\big)$ ($\Re m > 0$), its singular value function is given by
$$
\mu(t, A) = \inf \{s > 0 : \operatorname{Tr}(e_s^\perp(|A|)) \le t\}, \quad t > 0,
$$
where $e_s^\perp$ represents the spectral projection associated with the interval $(s, \infty)$.

The principal symbol of $|A|$ is $|\sigma(A)_{-m}|$, which belongs to
\[
L_{\infty}(\mathbb{R}^d)\overline{\otimes} L_{\infty}(\mathbb{S}^{d-1})\overline{\otimes}\mathcal{M}.
\]
Its singular value function is given by
$$
\mu(t, \sigma(A)_{-m}) = \inf \Bigg\{s > 0 : \int_{\mathbb{S}^{d-1} } \int_{\mathbb{R}^d} \tau(e_s^\perp(|\sigma(A)_{-m}(x, \xi)|)) \, dx \, d\xi \le t \Bigg\}, \quad t > 0.
$$

To study Weyl's law and Connes' trace theorem on the product space formed by the locally compact space $\mathbb{R}^d$ and the semifinite von Neumann algebra $\mathcal{M}$, we first consider the local case, and then extend the results to the full space and algebra via approximation. For the truncation in $\mathcal{M}$, we utilize a $\tau$-finite projection $\mathrm{p} \in \mathcal{M}$, while for the spatial cutoff, we utilize nonnegative functions from $C_c^\infty(\mathbb{R}^d)$.

The core of this section is the following

\begin{theorem} \label{asymp ellip}
	Let $A \in \mathrm{C}\Psi^{m}(\mathbb{R}^d; \mathcal{M})$ with $m > 0$ be elliptic, and assume that $|A|$ is strictly positive. Let $ \mathrm p \in \mathcal{M}$ be a $\tau$-finite projection.
	Then for every  $0\le g\in C_c (\mathbb{R}^d)$ the following hold:
	\begin{align*}
		&\lim_{t \rightarrow \infty} t^{\frac{m}{d}} \mu(t, M_{\mathrm{p}g}|A|^{-1}M_{\mathrm{p}g})\\
		= &(2\pi)^{-m} \lim_{t \rightarrow \infty} t^{\frac{m}{d}} \mu(t, \mathrm{p}g^{2}|\sigma(A)_m|^{-1}\mathrm{p}) \\
		=& d^{-\frac{m}{d}} (2\pi)^{-m} \left[ \int_{\mathbb{S}^{d-1} } \int_{\mathbb{R}^d} \tau\left(  g(x)^{\frac{2d}{m}} |\sigma(A)_m(x, \xi)|^{-\frac{d}{m}} \mathrm{p} \right) dx \, d\xi \right]^{\frac{m}{d}}.
	\end{align*}
	Furthermore, $M_{\mathrm{p}g}|A|^{-1}M_{\mathrm{p}g} \in \mathcal{L}_{\frac{d}{m}, \infty}\big(B(L_2(\mathbb{R}^d)) \overline{\otimes} \mathcal{M}\big)$.
\end{theorem}

The proof of Theorem \ref{asymp ellip}  is based on the analysis on the Riemann $\zeta$-functions provided in Section \ref{section-def-zeta}, and a type of Tauberian theorem. However, in order to use the classical Wiener-Ikehara Tauberian theorem introduced in the previous section, we have to determine whether the function
$$
\operatorname{Tr}\big((M_{\mathrm{p}g}|A|^{-1}M_{\mathrm{p}g})^z\big) - \frac{c}{z-p}, \quad \Re z > p= \frac d m
$$
admits a continuous extension to the closed half-plane $\Re z \ge \frac d m$. In fact, verifying this condition is nontrivial. Instead, it follows from Definition \ref{az zeta def} and Theorem \ref{zeta pole} that a more tractable condition is that the function
$$
\operatorname{Tr}\big(|A|^{-z}M_{\mathrm{p}g}^z\big) - \frac{c}{z-p}, \quad \Re z > p =\frac d m
$$
extends continuously to $\Re z \geq \frac  d m $. Therefore, due to the noncommutativity of $|A|^{-1} $ and $M_{\mathrm{p}g}$, the classical Wiener-Ikehara Tauberian theorem cannot be directly applied to the current situation.\ To overcome these difficulties, we apply the noncommutative Tauberian theorem \cite[Theorem 1.2]{MSZ22}, which is stated in the following lemma.

\begin{lemma}\label{nc-Tauberian}
	Let $p > 2$ and let $0 \le Y, Z \in \mathcal{L}_{\infty}$ be such that $Z \in \mathcal{L}_{p, \infty}$, $[Z, Y^{\frac{1}{2}}] \in \mathcal{L}_{\frac{p}{2}, \infty}$. Suppose there exists a constant $c \ge 0$ such that the function
	$$
	\operatorname{Tr}\big(Z^z Y^z\big) - \frac{c}{z-p}, \quad z \in \mathbb{C}, \quad \Re z > p,
	$$
	admits a continuous extension to the closed half-plane $\{z \in \mathbb{C} : \Re z \ge p\}$. Then
	$$
	\lim_{s \rightarrow 0} s^p \operatorname{Tr}\big(e_s^{\perp}(Y^{\frac{1}{2}} Z Y^{\frac{1}{2}})\big) = \lim_{h \rightarrow 0} h \operatorname{Tr}\big(e_h^{\perp}(Y^{\frac{p}{2}} Z^p Y^{\frac{p}{2}})\big) = \frac{c}{p}.
	$$
	Equivalently,
	$$
	\lim_{t \rightarrow \infty} t\ \mu(t, Y^{\frac{1}{2}} Z Y^{\frac{1}{2}})^p = \lim_{t \rightarrow \infty} t\ \mu(t, Y^{\frac{p}{2}} Z^p Y^{\frac{p}{2}}) = \frac{c}{p}.
	$$
\end{lemma}

Indeed, we need to make some adjustments of the above lemma to our situation. Firstly, the (weak) $\mathcal{L}_{p }$ norms are associated to the tensor trace $\operatorname{Tr}$ on $B(L_2(\mathbb{R}^d)) \overline{\otimes} \mathcal{M}$. Secondly, we have to consider the local version, namely, the trace $\operatorname{Tr}_{\mathrm{P}} := \operatorname{Tr}(M_{ \mathrm{P}} \cdot  M_{ \mathrm{P}}) $ on $B(L_2(\Omega)) \overline{\otimes} \mathcal{M}$. But since the proofs in \cite{MSZ22} work in the current situation {\textit{  mutatis mutandi}}, we omit the proof of the above lemma.

\begin{proof}[Proof of Theorem \ref{asymp ellip}]
	For the $\tau$-finite projection $ \mathrm{p} \in 	\mathcal{M}$, define the trace $\tau_{\mathrm{p}}$ for $a \in \mathcal{M}$ as
	$$
	\tau_{\mathrm{p}}(a) := \tau(\mathrm{p} a \mathrm{p}) \;\;\big(= \tau(a \mathrm{p}) = \tau(\mathrm{p} a)\big).
	$$
	The singular value function with respect to this localized trace is given by
	$$
	\mu_{\mathrm{p}}(t, a) := \inf \big\{ s > 0 : \tau_{\mathrm{p}}( e_s^\perp(|a|) ) \le t \big\}, \quad t > 0,
	$$
	or equivalently,
	$$	\mu_{\mathrm{p}}(t, a) := \inf \big\{\|ae\|: e\in\mathcal{M} \; \text{is a projection},\;  \tau_{\mathrm{p}} (e^\perp)\leq t \big\}.$$
	Using this equivalent definition of singular value function, we see clearly that
	$$\mu_{\mathrm{p}}(t, a) = \mu (t, a\mathrm{p}), $$
	where $\mu (t, \cdot)=\inf \big\{ s > 0 : \tau ( e_s^\perp(|\cdot |) ) \le t \big\}$ denotes the usual singular value function. Accordingly, if $a\in \mathcal{M} $ is positive,
	$$\mu_{\mathrm{p}}(t, a) = \mu (t, a\mathrm{p}) = \mu_{\mathrm{p}}(t, a\mathrm{p}) =\mu_{\mathrm{p}}(t, \mathrm{p}a)=\mu (t, \mathrm{p}a\mathrm{p}) .$$
	For the same reason, if we consider $	\Tr_{\mathrm{p}} =  \operatorname{tr} \otimes \tau_{\mathrm{p}}$ on $B(L_2(\mathbb{R}^d)) \,\overline{\otimes}\, \mathcal{M}$ and denote
	$$
	\mu_{\mathrm{p}}(t, A) := \inf \big\{ s > 0 : \Tr_{\mathrm{p}}( e_s^\perp(|A|) ) \le t \big\}, \quad t > 0,
	$$
	then we have
	\begin{equation}
		\label{svf-local=global}
		\mu(t, M_{\mathrm{p}g}|A|^{-1}M_{\mathrm{p}g})=\mu_{\mathrm{p}}(t, M_{ g}|A|^{-1}M_{ g}).
	\end{equation}

	We firstly assume that $0\le g\in C_c^\infty(\mathbb{R}^d)$, and let $\Omega $ be a compact set containing $ \operatorname{supp}(g) $. Assume also $m  \in  (0, 1 ] \cap (0, \frac d 2 )$. Consider the function
	$$z \mapsto  \operatorname{Tr}_{\mathrm{P}}  \Big(|A|^{-z}   M_{g^{2z}} \Big), \quad \Re z > \frac d m . $$
	By Definition \ref{az zeta def} and Theorem \ref{zeta pole}, this function admits a meromorphic continuation to $\Re z > 0$, with the right-most simple pole at $z = \frac{d}{m}$. The residue at this pole is
	$$
	\operatorname{Res}_{z=\frac{d}{m}}  \operatorname{Tr}_{\mathrm{P}}  \Big(|A|^{-z}   M_{g^{2z}} \Big)
	= \frac{1}{m (2\pi)^d} \int_{\mathbb{S}^{d-1} } \int_{\Omega} \tau\Big( g(x)^{\frac{2d}{m}} |\sigma(A)_m(x,\xi)|^{-\frac{d}{m}}  \mathrm{p}\Big)\, dx \, d\xi.
	$$
	Next, in order to apply Lemma \ref{nc-Tauberian}, we need to verify the assumptions for $p = \frac d m > 2 $ and $Z = |A|^{-1} , \ Y = M_{g^2 }$. Writing $ |A|^{-1} = |A|^{-1} J^m J^{-m}$, and using the fact
	$$   \|J^{-m} \|_{\mathcal{L}_{p,\infty}\big(B(L_2(\Omega))\overline{\otimes}    \mathcal{M}_\mathrm{p}\big)}= \|J^{-m} M_{\mathrm{P}}\|_{\mathcal{L}_{p,\infty}\big(B(L_2(\mathbb{R}^d))\overline{\otimes} \mathcal{M}\big)}, $$
	we know from Lemma \ref{Cwikel} and the H\"older inequality that the operator $ |A|^{-1}  \in  \mathcal{L}_{p,\infty}\big(B(L_2(\Omega))\overline{\otimes}\   \mathcal{M}_\mathrm{p}\big)$. For the commutator $[|A|^{-1},  M_{g }]$, notice that by Proposition \ref{cl M}, $[|A|^{-1},  M_{g }]\in \mathrm{C}\Psi^{-m-1}(\mathbb{R}^d; \mathcal{M})$. So the same argument as for $ |A|^{-1} $ gives the condition
	$$[|A|^{-1},  M_{g }]\in \mathcal{L}_{\frac{d}{m+1},\infty}\big(B(L_2(\Omega))\overline{\otimes}  \mathcal{M}_\mathrm{p}\big)\subset  \mathcal{L}_{\frac{d}{2m},\infty}\big(B(L_2(\Omega))\overline{\otimes} \mathcal{M}_\mathrm{p}\big).$$
	So we have verified that the operators $Z =   |A|^{-1}  $ and $Y = M_{g^2 }$ satisfy the conditions in Lemma \ref{nc-Tauberian}. Thus,
	\begin{equation}\label{asymp ellip-eq-s-l}
		\begin{split}
			&\lim_{t \to \infty} t^{\frac{m}{d}} \, \mu_{\mathrm{P}}\big(t, M_{g} |A|^{-1} M_{g} \big)\\
			= &\left[
			\frac{1}{(2\pi)^d} \int_{\mathbb{S}^{d-1}} \int_{\mathbb{R}^d} \tau \Big( g(x)^{\frac{2d}{m}} |\sigma(A)_m(x,\xi)|^{-\frac{d}{m}} \mathrm{p} \Big)\, dx \, d\xi
			\right]^{\frac{m}{d}}.
		\end{split}
	\end{equation}
	Then by \eqref{svf-local=global} (also holds true for $\mu_{\mathrm{P}}$), we get back to $\mu(t, \cdot )$, namely,
	\begin{equation}\label{asymp ellip-eq-s}
		\begin{split}
			&\lim_{t \to \infty} t^{\frac{m}{d}} \, \mu\big(t, M_{g \mathrm{p}} |A|^{-1} M_{g \mathrm{p}} \big)\\
			=&\left[
			\frac{1}{(2\pi)^d} \int_{\mathbb{S}^{d-1} } \int_{\mathbb{R}^d} \tau \Big( g(x)^{\frac{2d}{m}} |\sigma(A)_m(x,\xi)|^{-\frac{d}{m}} \mathrm{p} \Big)\, dx \, d\xi
			\right]^{\frac{m}{d}}.
		\end{split}
	\end{equation}

	Let us now get rid of the assumption $ g\in C_c^\infty(\mathbb{R}^d)$. Note that
	\begin{equation*}
		|A|^{-1} M_{g^2  } - M_{g } |A|^{-1} M_{g  }
		=    [ |A|^{-1} , M_{g}]    M_{g }  \in \Big(\mathcal{L}_{\frac{d}{m}, \infty}\Big)_0.
	\end{equation*}
	So \eqref{asymp ellip-eq-s-l} is equivalent to
	\begin{equation}\label{asymp ellip-eq-ns-l}
		\lim_{t \to \infty} t^{\frac{m}{d}} \, \mu_{\mathrm{P}}\big(t,   |A|^{-1} M_{g ^2 } \big)
		=
		\left[
		\frac{1}{(2\pi)^d} \int_{\mathbb{S}^{d-1} } \int_{\mathbb{R}^d} \tau \Big( g(x)^{\frac{2d}{m}} |\sigma(A)_m(x,\xi)|^{-\frac{d}{m}} \mathrm{p} \Big)\, dx \, d\xi
		\right]^{\frac{m}{d}}.
	\end{equation}
	and thus \eqref{asymp ellip-eq-s} is equivalent to
	\begin{equation}\label{asymp ellip-eq-ns}
		\lim_{t \to \infty} t^{\frac{m}{d}} \, \mu\big(t,    |A|^{-1} M_{g^2 \mathrm{p}}  \big)
		=
		\left[
		\frac{1}{(2\pi)^d} \int_{\mathbb{S}^{d-1} } \int_{\mathbb{R}^d} \tau \Big( g(x)^{\frac{2d}{m}} |\sigma(A)_m(x,\xi)|^{-\frac{d}{m}} \mathrm{p} \Big)\, dx \, d\xi
		\right]^{\frac{m}{d}}.
	\end{equation}
	With this formula, for a general continuous nonnegative function $\tilde{g}$, one can approximate $\tilde{g}$ by a sequence $\{g_k\}_{k\geq 1} \subset C_c^\infty(\mathbb{R}^d)$ uniformly. Taking the limit on $k$ in both \eqref{asymp ellip-eq-ns-l} and \eqref{asymp ellip-eq-ns}, we find that both \eqref{asymp ellip-eq-ns-l} and \eqref{asymp ellip-eq-ns} hold for general nonnegative $\tilde{g}\in  C_c(\mathbb{R}^d)$.

	The conclusion for all $m>0$ is done by iteration. Take the range $m  \in  (0, 2 ] \cap (0, d )$ for instance.
	For $A \in \mathrm{C}\Psi^{m}(\mathbb{R}^d; \mathcal{M})$ and $g \in C_c (\mathbb{R}^d)$, applying the proved formula \eqref{asymp ellip-eq-ns-l} to $|A|^{\frac 1 2 } \in \mathrm{C}\Psi^{\frac m  2}(\mathbb{R}^d; \mathcal{M})$ and $\sqrt{g}\in C_c (\mathbb{R}^d)$, we have
	\begin{equation*}
		\lim_{t \to \infty} t^{\frac{m}{2d}} \, \mu_{\mathrm{P}}\big(t,   |A|^{-\frac 1 2 } M_{g   } \big)
		=
		\left[
		\frac{1}{(2\pi)^d} \int_{\mathbb{S}^{d-1} } \int_{\mathbb{R}^d} \tau \Big( g(x)^{\frac{2d}{m}} |\sigma(A)_m(x,\xi)|^{-\frac{ d}{m}} \mathrm{p} \Big)\, dx \, d\xi
		\right]^{\frac{m}{2d}},
	\end{equation*}
	where we have substituted $\sigma(|A|^{\frac 1 2 })_{\frac m 2 }(x,\xi) $ by $|\sigma(A)_m(x,\xi) |^{\frac 1 2 }$.
	Then using the fact $ \mu  (t,   |x|^2)  = \mu  (t,  x)^2 $, we arrive at
	\begin{equation*}
		\begin{split}
			&\lim_{t \to \infty} t^{\frac{m}{2d}} \, \mu_{\mathrm{P}}\big(t,  M_{g   }  |A|^{-1} M_{g   } \big)^{\frac 1 2 }\\
			=
			&\left[
			\frac{1}{(2\pi)^d} \int_{\mathbb{S}^{d-1} } \int_{\mathbb{R}^d} \tau \Big( g(x)^{\frac{2d}{m}} |\sigma(A)_m(x,\xi)|^{-\frac{d}{m}} \mathrm{p} \Big)\, dx \, d\xi
			\right]^{\frac{m}{2d}}
		\end{split}
	\end{equation*}
	which is nothing but
	\begin{equation*}
		\begin{split}
			&
			\lim_{t \to \infty} t^{\frac{m}{ d}} \, \mu_{\mathrm{P}}\big(t,  M_{g   }  |A|^{-1} M_{g   } \big)\\
			=
			&\left[
			\frac{1}{(2\pi)^d} \int_{\mathbb{S}^{d-1} } \int_{\mathbb{R}^d} \tau \Big( g(x)^{\frac{2d}{m}} |\sigma(A)_m(x,\xi)|^{-\frac{d}{m}} \mathrm{p} \Big)\, dx \, d\xi
			\right]^{\frac{m}{d}} .
		\end{split}
	\end{equation*}

	The proof for $\lim_{t \rightarrow \infty} t^{\frac{m}{d}} \mu(t, \mathrm{p}g^{2}|\sigma(A)_m|^{-1}\mathrm{p}) $ is much simpler: When we consider the trace $\int_{\mathbb{R}^d} \otimes \int_{\mathbb{S}^{d-1} } \otimes  \tau_{\mathrm{p}}$ on  $L_{\infty}(\mathbb{R}^d)\overline{\otimes} L_{\infty}(\mathbb{S}^{d-1})\overline{\otimes}\mathcal{M}$, the involved elements $ g$ and $ |\sigma(A)_m|^{-1}$ are commutative, and the weak $L_p$ properties of $ |\sigma(A)_m|^{-1}$ are guaranteed by the decay of the $\Psi$DOs. Simply applying Theorem \ref{zeta pole} and then Lemma \ref{nc-Tauberian} concludes the desired result.
\end{proof}

By applying Lemma \ref{weak weyl} and Lemma \ref{ideal approximate}, we generalize the perturbation technique of Birman-Solomyak introduced in \cite[Section 4]{BS70}. This generalization extends the singular value sequence ($\mu(n,\cdot), n \in \mathbb{N}$) of a compact operator to the singular value function ($\mu(t,\cdot), t \in \mathbb{R}_{+}$) of an element in a semifinite von Neumann algebra. Furthermore, by adapting the proof strategy from \cite{BS77,BS79,BS79a} and \cite[Theorem 6.1]{P23}, we extend the singular spectral asymptotic limit for the inverses of strictly positive $\mathcal{M}$-valued elliptic $\Psi$DOs (of order $m> 0$) established in Theorem \ref{asymp ellip} to general $\mathcal{M}$-valued classical  $\Psi$DOs of negative order.


\begin{theorem}\label{asymp-nu-pn}
	Let $ T \in \mathrm{C}\Psi^{-m}\big(\mathbb{R}^d; \mathcal{M}\big) $ with $ m > 0 $, and suppose there exists a function $ \nu \in C_c (\mathbb{R}^d) $ and a $ \tau $-finite projection $ \mathrm{p} \in \mathcal{M} $ such that $ T = T M_{\mathrm{p}\nu} $. Then $ T \in \mathcal{L}_{\frac{d}{m}, \infty}\big(B(L_2(\mathbb{R}^d)) \overline{\otimes} \mathcal{M}\big)$, and
	\begin{align*}
		\lim_{t \to \infty} t^{\frac{m}{d}} \mu(t, T) &= (2 \pi)^{-m} \lim_{t \to \infty} t^{\frac{m}{d}} \mu(t, \sigma(T)_{-m}) \\
		&= d^{-\frac{m}{d}} (2 \pi)^{-m} \left[ \int_{\mathbb{S}^{d-1} } \int_{\mathbb{R}^d} \tau\Big(|\sigma(T)_{-m}(x, \xi)|^{\frac{d}{m}}\Big)\, dx \, d\xi \right]^{\frac{m}{d}}.
	\end{align*}
\end{theorem}

\begin{proof}
	By assumption, we have $|T|^{2} = M_{\mathrm{p}\bar{\nu}} |T|^{2} M_{\mathrm{p}\nu}$.
	Choose $0\le g \in C_c^\infty(\mathbb{R}^d)$ such that $g \equiv 1$ on an open neighborhood of $ \operatorname{supp}(\nu) $. Since $ T = T M_{\mathrm{p}\nu} $, it also holds true that $ T = T M_{\mathrm{p} g} $, so
	$$  |\sigma(T)_{-m}(x, \xi)|^{2} = g(x)\ \mathrm{p}\ |\sigma(T)_{-m}(x, \xi)|^{2}\ \mathrm{p}\  g(x).$$
	Define for $\varepsilon > 0$ the regularized operator as
	$$
	P_\varepsilon := \Big(|T|^2 + \varepsilon^2 M_{\mathrm{p}g} J^{-2m} M_{\mathrm{p}g}\Big)^{\frac{1}{2}}.
	$$
	By Proposition \ref{cl M}, $|T|^2 \in \mathrm{C}\Psi^{-2m}\big(\mathbb{R}^d ; \mathcal{M}\big)$, and
	$$
	P_\varepsilon^2 - |T|^2 = \varepsilon^2 M_{\mathrm{p}g} J^{-2m}  M_{\mathrm{p}g} \in \mathrm{C}\Psi^{-2m}\big(\mathbb{R}^d ; \mathcal{M}\big).
	$$
	It follows from Lemmas \ref{Cwikel} and \ref{Cwikel-0-2} that $P_\varepsilon^2 \to |T|^2$ in $\mathcal{L}_{\frac{d}{2m}, \infty}$ as $\varepsilon \to 0$.

	Let $A_\varepsilon \in \mathrm{C}\Psi^m\big(\mathbb{R}^d; \mathcal{M}\big)$ be an elliptic strictly positive operator, with principal symbol
	$$
	\sigma(A_\varepsilon)_m(x, \xi) = \big(|\sigma(T)_{-m}(x, \xi)|^2 + \varepsilon^2 |\xi|^{-2m} \big)^{-\frac{1}{2}};
	$$
	the ellipticity of $A_\varepsilon$ follows clearly from the above formula, while the strict positivity can be obtained by adding a small constant. So we are able to apply Theorem \ref{asymp ellip} to deduce that
	\begin{align*}
		&\lim_{t \to \infty} t^{\frac{2m}{d}} \mu(t, M_{\mathrm{p}g} |A_\varepsilon|^{-2} M_{\mathrm{p}g}) \\
		=& d^{-\frac{2m}{d}} (2 \pi)^{-2m} \left[ \int_{\mathbb{S}^{d-1} } \int_{\mathbb{R}^d} \tau\Big(\big(|\sigma(T)_{-m}(x, \xi)|^2 + \varepsilon^2 g(x)^2 \mathrm{p} \big)^{\frac{d}{2m}}\Big)\, dx \, d\xi\right]^{\frac{2m}{d}}.
	\end{align*}
	As $\varepsilon \to 0$, dominated convergence yields
	\begin{equation}\label{T square}
		\begin{split}
			&\lim_{\varepsilon \to 0} \lim_{t \to \infty} t^{\frac{2m}{d}} \mu(t, M_{\mathrm{p}g} |A_\varepsilon|^{-2} M_{\mathrm{p}g}) \\
			=& d^{-\frac{2m}{d}} (2 \pi)^{-2m} \left[ \int_{\mathbb{S}^{d-1} } \int_{\mathbb{R}^d} \tau\Big(|\sigma(T)_{-m}(x, \xi)|^{\frac{d}{m}}\Big)\, dx \, d\xi\right]^{\frac{2m}{d}}.
		\end{split}
	\end{equation}

	Let us now link the spectral asymptotic of $T$ with that in \eqref{T square}.
	The principal symbol of $|T|^2 + \varepsilon^2 J^{-2m}$ is $|\sigma(T)_{-m}(x, \xi)|^2 + \varepsilon^2 |\xi|^{-2m}$, so
	$$
	\sigma\big( |A_\varepsilon|^{-2}\big)_{-2m}(x, \xi) = \sigma\big(|T|^2 + \varepsilon^2J^{-2m}\big)_{-2m}(x, \xi).
	$$
	Hence,
	$$
	|T|^2 + \varepsilon^2 J^{-2m} - |A_\varepsilon|^{-2} \in \mathrm{C}\Psi^{-2m-1}\big(\mathbb{R}^d ; \mathcal{M}\big).
	$$
	Moreover, we have
	\begin{align*}
		P_\varepsilon^2 - M_{\mathrm{p}g} |A_\varepsilon|^{-2} M_{\mathrm{p}g} &= M_{\mathrm{p}g}\Big(|T|^2 + \varepsilon^2J^{-2m}-|A_\varepsilon|^{-2}\Big)J^{2m+1}J^{-2m-1} M_{\mathrm{p}g}.
	\end{align*}
	Applying  Lemmas \ref{Cwikel} and \ref{Cwikel-0-2}, we obtain
	$$
	P_\varepsilon^2 - M_{\mathrm{p}g} |A_\varepsilon|^{-2} M_{\mathrm{p}g} \in \mathcal{L}_{\frac{d}{2m+1}, \infty}\big(B(L_2(\mathbb{R}^d)) \overline{\otimes} \mathcal{M}\big) \subset \big(\mathcal{L}_{\frac{d}{2m}, \infty}\big)_0\big(B(L_2(\mathbb{R}^d)) \overline{\otimes} \mathcal{M}\big).
	$$
	Consequently, by Lemma \ref{weak weyl}, $P_\varepsilon^2$ and $M_{\mathrm{p}g} |A_\varepsilon|^{-2} M_{\mathrm{p}g}$ have the same asymptotic singular value behavior:
	$$
	\lim_{t \to \infty} t^{\frac{2m}{d}} \mu(t, P_\varepsilon^2) = \lim_{t \to \infty} t^{\frac{2m}{d}} \mu(t, M_{\mathrm{p}g} |A_\varepsilon|^{-2}M_{\mathrm{p}g}).
	$$
	Since $P_\varepsilon^2 \to |T|^2$ in $\mathcal{L}_{\frac{d}{2m}, \infty}$ as $\varepsilon \to 0$, Lemma \ref{ideal approximate} ensures
	$$
	\lim_{t \to \infty} t^{\frac{2m}{d}} \mu(t, |T|^2) = \lim_{\varepsilon \to 0} \lim_{t \to \infty} t^{\frac{2m}{d}} \mu(t, P_\varepsilon^2) = \lim_{\varepsilon \to 0} \lim_{t \to \infty} t^{\frac{2m}{d}} \mu(t, M_{\mathrm{p}g} |A_\varepsilon|^{-2} M_{\mathrm{p}g}).
	$$
	It follows from \eqref{T square} that
	$$
	\lim_{t \to \infty} t^{\frac{m}{d}} \mu(t, T) = d^{-\frac{m}{d}} (2 \pi)^{-m} \left[ \int_{\mathbb{S}^{d-1} } \int_{\mathbb{R}^d} \tau\Big(|\sigma(T)_{-m}(x, \xi)|^{\frac{d}{m}}\Big)\, dx \, d\xi\right]^{\frac{m}{d}}.
	$$

	The proof for $
	\lim _{t \rightarrow \infty} t^{\frac{m}{d}} \mu(t, \sigma(T)_{-m})$ is similar to that in Theorem \ref{asymp ellip}; the details are therefore omitted.
\end{proof}

We now extend this result to the broader class of $\mathcal{M}$-valued $\Psi$DOs that are compactly supported from the right,
following \cite[Section~6]{FSZ24}.
\begin{definition}
	An $\mathcal{M}$-valued $\Psi$DO $T \in \mathrm{C}\Psi^{-m}\big(\mathbb{R}^d; \mathcal{M}\big)$ with $m > 0$ is said to be compactly supported from the right if there exists $f \in C_c(\mathbb{R}^d; L_{\frac{d}{m}}(\mathcal{M}) \cap \mathcal{M})$ such that $T = T M_f$.
\end{definition}

\begin{theorem} \label{asymp most}
	Let $T \in \mathrm{C}\Psi^{-m}\big(\mathbb{R}^d; \mathcal{M}\big)$ with $m > 0$, and $f \in C_c(\mathbb{R}^d; L_{\frac{d}{m}}(\mathcal{M}) \cap \mathcal{M})$. Then
	\[
	T M_f \in \mathcal{L}_{\frac{d}{m}, \infty}\bigl(B(L_2(\mathbb{R}^d)) \overline{\otimes} \mathcal{M}\bigr)
	\]
	and
	\begin{align*}
		\lim_{t \to \infty} t^{\frac{m}{d}} \mu(t, T M_f) &= (2 \pi)^{-m} \lim_{t \to \infty} t^{\frac{m}{d}} \mu(t, \sigma(T)_{-m}f) \\
		&= d^{-\frac{m}{d}} (2 \pi)^{-m} \left[ \int_{\mathbb{S}^{d-1} } \int_{\mathbb{R}^d} \tau\Big(|\sigma(T)_{-m}(x, \xi)f(x)|^{\frac{d}{m}}\Big)\, dx \, d\xi \right]^{\frac{m}{d}}.
	\end{align*}
	In particular, if $T$ is compactly supported from the right. Then
	\[
	T \in \mathcal{L}_{\frac{d}{m}, \infty}\bigl(B(L_2(\mathbb{R}^d)) \overline{\otimes} \mathcal{M}\bigr)
	\]
	and
	\begin{align*}
		\lim_{t \to \infty} t^{\frac{m}{d}} \mu(t, T) &= (2 \pi)^{-m} \lim_{t \to \infty} t^{\frac{m}{d}} \mu(t, \sigma(T)_{-m}) \\
		&= d^{-\frac{m}{d}} (2 \pi)^{-m} \left[ \int_{\mathbb{S}^{d-1} } \int_{\mathbb{R}^d} \tau\Big(|\sigma(T)_{-m}(x, \xi)|^{\frac{d}{m}}\Big)\, dx \, d\xi \right]^{\frac{m}{d}}.
	\end{align*}
\end{theorem}

\begin{proof}
	Let $(\mathrm{p}_\lambda)_{\lambda \in \Lambda} \subset \mathcal{M}$ be an increasing net of projections such that $\mathrm{p}_\lambda \nearrow 1$ in the strong operator topology and $\tau(\mathrm{p}_\lambda) < \infty$ for all $\lambda$, and set $\mathrm{q}_\lambda := 1 - \mathrm{p}_\lambda$.
	Let $a \in L_\frac{d}{m}(\mathcal{M}) \cap \mathcal{M}$. Then for any $ p >0$,
	\begin{equation}\label{lem:finite-trace-approx}
		\|a(1 - \mathrm{p}_\lambda)\|_p \longrightarrow 0 \quad \text{as } \lambda \to \infty.
	\end{equation}

	For $f \mathrm{p}_\lambda \in C_c(\mathbb{R}^d; L_\frac{d}{m}(\mathcal{M}) \cap \mathcal{M})$, define
	$$
	T_\lambda := T M_{f \mathrm{p}_\lambda}.
	$$
	Since $\mathrm{p}_\lambda$ is $\tau$-finite, by Theorem \ref{asymp-nu-pn} we have
	\begin{equation}\label{eq:Tn-asymp}
		T_\lambda \in \mathcal{L}_{\frac{d}{m}, \infty}, \qquad
		\lim_{t \to \infty} t^{\frac{m}{d}} \mu(t, T_\lambda) =: C_\lambda,
	\end{equation}
	with
	\begin{equation}\label{eq:Cn2}
		\begin{split}
			C_\lambda & = d^{-\frac{m}{d}} (2\pi)^{-m}
			\left[ \int_{\mathbb{S}^{d-1} } \int_{\mathbb{R}^d}
			\tau\Big(|\sigma(T_\lambda)_{-m}(x,\xi)|^{\frac{d}{m}}\Big)\, dx \, d\xi \right]^{\frac{m}{d}}\\
			&= d^{-\frac{m}{d}} (2\pi)^{-m}
			\left[ \int_{\mathbb{S}^{d-1} } \int_{\mathbb{R}^d}
			\tau\Big(|\sigma(T)_{-m}(x,\xi)f(x)\, \mathrm{p}_\lambda|^{\frac{d}{m}}\Big)\, dx \, d\xi \right]^{\frac{m}{d}}.
		\end{split}
	\end{equation}
	
	We now show that $T_\lambda$ approximates $T M_f$ in $\mathcal{L}_{\frac{d}{m}, \infty}$. Write
	$$
	T M_f = T_\lambda + R_\lambda, \qquad R_\lambda := T M_{f \mathrm{q}_\lambda}.
	$$
	We claim that $\|R_\lambda\|_{\mathcal{L}_{\frac{d}{m}, \infty}} \to 0$ as $\lambda \to \infty$. To prove this, we apply Lemmas \ref{Cwikel} and \ref{Cwikel-0-2}: for any $f \mathrm{q}_\lambda \in C_c(\mathbb{R}^d; L_\frac{d}{m}(\mathcal{M}) \cap \mathcal{M})$,
	$$
	\|R_\lambda\|_{\mathcal{L}_{\frac{d}{m}, \infty}}
	\lesssim \|J^{-m} M_{f \mathrm{q}_\lambda}\|_{\mathcal{L}_{\frac{d}{m}, \infty}}
	\lesssim \|f \mathrm{q}_\lambda\|_{L_{\frac{d}{m}}(\mathcal{N})},
	$$
	for $  \frac{d}{m} >2$.
	The norm on the right-hand side is interpreted as $\ell_p(L_2)(\mathbb{R}^d; \mathcal{M})$ when $0 < \frac{d}{m} < 2$, and as $\ell_{2,\log}(L_\infty)(\mathbb{R}^d; \mathcal{M})$ when $\frac{d}{m} = 2$. Since $f \in C_c (\mathbb{R}^d; L_\frac{d}{m}(\mathcal{M}) \cap \mathcal{M})$, for each fixed $x \in \mathbb{R}^d$, we have $f(x) \in L_\frac{d}{m}(\mathcal{M}) \cap \mathcal{M}$. By \eqref{lem:finite-trace-approx},
	$$
	\|f(x)   \mathrm{q}_\lambda\|_{L_{\frac{d}{m}}(\mathcal{M})} \longrightarrow 0 \quad \text{as } \lambda \to \infty.
	$$
	Integrating over $x \in \mathbb{R}^d$, we conclude
	$$
	\|f \mathrm{q}_\lambda\|_{L_{\frac{d}{m}}(\mathcal{N})} \longrightarrow 0,
	$$
	with the norm understood in the sense of $\ell_p(L_2)(\mathbb{R}^d; \mathcal{M})$ for $0 < \frac{d}{m} < 2$ and $\ell_{2,\log}(L_\infty)(\mathbb{R}^d; \mathcal{M})$ for $\frac{d}{m} = 2$. Therefore,
	$$
	\|R_\lambda\|_{\mathcal{L}_{\frac{d}{m}, \infty}} \longrightarrow 0 \quad \text{as } \lambda \to \infty.
	$$
	Consequently, by Lemma~\ref{ideal approximate},
	$$
	\lim_{t \to \infty} t^{\frac{m}{d}} \mu(t, T M_f) =  \lim_{\lambda \to \infty} \lim_{t \to \infty} t^{\frac{m}{d}} \mu(t, T_\lambda).
	$$
	Letting $\lambda \to \infty$ and using \eqref{eq:Tn-asymp}, we obtain
	\begin{equation}\label{eq:swap-limits}
		\lim_{t \to \infty} t^{\frac{m}{d}} \mu(t, T M_f) = \lim_{\lambda \to \infty} C_\lambda.
	\end{equation}
	
	Finally, for each $(x,\xi)$, we have
	$$
	\tau\Big(|\sigma(T)_{-m}(x,\xi)f(x) \, \mathrm{p}_\lambda|^{\frac{d}{m}}\Big) \nearrow \tau\Big(|\sigma(T)_{-m}(x,\xi)f(x)|^{\frac{d}{m}}\Big).
	$$
	The monotone convergence theorem then gives
	\begin{align*}
		&\int_{\mathbb{S}^{d-1}} \int_{\mathbb{R}^d}
		\tau\Bigl( \bigl| \sigma(T)_{-m}(x,\xi) f(x) \, \mathrm{p}_\lambda \bigr|^{\frac{d}{m}} \Bigr)
		\, dx \, d\xi \\
		&\qquad \nearrow
		\int_{\mathbb{S}^{d-1}} \int_{\mathbb{R}^d}
		\tau\Bigl( \bigl| \sigma(T)_{-m}(x,\xi) f(x) \bigr|^{\frac{d}{m}} \Bigr)
		\, dx \, d\xi .
	\end{align*}
	Combining this with \eqref{eq:Cn2} and \eqref{eq:swap-limits} completes the proof:
	$$
	\lim_{t \to \infty} t^{\frac{m}{d}} \mu(t, T M_f)
	= d^{-\frac{m}{d}} (2 \pi)^{-m}
	\left[ \int_{\mathbb{S}^{d-1} } \int_{\mathbb{R}^d} \tau\Big(|\sigma(T)_{-m}(x, \xi)f(x)|^{\frac{d}{m}}\Big)\, dx \, d\xi \right]^{\frac{m}{d}}.
	$$
\end{proof}

By following the $A_\varepsilon$-regularization strategy employed for $\mathcal{M}$-valued elliptic $\Psi$DOs in Theorem~\ref{asymp-nu-pn}, we can extend the singular value asymptotics to the positive and negative parts of a self-adjoint operator. This formulation aligns more closely with the original Birman--Solomyak approach, which focuses on the distribution of positive and negative eigenvalues rather than singular values.

\begin{corollary}
	Let $T \in \mathrm{C}\Psi^{-m}(\mathbb{R}^d;\mathcal{M})$ with $m>0$, and suppose there exists a function $\nu \in C_c(\mathbb{R}^d)$ and a $\tau$-finite projection $\mathrm{p} \in \mathcal{M}$ such that $T = T M_{\mathrm{p}\nu}$.
	Then $T \in \mathcal{L}_{\frac{d}{m},\infty}\big(B(L_2(\mathbb{R}^d)) \overline{\otimes} \mathcal{M}\big)$, and
	$$
	\lim_{t \to \infty} t^{\frac{m}{d}} \mu^\pm(t, T)
	= d^{-\frac{m}{d}} (2\pi)^{-m} \left[ \int_{\mathbb{S}^{d-1} } \int_{\mathbb{R}^d}
	\tau\Big( (\sigma(T)_{-m}(x,\xi)^\pm)^{\frac{d}{m}} \Big) \, dx \, d\xi \right]^{\frac{m}{d}},
	$$
	where $\mu^\pm(t,T)$ denote the positive and negative eigenvalue functions of $T$, respectively, and $\sigma(T)_{-m}^\pm$ denote the positive and negative parts of the principal symbol $\sigma(T)_{-m}$.
\end{corollary}

As a direct consequence of Theorem~\ref{asymp most}, this result extends to the broader class of $\mathcal{M}$-valued $\Psi$DOs that are compactly supported from the right.

\begin{corollary}
	Let $T \in \mathrm{C}\Psi^{-m}(\mathbb{R}^d;\mathcal{M})$ with $m>0$ be a self-adjoint operator that is compactly supported from the right.
	Then $T \in \mathcal{L}_{\frac{d}{m},\infty}\big(B(L_2(\mathbb{R}^d)) \overline{\otimes} \mathcal{M}\big)$, and
	$$
	\lim_{t \to \infty} t^{\frac{m}{d}} \mu^\pm(t, T)
	= d^{-\frac{m}{d}} (2\pi)^{-m} \left[ \int_{\mathbb{S}^{d-1} } \int_{\mathbb{R}^d}
	\tau\Big( (\sigma(T)_{-m}(x,\xi)^\pm)^{\frac{d}{m}} \Big) \, dx \, d\xi \right]^{\frac{m}{d}}.
	$$
\end{corollary}

\subsection{Trace formulas}
In this section we briefly explain why Theorem \ref{asymp most} implies a trace formula for operators of order $-d.$ For general details on traces we direct the reader to \cite{LSZ12} and more specifically to \cite{LevitinaUsachevI,LevitinaUsachevII,LevitinaUsachevIII}.

Denote $\mathfrak{B}= B(L_2(\mathbb{R}^d)) \overline{\otimes} \mathcal{M}$. Recall that a densely-defined linear operator $T$ is said to be $\Tr$-compact if for every $\varepsilon>0$ there exists a projection $P\in\mathfrak{B}$ with $\Tr(P)<\infty$ such that
\[
\|T(1-P)\|_{\mathfrak{B}} < \varepsilon.
\]
The singular value functional $\mu(T) = \{\mu(t,T)\}_{t\geq 0}$ of a $\Tr$-compact operator $T$ is defined as
\[
\mu(t,T) = \inf\{\|T(1-P)\|_{\infty}\;:\; P\text{ is a projection such that}\Tr(P)\leq t\}.
\]
Note that we might have $\mu(t,T)=\infty$ for some $t,$ however by definition $\mu(t,T)\to 0$ as $t\to\infty.$

The operator space $\mathcal{L}_{1,\infty}(\mathfrak{B},\Tr)$ is a bimodule of $\mathfrak{B}$ which be defined as the set of $\mathrm{Tr}$-measurable operators $T$ whose singular value function $\mu(t,T)$ obeys
\[
\|T\|_{\mathcal{L}_{1,\infty}(\mathfrak{B})} := \sup_{t>0} t\mu(t,T) < \infty.
\]
We also have
\[
\|T\|_{\mathcal{L}_{1,\infty}(\mathfrak{B})} = \sup_{t>0} t\mathrm{Tr}(\chi_{t,\infty}(|T|)) < \infty.
\]
A symmetric functional on $\mathcal{L}_{1,\infty}(\mathfrak{B})$ is a linear functional $\varphi$ with the property
\[
\mu(T)=\mu(S)\Rightarrow \varphi(T)=\varphi(S).
\]
and a symmetric functional is continuous if it is continuous in the $\mathcal{L}_{1,\infty}(\mathfrak{B})$ quasinorm, i.e.
\[
|\varphi(T)| \leq C\|T\|_{\mathcal{L}_{1,\infty}(\mathfrak{B})}.
\]
Following the terminology of \cite[Definition 4.11]{LevitinaUsachevI}, we say that a symmetric functional is \emph{supported at infinity} if $\varphi(X\chi_{(a,\infty)}(|X|))=0$ for all $a>0$ and $X \in \mathcal{L}_{1,\infty}(\mathfrak{B},\Tr).$

An example of a continuous symmetric functional supported at infinity is the Dixmier-type trace, defined on $0\leq T \in \mathcal{L}_{1,\infty}(\mathfrak{B},\Tr)$ by
\begin{equation}\label{dixmier_definition}
	\Tr_{\omega}(T) := \omega\left(\{\frac{1}{\log(N)}\sum_{n=0}^N \mu(n,T)\}\right)
\end{equation}
where $\omega\in \ell_{\infty}(\mathbb{N})^*$ is an extended limit, i.e. a positive functional for which $\omega(x)=\lim_{n\to\infty} x_n$ whenever the limit exists. Because of the possible divergence of $\mu(t,T)$ at $t=0,$ there is no universal convention for the definition of Dixmier traces on $\mathcal{L}_{1,\infty}(\mathfrak{B},\Tr).$ We follow \cite[Definition 4.1]{LevitinaUsachevII}.

As a direct consequence of Theorem~\ref{asymp most}, we obtain the Dixmier trace formula for operator-valued $\Psi$DOs of order $-d$.
In other words, the following result can be viewed as the operator-valued version of Connes' trace theorem,
generalizing the scalar-valued case when $\mathcal{M} = \mathbb{C}.$

\begin{corollary}
	Let $T \in \mathrm{C}\Psi^{-d}\big(\mathbb{R}^d; \mathcal{M}\big)$ be compactly supported from the right. Then $T \in \mathcal{L}_{1,\infty}(\mathfrak{B},\Tr)$ and for any symmetric functional $\varphi$ supported at infinity, the value of $\varphi(T)$ is independent of $\varphi$ up to normalization in the following sense: if $\varphi,\psi$ are two symmetric functionals supported at infinity, then
	\[
	\varphi(T)\psi(S) = \varphi(S)\psi(T)
	\]
	for all $T,S \in \mathrm{C}\Psi^{-d}\big(\mathbb{R}^d;\mathcal{M}\big)$ compactly supported from the right.
	
	If $\Tr_{\omega}$ is a Dixmier trace defined as in \eqref{dixmier_definition}, we have
	$$
	\Tr_{\omega}(T) = d^{-1}(2\pi)^{-d} \int_{\mathbb{S}^{d-1} } \int_{\mathbb{R}^d} \tau \big( \sigma(T)_{-d}(x,\xi) \big)\,dx\,d\xi.
	$$
	In particular, the value is independent of the choice of $\omega.$
\end{corollary}
\begin{proof}
	Given $0\leq T\in \mathrm{C}\Psi^{-d}\big(\mathbb{R}^d; \mathcal{M}\big)$ with the property that $\sigma_{-d}(x,\xi)\geq 0$ for all $(x,\xi) \in \mathbb{R}^d\times \mathbb{R}^d\setminus \{0\},$ define
	\[
	c(T) := d^{-1} (2 \pi)^{-d}\int_{\mathbb{S}^{d-1} } \int_{\mathbb{R}^d} \tau\Big(\sigma(T)_{-d}(x, \xi)\Big)\, dx \, d\xi.
	\]
	By Lemma \ref{asymp most}, if $c(T)=0,$ then
	\[
	\lim_{t\to\infty} t\mu(t,T) = 0
	\]
	and hence $\varphi(T)=0$ for all continuous symmetric functionals $\varphi$ supported at infinity. Let $0\leq S,T \in \mathrm{C}\Psi^{-d}\big(\mathbb{R}^d;\mathcal{M}\big)$ with non-negative principal symbol, we have
	\[
	c(c(T)S-c(S)T) = 0
	\]
	and hence
	\[
	c(T)\varphi(S)=c(S)\varphi(T)
	\]
	for all continuous symmetric functionals $\varphi$ supported at infinity. Since for any other continuous symmetric functional $\psi$ supported at infinity we have
	\[
	c(T)\psi(S)=c(S)\psi(T)
	\]
	we deduce that
	\[
	c(T)c(S)\varphi(S)\psi(T) = c(S)c(T)\varphi(T)\psi(S).
	\]
	If either of $c(T),c(S)=0,$ then $\varphi(S)\varphi(T)=0$ so we can exclude this case. It follows that
	\[
	\varphi(T)\psi(S) = \varphi(S)\psi(T).
	\]
	Next, let $T$ be selfadjoint with selfadjoint principal symbol. By the decay of $\sigma(T)_{-d}(x,\xi)$, we may find $C>0 $ large enough so that $ |\sigma(T)_{-d}(x,\xi)|\leq C(1+|\xi|^2) ^{-\frac d 2}$; since $T=T M_f$ is compactly supported, we may also require $C J^{-d} \geq |T|$. Thus,
	$$\widetilde{T} :=M_{f^*} (C J^{-d} -T)M_{f} =C M_{f^*}  J^{-d} M_{f} -T $$ is nonnegative with non-negative principal symbol. Hence by the preceding paragraph,
	\[
	\varphi(\widetilde{T})\psi(S) = \varphi(S)\psi(\widetilde{T}).
	\]
	Since $\varphi$ and $\psi$ are linear, we deduce the result for self-adjoint $T.$ Similarly, for non-self-adjoint $T,$ we write $T$ as a linear combination of its real and imaginary parts. Reasoning the same for $S,$ we deduce that
	\[
	\varphi(T)\psi(S) = \varphi(S)\psi(T)
	\]
	for all $T,S \in \mathrm{C}\Psi^{-d}\big(\mathbb{R}^d;\mathcal{M}\big).$
	
	Next we verify the Dixmier trace formula. Once again we consider the case $T \ge 0$ and $\sigma(T)_{-d}(x,\xi)\geq 0$ pointwise. Applying Theorem \ref{asymp most} to $ T$ with $m =  d$, we have
	$$
	\lim_{t \to \infty} t\,\mu(t,T) =: A,
	$$
	where, according to Theorem \ref{asymp most},
	$$
	A = (2\pi)^{-d} d^{-1} \int_{\mathbb{S}^{d-1} } \int_{\mathbb{R}^d} \tau \big(  \sigma(T)_{-d}(x,\xi)  \big)\,dx\,d\xi.
	$$
	Thus, for large $t$, $\mu(t,T) = A t^{-1} + o(t^{-1})$. Integrating yields
	$$
	\sum_{n=0}^N \mu(n,T)\,dt = A \log N + o(\log N), \quad N \to \infty.
	$$
	Therefore, for any extended limit functional $\omega$,
	$$
	\Tr_\omega(T) = A.
	$$
	This establishes the claimed formula for $T \ge 0$ with $\sigma(T)_{-d}(x,\xi)\geq 0$ pointwise.
	
	Splitting $T$ into a sum of four positive operators with positive principal symbol as above, we deduce the result for general $T.$
	%
	%
\end{proof}

\section{Weyl's law for $\mathcal{M}$-valued commutators}\label{op-valued-commutator}

In this section, we apply general Weyl's law to commutators of the form $[T_\phi, M_f]$, where $\phi$ is a homogeneous function on $\mathbb{R}^d$ (or $\mathbb{S}^{d-1}$), and $f$ is an operator-valued function. We will also give Weyl's law for the commutator $[I^\alpha,M_f]$, where $I^\alpha=(- \Delta)^{\frac{\alpha}{2}}$ denotes the Riesz potential.


Referring to the definition provided in \cite[Definition 2.1]{FSZ24} for the scalar-valued case, we formally define $\nabla$ as the self-adjoint gradient operator on $L_2(L_\infty(\mathbb{R}^d)\overline{\otimes} \mathcal{M})$, explicitly given by
$$
\nabla = \Big(\frac{1}{i} \frac{\partial}{\partial x_1}, \cdots, \frac{1}{i} \frac{\partial}{\partial x_d}\Big) = \big(D_{1}, \cdots, D_{d}\big).
$$

From the perspective of functional calculus, we may write the Fourier multiplier $T_\phi$ with symbol $\phi: \mathbb{R}^d \rightarrow \mathbb{C}$ as
\begin{equation*}
	T_\phi  = \phi (\nabla).
\end{equation*}
In the same spirit, for the tuple $ \nabla (-\Delta)^{-\frac{1}{2}}  =\big\{R_{ k} \big\}_{k=1}^d=\big\{D_{k} (-\Delta)^{-\frac{1}{2}}\big\}_{k=1}^d$ whose spectrum lies in $\mathbb{S}^{d-1}$, we can apply the functional calculus by $\phi \in C^\infty(\mathbb{S}^{d-1}) $ and get a Fourier multiplier $T_\phi= \phi( \nabla (-\Delta)^{-\frac{1}{2}} )$ whose symbol is a homogeneous function on $\mathbb{R}^d$ of order $0$.

For a real number $\alpha$, the Riesz potential $I^\alpha=(- \Delta)^{\frac{\alpha}{2}}$  and Bessel potential $J^\alpha=(1- \Delta)^{\frac{\alpha}{2}}$ are the operators  defined on $\mathscr{S}^{\prime}(\mathbb{R}^d ; L_1(\mathcal{M})+\mathcal{M})$, where $\Delta$ denotes the Laplacian on $\mathbb{R}^d$. Their symbols are the Fourier multipliers  $I_\alpha(\xi)=|\xi|^{\alpha}$ and $J_\alpha(\xi)=(1+|\xi|^2)^{\frac{\alpha}{2}}$.

We now introduce the definition of operator-valued Sobolev spaces:

\begin{definition}
	For $1 \le p \le \infty$ and an integer $k \ge 1$, the  Sobolev space $W_p^k(\mathbb{R}^d; L_p(\mathcal{M}))$
	is defined as the set of all $\mathcal{M}$-valued tempered distributions $f \in \mathscr{S}'(\mathbb{R}^d; L_1(\mathcal{M}) + \mathcal{M})$ such that all distributional derivatives $D^\beta f$ with $|\beta| \le k$ belong to $L_p(\mathbb{R}^d; L_p(\mathcal{M}))$.
	This space is equipped with the norm
	$$
	\|f\|_{W_p^k(\mathbb{R}^d; L_p(\mathcal{M}))}
	:=
	\begin{cases}
		\Bigg( \sum\limits_{|\beta| \le k} \|D^\beta f\|_{L_p(\mathbb{R}^d; L_p(\mathcal{M}))}^p \Bigg)^{\frac{1}{p}}, & 1 \le p < \infty, \\[2mm]
		\max\limits_{|\beta| \le k} \|D^\beta f\|_{L_\infty(\mathbb{R}^d)\overline{\otimes} \mathcal{M}}, & p = \infty.
	\end{cases}
	$$
\end{definition}

Let us give the definition of homogeneous  Sobolev spaces that will be used in this section.


\begin{definition}
	The homogeneous Sobolev space $\dot{W}_p^1(\mathbb{R}^d; L_p(\mathcal{M})), 1 \le p \le \infty$, is defined as the set of all functions $f \in \mathscr{S}^{\prime}(\mathbb{R}^d ; L_1(\mathcal{M})+\mathcal{M})$
	such that their distributional gradient lies in $L_p(\mathcal{N})$. This is a space of $\mathcal{M}$-valued functions modulo $1_{\mathbb{R}^d} \otimes m \, (m \in \mathcal{M})$, equipped with the norm
	$$
	\|f\|_{\dot{W}_p^1(\mathbb{R}^d; L_p(\mathcal{M}))} :=\Big(\sum_{j=1}^{d} \big\|D_k f\big\|_{L_p(\mathcal{N})}^{p}\Big)^{\frac{1}{p}}.
	$$
\end{definition}

\begin{rk}
	In the scalar-valued case, it follows from \cite[Theorem 11.43]{Leo17} that $C_c^\infty(\mathbb{R}^d)$ is dense in $\dot{W}_p^1(\mathbb{R}^d)$ when either $d \ge 2$ or $p > 1$. Since $L_p(\mathcal{M}) \cap \mathcal{M}$ is dense in $L_p(\mathcal{M})$ for $p>0$, this result extends to the vector-valued function space. Specifically, $C_c^{\infty}(\mathbb{R}^d;L_p(\mathcal{M})\cap\mathcal{M})$ is dense in $\dot{W}_p^1(\mathbb{R}^d; L_p(\mathcal{M}))$ when either $d \ge 2$ or $p > 1$.
\end{rk}

\subsection{ Commutator estimates}\label{Commutator estimates}

Before stating and proving our main results on Weyl's law for $\mathcal{M}$-valued commutators, we provide some a priori estimates of the commutators in this section, by extending the commutator estimates in \cite{LMSZ17,FSZ24} to operator-valued setting. Comparing to the main results therein, one feature of the commutator estimates in the current operator-valued setting is that we have to consider the operator-valued Schatten norm $\mathcal{L}_{p, \infty}\big(B(L_2(\mathbb{R}^d))\overline{\otimes} \mathcal{M}\big)$ of the commutators.

Based on the Lemma \ref{Cwikel}, the commutator estimates for Riesz transforms in \cite{LMSZ17, MSX19, MSX20} remain valid in our setting with a similar proof.

\begin{lemma}\label{Riesz upper}
	Let  $f \in \dot{W}_d^1(\mathbb{R}^d; L_d(\mathcal{M}))$ with $d \ge 2$. Then for $1\le  k \le d$, the commutator $[R_k, M_f]\in \mathcal{L}_{d, \infty}\big(B(L_2(\mathbb{R}^d))\overline{\otimes} \mathcal{M}\big)$ with
	$$
	\|[R_k, M_f]\|_{\mathcal{L}_{d, \infty}} \le C_{  d}\|f\|_{\dot{W}_{d}^1(\mathbb{R}^d; L_{d}(\mathcal{M}))}.
	$$
\end{lemma}
\begin{proof}
	Let us sketch the idea of the proof following closely the argument in \cite{LMSZ17}, and point out the main difference from that in \cite{LMSZ17}. We need to estimate the commutator
	$$\big[  \sum_{k=1}^d  \gamma_k \otimes  R_k , 1\otimes M_f\big]  =  \big[ {\rm sgn}(\mathcal{D}), 1\otimes M_f \big]$$
	as an operator affiliated to $\mathbb{C}^N\otimes  B(L_2(\mathbb{R}^d))\overline{\otimes} \mathcal{M}  $, where $\gamma_k$'s are Pauli matrices on $\mathbb{C}^N$ and $\mathcal{D}= \sum_{k=1}^d  \gamma_k \otimes  D_{x_k} $.
	
	Write
	$$  \big[ {\rm sgn}(\mathcal{D}), 1\otimes M_f \big]   = \big[ {\rm sgn}(\mathcal{D})-g(\mathcal{D}), 1\otimes M_f \big]   + \big[ g(\mathcal{D}), 1\otimes M_f \big]  $$
	with $g(t) =  t/(1+t^2)^{1/2}$.
	Similar to \cite[Lemma~7]{LMSZ17}, the first term is estimated as
	$$ \|\big[ {\rm sgn}(\mathcal{D})-g(\mathcal{D}), 1\otimes M_f \big]\|_{\mathcal{L}_{d, \infty}}\le C_{d } \|f\|_{ L_d (\mathbb{R}^d; L_d(\mathcal{M}))}.$$
	The term $ \big[ g(\mathcal{D}), 1\otimes M_f \big]  $ is treated as in \cite[Section~3]{LMSZ17}:  write
	$$\big[ g(\mathcal{D}), 1\otimes M_f \big]  =  T_{g^{[1]}}^{\mathcal{D},\mathcal{D}} \big[  \mathcal{D}, 1\otimes M_f \big] $$
	with $g^{[1]}(s, t ) =  \frac{g(s)-g(t)}{s-t}$.
	Here, for general self-adjoint (potentially unbounded) operators $\mathcal{E}_1$ and $\mathcal{E}_2$, and a two-variable funtion $\phi$, the double operator integral $ T_{\phi }^{\mathcal{E}_1,\mathcal{E}_2} $ is defined (see \cite{PSW02, PS09}) as
	$$T_{\phi }^{\mathcal{E}_1,\mathcal{E}_2} (A) =  \int_{\mathbb{R}^2} \phi(s, t )  \  d\mu_1(s) \ A \ d\mu_2(t) $$
	with $d\mu_j$ ($j=1,2$) the associated spectral measures of $\mathcal{E}_1$ and $\mathcal{E}_2$. The transformer $ T_{g^{[1]}}^{\mathcal{D},\mathcal{D}} $ is decomposed as
	$$ T_{g^{[1]}}^{\mathcal{D},\mathcal{D}}  =   T_{\psi_1}^{\mathcal{D},\mathcal{D}} \circ  T_{\psi_2}^{\mathcal{D},\mathcal{D}}  \circ T_{\psi_3}^{\mathcal{D},\mathcal{D}} $$
	with
	\begin{align*}
		\psi_1(s,t) &= 1 + \frac{1 - s t}{(1+s^2)^{\frac{1}{2}} (1+t^2)^{\frac{1}{2}}}, \\[1ex]
		\psi_2(s,t) &= \frac{(1+s^2)^{\frac{1}{4}} (1+t^2)^{\frac{1}{4}}}{(1+s^2)^{\frac{1}{2}} + (1+t^2)^{\frac{1}{2}}}, \\[1ex]
		\psi_3(s,t) &= \frac{1}{(1+s^2)^{\frac{1}{4}} (1+t^2)^{\frac{1}{4}}}.
	\end{align*}
	
	Unlike in \cite{LMSZ17}, we require the complete boundedness of $ T_{\psi_j}^{\mathcal{D},\mathcal{D}} $'s in our operator-valued setting. Note that, operators in $B(L_2(\mathbb{R}^d))$ (or $B(L_2(\mathbb{R}^d))\overline{\otimes} \mathcal{M}  $) can be represented as infinite matrices, and transformers $ T_{\psi_j}^{\mathcal{D},\mathcal{D}} $ can then be viewed as Schur multipliers acting on these matrices. For $j=1,3$, the functions $\psi_j$ can be written as a linear combination of terms of the form $a(s) b(t)$ for some bounded functions $a, b$ on $\mathbb{R}$, so $ T_{\psi_j}^{\mathcal{D},\mathcal{D}} $ are completely bounded Schur multipliers (see \cite[Theorem~5.1]{Pisier95}, or \cite[Chapter~8]{Pisier98}). For $\psi_2(s,t )$, we refer to \cite[Lemma~8]{LMSZ17}, which transfers the operator $T_{\psi_2}^{\mathcal{D},\mathcal{D}} $ to $ T_{\psi_0}^{\mathcal{B},\mathcal{B}} $ for a certain operator $\mathcal{B}$ and the function $\psi_0(s,t ) = (2\cosh (s-t))^{-1} =  h(s-t)$. Since $\widehat{h} \in \mathscr{S}(\mathbb{R})\subset L_1(\mathbb{R})$, $h$ is a completely bounded Fourier multiplier on $L_p(\mathbb{R})$ for any $1\le p \le \infty$, so by the transference between Fourier and Schur multipliers (see \cite{NR11} for discrete case and \cite{CS15} for locally compact group case), $\psi_0$ is a completely bounded Schur multiplier on $\mathcal{S}_p $.
	
	With the complete boundedness of $T_{\psi_j}^{\mathcal{D},\mathcal{D}} $, we are able to conclude the proof of the lemma. Noting that
	\begin{equation*}
		\begin{split}
			T_{\psi_3}^{\mathcal{D},\mathcal{D}} \big[  \mathcal{D}, 1\otimes M_f \big] &=  (1+ \mathcal{D}^2)^{-\frac 1 4 }\big[  \mathcal{D}, 1\otimes M_f \big]   (1+  \mathcal{D}^2)^{-\frac 1 4 }\\
			&= \sum_{k=1}^d \Big(1\otimes   (1 - \Delta)^{-\frac 1 4 }\Big)\big[ \gamma_k \otimes D_{x_k}, 1\otimes M_f \big]  \Big(1\otimes   (1 - \Delta)^{-\frac 1 4 }\Big)\\
			&= \sum_{k=1}^d \Big(1\otimes   (1 - \Delta)^{-\frac 1 4 }\Big) \big(  \gamma_k \otimes M_{D_{x_k}f } \big)  \Big(1\otimes   (1 - \Delta)^{-\frac 1 4 }\Big),
		\end{split}
	\end{equation*}
	it follows from \eqref{Cwikel-wp} that
	$$\| T_{\psi_3}^{\mathcal{D},\mathcal{D}} \big[  \mathcal{D}, 1\otimes M_f \big] \|_{\mathcal{L}_{d, \infty}} \le C_{  d}\|f\|_{\dot{W}_{d}^1(\mathbb{R}^d; L_{d}(\mathcal{M}))},$$
	and hence
	$$\| \big[ g(\mathcal{D}), 1\otimes M_f \big] \|_{\mathcal{L}_{d, \infty}}  = \| T_{\psi_1}^{\mathcal{D},\mathcal{D}}\circ   T_{\psi_2}^{\mathcal{D},\mathcal{D}} \circ  T_{\psi_3}^{\mathcal{D},\mathcal{D}} \big[  \mathcal{D}, 1\otimes M_f \big] \|_{\mathcal{L}_{d, \infty}} \le C_{  d}\|f\|_{\dot{W}_{d}^1(\mathbb{R}^d; L_{d}(\mathcal{M}))}.$$
	Combining with the estimate for $\big[ {\rm sgn}(\mathcal{D})-g(\mathcal{D}), 1\otimes M_f \big]$, we get
	$$
	\|[R_k, M_f]\|_{\mathcal{L}_{d, \infty}} \le C_{  d}\big(\|f\|_{ L_d (\mathbb{R}^d; L_d(\mathcal{M}))}+ \|f\|_{\dot{W}_{d}^1(\mathbb{R}^d; L_{d}(\mathcal{M}))} \big).
	$$
	Finally, a dilation argument as in the proof of \cite[Theorem~11]{LMSZ17} can get rid of the $L_d$-norm of $f$, completing the proof of the lemma.
\end{proof}

The result in \cite[Proposition 4.5.2]{Zeng23} remains valid in our setting with an identical proof. Hence, we state the analogue of \cite[Proposition 4.5.2]{Zeng23} in our setting as the following lemma.

\begin{lemma}\label{calderon upper}
	Let  $\phi \in C^\infty(\mathbb{S}^{d-1})$, $f \in \dot{W}_d^1(\mathbb{R}^d; L_d(\mathcal{M}))$, and $d \ge 2$. Then the commutator $[T_\phi, M_f]\in \mathcal{L}_{d, \infty}\big(B(L_2(\mathbb{R}^d))\overline{\otimes} \mathcal{M}\big)$ and
	$$
	\|[T_\phi, M_f]\|_{\mathcal{L}_{d, \infty}} \le C_{\phi, d}\|f\|_{\dot{W}_{d}^1(\mathbb{R}^d; L_{d}(\mathcal{M}))}.
	$$
\end{lemma}
\begin{proof}
	The proof for \cite[Lemma~3.5]{SXZ23} holds true for commutators in $\mathcal{L}_{d, \infty}\big(B(L_2(\mathbb{R}^d))\overline{\otimes} \mathcal{M}\big)$ without any changes. So the proof for \cite[Proposition 4.5.2]{Zeng23} remains valid in our setting.
\end{proof}

In the scalar-valued case, it follows from \cite[Proposition 4.9]{FSZ24} that when $d \ge 2$ and $\alpha \in (-\frac{d}{2}, 1)$ (alternatively, let $d=1$ and $\alpha \in (0,1)$), the operator $[I^{\alpha}, M_f]: C_c^{\infty}(\mathbb{R}^d) \to (C_c^{\infty})'(\mathbb{R}^d)$ has a unique bounded extension in $B(L_2(\mathbb{R}^d))$. Therefore, in our setting, we adopt the same conditions as in the scalar-valued case.

The results in \cite[Theorem 1.1 (i), Theorem 1.2 (i)]{FSZ24} and their proofs hold verbatim in our setting. Hence, we state the counterparts of \cite[Theorem 1.1 (i), Theorem 1.2 (i)]{FSZ24} in our setting as the following lemma.

\begin{lemma}\label{endpoint upper}
	Let $d \ge 2$ and let $\alpha \in (-\frac{d}{2}, 0) \cup (0,1)$ (alternatively, let $d=1$ and $\alpha \in(0,1)$ ).
	If $f \in \dot{W}_{\frac{d}{1-\alpha}}^1\Big(\mathbb{R}^d; L_{\frac{d}{1-\alpha}}(\mathcal{M})\Big)$, then $[I^{\alpha}, M_f] \in \mathcal{L}_{\frac{d}{1-\alpha}, \infty}$ and
	$$
	\big\|[I^{\alpha}, M_f]\big\|_{\frac{d}{1-\alpha}, \infty} \le c_{d, \alpha}\|f\|_{\dot{W}_{\frac{d}{1-\alpha}}^1\Big(\mathbb{R}^d; L_{\frac{d}{1-\alpha}}(\mathcal{M})\Big)} .
	$$
\end{lemma}

\begin{proof}
	As Lemma \ref{Riesz upper}, the proof of this lemma is similar to that of \cite[Theorem 1.1 (i), Theorem 1.2 (i)]{FSZ24}, the only difference being that the involved double operator integrals (or Schur multipliers) are completely bounded. Let us give a simplified proof below.
	
	The starting point is still the Cwikel-type estimate: Let $ \beta,\gamma \in(0,\frac{d}{4})$ and $p=\frac{d}{2(\beta+\gamma)}$, then for every $f\in L_p(  (\mathbb{R}^d; L_p(\mathcal{M})))$, we have
	\begin{equation}
		\label{Cwikel-unsymm}
		\Big\|(-\Delta)^{-\beta}M_f(-\Delta)^{-\gamma}\Big\|_{\mathcal{L}_{\frac{d}{2(\beta+\gamma)},\infty}}\le C_{d,\beta,\gamma}\|f\|_{L_p(  (\mathbb{R}^d; L_p(\mathcal{M})))}.
	\end{equation}
	To see this, we just need to use polar decomposition to find a partial isometry $u\in L_{\infty}(\mathbb{R}^d)\overline{\otimes} \mathcal{M} $ such that
	$$f=|f^{\ast}|^{\frac{\beta }{\beta +\gamma}}\ u\ |f|^{\frac{\gamma}{\beta +\gamma}},$$
	so that
	$$(-\Delta)^{-\beta }M_f(-\Delta)^{-\gamma}=(-\Delta)^{-\beta }M_{|f^{\ast}|^{\frac{\beta }{\beta +\gamma}}}\cdot M_u\cdot M_{|f|^{\frac{\gamma}{\beta +\gamma}}}(-\Delta)^{-\gamma}.$$
	Applying \eqref{Cwikel-wp} and the noncommutative H\"older inequality, we conclude \eqref{Cwikel-unsymm}. Next, let $-\frac{d}{2}< \alpha <1$ and let $f\in L_{\frac{d}{1-\alpha }}(  (\mathbb{R}^d; L_{\frac{d}{1-\alpha }}(\mathcal{M})))$. Denote
	\begin{equation}
		\label{Fa}
		F_{\alpha }=|\mathcal{D} |^{\frac{\alpha -1}{2}}[\mathcal{D} ,1\otimes M_f]|\mathcal{D} |^{\frac{\alpha -1}{2}},
	\end{equation}
	and
	\begin{equation}
		\label{Ga}
		G_\alpha =\sum_{l=1}^{\lceil\frac{1-\alpha }{2}\rceil-1}|\mathcal{D} |^{\alpha +l-1}[\mathcal{D} ,1\otimes M_f]|\mathcal{D} |^{-l}+|\mathcal{D} |^{-l}[\mathcal{D} ,1\otimes M_f]|\mathcal{D} |^{\alpha +l-1}.
	\end{equation}
	Then we deduce from \eqref{Cwikel-unsymm} that
	\begin{equation}
		\label{estimateFG}
		\big\|F_\alpha \big\|_{\mathcal{L}_{\frac{d}{1-\alpha },\infty}}, \; \big\|G_\alpha \big\|_{\mathcal{L}_{\frac{d}{1-\alpha },\infty}} \le C_{\alpha , d } \|f\|_{ L_{\frac{d}{1-\alpha }}(  (\mathbb{R}^d; L_{\frac{d}{1-\alpha }}(\mathcal{M})))}.
	\end{equation}

	The next step is to link $1\otimes \big[ I^\alpha , M_f \big] =\big[ | \mathcal{D} |^{\alpha },1\otimes M_f\big]$ with the above constructed $F_\alpha$ and $G_\alpha$ via double operator integrals (or Schur multipliers). Set
	\begin{equation*}
		\Psi_\epsilon(s ,t )=\frac{|s |^{\epsilon}-|t |^{\epsilon}}{s -t  }\cdot |s |^{ \frac{1-\epsilon}{2}}|t |^{\frac{1-\epsilon}{2}},\quad  s ,t \in\mathbb{R}
	\end{equation*}
	when $s  \neq t $, and $\Psi_\epsilon(s ,t )=\epsilon$ when $s  = t $.
	Denote $\psi(s) =|s|$ and set (with the convention $\psi^{[1]}(s , s )=0 $)
	$$\psi^{[1]}(s , t )  =  \frac{|s  |- |t |}{s -t }.$$
	Then the relation between $\big[ | \mathcal{D} |^{\alpha },1\otimes M_f\big]$ and $F_\alpha,G_\alpha$ is
	\begin{equation}\label{link-commutator-FG}
		\big[|\mathcal{D} |^{\alpha},1\otimes M_f\big]=T^{\mathcal{D},\mathcal{D} }_{\Psi_{\delta}}(F_\alpha)-T^{\mathcal{D} ,\mathcal{D} }_{\psi^{[1]}}(G_\alpha),
	\end{equation}
	where $ \delta=\alpha-2+2\lceil\frac{1-\alpha}{2}\rceil \in [-1,1)$. The reason for \eqref{link-commutator-FG} is the following easily verified arithmetic equation:
	$$|s |^{\alpha }-|t |^{\alpha }=\Psi_{\delta}(s ,t )\cdot X_{\alpha }(s ,t )-\psi^{[1]}(s ,t )\cdot Y_{\alpha }(s ,t ),$$
	where
	$$X_{\alpha }(s ,t )=|s |^{\frac{\alpha -1}{2}}|t |^{\frac{\alpha -1}{2}}(s -t ),\quad s ,t \in\mathbb{R},$$
	$$Y_{\alpha }(s ,t )=\Big(\sum_{l=1}^{\lceil\frac{1-\alpha }{2}\rceil-1}|s |^{\alpha +l-1}|t |^{-l}+|s |^{-l}|t |^{\alpha +l-1}\Big)(s -t ),\quad s ,t \in\mathbb{R}.$$

	By virtue of \eqref{estimateFG} and \eqref{link-commutator-FG}, we are reduced to showing the complete boundedness of the two transformers $T^{\mathcal{D},\mathcal{D} }_{\Psi_{\delta}} $ and $T^{\mathcal{D} ,\mathcal{D} }_{\psi^{[1]}}$. For $T^{\mathcal{D} ,\mathcal{D} }_{\psi^{[1]}}$, it is clear by \cite[Section~3.2]{CGPT23} that $\psi^{[1]}(s, t)$ is completely bounded Schur multiplier on $\mathcal{S}_{p,\infty}$ for $1<p<\infty$. For $T^{\mathcal{D},\mathcal{D} }_{\Psi_{\delta}} $ with $\delta  \in [-1,1)$, we may set	
	$$\Phi_\delta  ( s ,  t  )=\frac{ s ^{\delta  }- t ^{ \delta  }}{ s - t }\cdot   s ^{\frac{1-\delta  }{2}}  t ^{\frac{1-\delta  }{2}},\quad  s , t >0$$
	and write
	$$\Psi_{\delta  }( s , t )=\Phi_{\delta  }(| s |,| t |)\cdot\psi^{[1]}( s , t ),\quad  s , t \in\mathbb{R}.$$
	It is immediate that
	$$T_{\Psi_\delta  }^{\mathcal{D},\mathcal{D} }=T_{\Phi_\delta  }^{| \mathcal{D}|,|\mathcal{D} |}\circ T_{\psi^{[1]}}^{\mathcal{D},\mathcal{D} }.$$
	The complete boundedness is proceeded as in \cite[Lemma~4.2]{FSZ24}, following \cite[Lemma~9]{PS09}:
	$$\Phi_\delta (s,t ) =  h_\delta (\frac s t )$$
	with $h_\delta \circ \exp:=g $ being a Schwartz function on $\mathbb{R}$.
	Then we may write
	$$\Phi_\delta (s,t ) =  \int_{\mathbb{R}}  \widehat{g}(\xi) s^{i\xi} t^{-i\xi} d\xi, $$
	forcing the complete boundedness of $T_{\Psi_\delta  }^{\mathcal{D},\mathcal{D} }$ on $\mathcal{S}_{p,\infty}$ for $1<p<\infty$.
\end{proof}

\subsection{ Weyl's law for the  commutator $[T_\phi , M_f]$ }

In the scalar-valued case, Calder\'{o}n--Zygmund  singular integral operators are characterized by kernels that have a mean value of zero on the unit sphere $\mathbb{S}^{d-1}$ and are homogeneous of degree $-d$. It follows from \cite[Theorem II.4.2]{St70} that each Calder\'{o}n-Zygmund  singular integral operator can be expressed as a Fourier multiplier $T_\phi$, where the symbol $ \phi$ is smooth and homogeneous of degree 0. In the sequel, we fix a function $\phi \in C^{\infty}(\mathbb{S}^{d-1} )$ which is homogeneous of degree 0 and non identically zero. We set $\phi(0)=0$, $\phi$ can be viewed as a function in $C^{\infty}(\mathbb{S}^{d-1})$.

The main result in this subsection is the following

\begin{theorem}\label{asymp limit cz commutator}
	Let $f \in \dot{W}_d^1(\mathbb{R}^d; L_d(\mathcal{M}))$ and $d\ge 2$. Then
	\begin{equation}\label{asy limit cz}
		\lim_{t \to \infty} t^{\frac{1}{d}} \mu(t, [T_\phi, M_f]) =(2\pi)^{-1} d^{-\frac{1}{d}}  \Bigg(\int_{\mathbb{S}^{d-1}} \int_{\mathbb{R}^d}  \tau \Big( \Big| \sum_{|\gamma|=1} \partial_s^\gamma \phi(s) D_x^\gamma f(x) \Big|^d \Big)\, dx \, ds \Bigg)^{\frac{1}{d}}.
	\end{equation}
	where the integral over $\mathbb{S}^{d-1}$ is taken with respect to the rotation-invariant measure $ds$ on $\mathbb{S}^{d-1}$, and $s = (s_1, \ldots, s_d)$.
\end{theorem}

The main idea for the proof of Theorem \ref{asymp limit cz commutator} is the application of Theorem \ref{asymp most}. However, since neither $\phi$ nor $f$ is globally smooth on $\mathbb{R}^d$, neither of them is viewed as $\Psi$DOs. So the first step of the proof is to replace $\phi\in C^{\infty}(\mathbb{S}^{d-1})$ and $f\in \dot{W}_d^1(\mathbb{R}^d; L_d(\mathcal{M}))$ with smooth elements.

To eliminate the singularity of $\phi\in C^{\infty}(\mathbb{S}^{d-1})$ at the origin, we replace $T_\phi$ with another Fourier multiplier $T_{\widetilde{\phi}}$ whose symbol $\widetilde{\phi}$ is smooth on the entire $\mathbb{R}^d$. Such a function $\widetilde{\phi}$ is constructed as the following: let $\eta \in C^{\infty}(\mathbb{R}^d)$ be such that
\begin{equation}\label{cutoff-eta}
	\eta(\xi) =
	\begin{cases}
		0, & |\xi| \le \frac{1}{4},\\[1mm]
		1, & |\xi| \ge \frac{1}{2},
	\end{cases}
	\quad 0 \le \eta(\xi) \le 1 \text{ for all } \xi \in \mathbb{R}^d,
\end{equation}
and take $\widetilde{\phi}=\eta \phi$.

The operator-valued function $f\in \dot{W}_d^1(\mathbb{R}^d; L_d(\mathcal{M}))$ is approximated by operator-valued functions in $  C_c^{\infty}(\mathbb{R}^d; L_d(\mathcal{M}) \cap \mathcal{M})$, such approximation works due to the a priori estimate Lemma \ref{calderon upper}.

However, given $\widetilde{\phi}\in C ^{\infty}(\mathbb{R}^d)  $ and $f \in C_c^{\infty}(\mathbb{R}^d; L_d(\mathcal{M}) \cap \mathcal{M})$, the commutator $[T_{\widetilde{\phi}}, M_f]$ is not  right-compactly supported. In order to apply Theorem \ref{asymp most}, we need to remove the right-compact support restriction.
The following lemma provides the required reduction, by showing that the Weyl law for $[T_\phi, M_f]$ follows from the corresponding result for right-compactly supported functions.

\begin{lemma}\label{right support for compact f}
	Let $\phi \in C ^{\infty}(\mathbb{R}^d) $, $f \in C_c^{\infty}(\mathbb{R}^d; L_d(\mathcal{M}) \cap \mathcal{M})$, and set $A := [T_\phi, M_f]$. Let $(\mathrm{p}_\lambda)_{\lambda \in \Lambda} \subset \mathcal{M}$ be an increasing sequence of $\tau$-finite projections such that $\tau(\mathrm{p}_\lambda) < \infty$ and $\mathrm{p}_\lambda \nearrow 1_{\mathcal{M}}$ in the strong operator topology. Then the following statements hold:
	
	\begin{enumerate}[$\rm (i)$]
		\item There exists a compact set $K \subset\subset \mathbb{R}^d$ with $\operatorname{supp} f \subset K$ and a cutoff function $\chi \in C_c^\infty(\mathbb{R}^d)$ such that $\chi \equiv 1$ on a neighborhood of $K$, for which the truncated operator
		$$
		A_{\chi,\mathrm{p}_\lambda} := A M_{\chi \mathrm{p}_\lambda}
		$$
		is compactly supported from the right, with right-compact support $\chi \mathrm{p}_\lambda \in C_c^\infty(\mathbb{R}^d; L_d(\mathcal{M}) \cap \mathcal{M})$, and
		$$
		A_{\chi,\mathrm{p}_\lambda} \to A \text{ in } \mathcal{L}_{d,\infty}.
		$$
		
		\item For every $\lambda$, the operator $A_{\chi,\mathrm{p}_\lambda}$ satisfies the Weyl's law
		$$
		\lim_{t \to \infty} t^{\frac{1}{d}} \mu(t, A_{\chi,\mathrm{p}_\lambda}) = C_\lambda ,
		$$
		and the limit for $A$ exists and is equal to the limit of the sequence $C_\lambda $, i.e.,
		$$
		\lim_{t \to \infty} t^{\frac{1}{d}} \mu(t, A) = \lim_{\lambda \to \infty} C_\lambda  = (2\pi)^{-1} d^{-\frac{1}{d}} \Bigg( \int_{\mathbb{S}^{d-1} } \int_{\mathbb{R}^d} \tau\big(|\sigma(A)_{-1}(x,\xi)|^d \big)\,dx\,d\xi \Bigg)^{\frac{1}{d}}.
		$$
	\end{enumerate}
\end{lemma}

\begin{proof}
	(i)
	Since $f \in C_c^\infty(\mathbb{R}^d; L_d(\mathcal{M}) \cap \mathcal{M})$, its $\mathbb{R}^d$-support
	$$
	K := \operatorname{supp}_x(f) =\overline{   \{x\in \mathbb{R}^d:  f(x) \neq 0 \} }
	$$
	is compact. Choose $\chi \in C_c^\infty(\mathbb{R}^d)$ such that $\chi \equiv 1$ on an open neighborhood of $K$, and denote the support of $\chi$ by $\Omega$.
	
	Consider the multiplication operator $ M_{\chi \mathrm{p}_\lambda} $, where $ \chi \in C_c^\infty(\mathbb{R}^d) $ and $ \mathrm{p}_\lambda \in L_d(\mathcal{M}) \cap \mathcal{M} $, so that $ \chi \mathrm{p}_\lambda \in C_c^\infty(\mathbb{R}^d; L_d(\mathcal{M}) \cap \mathcal{M}) $. By construction, $ A_{\chi,\mathrm{p}_\lambda} = A M_{\chi \mathrm{p}_\lambda} $ is compactly supported from the right. According to Proposition~\ref{cl M}, the principal symbol of $ A_{\chi,\mathrm{p}_\lambda} $ is given by
	$$
	\sigma(A_{\chi,\mathrm{p}_\lambda})_{-1}(x,\xi) = \sigma(A)_{-1}(x,\xi)\,\chi(x)\,\mathrm{p}_\lambda = \sigma(A)_{-1}(x,\xi)\,\mathrm{p}_\lambda,
	$$
	since $ \chi \equiv 1 $ on $ \operatorname{supp}_x \sigma(A)_{-1} $. Moreover,
	$$
	A - A_{\chi,\mathrm{p}_\lambda} = A M_{1-\chi} + A M_{\chi(1-\mathrm{p}_\lambda)}.
	$$
	
	Since $ \chi \equiv 1 $ on $ \operatorname{supp}_x(f) $, it follows that $ f(x)(1 - \chi(x)) = 0 $ for all $ x \in \mathbb{R}^d $. Consequently, the principal symbol of $ A M_{1-\chi} =-M_{f}T_\phi M_{1-\chi}$ vanishes:
	$$
	\sigma(A M_{1-\chi})_{-1}(x,\xi) = \sigma(A)_{-1}(x,\xi) \cdot (1 - \chi(x)) = 0.
	$$
	Furthermore, all higher-order homogeneous components of the symbol vanish outside $ \operatorname{supp}_x(f) $, which implies
	$$
	A M_{1-\chi} \in \mathrm{C}\Psi^{-\infty}\big(\mathbb{R}^d; \mathcal{M}\big)\subset B(L_2(\mathbb{R}^d)) \overline{\otimes} \mathcal{M}.
	$$
	
	Lemmas \ref{Cwikel} and \ref{Cwikel-0-2} imply that
	$$
	J^{-4d}M_{f}T_\phi M_{1-\chi} \in \mathcal{L}_{\frac{1}{4},\infty}\big(B(L_2(\mathbb{R}^d)) \overline{\otimes} \mathcal{M}\big).
	$$
	Since
	$$
	A M_{1-\chi}J^{4d} \in \mathrm{C}\Psi^{-\infty}\big(\mathbb{R}^d; \mathcal{M}\big)\subset B(L_2(\mathbb{R}^d)) \overline{\otimes} \mathcal{M},
	$$
	Then,
	\begin{align*}
		|A M_{1-\chi}|
		&= \bigl( A M_{1-\chi} J^{4d} J^{-4d} M_{f} T_\phi M_{1-\chi} \bigr)^{\frac{1}{2}} \\
		&\in \mathcal{L}_{\frac{1}{2},\infty}\bigl( B(L_2(\mathbb{R}^d)) \overline{\otimes} \mathcal{M} \bigr) \\
		&\subset (\mathcal{L}_{d,\infty})_0 \bigl( B(L_2(\mathbb{R}^d)) \overline{\otimes} \mathcal{M} \bigr).
	\end{align*}
	Thus, by Lemma \ref{weak weyl},
	$$
	\lim_{t \to \infty} t^{\frac{1}{d}} \mu(t, A - A_{\chi,\mathrm{p}_\lambda}) = \lim_{t \to \infty} t^{\frac{1}{d}} \mu(t, A M_{\chi(1-\mathrm{p}_\lambda)}).
	$$
	Since
	$$
	A M_{\chi(1-\mathrm{p}_\lambda)} = [T_\phi, M_f] M_{\chi(1-\mathrm{p}_\lambda)} = [T_\phi, M_{f(1-\mathrm{p}_\lambda)}] M_{\chi},
	$$
	an application of Lemma \ref{calderon upper} yields
	$$
	\|A M_{\chi(1-\mathrm{p}_\lambda)}\|_{\mathcal{L}_{d, \infty}} \le C_{\phi, d} \|f(1-\mathrm{p}_\lambda)\|_{\dot{W}_{d}^1(\mathbb{R}^d; L_{d}(\mathcal{M}))}.
	$$
	Finally, using the fact that $ \mathrm{p}_\lambda \nearrow 1_{\mathcal{M}} $ in the strong operator topology and that $ f \in C_c^\infty(\mathbb{R}^d; L_1(\mathcal{M}) \cap \mathcal{M}) $, we conclude that
	$$
	A_{\chi,\mathrm{p}_\lambda} \to A \quad \text{in } \mathcal{L}_{d,\infty}.
	$$
	
	(ii) For each $\lambda$, define
	$$
	A_\lambda := A_{\chi,\mathrm{p}_\lambda} = A M_{\chi \mathrm{p}_\lambda}.
	$$
	By direct application of Theorem \ref{asymp most} to these right-compactly supported operators, each $A_n$ satisfies the Weyl's law
	$$
	\lim_{t \to \infty} t^{\frac{1}{d}} \mu(t, A_\lambda) = C_\lambda ,
	$$
	where, by Theorem \ref{asymp most},
	$$
	C_\lambda  = (2\pi)^{-1} d^{-\frac{1}{d}} \Bigg( \int_{\mathbb{S}^{d-1} } \int_{\mathbb{R}^d} \tau\big(|\sigma(A)_{-1}(x,\xi)|^{d} \mathrm{p}_\lambda\big)\,dx\,d\xi \Bigg)^{\frac{1}{d}}.
	$$
	Since $A_{\chi,\mathrm{p}_\lambda} \to A$ in $\mathcal{L}_{d,\infty}$, Lemma \ref{ideal approximate} gives
	$$
	\lim_{t \to \infty} t^{\frac{1}{d}} \mu(t, A) = \lim_{\lambda \to \infty} C_\lambda .
	$$
	Apply the monotone convergence theorem along the increasing sequence $\mathrm{p}_\lambda \nearrow 1_{\mathcal{M}}$: for the nonnegative integrand $\tau\big(|\sigma(A)_{-1}|^{d} \mathrm{p}_\lambda\big)$, we have
	$$
	\int_{\mathbb{S}^{d-1} } \int_{\mathbb{R}^d} \tau\big(|\sigma(A)_{-1}(x,\xi)|^{d} \mathrm{p}_\lambda\big)\,dx\,d\xi
	\nearrow
	\int_{\mathbb{S}^{d-1} } \int_{\mathbb{R}^d} \tau\big(|\sigma(A)_{-1}(x,\xi)|^{d}\big)\,dx\,d\xi,
	$$
	by normality of $\tau$ and Fubini's theorem. Hence $C_\lambda  \nearrow C$ with
	$$
	C = (2\pi)^{-1} d^{-\frac{1}{d}} \Bigg( \int_{\mathbb{S}^{d-1} } \int_{\mathbb{R}^d} \tau\big(|\sigma(A)_{-1}(x,\xi)|^{d}\big)\,dx\,d\xi \Bigg)^{\frac{1}{d}}.
	$$
	It follows that
	$$
	\lim_{t \to \infty} t^{\frac{1}{d}} \mu(t, A) = C,
	$$
	which is the desired Weyl's law for $A = [T_\phi, M_f]$.
\end{proof}


We can now proceed to complete the proof of Theorem \ref{asymp limit cz commutator}.

\begin{proof}[Proof of Theorem \ref{asymp limit cz commutator}]
	Assume firstly $f \in C_c^\infty(\mathbb{R}^d; L_d(\mathcal{M}) \cap \mathcal{M})$. For the two $\Psi$DOs $T_{\widetilde{\phi}}$ and $M_f$, both of order $0$, we apply Theorem~\ref{asympcomthm} to obtain the asymptotic expansion of the principal symbol of $[T_{\widetilde{\phi}}, M_f]$. The symbol of $M_f T_{\widetilde{\phi}}$ is simply $\widetilde{\phi}(\xi) f(x)$, while the symbol of $T_{\widetilde{\phi}} M_f$ has the asymptotic expansion
	$$
	\sigma(T_{\widetilde{\phi}} M_f)(x,\xi)\sim \sum_{j \ge 0} \sum_{|\gamma|=j} \frac{1}{\gamma!} \partial_{\xi}^\gamma \widetilde{\phi}(\xi) D^\gamma f(x).
	$$
	By the homogeneity of $\phi$ and the construction of $\widetilde{\phi}$, each term $\partial_{\xi}^\gamma \widetilde{\phi}(\xi) D^\gamma f(x)$ coincides with the homogeneous symbol $\partial_{\xi}^\gamma \phi(\xi) D^\gamma f(x)$ of order $-|\gamma|$ when $|\xi| \ge 1$. Therefore, the commutator $[T_{\widetilde{\phi}}, M_f] \in \mathrm{C}\Psi^{-1}\big(\mathbb{R}^d; \mathcal{M}\big)$, with principal symbol given by
	\begin{equation}\label{tilde-comm-prin}
		\sigma([T_{\widetilde{\phi}}, M_f])_{-1}(x,\xi)= \sum_{k=1}^d \partial_{\xi_k} \widetilde{\phi}(\xi) D_k f(x).
	\end{equation}
	Then, by Lemma \ref{right support for compact f}, we obtain
	$$
	\lim_{t \to \infty} t^{\frac{1}{d}} \mu(t, [T_{\widetilde{\phi}}, M_f]) = d^{-\frac{1}{d}} (2 \pi)^{-1} \left[ \int_{\mathbb{S}^{d-1} } \int_{\mathbb{R}^d} \tau\Big(\big|\sum_{k=1}^d \partial_{\xi_k} \phi(\xi) D_k f(x)\big|^{d}\Big)\, dx \, d\xi \right]^{\frac{1}{d}}.
	$$
	
	By construction, $\phi - \widetilde{\phi} \in L_p(\mathbb{R}^d)$ for all $p \ge 1$. Therefore, by Lemma~\ref{Cwikel} and the triangle inequality,
	$$
	[T_{\phi}, M_f] - [T_{\widetilde{\phi}}, M_f] = T_{\phi-\widetilde{\phi}}M_f - M_fT_{\phi-\widetilde{\phi}} \in \mathcal{L}_{d} \subset (\mathcal{L}_{d,\infty})_{0}.
	$$
	Hence, it follows from Lemma \ref{weak weyl} that
	$$
	\lim_{t \to \infty} t^{\frac{1}{d}} \mu(t, [T_{\phi}, M_f]) = d^{-\frac{1}{d}} (2 \pi)^{-1} \left[ \int_{\mathbb{S}^{d-1} } \int_{\mathbb{R}^d} \tau\Big(\big|\sum_{k=1}^d \partial_{\xi_k} \phi(\xi) D_k f(x)\big|^{d}\Big)\, dx \, d\xi \right]^{\frac{1}{d}}.
	$$
	
	So far, we have established \eqref{asy limit cz} for $f \in C_c^\infty(\mathbb{R}^d; L_d(\mathcal{M}) \cap \mathcal{M})$. Since $C_c^\infty(\mathbb{R}^d; L_d(\mathcal{M}) \cap \mathcal{M})$ is norm dense in $ \dot{W}_d^1(\mathbb{R}^d; L_d(\mathcal{M}))$, the result extends to $f \in \dot{W}_d^1(\mathbb{R}^d; L_d(\mathcal{M}))$ by Lemmas \ref{ideal approximate} and \ref{calderon upper}, which completes the proof.
\end{proof}

The $j$-th Riesz transform $R_j = D_{x_j} (-\Delta)^{-\frac{1}{2}}$ is defined on $\mathscr{S}'(\mathbb{R}^d; L_1(\mathcal{M}) + \mathcal{M})$, where $\Delta$ denotes the Laplacian on $\mathbb{R}^d$. If we choose $T_\phi = R_j$, then $[R_j, M_f]$ is an $\mathcal{M}$-valued classical  $\Psi$DO of order $-1$, with principal symbol $\sum_{|\gamma|=1} \partial_\xi^\gamma \frac{\xi_j}{|\xi|} D_x^\gamma f(x)$. Therefore, by Theorem \ref{asymp limit cz commutator}, we obtain the following corollary.

\begin{corollary}
	Let $f\in \dot{W}_d^1(\mathbb{R}^d; L_d(\mathcal{M}))$ and $d\ge 2$. Then
	\begin{equation}\label{asy limit riesz}
		\begin{split}
			&\lim_{t \to \infty} t^{\frac{1}{d}} \mu(t, [R_j, M_f])\\
			=& (2\pi)^{-1}d^{-\frac{1}{d}}  \Bigg(  \int_{\mathbb{S}^{d-1}} \int_{\mathbb{R}^d}\tau \Big( \Big| D_j f(x) - \sum_{k=1}^{d} s_j s_k D_k f(x) \Big|^d \Big)\, dx \, ds \Bigg)^{\frac{1}{d}},
		\end{split}
	\end{equation}
	where the integral over $\mathbb{S}^{d-1}$ is taken with respect to the rotation-invariant measure $ds$ on $\mathbb{S}^{d-1}$, and $s = (s_1, \ldots, s_d)$.
\end{corollary}

\begin{rk}
	Simply using the triangular inequality, we see that
	$$\Bigg(  \int_{\mathbb{S}^{d-1}} \int_{\mathbb{R}^d}\tau \Big( \Big| D_j f(x) - \sum_{k=1}^{d} s_j s_k D_k f(x) \Big|^d \Big)\, dx \, ds \Bigg)^{\frac{1}{d}}\leq C_d \|f\|_{\dot{W}_d^1}.$$
	It is also proved in the appendix of \cite{SXZ23} that
	$$\sum_{1\leq j \leq d}\Bigg(  \int_{\mathbb{S}^{d-1}} \int_{\mathbb{R}^d}\tau \Big( \Big| D_j f(x) - \sum_{k=1}^{d} s_j s_k D_k f(x) \Big|^d \Big)\, dx \, ds \Bigg)^{\frac{1}{d}}\approx   \|f\|_{\dot{W}_d^1}$$
	with relevant constants depending only on $d$.
	However, the reverse inequality does not seem to hold for a single $j$.
\end{rk}

\subsection{ Weyl's law for the  commutator $[I^{\alpha} , M_f]$ }

\begin{theorem}
	Let $d \ge 2$ and let $\alpha \in (-\frac{d}{2}, 0) \cup (0,1)$ (alternatively, let $d=1$ and $\alpha \in(0,1)$ ). If $f \in \dot{W}_{\frac{d}{1-\alpha}}^1\Big(\mathbb{R}^d; L_{\frac{d}{1-\alpha}}(\mathcal{M})\Big) $, then
	\begin{equation}\label{asymp-riesz-poten}
		\lim_{t \to \infty} t^{\frac{1-\alpha}{d}} \mu(t, [I^{\alpha}, M_f]) = C_{d,\alpha} \Bigg( \int_{\mathbb{S}^{d-1}}\int_{\mathbb{R}^d}  \tau\Big(\left|s \cdot \nabla f(x)\right|^{\frac{d}{1-\alpha}}\Big)\, dx \, ds \Bigg)^{\frac{1-\alpha}{d}}
	\end{equation}
	with
	$$
	C_{d,\alpha}=|\alpha| d^{\frac{\alpha-1}{d}} (2\pi)^{\alpha-1}.
	$$
\end{theorem}

\begin{proof}
	As in Theorem \ref{asymp limit cz commutator}, to eliminate the singularity, we replace $I^\alpha$ with $J^\alpha$, and assume firstly $f \in C_c^{\infty}(\mathbb{R}^d; L_{\frac{d}{1-\alpha}}(\mathcal{M}) \cap \mathcal{M})$. By Theorem~\ref{asympcomthm}, the symbol of the product $M_f J^\alpha$ is given by
	$$
	\sigma(M_f J^\alpha)(x,\xi) = (1+|\xi|^2)^{\frac{\alpha}{2}} f(x),
	$$
	while the symbol of $J^\alpha M_f$ has the asymptotic expansion
	$$
	\sigma(J^\alpha M_f)(x,\xi)
	\sim (1+|\xi|^2)^{\frac{\alpha}{2}} f(x)
	+ \sum_{j\ge1}\sum_{|\gamma|=j} \frac{1}{\gamma!}\,
	\partial_\xi^\gamma\!\big((1+|\xi|^2)^{\frac{\alpha}{2}}\big)\,
	D^\gamma f(x).
	$$
	Hence,
	$$
	\sigma([J^\alpha,M_f])(x,\xi)
	\sim
	\sum_{|\gamma|\ge1} \frac{1}{\gamma!}\,
	\partial_\xi^\gamma\!\big((1+|\xi|^2)^{\frac{\alpha}{2}}\big)\,
	D^\gamma f(x).
	$$
	
	To obtain the homogeneous expansion, we use the standard binomial expansion
	$$
	(1+|\xi|^2)^{\frac{\alpha}{2}}
	= |\xi|^{\alpha}\Big(1+|\xi|^{-2}\Big)^{\frac{\alpha}{2}}
	\sim
	|\xi|^{\alpha}\sum_{m=0}^\infty
	\binom{\frac{\alpha}{2}}{m}|\xi|^{-2m},
	\qquad |\xi|\to\infty.
	$$
	Differentiating termwise yields
	$$
	\partial_\xi^\gamma\!\big((1+|\xi|^2)^{\frac{\alpha}{2}}\big)
	\sim
	\sum_{m=0}^\infty
	\binom{\frac{\alpha}{2}}{m}\,
	\partial_\xi^\gamma\!\big(|\xi|^{\alpha-2m}\big),
	\qquad |\xi|\to\infty.
	$$
	Substituting this into the previous expression gives the homogeneous asymptotic expansion
	$$
	\sigma([J^\alpha, M_f])(x,\xi)
	\sim
	\sum_{|\gamma|\ge 1} \frac{1}{\gamma!}
	\sum_{m=0}^\infty
	\binom{\frac{\alpha}{2}}{m}\,
	\partial_\xi^\gamma\!\big(|\xi|^{\alpha-2m}\big)\,
	D^\gamma f(x).
	$$
	Each term
	$$
	\partial_\xi^\gamma(|\xi|^{\alpha-2m})\,D^\gamma f(x)
	$$
	is homogeneous of degree
	$$
	\alpha - |\gamma| - 2m
	$$
	in $\xi$. Therefore $[J^\alpha,M_f]$ is a classical $\Psi$DO of order $\alpha-1$.
	In particular, the principal symbol  corresponds to $|\gamma|=1, m=0$:
	$$
	\sigma([J^\alpha,M_f])_{\alpha-1}(x,\xi)
	= \sum_{k=1}^d \partial_{\xi_k}(|\xi|^{\alpha})\,\partial_{x_k}f(x)
	= \alpha\,|\xi|^{\alpha-2}\,\xi\cdot\nabla_x f(x).
	$$
	
	The proof of Lemma \ref{right support for compact f} easily extends to the situation where $\phi: \mathbb{R}^d \rightarrow \mathbb{C}$ is smooth but of order $\alpha$. In particular, for $\phi(\xi )  = (1+|\xi|^2)^{\frac{\alpha}{2}}$ and $f \in C_c^{\infty}(\mathbb{R}^d;L_{\frac{d}{1-\alpha}}(\mathcal{M})\cap\mathcal{M})$,
	by  Lemma \ref{right support for compact f} for $m = 1 - \alpha$, we obtain
	\begin{align*}
		&\lim_{t \to \infty} t^{\frac{1-\alpha}{d}} \mu(t, [J^{\alpha}, M_f])\\
		=& |\alpha| d^{\frac{\alpha-1}{d}} (2\pi)^{\alpha-1} \Bigg( \int_{\mathbb{S}^{d-1}}\int_{\mathbb{R}^d}  \tau\Big(\left|s \cdot \nabla f(x)\right|^{\frac{d}{1-\alpha}}\Big)\, dx \, ds \Bigg)^{\frac{1-\alpha}{d}}.
	\end{align*}

	For $d \ge 2$ and $\alpha \in (-\frac{d}{2}, 0) \cup (0,1)$, or for $d=1$ and $\alpha \in (0,1)$, we have in either case that $\frac{d}{1-\alpha} > 1$.
	Since $f \in C_c^{\infty}(\mathbb{R}^d; L_p(\mathcal{M}) \cap \mathcal{M})$, it follows that $f \in L_p(\mathbb{R}^d; L_p(\mathcal{M}))$ for $2 \le p < \infty$, and $f \in \ell_p(L_2)(\mathbb{R}^d; \mathcal{M})$ for $1 < p < 2$.  Moreover, the function $\xi \mapsto (1+|\xi|^2)^{\frac{\alpha}{2}} - |\xi|^\alpha$ belongs to $L_{\frac{d}{1-\alpha}}(\mathbb{R}^d)$ when $2 \le\frac{d}{1-\alpha} <\infty$, and to $\ell_{\frac{d}{1-\alpha}}(L_2)(\mathbb{R}^d)$ when $1 < \frac{d}{1-\alpha} < 2$.
	Therefore, by Lemmas \ref{Cwikel} and \ref{Cwikel-0-2},
	$$
	\begin{gathered}
		[J^{\alpha}, M_f] - [I^{\alpha}, M_f] = \\
		(J^{\alpha} - I^{\alpha}) M_f - M_f(J^{\alpha} - I^{\alpha}) \in \mathcal{L}_{\frac{d}{1-\alpha}} \subset \Big(\mathcal{L}_{\frac{d}{1-\alpha}, \infty}\Big)_0.
	\end{gathered}
	$$
	Using Lemma \ref{weak weyl}, we obtain \eqref{asymp-riesz-poten}.
	
	Finally, applying Lemmas \ref{ideal approximate} and \ref{endpoint upper}, then \eqref{asymp-riesz-poten} also holds for $f \in \dot{W}_{\frac{d}{1-\alpha}}^1\Big(\mathbb{R}^d; L_{\frac{d}{1-\alpha}}(\mathcal{M})\Big)$.
\end{proof}

\begin{rk}
	If $f$ is scalar-valued, then it is pointed out in \cite{FSZ24} that, by rotation invariance,
	$$\Bigg( \int_{\mathbb{S}^{d-1}}\int_{\mathbb{R}^d}  \tau\Big(\left|s \cdot \nabla f(x)\right|^{\frac{d}{1-\alpha}}\Big)\, dx \, ds \Bigg)^{\frac{1-\alpha}{d}}= C_{d,\alpha} \|\nabla f(x)\|_{\frac{d}{1-\alpha}}$$
	for a certain positive constant depending only on $\alpha$ and $d$.
	In the operator-valued setting, it also follows from the triangular inequality that
	$$\Bigg( \int_{\mathbb{S}^{d-1}}\int_{\mathbb{R}^d}  \tau\Big(\left|s \cdot \nabla f(x)\right|^{\frac{d}{1-\alpha}}\Big)\, dx \, ds \Bigg)^{\frac{1-\alpha}{d}}\leq  C_{d,\alpha} \| f\|_{ \dot{W}_{\frac{d}{1-\alpha}}^1}.$$
	The reverse inequality is also true: for a fixed $f \in \dot{W}_{\frac{d}{1-\alpha}}^1\Big(\mathbb{R}^d; L_{\frac{d}{1-\alpha}}(\mathcal{M})\Big) $, select $1\leq j \leq d $ such that $\|D_j f\|_{\frac{d}{1-\alpha}} \geq\|D_k f\|_{\frac{d}{1-\alpha}}  $ for $k\neq j$, then select a small neighborhood $\Omega \subset \mathbb{S}^{d-1}$ of the $j$-th pole $(0,\cdots, 0, 1, 0,\cdots, 0)$ such that $|s_j |\geq d |s_k|$ for $k\neq j$. When $s\in \Omega$, we have
	$$\tau\Big(\left|s \cdot \nabla f(x)\right|^{\frac{d}{1-\alpha}}\Big) \geq \big(|s_j|\ \|D_j f\|_{\frac{d}{1-\alpha}}  \big)^{\frac{d}{1-\alpha}},$$
	whence
	$$ \| f\|_{ \dot{W}_{\frac{d}{1-\alpha}}^1}\leq  C_{d,\alpha}\Bigg( \int_{\Omega}\int_{\mathbb{R}^d}  \tau\Big(\left|s \cdot \nabla f(x)\right|^{\frac{d}{1-\alpha}}\Big)\, dx \, ds \Bigg)^{\frac{1-\alpha}{d}}.$$
	So we conclude
	$$\Bigg( \int_{\mathbb{S}^{d-1}}\int_{\mathbb{R}^d}  \tau\Big(\left|s \cdot \nabla f(x)\right|^{\frac{d}{1-\alpha}}\Big)\, dx \, ds \Bigg)^{\frac{1-\alpha}{d}}\approx \| f\|_{ \dot{W}_{\frac{d}{1-\alpha}}^1}$$
	with relevant constants depending only on $\alpha$ and $d$.
\end{rk}

\begin{rk}
	The restriction of $\alpha \in (-\frac{d}{2}, 0) \cup (0,1)$ is due to the singularity of $I^\alpha$: It is pointed out by \cite[Lemma~3.3]{FSZ24} that for $\alpha\in (-d, -\frac d 2 ]$, $[I^{\alpha}, M_f]$ does not even map $L_2(\mathbb{R}^d)$ to itself for some $f\in L_1(\mathbb{R}^d)$. For $J^\alpha$, we see from the above proof that
	$$
	\lim_{t \to \infty} t^{\frac{1-\alpha}{d}} \mu(t, [J^{\alpha}, M_f])
	= |\alpha| d^{\frac{\alpha-1}{d}} (2\pi)^{\alpha-1} \Bigg( \int_{\mathbb{S}^{d-1}}\int_{\mathbb{R}^d}  \tau\Big(\left|s \cdot \nabla f(x)\right|^{\frac{d}{1-\alpha}}\Big)\, dx \, ds \Bigg)^{\frac{1-\alpha}{d}}
	$$
	for any $\alpha \in (-\infty, 0) \cup (0,1)$.
\end{rk}

\section{Examples}

This section presents applications of the general theory to concrete noncommutative models. We consider crossed product constructions as well as the abelian case $\mathcal{M}=L_{\infty}(\Omega)$, leading to random and almost periodic operators. In these settings, the operator-valued Weyl's law yield explicit formulas for spectral quantities such as the density of states. These examples demonstrate the scope and applicability of the abstract results.

\subsection{Crossed products}
Throughout this section, let $\mathcal{M}\subset B(H)$ be a von Neumann algebra with normal semifinite trace $\tau$, and let
\[
\alpha:\mathbb{R}^d\to \mathrm{Aut}(\mathcal{M})
\]
denote an action of $\mathbb{R}^d$ on $\mathcal{M}$ by trace-preserving automorphisms. The von Neumann algebra crossed product $\mathcal{M}\rtimes_{\alpha} \mathbb{R}^d$ can be defined as the closure in the weak operator topology of the algebra generated by the operators
\[
\pi_{\alpha}(X)\xi(x) = \alpha(x)(X)\xi(x),\quad \xi\in L_2(\mathbb{R}^d;H),\, X \in \mathcal{M}
\]
and
\[
U(y)\xi(x) = \xi(x-y),\quad \xi \in L_2(\mathbb{R}^d;H).
\]
The dual weight $\widehat{\tau}$ on the crossed product is described in \cite{Takesaki}.

\begin{definition}\label{equivariant_definition}
	Let $T_yu(x) = u(x-y)$ be the operator of translation by $y \in \mathbb{R}^d$ on $u\in L_2(\mathbb{R}^d).$
	Define
	\[
	\widetilde{\alpha}:\mathbb{R}^d\to \mathrm{Aut}(B(L_2(\mathbb{R}^d))\otimes \mathcal{M}).
	\]
	by
	\[
	\widetilde{\alpha}(x)(A\otimes X) = (T_x^*AT_x)\otimes \alpha(x)(X),\quad A \in B(L_2(\mathbb{R}^d)),\; X \in \mathcal{M},\; x \in \mathbb{R}^d.
	\]
\end{definition}

\begin{lemma}\label{crossed_product_as_fixed_point}
	Identify $B(L_2(\mathbb{R}^d))\otimes \mathcal{M}$ with its image in $B(L_2(\mathbb{R}^d;H)).$
	\[
	\mathcal{M}\rtimes_{\alpha} \mathbb{R}^d = (B(L_2(\mathbb{R}^d))\otimes \mathcal{M})^{\widetilde{\alpha}}
	\]
	where the superscript $\widetilde{\alpha}$ means the fixed subalgebra.
\end{lemma}
\begin{proof}
	In \cite[Corollary X.1.22]{Takesaki} it is proved that the two algebras are isomorphic. Moreover, the specific isomorphism between the double crossed product $(\mathcal{M}\rtimes_{\alpha}\mathbb{R})\rtimes_{\widehat{\alpha}} \mathbb{R}$ and $B(L_2(\mathbb{R}^d)) \otimes \mathcal{M}$ given in \cite[Theorem X.2.3]{Takesaki} is the identity map due to the specific representation of the crossed product and the tensor product used here.
\end{proof}

We connect this to $\Psi$DOs via the following definition.
\begin{definition}\label{equivariant_symbol_definition}
	Let $\sigma\in S^{m}_{\rho,\delta}(\mathbb{R}^d\times \mathbb{R}^d;\mathcal{M}).$ We say that $\sigma$ is equivariant if
	\[
	\alpha(y)\sigma(x,\xi) = \sigma(x-y,\xi),\quad x,y,\xi\in \mathbb{R}^d.
	\]
\end{definition}

\begin{lemma}\label{equivariant_symbols_are_equivariant_operators}
	If $\sigma\in S^{0}_{\rho,\delta}(\mathbb{R}^d\times \mathbb{R}^d;\mathcal{M})$ is $\alpha$-equivariant where $0\leq \delta \leq \rho \leq 1,$ then $\mathrm{Op}(\sigma) \in \mathcal{M}\rtimes_{\alpha} \mathbb{R}^d.$
\end{lemma}
\begin{proof}
	By Lemma \ref{crossed_product_as_fixed_point}, it suffices to show that $\mathrm{Op}(\sigma)$ is fixed under $\widetilde{\alpha}.$ Denoting $T_y$ for translation by $y,$ we have
	\[
	T_y^*\mathrm{Op}(\sigma)T_y = \mathrm{Op}(\sigma_y)
	\]
	where $\sigma_y(x,\xi) = \sigma(x+y,\xi),$ Hence
	\[
	\widetilde{\alpha}(y)\mathrm{Op}(\sigma) = \mathrm{Op}(\alpha(y)\sigma_y).
	\]
	So $\mathrm{Op}(\sigma)$ is fixed under $\widetilde{\alpha}$ if and only if $\sigma = \alpha(y)(\sigma_y)$ for all $y,$ which is precisely that $\sigma$ is $\alpha$-equivariant according to Definition \ref{equivariant_symbol_definition}.
\end{proof}

\begin{corollary}
	Let $m>0$ and let $A \in \mathrm{C}\Psi^m(\mathbb{R}^d;\mathcal{M})$ have $\alpha$-equivariant symbol. Then $A$ is affiliated to the crossed product $\mathcal{M}\rtimes_{\alpha} \mathbb{R}^d.$
\end{corollary}
\begin{proof}
	It suffices that the resolvents of $A$ belong to $\mathcal{M}\rtimes_{\alpha}\mathbb{R}^d,$ and this is a consequence of Theorem \ref{inverse resolvent} and Lemma \ref{equivariant_symbols_are_equivariant_operators}.
\end{proof}

The dual trace $\widehat{\tau}$ on $\mathcal{M}\rtimes_{\alpha} \mathbb{R}^d$ is not the same as the restriction of the tensor product trace $\Tr = \mathrm{tr}\otimes \tau$ on $B(L_2(\mathbb{R}^d))\otimes \mathcal{M}.$ In the next subsection, we will be able to prove results about the spectral counting function of $\alpha$-invariant elliptic operators, as measured with the dual trace, using an ergodicity assumption.

\subsection{Random and almost periodic operators}\label{random_section}
In this section we specialize to $\mathcal{M} = L_\infty(\Omega),$ where $(\Omega,\Sigma,\mathbb{P})$ is a probability space. Write $\mathbb{E}(\cdot) = \int_{\Omega}\cdot d\mathbb{P}.$ Operators valued in $L_\infty(\Omega)$ can be thought of as random operators, and the tensor product trace $\mathrm{Tr}$ on $\mathcal{N} = L_{\infty}(\Omega)\otimes B(L_2(\mathbb{R}^d))$ is the expected value of the operator trace, i.e.
\[
\mathrm{Tr}(A) = \mathbb{E}(\mathrm{tr}(A)).
\]
We single out the important special case given by $\Omega = \mathbb{R}^d_B,$ the Bohr compactification of $\mathbb{R}^d,$ where $\mathbb{P}$ is the normalized Haar measure.

We will restrict attention to random operators that are ergodic under translation in the following sense, which is a special case of the notion of an ergodic family from \cite[Section V.2]{CarmonaLacroix}.
\begin{definition}\label{gamma_equivariant_definition}
	Let $G\subseteq \mathbb{R}^d$ be a closed subgroup\footnote{We only need to consider $G = \mathbb{R}^d$ and $G = \mathbb{Z}^d$.} A group homomorphism $\gamma:G\to \mathrm{Aut}(\Omega)$ is said to be measure preserving if $\mathbb{P}(\gamma(k)A) = \mathbb{P}(A)$ for all $A \in \Sigma$ and $k\in G.$ We say that $\gamma$ is ergodic if the only measurable subsets of $\Omega$ invariant under $\gamma$ have measure zero or one, i.e.
	\[
	\gamma(k)A=A\text{ for all }k\in G\Longleftrightarrow \mathbb{P}(A) \in \{0,1\}.
	\]
	Given $k\in G,$ let $T_k$ denote the unitary operator of translation by $k$ on $L_2(\mathbb{R}^d).$ That is,
	\[
	T_ku(x) = u(x-k),\quad x \in \mathbb{R}^d,\; u\in L_2(\mathbb{R}^d).
	\]
	Denote the pullback of $\gamma(k)$ by $\gamma(k)^*,$ which is defined on $L_{\infty}(\Omega)$ by
	\[
	\gamma(k)^*f(\omega) = f(\gamma(k)(\omega)),\quad k\in G,\; \omega\in \Omega.
	\]
	An operator $P \in B(L_2(\mathbb{R}^d))\otimes L_{\infty}(\Omega).$ will be called $\gamma$-\emph{equivariant} if for all $k\in G$ we have
	\[
	(1\otimes \gamma(k)^*)P = (T_k\otimes 1)P(T_{-k}\otimes 1).
	\]
	If $P$ is unbounded, the same definition applies provided that the domain of $P$ is invariant under the maps $1\otimes \gamma(k)$ and $T_k\otimes 1$ for all $k\in G.$
\end{definition}
Note that if $G = \mathbb{R}^d,$ then Lemma \ref{crossed_product_as_fixed_point} shows that the $G$-invariant operators are precisely the elements of the crossed product $L_{\infty}(\Omega)\rtimes_{\gamma} \mathbb{R}^d.$

An example with discrete $G$ is as follows.
\begin{example}
	Let $\Omega = \{-1,1\}^{\mathbb{Z}^d}$ be equipped with the product measure. The action
	\[
	\gamma(k)(\{\varepsilon_n\}_{n\in \mathbb{Z}^d}) = \{\varepsilon_{n-k}\}_{n\in \mathbb{Z}^d}
	\]
	is ergodic. Let $V \in C^\infty_c(\mathbb{R}^d).$ The function
	\[
	V(x,\varepsilon) := \sum_{n\in \mathbb{Z}^d} V(x+n)\varepsilon_n,\quad \varepsilon = \{\varepsilon_n\}_{n\in \mathbb{Z}^d} \in \Omega
	\]
	obeys
	\[
	V(x+k,\varepsilon) = V(x,\gamma(k)\varepsilon),\quad k\in \mathbb{Z}^d.
	\]
	Therefore the operator
	\[
	M_{V} \in B(L_2(\mathbb{R}^d))\otimes L_{\infty}(\Omega)
	\]
	of pointwise multiplication by $V$ is $\gamma$-equivariant.
\end{example}

\begin{lemma}
	Let $G\subseteq \mathbb{R}^d$ be a closed subgroup, and let $\gamma:G\to \mathrm{Aut}(\Omega)$ be group action. If $\sigma$ is a symbol function satisfying
	\begin{equation}\label{equivariant_symbol}
		\sigma(x,\xi,\omega) = \sigma(x-k,\xi,\gamma(k)\omega),\quad x,\xi \in \mathbb{R}^d,\; \text{ almost every }\omega\in \Omega
	\end{equation}
	for all $k\in G$ then $\mathrm{Op}(\sigma)$ is $\gamma$-equivariant in the sense of Definition \ref{gamma_equivariant_definition}.
	
	Conversely, if $P$ is a $\gamma$-equivariant $L_{\infty}(\Omega)$-valued $\Psi$DO, then its symbol $\sigma$ obeys \eqref{equivariant_symbol}.
\end{lemma}

In the event that $\Omega = \mathbb{R}^d_B$ is the Bohr compactification, $G=\mathbb{R}^d$ and $\gamma$ is the translation induced by the embedding of $\mathbb{R}^d$ into $\mathbb{R}^d_B,$ we will call a $\gamma$-equivariant operator \emph{almost periodic}. For further details on almost-periodic operators see \cite{Shubin1979} and \cite{PasturFigotin}.

\begin{example}
	Recall that the space $\mathrm{CAP}(\mathbb{R}^d)$ of continuous almost periodic functions on $\mathbb{R}^d$ as the closure of the linear span of the functions $x\mapsto \exp(i(x,\xi))$ in the uniform norm. We can identify $\mathrm{CAP}(\mathbb{R}^d)$ with the continuous functions on the Bohr compactification of $\mathbb{R}^d.$
	If $V \in \mathrm{CAP}(\mathbb{R}^d),$ define a function $\widetilde{V}$ on $\mathbb{R}^d\times \mathbb{R}^d_B$ by
	\[
	\widetilde{V}(x,g) = V(x+g),\quad x \in \mathbb{R}^d,\, g \in \mathbb{R}^d_B.
	\]
	The operator $M_V$ of pointwise multiplication by $V$ belongs to
	\[
	B(L_2(\mathbb{R}^d))\otimes L_{\infty}(\mathbb{R}^d_B)
	\]
	and is equivariant under the action of $\mathbb{R}^d$ on $\mathbb{R}^d_B$ induced by the inclusion. Since the $\mathbb{R}^d$ action on $\mathbb{R}^d_B$ is ergodic, this is another example of Definition \ref{gamma_equivariant_definition}.
\end{example}

We will use the following ergodic theorem:
\begin{theorem}\cite[Theorem 1.1, page 23]{Petersen1983}, \cite[Theorem 2.3, page 30]{Petersen1983}
	Let $f \in L_{1}(\Omega),$ and let $\gamma:\mathbb{Z}^d\to \mathrm{Aut}(\Omega)$ be ergodic. For almost every $\omega\in \Omega,$ we have
	\[
	\lim_{n\to\infty} \frac{1}{(2n+1)^d}\sum_{|k|_\infty\leq n} f(\gamma(k)\omega) = \mathbb{E}(f).
	\]
	If $f \in L_2(\Omega),$ then we have the same in the norm topology of $L_2(\Omega).$
	
	Similarly, if $\gamma:\mathbb{R}^d\to \mathrm{Aut}(\Omega)$ is ergodic, then
	\[
	\lim_{R\to\infty} \frac{1}{(2R)^d}\int_{[-R,R]^d} f(\gamma(x)\omega)\,dx = \mathbb{E}(f)
	\]
	in the same sense.
\end{theorem}

For symbols $\sigma \in S^m_{\rho,\delta}(\mathbb{R}^d\times \mathbb{R}^d;L_{\infty}(\Omega))$ we will normally write
\[
\sigma(x,\xi,\omega),\quad (x,\xi)\in \mathbb{R}^d\times \mathbb{R}^d,\; \omega\in \Omega
\]
for $\sigma(x,\xi)(\omega).$

\begin{remark}
	If $\sigma$ is the symbol of an almost-periodic operator, then we have
	\[
	\sigma(x,\xi,\omega) = \sigma(0,\xi,\omega+x),\quad x,\xi \in \mathbb{R}^d,\; \omega\in \mathbb{R}^d_B.
	\]
	and hence we may identify $\sigma$ with a function on $\mathbb{R}^d_B\times \mathbb{R}^d.$
\end{remark}

\begin{theorem}\label{existence_of_DOS}
	Let $\gamma:G \to \mathrm{Aut}(\Omega)$ be an ergodic action, where $G=\mathbb{R}^d$ or $G = \mathbb{Z}^d.$
	Let $m>0$ and let $P$ be non-negative and $\gamma$-equivariant. Assume that for almost-every $\omega\in \Omega$ and for all $\lambda>0$ we have
	\[
	M_{\chi_{[-1,1]^d}}\chi_{[0,\lambda]}(P(\omega)) \in \mathcal{S}_1(L_2(\mathbb{R}^d)).
	\]
	The \emph{density of states} of $P$, defined as the random function
	\[
	N(\lambda,P)(\omega) := \lim_{R\to\infty} (2R)^{-d}\mathrm{tr}(M_{\chi_{[-R,R]^d}}\chi_{[0,\lambda]}(P(\omega)))
	\]
	exists for almost every $\omega\in \Omega,$ and coincides with the constant function
	\[
	N(\lambda,P)(\omega) = \mathbb{E}(\mathrm{tr}(M_{\chi_{[0,1]^d}}\chi_{[0,\lambda]}(P))).
	\]
\end{theorem}
\begin{proof}
	Let
	\[
	N_R(\lambda,P)(\omega) = (2R)^{-d}\mathrm{tr}(M_{\chi_{[-R,R]^d}}\chi_{[0,\lambda]}(P(\omega))).
	\]
	We will show that $\lim_{R\to\infty}N_R(\lambda,P)(\omega)$ exists. Assume that $R$ ranges over integers (the general case is not different). By equivariance we have
	\[
	N_R(\lambda,P)(\omega) = \frac{1}{(2n+1)^d}\sum_{|k|\leq n} \mathrm{tr}(M_{\chi_{[0,1]^d}}\chi_{[0,\lambda]}(P(\omega))).
	\]
	By the pointwise ergodic theorem, we conclude that for almost every $\omega\in \Omega,$ we have
	\[
	\lim_{R\to\infty} N_R(\lambda,P)(\omega) = \Tr(M_{\chi_{[0,1]^d}}\chi_{[0,\lambda]}(P)).
	\]
\end{proof}

\begin{lemma}
	Let $\gamma:\mathbb{R}^d\to \mathrm{Aut}(\Omega)$ be an ergodic action. Denote $\tau$ for the integral on $L_{\infty}(\Omega).$ Let $A$ be an equivariant operator affiliated with $L_{\infty}(\Omega)\otimes B(L_2(\mathbb{R}^d).$ We have
	\[
	N(\lambda,A)(\omega) = \widehat{\tau}(\chi_{(-\infty,\lambda)}(A))
	\]
	where $\widehat{\tau}$ is the dual trace on the crossed product $L_{\infty}(\Omega)\rtimes_{\gamma} \mathbb{R}^d.$
\end{lemma}

\begin{theorem}
	Let $G\subseteq \mathbb{R}^d$ be a closed subgroup, and $\gamma:G\to \mathrm{Aut}(\Omega)$ be an ergodic action.
	Let $m>0$ and let $A \in \mathrm{C}\Psi^m(\mathbb{R}^d;L_{\infty}(\Omega))$ be non-negative, elliptic and $\gamma$-equivariant. If $G = \mathbb{Z}^d,$ then as $\lambda\to\infty$ we have
	\[
	N(\lambda,A) = \lambda^{\frac{d}{m}}\frac{m}{d(2\pi)^d}\int_{|\xi|=1}\int_{[0,1]^d} \mathbb{E}(\sigma(A)_{m}(x,\xi,\cdot)^{-\frac{d}{m}})\,dxd\xi+o(\lambda^{\frac{d}{m}}).
	\]
	If $G = \mathbb{R}^d,$ then
	\[
	N(\lambda,A) = \lambda^{\frac{d}{m}}\frac{m}{d(2\pi)^d}\int_{|\xi|=1}\mathbb{E}(\sigma(A)_{m}(0,\xi,\cdot)^{-\frac{d}{m}})\,d\xi+o(\lambda^{\frac{d}{m}}).
	\]
\end{theorem}
\begin{proof}
	Immediate combination of Theorems \ref{microlocal_weyl_law} and \ref{existence_of_DOS}.
\end{proof}
A similar result for differential operators first obtained by Gusev \cite[Theorem 5.2]{Gusev1977}. Compare \cite[Theorem 4.6]{Shubin1979}, which obtains the same result in the almost periodic case with a sharper remainder term.

As a special case of Theorem \ref{asymp most}, we also have the following:
\begin{theorem}
	Let $G\subseteq \mathbb{R}^d$ be a closed subgroup, and $\gamma:G\to \mathrm{Aut}(\Omega)$ be an ergodic action.
	Let $T \in \mathrm{C}\Psi^{-m}(\mathbb{R}^d;L_{\infty}(\Omega)),$ $m>0$ be compactly supported from the right and $\gamma$-equivariant. If $G = \mathbb{Z}^d,$ then
	\[
	\lim_{t\to\infty} t^{\frac{m}{d}}\mu(t,T)(\omega) = d^{-\frac{m}{d}}(2\pi)^{-m}\left[ \int_{|\xi|=1} \int_{\mathbb{R}^d} \mathbb{E}\Big(|\sigma(T)_{-m}(x, \xi,\cdot)|^{\frac{d}{m}}\Big)\, dx \, d\xi \right]^{\frac{m}{d}}.
	\]
\end{theorem}

\begin{rk}
	The preceding theorem implies that if $T\in \mathrm{C}\Psi^{-m}(\mathbb{R}^d,L_{\infty}(\mathbb{R}^d_B))$ is equivariant with respect to the translation action of $\mathbb{R}^d$ on $\mathbb{R}^d_B$ (i.e., $T$ is an almost periodic $\Psi$DO), and $\tau_{\omega}$ is a Dixmier trace
	on the von Neumann algebra $L_{\infty}(\mathbb{R}^d_B)\rtimes \mathbb{R}^d,$ then
	\[
	\tau_{\omega}(T) = \frac{1}{d(2\pi)^d}\int_{\mathbb{R}^d_B}\int_{|\xi|=1}\sigma(T)_{-m}(x,\xi)\, dx \, d\xi.
	\]
	The same result was stated in \cite[Theorem 9.1]{BCPRSW2005}, but the proof contains an oversight in that it refers to \cite{Shubin1976} for the eigenvalue asymptotics but the results there only apply to positive order elliptic operators.
\end{rk}



\begin{thebibliography}{99}
\bibitem{AS63} M.~Atiyah and I.~Singer.
\newblock The index of elliptic operators on compact manifolds.
\newblock {\it Bull. Amer. Math. Soc.}, 69 (1963) 422-433.

\bibitem{A69} M.~Atiyah.
\newblock Global theory of elliptic operators.
\newblock In: {\it Proceedings of the International Conference on Functional Analysis and Related Topics}, (Tokyo) (1969) pp. 21-30.

\bibitem{Baaj88} S.~Baaj.
\newblock Calcul pseudodiff\'{e}rentiel et produits crois\'{e}s de $C^*$-alg\`{e}bres, I and II.
\newblock {\it C. R. Acad. Sci. Paris, Ser. I}, 307 (1988) 581-586 and 663-666.

\bibitem{BCPRSW2005} M.-T. Benameur, A. L. Carey, J. Phillips, A. Rennie, F. Sukochev and K. Wojciechowski.
\newblock An analytic approach to spectral flow in von Neumann algebras.
\newblock {\it Analysis, geometry and topology of elliptic operators}, (2005) 297--352.

\bibitem{BT06} M.~Benameur and T.~Fack.
\newblock Type II non-commutative geometry. I: Dixmier trace in von Neumann algebras.
\newblock {\it Adv. Math.}, 199:1 (2006) 29-87.

\bibitem{BGS90} M.S.~Birman, G.E.~Karadzhov, and M.Z.~Solomyak.
\newblock Boundedness conditions and spectrum estimates for the operators $b(X)a(D)$ and their analogs.
\newblock In: {\it Estimates and Asymptotics for Discrete Spectra of Integral and Differential Equations}, Leningrad, 1989-90, Adv. Soviet Math. 7, Amer. Math. Soc., Providence, RI, 1991, pp. 85-106.


\bibitem{BS70} M.S.~Birman and M.Z.~Solomjak.
\newblock The principal term of the spectral asymptotics for ``nonsmooth'' elliptic problems.
\newblock {\it Funct. Anal. Appl.}, 4 (1970) 265-275.

\bibitem{BS77i} M.S.~Birman and M.Z.~Solomyak.
\newblock Estimates for the singular numbers of integral operators.
\newblock {\it Usp. Mat. Nauk}, 32 (193) (1977) 17-84, 1 (Russian), {\it Russ. Math. Surv.}, 32 (1977), no. 1, 15-89 (English).

\bibitem{BS77} M.S.~Birman and M.Z.~Solomyak.
\newblock Asymptotic behavior of the spectrum of pseudodifferential operators with anisotropically homogeneous symbols.
\newblock {\it Vestn. Leningr. Univ.}, 13 (3) (1977) 13-21 (Russian), {\it Vestn. Leningr. Univ., Math.}, 10 (1982) 237-247 (English).

\bibitem{BS79} M.S.~Birman and M.Z.~Solomyak.
\newblock Asymptotic behavior of the spectrum of pseudodifferential operators with anisotropically homogeneous symbols. II.
\newblock {\it Vestn. Leningr. Univ. Mat. Mekh. Astron.}, 13 (3) (1979) 5-10 (Russian).

\bibitem{BS79a} M.S.~Birman and M.Z.~Solomjak.
\newblock Asymptotic behavior of the spectrum of variational problems on solutions of elliptic equations.
\newblock {\it Sibirsk. Mat. Zh.}, 20 (1979) no. 1, 3-22, 204.

\bibitem{BS87} M.S.~Birman and M.Z.~Solomyak.
\newblock Spectral Theory of Selfadjoint Operators in Hilbert Space.
\newblock {\it Mathematics and Its Applications (Soviet Series)}, D. Reidel Publishing Co., Dordrecht, 1987.

\bibitem{B98} V.I.~Burenkov.
\newblock Sobolev Spaces on Domains.
\newblock {\it Teubner-Texte Math.}, 137, B. G. Teubner Verlagsgesellschaft, Stuttgart, 1998.

\bibitem{CPRS04} A. L. Carey, J. Phillips, A. Rennie, and F. A. Sukochev.
\newblock The Hochschild class of the Chern character for semifinite spectral triples.
\newblock {\it J. Funct. Anal.}, 213:1 (2004), 111-153.

\bibitem{CPRS06i} A. L. Carey, J. Phillips, A. Rennie, and F. A. Sukochev.
\newblock The local index formula in semifinite von Neumann algebras. I: Spectral flow.
\newblock {\it Adv. Math.}, 202:2 (2006), 451-516.

\bibitem{CPRS06ii} A. L. Carey, J. Phillips, A. Rennie, and F. A. Sukochev.
\newblock The local index formula in semifinite von Neumann algebras. II: The even case.
\newblock {\it Adv. Math.}, 202:2 (2006), 517-554.

\bibitem{CGRS14} A.L.~Carey, V.~Gayral, A.~Rennie, and F.~Sukochev.
\newblock Index theory for locally compact noncommutative geometries.
\newblock {\it Mem. Amer. Math. Soc.}, 231 (2014) no. 1085, vi+130.

\bibitem{CarmonaLacroix} R.~Carmona and J.~Lacroix.
\newblock Spectral Theory of Random Operators
\newblock {\it Probability and Its Applications} (1990)

\bibitem{CS15} M. Caspers and M. de la Salle.
\newblock Schur and Fourier multipliers of an amenable group acting on non-commutative $L^p$-spaces.
\newblock {\em Tran. Amer. Math. Soc.}, 367 (2015), no. 10, 6997-7013.


\bibitem{CGPT23} J. Conde-Alonso, A. Gonz\'alez-P\'erez, J. Parcet and E. Tablate.
\newblock Schur multipliers in Schatten-von Neumann classes.
\newblock {\em Ann. of Math. (2)}, 198 (2023), no. 3, 1229-1260.

\bibitem{AC80} A.~Connes.
\newblock $C^*$-alg\`{e}bres et g\'{e}om\'{e}trie differentielle.
\newblock {\it C. R. Acad. Sci. Paris, S\'{e}r. A}, 290 (1980) 599-604.


\bibitem{AC85} A.~Connes.
\newblock Noncommutative differential geometry.
\newblock {\it Inst. Hautes \'{E}tudes Sci. Publ. Math.}, 62 (1985) 257-360.

\bibitem{AC88} A.~Connes.
\newblock The action functional in non-commutative geometry.
\newblock {\it Commun. Math. Phys.}, 117 (1988) 673-683.

\bibitem{AC94} A.~Connes.
\newblock Noncommutative Geometry.
\newblock Academic Press, San Diego, 1994.

\bibitem{AC95} A.~Connes.
\newblock Noncommutative geometry and reality.
\newblock {\it J. Math. Phys.}, 36 (11) (1995) 6194-6231.

\bibitem{AC21} A.~Connes.
\newblock Noncommutative Geometry, the Spectral Standpoint.
\newblock In: New Spaces in Physics, vol. 2, Cambridge University Press, 2021, pp. 23-84.

\bibitem{CF19} A.~Connes and F.~Fathizadeh.
\newblock The term $a_4$ in the heat kernel expansion of noncommutative tori.
\newblock {\it M\"{u}nster J. Math.}, 12 (2019) 239-410.

\bibitem{CMSZ17} A.~Connes, E.~McDonald, F.~Sukochev and D.~Zanin.
\newblock Conformal trace theorem for Julia sets of quadratic polynomials.
\newblock {\it Ergod. Theory Dyn. Syst.}, 39 (2017) 1-26.


\bibitem{CM14} A.~Connes and H.~Moscovici.
\newblock Modular curvature for noncommutative two-tori.
\newblock {\it J. Am. Math. Soc.}, 27 (3) (2014) 639-684.

\bibitem{CSZ17} A.~Connes, F.~Sukochev and D.~Zanin.
\newblock Trace theorem for quasi-Fuchsian groups.
\newblock {\it Mat. Sb.}, 208, (2017) 59-90.


\bibitem{CST94} A. Connes, D. Sullivan and N. Teleman.
\newblock Quasiconformal mappings, operators on Hilbert space, and local formulae for characteristic classes.
\newblock {\it Topology}, 33 (1994), 663-681.

\bibitem{CT11} A.~Connes and P.~Tretkoff.
\newblock The Gauss-Bonnet theorem for the noncommutative two torus.
\newblock In: Noncommutative Geometry, Arithmetic, and Related Topics, Johns Hopkins Univ. Press, Baltimore, MD, 2011, pp. 141-158.

\bibitem{Co55} M.~Cotlar.
\newblock A combinatorial inequality and its applications to $L_2$-spaces.
\newblock {\it Rev. Mat. Cuyana}, 1 (1955) 41-55 (1956).

\bibitem{Cwikel77} M.~Cwikel.
\newblock Weak type estimates for singular values and the number of bound states of Schr\"odinger operators.
\newblock {\it Ann. of Math.}, 106 (1977) 93-100.



\bibitem{DR87} M.~Dauge and D.~Robert.
\newblock Weyl's formula for a class of pseudodifferential operators with negative order on $L^2(\mathbf{R}^n)$.
\newblock In: Pseudodifferential operators (Oberwolfach, 1986), {\it Lecture Notes in Math.}, 1256, Springer, Berlin, 1987, pp. 91-122.

\bibitem{D66} J.~Dixmier.
\newblock Existence de traces non normales.
\newblock {\it C. R. Acad. Sci. Paris S\'{e}r. A-B}, 262 (1966) A1107-A1108.

\bibitem{DPS23}P.~Dodds, B.de~Pagter, and F.Sukochev.
\newblock Noncommutative Integration and Operator Theory.
\newblock Progress in Mathematics, volume 349, Birkh\"{a}user, Cham, Switzerland, 2023.

\bibitem{FSZ23} R.~Frank, F.~Sukochev, and D.~Zanin.
\newblock Asymptotics of singular values for quantum derivatives.
\newblock {\it Trans. Amer. Math. Soc.}, 376 (3) (2023) 2047-2088.

\bibitem{FSZ24} R.~Frank, F.~Sukochev, and D.~Zanin.
\newblock Endpoint Schatten class properties of commutators.
\newblock {\it Adv. Math.}, 450 (2024) 109738.

\bibitem{GJM22} L.~Gao, M.~Junge, and E.~McDonald.
\newblock Quantum Euclidean spaces with noncommutative derivatives.
\newblock {\it J. Noncommut. Geom.}, 16 (2022) 153-213.

\bibitem{G04} V.~Gayral, J.~Gracia-Bond\'{i}a, B.~Iochum, T.~Sch\"{u}cker, and J.~V\'{a}rilly.
\newblock Moyal planes are spectral triples.
\newblock {\it Commun. Math. Phys.}, 246 (2004) 569-623.

\bibitem{GJP21} A.M.~Gonz\'{a}lez-P\'{e}rez, M.~Junge, and J.~Parcet.
\newblock Singular integral in quantum Euclidean spaces.
\newblock {\it Mem. Amer. Math. Soc.}, 272 (2021) no. 1334, vii+90 pp.

\bibitem{Gusev1977} A.~I.~Gusev.
\newblock Density of states and other spectral invariants of self-adjoint elliptic operators with random coefficients. \newblock {\it Mat. Sb.} 104, 1977, 207--226.

\bibitem{HLP19i} H.~Ha, G.~Lee, and R.~Ponge.
\newblock Pseudodifferential calculus on noncommutative tori, I. Oscillating integrals.
\newblock {\it Int. J. Math.}, 30 (8) (2019) 1950033.

\bibitem{HLP19ii} H.~Ha, G.~Lee, and R.~Ponge.
\newblock Pseudodifferential calculus on noncommutative tori, II. Main properties.
\newblock {\it Int. J. Math.}, 30 (8) (2019) 1950034.

\bibitem{HR00} N.~Higson and J.~Roe.
\newblock Analytic K-homology.
\newblock {\it Oxford Mathematical Monographs.}, (2000) pp. xviii+405.

\bibitem{Hyt16} T.~Hytnen, J.~Neerven, M.~Veraar, and L.~Weis.
\newblock Analysis in Banach spaces, I, Martingales and Littlewood-Paley theory.
\newblock Springer, 2016.

\bibitem{Hyt23} T.~Hytnen, J.~Neerven, M.~Veraar, and L.~Weis.
\newblock Analysis in Banach spaces, III, Harmonic Analysis and Spectral Theory.
\newblock Springer, 2023.

\bibitem{Ike31} S.~Ikehara.
\newblock An extension of Landau's theorem in the analytic theory of numbers.
\newblock {\it J. Math. Phys. Mass. Inst. Technol.}, 10 (1931) 1-12.

\bibitem{L.H65} L.~H\"{o}rmander.
\newblock Pseudo-Diff. Oper., 18 (3) (1965) 501-517.



\bibitem{JW82} S. Janson and T.H. Wolff.
\newblock Schatten classes and commutators of singular integral operators.
\newblock {\it Ark. Mat.}, 20 (1982), 301--310.



\bibitem{KLPS13} N.~Kalton, S.~Lord, D.~Potapov, and F.~Sukochev.
\newblock Traces of compact operators and the noncommutative residue.
\newblock {\it Adv. Math.}, 235 (2013) 1-55.


\bibitem{KN65} J.J.~Kohn and L.~Nirenberg.
\newblock An algebra of pseudo-differential operators.
\newblock {\it Commun. Pure Appl. Math.}, 18 (1965) 269-305.

\bibitem{LP20} G.~Lee and R.~Ponge.
\newblock Functional calculus for elliptic operators on noncommutative tori, I.
\newblock {\it J. PseudoDiffer. Oper. Appl.}, 11 (2020) 935-1004.


\bibitem{LL25} G.~Lee and M.~Lesch.
\newblock Weakly parametric pseudodifferential calculus for twisted $C^*$-dynamical systems.
\newblock {\it J. Noncommut. Geom.}, 19 (2025) no.~3, 847-877


\bibitem{Leo17} G.~Leoni.
\newblock A First Course in Sobolev Spaces, second edition.
\newblock {\it Graduate Studies in Mathematics}, vol. 181, AMS, Providence, RI, 2017.

\bibitem{LM16} M.~Lesch and H.~Moscovici.
\newblock Modular curvature and Morita equivalence.
\newblock {\it Geom. Funct. Anal.}, 26 (3) (2016) 818-873.

\bibitem{LSZ20} G.~Levitina, F.~Sukochev, and D.~Zanin.
\newblock Cwikel estimates revisited.
\newblock {\it Proc. Lond. Math. Soc.}, 120 (3) (2020) no. 2, 265-304.


\bibitem{LevitinaUsachevI} G. Levitina and A.~S. Usachev
\newblock Pietsch correspondence for symmetric functionals on Calkin operator spaces associated with semifinite von Neumann algebras.
\newblock {\it J. Operator Theory} {\bf 89} (2023), no.~2, 305--342.

\bibitem{LevitinaUsachevII} G. Levitina and A.~S. Usachev,
\newblock Dixmier-type traces on symmetric spaces associated with semifinite von Neumann algebras.
\newblock {\it Ann. Funct. Anal.} {\bf 15} (2024), no.~2, Paper No. 45, 14 pp.

\bibitem{LevitinaUsachevIII} G. Levitina and A.~S. Usachev
\newblock Symmetric functionals on simply generated symmetric spaces.
\newblock {\it J. Math. Anal. Appl.} {\bf 546} (2025), no.~1, Paper No. 129184, 30 pp.


\bibitem{LMSZ17} S.~Lord, E.~McDonald, F.~Sukochev, and D.~Zanin.
\newblock Quantum differentiability of essentially bounded functions on Euclidean space.
\newblock {\it J. Funct. Anal.}, 273 (2017) 2353-2387.

\bibitem{LSZ12} S.~Lord, F.~Sukochev, and D.~Zanin.
\newblock Singular Traces: Theory and Applications.
\newblock Walter de Gruyter, vol. 46, 2012.

\bibitem{LSZ19} S.~Lord, F.~Sukochev, and D.~Zanin.
\newblock Advances in Dixmier Traces and Applications.
\newblock In: Advances in Noncommutative Geometry, Springer, 2019, pp. 491-583.

\bibitem{LSZ21} S.~Lord, F.~Sukochev, and D.~Zanin.
\newblock Singular Traces, vol 1. Theory.
\newblock {\it Walter de Gruyter}, vol. 46/1, 2021.

\bibitem{MP22} E.~McDonald and R.~Ponge.
\newblock Cwikel estimates and negative eigenvalues of Schr\"odinger operators on noncommutative tori.
\newblock {\it J. Math. Phys.}, 61 (2022) 043503.

\bibitem{MP23} E.~McDonald and R.~Ponge.
\newblock Dixmier trace formulas and negative eigenvalues of Schr\"odinger operators on curved noncommutative tori.
\newblock {\it Adv. Math.}, 412 (2023) 108815.

\bibitem{MSX19} E.~McDonald, F.~Sukochev, and X.~Xiong.
\newblock Quantum differentiability on quantum tori.
\newblock {\it Commun. Math. Phys.}, 371 (2019) 1231-1260.

\bibitem{MSX20} E.~McDonald, F.~Sukochev, and X.~Xiong.
\newblock Quantum differentiability on noncommutative Euclidean spaces.
\newblock {\it Commun. Math. Phys.}, 379 (2020) 491-542.

\bibitem{MSX23} E.~McDonald, F.~Sukochev, and X.~Xiong.
\newblock Quantum differentiability-the analytical perspective.
\newblock In: {\it Proc. Sympos. Pure Math.}, 105, AMS, Providence, RI (2023) 257-280.

\bibitem{MSZ19} E.~McDonald, F.~Sukochev, and D.~Zanin.
\newblock A $C^{\ast}$-algebraic approach to the principal symbol II.
\newblock {\it Math. Ann.}, 374 (2019) no. 1-2, 273-322.

\bibitem{MSZ22} E.~McDonald, F.~Sukochev, and D.~Zanin.
\newblock A noncommutative Tauberian theorem and Weyl asymptotics in noncommutative geometry.
\newblock {\it Lett. Math. Phys.}, 122 (2022) 77, 45 pp.

\bibitem{NR11} S. Neuwirth and \'{E}. Ricard.
\newblock Transfer of Fourier multipliers into Schur multipliers and sumsets in a discrete group.
\newblock {\it Canad. J. Math.} 63 (2011), 1161-1187.

\bibitem{PSW02} B.~de Pagter, F.~Sukochev, and H.~Witvliet.
\newblock Double operator integrals.
\newblock {\it J. Funct. Anal.}, 192 (2002) 52-111.

\bibitem{PasturFigotin} L.~Pastur and A.~Figotin.
\newblock Spectra of Random and Almost-Periodic Operators
\newblock {\it Grundlehren der mathematischen Wissenschaften} 297 (1992).

\bibitem{Petersen1983} K.~Petersen
\newblock Ergodic Theory.
\newblock {\it Cambridge studies in Advanced Mathematics} 2 (1983).

\bibitem{Pisier95} G. Pisier.
\newblock Similarity problems and completely bounded maps.
\newblock {\it Springer Lecture Notes}. 1618 (1995).

\bibitem{Pisier98} G. Pisier.
\newblock Noncommutative vector-valued $L_p$ spaces and completely $p$-summing maps.
\newblock {\it Ast\'erisque}. 247 (1998), vi+131 pp.

\bibitem{PX03} G.~Pisier and Q.~Xu.
\newblock Noncommutative $L_p$-spaces.
\newblock In: Handbook of the Geometry of Banach Spaces, vol. 2, North-Holland, Amsterdam, 2003, pp. 1459-1517.


\bibitem{P08} R.~Ponge.
\newblock Heisenberg calculus and spectral theory of hypoelliptic operators on Heisenberg manifolds.
\newblock {\it Mem. Amer. Math. Soc.}, 194 (906) viii+134, 2008.

\bibitem{P20i} R.~Ponge.
\newblock Connes's trace theorem for curved noncommutative tori. Application to scalar curvature.
\newblock {\it J. Math. Phys.}, 61 (2020) 042301, 27 pp.

\bibitem{P20ii} R.~Ponge.
\newblock Noncommutative residue and canonical trace on noncommutative tori. Uniqueness results.
\newblock {\it SIGMA Symmetry Integrability Geom. Methods Appl.}, 16 (2020) 061, 31 pp.

\bibitem{P22} R.~Ponge.
\newblock Weyl's laws and Connes' integration formulas for matrix-valued $L \log L$-Orlicz potentials.
\newblock {\it Math. Phys. Anal. Geom.}, 25 (10) (2022), 33 pp.

\bibitem{P23} R.~Ponge.
\newblock Connes integration and Weyls laws.
\newblock {\it J. Noncommut. Geom.}, 17 (2023) no.2, 719-767.

\bibitem{PS09} D.~Potapov and F.~Sukochev.
\newblock Unbounded Fredholm modules and double operator integrals.
\newblock {\it J. Reine Angew. Math.}, 626 (2009) 159-185.

\bibitem{R22} G.~Rozenblum.
\newblock Eigenvalues of singular measures and Connes noncommutative integration.
\newblock {\it J. Spectr. Theory}, 12 (2022) 259-300.

\bibitem{RT10} M.~Ruzhansky and V.~Turunen.
\newblock Pseudo-Differential Operators and Symmetries. Background Analysis and Advanced Topics.
\newblock {\it Pseudo-Differential Operators. Theory and Applications}, vol. 2, Birkh\"{a}user Verlag, Basel, 2010.

\bibitem{Seeley67} R.T.~Seeley.
\newblock Complex powers of an elliptic operator.
\newblock In: {\it Proc. Sympos. Pure Math.}, Vol. X, AMS, Providence, RI, 1967, pp. 288-307.

\bibitem{SSUZ15} E. Semenov, F. Sukochev, A. Usachev and D. Zanin.
\newblock Banach limits and traces on $\mathcal{L}_{1, \infty}$.
\newblock {\it Adv. Math.}, 285 (2015) 568-628.


\bibitem{Shubin1976} M. A.~Shubin.
\newblock Almost periodic pseudodifferential operators, and von Neumann algebras.
\newblock {\it Trudy Moskov. Mat. Ob\v s\v c.} {\bf 35} (1976), 103--164.

\bibitem{Shubin1979} M.A.~Shubin.
\newblock The spectral theory and the index of elliptic operators with almost periodic coefficients
\newblock {\it Russ. Math. Surv.} 34:2 (1979)

\bibitem{Shu01} M.A.~Shubin.
\newblock Pseudodifferential operators and spectral theory, Second Edition.
\newblock Springer-Verlag, Berlin, 2001.

\bibitem{Sim15} B.~Simon.
\newblock Operator theory. A Comprehensive Course in Analysis, Part 4.
\newblock AMS, Providence, RI, 2015.

\bibitem{Si05} B.~Simon.
\newblock Trace Ideals and Their Applications, 2nd edn.
\newblock AMS, Providence, 2005.


\bibitem{St70} E.M.~Stein.
\newblock Singular integrals and differentiability properties of functions.
\newblock Princeton Univ. Press, Princeton, 1970.

\bibitem{St93} E.M.~Stein.
\newblock Harmonic analysis: real-variable methods, orthogonality, and oscillatory integrals.
\newblock In: Monographs in Harmonic Analysis, III, Princeton Mathematical Series, vol. 43, Princeton Univ. Press, 1993.


\bibitem{FZ18} F.~Sukochev and D.~Zanin.
\newblock A $C^*$-algebraic approach to the principal symbol I.
\newblock {\it J. Oper. Theory.}, 80 (2018) 481-522.

\bibitem{FZ21} F.~Sukochev and D.~Zanin.
\newblock Optimal constants in non-commutative H\"{o}lder inequality for quasinorms.
\newblock {\it Proc. Amer. Math. Soc.}, 149 (2021) no. 9, 3813-381

\bibitem{FZ23} F.~Sukochev and D.~Zanin.
\newblock Connes integration formula - a constructive approach.
\newblock {\it Funktsional. Anal. i Prilozhen.}, 57:1 (2023) 52-76.

\bibitem{SXZ23} F.~Sukochev, X.~Xiong, and D.~Zanin.
\newblock Asymptotics of singular values for quantised derivatives on noncommutative tori.
\newblock {\it J. Funct. Anal.}, 285 (2023) 110021.

\bibitem{Takesaki} M.~Takesaki.
\newblock {\it Theory of operator algebras. II},
\newblock Encyclopaedia of Mathematical Sciences Operator Algebras and Non-commutative Geometry, 125 6, Springer, Berlin, 2003.

\bibitem{Tao18} J.~Tao.
\newblock The theory of pseudo-differential operators on the noncommutative $n$-torus.
\newblock {\it Inst. Phys. Conf. Ser.}, 965 (2018) 1-12.

\bibitem{Ta81} M.E.~Taylor.
\newblock Pseudodifferential Operator.
\newblock Princeton Univ. Press, 1981.

\bibitem{Tian25} Y.~Tian.
\newblock Spectral asymptotics for quantized derivatives on quantum Euclidean spaces.
\newblock {\it J. Noncommut. Geom.}, (2025) DOI: 10.4171/JNCG/644.

\bibitem{Tre67} F. Tr\`eves.
\newblock  Topological vector spaces, distributions and kernels.
\newblock  Academic Press, New York-London, 1967.

\bibitem{VG88} J.C.~V\'{a}rilly and J.M.~Gracia-Bond\'{i}a.
\newblock Algebras of distributions suitable for phase-space quantum mechanics. II. Topologies on the Moyal algebra.
\newblock {\it J. Math. Phys.}, 29 (1988) no. 4, 880-887.

\bibitem{W11} H.~Weyl.
\newblock Ueber die asymptotische verteilung der eigenwerte.
\newblock {\it Nachrichten von der Gesellschaft der Wissenschaften zu G\"{o}ttingen, Mathematisch-Physikalische Klasse}, 1911:110-117, 1911.

\bibitem{Wie32} N.~Wiener.
\newblock Tauberian Theorems.
\newblock {\it Ann. Math.}, 33 (1) (1932) 1-100.

\bibitem{XX19} R.~Xia and X.~Xiong.
\newblock Mapping properties of operator-valued pseudo-differential operators.
\newblock {\it J. Funct. Anal.}, 277 (2019) 2918-2980.


\bibitem{Xu07} Q.~Xu.
\newblock Noncommutative $L_p$-spaces and martingale inequalities, unpublished manuscript, 2007.

\bibitem{Zeng23} K.~Zeng.
\newblock Some problems in harmonic analysis on twsited crossed products.
\newblock Ph.D. Thesis, Universit\'{e} Bourgogne Franche-Comt\'{e}, 2023. English. NNT: 2023UBFCD048. tel04414142.

\bibitem{JYE25} J.~Zweck, Y.~Latushkin, and E.~Gallo.
\newblock A regularity condition under which integral operators with operator-valued kernels are trace class.
\newblock {\it Bol. Soc. Mat. Mex.}, 31, 38 (2025).
\end{thebibliography}
\end{document}